\documentclass[12pt,notitlepage]{report}
\usepackage{amsmath}
\usepackage{amssymb}
\usepackage{amsthm}
\usepackage{geometry}
\usepackage[backend=biber,style=alphabetic,sorting=nyt]{biblatex}
\bibliography{citations}
\theoremstyle{plain}
\newtheorem{thm}{Theorem}[section] 
\newtheorem{prop}[thm]{Proposition}
\newtheorem{lemma}[thm]{Lemma} 
\newtheorem{cor}[thm]{Corollary}

\theoremstyle{definition}
\newtheorem{defn}[thm]{Definition} 
\newtheorem{eg}{Example}

\newtheorem{claim}[thm]{Claim}
\newtheorem{conj}[thm]{Conjecture}
\newtheorem{notations}[thm]{Notations}
\newtheorem{conventions}[thm]{Conventions}
\theoremstyle{remark} 
\newtheorem{rmk}[thm]{Remark}

\newtheorem{mycom}[thm]{Comment}

\usepackage{soul} 
\usepackage{quiver} 
\usepackage[hidelinks]{hyperref} 
\usepackage{physics}
\usepackage{bbm} 
\usepackage{import}
\usepackage{pdfpages}
\usepackage{transparent}
\usepackage{xcolor}

\usepackage{graphicx} 
\graphicspath{ {./figures/} }

\begin{document}
\title{Topological Field Theories and the Algebraic Structures of the Two-Sphere}
\author{Chris Li}
\maketitle
\abstract{We give two presentations for bordisms of $S^2$ in the 3-dimensional oriented bordism category $\operatorname{Cob}(3) $, encoding the algebraic structures on $S^2$. After passing through topological field theories, we define two kinds of monoids which we call P-monoids and L-monoids. In addition to both being commutative Frobenius monoids, P-monoids are equipped with a class of endomorphisms while L-monoids are equipped with a class of unit morphisms, all of which are labelled by closed oriented irreducible prime 3-manifolds. They turn out to be equivalent. The new prime structures satisfy some countable relations with the commutative Frobenius structure, the most notable of which we call ``legs relations.'' We then restrict to the setting of algebras and show that the legs relations place strong constraints on the new prime endomorphisms which forces them to act by multiplications by prime units, rendering the additional prime structures remarkably simple. We also propose an $\infty$-operad which encodes these prime structures and contains the $\infty$-little 3-cube operad as a sub-operad.
\newpage
\tableofcontents

\chapter{Introduction}
\subsubsection{Motivations}
One of the greatest scientific achievements in the 20th century is the development of quantum field theories, which model the interaction of elementary particles with great precision through \textit{quantum fields}. Quantum fields are , roughly speaking, classes of functions $\phi : M \to \mathbb{C}$ where $M$ is a manifold representing the underlying space-time. Unfortunately, developing a mathematically rigorous definition of quantum field theories has proven extremely challenging, leaving one of the most foundational subjects of modern physics on unstable soil. 

\textit{Topological field theories} are toy models of quantum field theories. In some sense, they are the most constrained version of quantum field theories. Generally speaking, the more constrained an object is, the easier it is to study-- for example, it is much easier to study finite dimensional inner product spaces over $\mathbb{R}$ than to study modules over a not necessarily commutative ring $R$. However, just because it is easier doesn't mean it is useless. A student likely learns about bases for a vector spaces first, and then carries that rather helpful intuition to the study of modules. So the hope is that by understanding topological field theories, the highly constrained version of quantum field theories, one may gain a glimpse at the full mathematical complexity of quantum field theory.

It turns out that topological field theories can be rigorously formulated in mathematics, and the structures one discovers are already extremely interesting, intricate, and ingenious. This paper explores one aspect of the structures of topological field theories of dimension $\geq 3$.

Mathematically, the vanilla definition of a topological field theory is
\begin{defn} \label{defn:TFT}
An $n$-dimensional topological field theory is a symmetric monoidal functor: \[
Z: \operatorname{Cob}( n) \to \operatorname{Vect}_k
.\]
where $\operatorname{Cob}( n)$ is the category of closed oriented $n$-manifolds and cobordisms modulo orientation and boundary preserving diffeomorphisms, $\operatorname{Vect}_k$ is the category of (insert your favorite adjectives here such as topological) $k$-vector spaces.
\end{defn}
Historically, this definition was given by \cite{atiyah1988topological}, possibly inspired by \cite{segal1988definition}'s attempt to define a conformal field theory and \cite{witten1982supersymmetry} on the relation between Morse theory and supersymmetry.

Unpacking the definition, a topological field theory is a nice assignment $Z$ which assigns to a closed oriented  $(n-1)$-manifold $X$ a vector space $Z(X)$, and to a bordism $M$ from $X$ to $Y$ a linear map $Z(M):Z(X) \to Z(Y)$. In particular, $Z(\emptyset)=k$ where $\emptyset$ is the empty $(n-1)$-dimensional closed oriented manifold. Therefore $Z$ sends a closed oriented $n$ manifold $M$, which is thought of as a cobordism between the empty set and itself, to a linear map $Z(M): k\to k$, which can be identified with its image of $1$. So $Z(M)$ is just a number.

This definition closely aligns with physics: $Z(X)$ is supposed to be the (Hilbert) space of states $\mathcal{H}_X$ on $X$, and the linear map $Z(M): Z(X) \to Z(X)$ is supposed to be time evolution on $\mathcal{H}_X$. In the case $X$ is empty, the number $Z(M)$ is supposed to be the partition function $Z$ of the theory on spacetimes $M$. Some excellent expositions by \cite{Carqueville_Runkel_2018} and \cite{Freed_2009} goes into the physics side in greater detail.

The vanilla definition has since then been enhanced. The modern mathematical perspective is
\begin{defn} \label{defn:fullyextendedTFT}
A fully-extended $n$-dimensional topological field theory is a symmetric monoidal functor of $(\infty,n)$-categories \[
Z: \operatorname{Bord}_n \to \mathcal{C}
.\]
\end{defn}
The $(\infty,n)$ category $\operatorname{Bord}_n$ has objects $0$-dimensional manifolds, $1$-morphisms are cobordisms between $0$-dimensional manifolds, $2$-morphisms are cobordisms between cobordisms, etc. The $n$-morphisms can be interpreted as an $n$-dimensional manifold with corners. The $(n+1)$-morphisms are diffeomorphisms between two $n$-manifold with corners, and the $(n+2)$-morphisms are isotopies between diffeomorphisms, etc. This construction was proposed by \cite{lurie_classification_2008} and later corrected by \cite{Calaque_2019} and appeared as the thesis of \cite{scheimbauer2014factorization}.

Notice two features about this modern definition: one does not mod out the diffeomorphisms in the cobordisms and instead store these information in $(n+1)$-morphisms and higher; the vanilla 1-categorical definition sits inside the $(n-1)$-morphisms and the $n$-morphisms, and can be recovered once one passes to the homotopy category.

Physically, an extended topological field theory allows one to encode more information than just the space of states and the partition function. For example, a $2$-dimensional fully extended topological field theory also associates to a zero-dimensional manifold (aka a point) some additional data, usually a category. For the topological $B$-model on a complex manifold $Y$, this is the category $D^b \operatorname{Coh}(Y)$. For the topological $A$ model with symplectic K\"ahler target $X$, this is the category $\operatorname{Fuk}(X)$. Homological mirror symmetry is equivalent to the equivalence between the topological $A$ model as a $2$-dimensional fully extended TFT $Z_A$ and the topological $B$ model $Z_B$. For physical interpretations of these higher structures in general, see \cite{kapustin2010topological}, \cite{Carqueville_Runkel_2018} and \cite{Freed_2009}.


Topological field theories have also proven to be more than just highly constrained toy models, especially in the field of mathematics. Much of the attention came from the Geometric Langlands correspondence, which can be reinterpreted as a conjectural equivalence of two topological field theories by \cite{kapustin2007electric, witten2010geometric}. There is a conjectured equivalence for relative Langlands as well \cite{ben2024relative}. Therefore much of the modern formal interests in topological field theories have come from the mathematical side. 

Some of the most interesting results and insights have come from attempts to classify topological field theories. Historically the most famous theorem is perhaps the equivalence between $2$-dimensional topological field theories (up to equivalence) and commutative Frobenius algebras (up to isomorphism):
\begin{thm} \label{thm:2dtftCFA}
The 2-dimensional oriented bordism category $\operatorname{Cob}( 2) $  is the free symmetric monoidal category with a single commutative Frobenius algebra object. In particular the category of oriented 2-dimensional topological field theories is equivalent to the category of commutative Frobenius algebras.
\end{thm}

This theorem was widely known and formal proofs appeared in numerous papers such as \cite{dijkgraaf1989geometrical}, \cite{abrams1996two}, and J. \cite{kock_frobenius_2003}. The category $\operatorname{Cob}( 2)$ has objects closed oriented $1$-dimensional manifolds (namely, a disjoint union of oriented circles), and morphisms are cobordisms between the disjoint unions of circles modulo diffeomorphisms. The single commutative Frobenius algebra object refers to the circle $S^1$. By a commutative Frobenius algebra we mean a $k$-vector space $A$ together with the following data
\begin{align*}
	\text{Multiplication } m&: A \otimes A \to A\\
	\text{Unit } 1&: k\to A\\
	\text{Co-multiplication } m^\vee &: A \to A \otimes A\\
	\text{Trace/co-unit } tr&: A \to k
\end{align*}
satisfying:
\begin{enumerate}
\item The multiplication $m$ is associative and commutative.
\item The unit $1$ is the unit map with respect to the multiplication $m$.
\item The co-multiplication $m^\vee$ is co-associative and co-commutative.
\item The co-unit $tr$ is the co-unit map with respect to the co-multiplication $m^\vee$.
\item The Frobenius condition: $(m \otimes id) \circ (id \otimes m^\vee) = m^\vee \circ m$.
\end{enumerate}
In $\operatorname{Cob}( 2) $ , these maps are given by the morphisms in figure \ref{fig:cob2cf}. 

\begin{figure}[ht]
    \centering
\def\svgwidth{1\columnwidth} 
\begingroup%
  \makeatletter%
  \providecommand\color[2][]{%
    \errmessage{(Inkscape) Color is used for the text in Inkscape, but the package 'color.sty' is not loaded}%
    \renewcommand\color[2][]{}%
  }%
  \providecommand\transparent[1]{%
    \errmessage{(Inkscape) Transparency is used (non-zero) for the text in Inkscape, but the package 'transparent.sty' is not loaded}%
    \renewcommand\transparent[1]{}%
  }%
  \providecommand\rotatebox[2]{#2}%
  \newcommand*\fsize{\dimexpr\f@size pt\relax}%
  \newcommand*\lineheight[1]{\fontsize{\fsize}{#1\fsize}\selectfont}%
  \ifx\svgwidth\undefined%
    \setlength{\unitlength}{680.31496063bp}%
    \ifx\svgscale\undefined%
      \relax%
    \else%
      \setlength{\unitlength}{\unitlength * \real{\svgscale}}%
    \fi%
  \else%
    \setlength{\unitlength}{\svgwidth}%
  \fi%
  \global\let\svgwidth\undefined%
  \global\let\svgscale\undefined%
  \makeatother%
  \begin{picture}(1,0.33333333)%
    \lineheight{1}%
    \setlength\tabcolsep{0pt}%
    \put(0,0){\includegraphics[width=\unitlength,page=1]{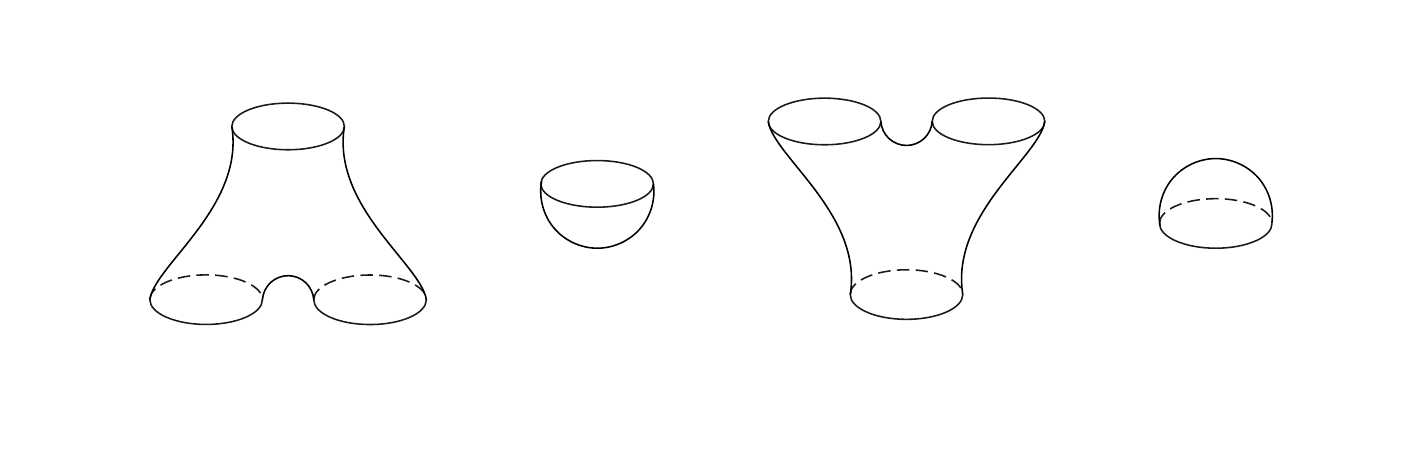}}%
    \put(0.19465906,0.05483659){\makebox(0,0)[lt]{\lineheight{1.25}\smash{\begin{tabular}[t]{l}$m$\end{tabular}}}}%
    \put(0.4160841,0.06727912){\makebox(0,0)[lt]{\lineheight{1.25}\smash{\begin{tabular}[t]{l}$1$\end{tabular}}}}%
    \put(0.62335847,0.06442903){\makebox(0,0)[lt]{\lineheight{1.25}\smash{\begin{tabular}[t]{l}$m^\vee$\end{tabular}}}}%
    \put(0.85015456,0.06442903){\makebox(0,0)[lt]{\lineheight{1.25}\smash{\begin{tabular}[t]{l}$tr$\end{tabular}}}}%
  \end{picture}%
\endgroup%

    \caption[Commutative Frobenius generators of $\operatorname{Cob}(2) $]{The commutative Frobenius generators of $\operatorname{Cob}( 2) $.}
    \label{fig:cob2cf}
\end{figure}

Physically, this means that the space of states associated to a circle of any 2-dimensional topological field theory must be a commutative Frobenius algebra. Unsurprisingly, being a commutative Frobenius algebra is a strong condition which forces many desirable properties, the most notable of which is finite dimensionality.

Other classification theorems exist in dimensions $>2$ as well, putting TFTs in very rigid mathematical structures. In dimension $3$, \cite{bartlett2014extended, bartlett2015modular} showed that there is an equivalence between $(1+1+1)$-dimensional TFT and modular tensor categories. \cite{juhasz2018defining} gave generator and relations for the cobordism category $\operatorname{Cob}( n)$ for any $n$, in particular for $n=3$ it is shown that there is an equivalence between the category of $3$-dimensional oriented TFTs and the category of \textit{J-algebras}. \cite{lurie_classification_2008} sketched a proof of the \textit{cobordism hypothesis} of \cite{Baez_Dolan_1995} that an $n$-dimensional fully extended framed TFT is determined by its value on the framed point. 

However none of these classification theorems provides the answer to the most naive question one might ask upon learning theorem \ref{thm:2dtftCFA}:
\begin{center}
\begin{minipage}{0.7\textwidth}
\centering
\textit{If the circle is a commutative Frobenius object in} $\operatorname{Cob}( 2) $, \textit{what kind of algebraic structure does the sphere have in }$\operatorname{Cob}( 3) $ \textit{?} 
\end{minipage}
\end{center}


In this thesis, we provide the answer 
\begin{center}
\begin{minipage}{0.7\textwidth}
\centering
\textit{The sphere has the commutative Frobenius structure as well as new ``prime'' structures.}
\end{minipage}
\end{center}

By prime structures here we mean structures labelled by diffeomorphism classes of oriented irreducible prime 3-manifolds. Physically, this means that the space of states for a $3$-dimensional topological field theory is not only a commutative Frobenius algebra but it also has additional structures labelled by these prime 3-manifolds.

\subsubsection{Overview of Contents}
The result of this thesis can be summarized as the \textbf{slogan}: in $\operatorname{Cob}( 3)$, the category of oriented closed $2$-manifolds and their bordisms, the object $S^2$ has the usual structures of a commutative Frobenius object, as well as \textbf{additional new maps} $p^{\times \times} : S^2 \to S^2$ and $p^\times : \emptyset \to S^2$, where $p$ is a connected closed irreducible oriented $3$-manifold. Besides the commutative Frobenius relations, the only other necessary relations are the ``legs relations'' $m\circ (p^{\times \times} \otimes id_{S^2} )= m\circ (id_{S^2} \otimes p^{\times \times} )$. In particular, there are \textbf{no more structures} on $S^2$-- every other algebraic structure is generated by commutative Frobenius structures and  $\left\{ p^{\times \times}  \right\} $ which satisfy the legs relations, or commutative Frobenius structures and $\left\{ p^\times \right\} $. 

We shall explain what this means in the remainder of the Introduction.

The classification of $2$-dimensional topological field theories $Z: \operatorname{Cob}( 2) \to \operatorname{Vect}_k$ requires that one show $S^1$ has the structures of commutative Frobenius objects in $\operatorname{Cob}( 2) $ . This means that one needs to show that the maps $(m,1,m^\vee, tr)$ exist, and that they satisfy the usual commutative Frobenius conditions. This part is easy. The difficult part is showing that these maps and the commutative Frobenius conditions are \textit{sufficient and necessary}. In the language of category theory, one needs to construct a \textit{presentation} of the bordism category $\operatorname{Cob}( 2)$. By a \textit{presentation} we mean generators and relations. The generators being $(m,1,m^\vee, tr)$, and the relations being the commutative Frobenius relations. One inevitably has to dig into the Morse theory of $2$-dimensional manifolds to show, among other things, that two different ways of gluing the generators into the same cobordism $[M]$ must be related by only commutative Frobenius relations. This is most naturally proven using Cerf theory: the two different gluings correspond to Morse functions $f_0 , f_1:M \to \mathbb{R}$, and \cite{cerf_stratification_1970} showed that there is a path of functions $f_t$ from $f_0 $ to $f_1 $ such that along this path the function $f_t$ only fails to be ``excellent Morse'' at finitely many times. In particular, before and after those singularities, the Morse functions induce a commutative Frobenius type of relations.

We will essentially take a similar approach to understand the algebraic structures of $S^2$. In order to do so, we need to consider all possible bordisms between disjoint unions of  $S^2$. These span a subcategory $\operatorname{Cob}( 3)_{S^2}$ of $\operatorname{Cob}( 3) $, defined as
\begin{defn} 
Consider the symmetric monoidal category $\operatorname{Cob}(3) $. It has a subcategory whose objects are disjoint unions of $S^2$. Call this category $\operatorname{Cob}(3)_{S^2}$ and note that it inherits the symmetric monoidal structure.
\end{defn}
In particular, the morphisms in $\operatorname{Cob}( 3)_{S^2} $ all have spherical boundaries. For example, here's a list of morphisms that will play a central role in defining the presentations of $\operatorname{Cob}( 3)_{S^2} $:
\begin{align*}
id_{S^2} &= \left[ S^2 \times I \right]  \\
m &= \left[ S^3 \setminus \operatorname{int}(\sqcup_{}^3 D^3) \right]  : S^2 \sqcup S^2 \to S^2 \\
1 &= \left[ D^3  \right] : \emptyset \to S^2\\
m^\vee &= \left[ S^3 \setminus \operatorname{int}(\sqcup_{}^3 D^3)  \right] : S^2  \to S^2\sqcup S^2 \\
\tr &= \left[ D^3 \right]  :S^2 \to \emptyset\\
p^{\times \times} &= [p \setminus \operatorname{int}(D^3 \sqcup D^3)]: S^2 \to S^2\\
p^\times &= [p \setminus \operatorname{int}(D^3)] : \emptyset \to S^2
\end{align*}
where $p$ is any diffeomorphism class of connected closed oriented irreducible 3-manifold. Note that the only $3$-manifold which is prime but reducible is $S^2\times S^1$. So,equivalently, $p$ is any diffeomorphism class of connected closed oriented prime $3$-manifold which is not $S^2\times S^1$. See figure \ref{fig:generators} for a sketch of these morphisms. 

\begin{figure}[ht]
    \centering
\def\svgwidth{1\columnwidth} 
\begingroup%
  \makeatletter%
  \providecommand\color[2][]{%
    \errmessage{(Inkscape) Color is used for the text in Inkscape, but the package 'color.sty' is not loaded}%
    \renewcommand\color[2][]{}%
  }%
  \providecommand\transparent[1]{%
    \errmessage{(Inkscape) Transparency is used (non-zero) for the text in Inkscape, but the package 'transparent.sty' is not loaded}%
    \renewcommand\transparent[1]{}%
  }%
  \providecommand\rotatebox[2]{#2}%
  \newcommand*\fsize{\dimexpr\f@size pt\relax}%
  \newcommand*\lineheight[1]{\fontsize{\fsize}{#1\fsize}\selectfont}%
  \ifx\svgwidth\undefined%
    \setlength{\unitlength}{680.31496063bp}%
    \ifx\svgscale\undefined%
      \relax%
    \else%
      \setlength{\unitlength}{\unitlength * \real{\svgscale}}%
    \fi%
  \else%
    \setlength{\unitlength}{\svgwidth}%
  \fi%
  \global\let\svgwidth\undefined%
  \global\let\svgscale\undefined%
  \makeatother%
  \begin{picture}(1,0.16666667)%
    \lineheight{1}%
    \setlength\tabcolsep{0pt}%
    \put(0,0){\includegraphics[width=\unitlength,page=1]{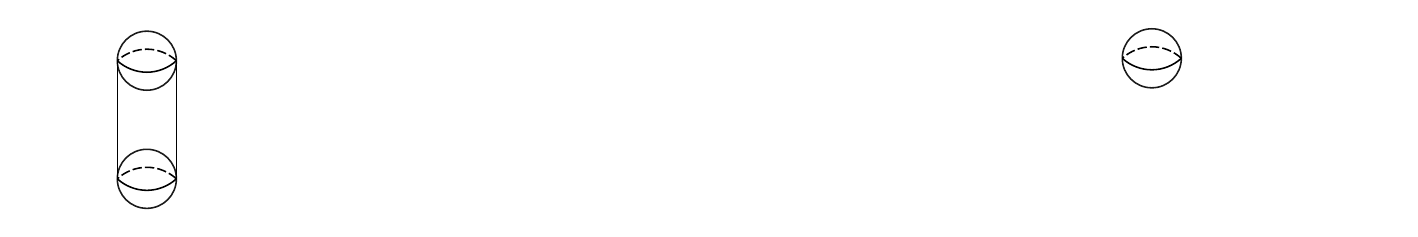}}%
    \put(0.79252224,0.07641748){\makebox(0,0)[lt]{\lineheight{1.25}\smash{\begin{tabular}[t]{l}$p^{\times \times}$\end{tabular}}}}%
    \put(0,0){\includegraphics[width=\unitlength,page=2]{generators.pdf}}%
    \put(0.93194472,0.07641748){\makebox(0,0)[lt]{\lineheight{1.25}\smash{\begin{tabular}[t]{l}$p^{\times}$\end{tabular}}}}%
    \put(0,0){\includegraphics[width=\unitlength,page=3]{generators.pdf}}%
  \end{picture}%
\endgroup%

    \caption[A list of relevant generating morphisms]{A list of generating morphisms. From left to right: $id_{S^2}$, multiplication $m$, unit $1$, comultiplication $m^\vee$, trace $tr$, prime endomorphism $p^{\times \times} $, prime unit $p^\times$, one for each connected oriented irreducible prime $3$-manifold $p$. All but the last kind make up the generating set $G_2 $. All but the second to last kind make up the generating set $G_1 $.}
    \label{fig:generators}
\end{figure}

Our aim is to construct presentations of $\operatorname{Cob}( 3)_{S^2} $. It turns out that $\operatorname{Cob}( 3)_{S^2}$ has two sets of generators:
\begin{align*}
G_1& = C \cup F \cup L \cup \text{orientation reversal}\\
G_2 &= C \cup F \cup P \cup \text{orientation reversal}
\end{align*}
where $C = \left\{ 1,m \right\} , F=\left\{ tr, m^\vee \right\} , P = \left\{ p^{\times \times}  \right\}_{p \in \mathcal{I}} , L=\left\{ p^\times \right\}_{p \in \mathcal{I}} $. Here $\mathcal{I}$ denotes the set $ \left\{ \text{connected closed irreducible 3-manifold} \right\} / \text{orientation-preserving diffeomorphisms} $ .

The presentation using $G_1 $ generators is usually associated with the letter ``L'' because of the distinct L-shaped prime factor attachments in figure \ref{fig:g1generate}. The presentation using $G_2 $ generators is usually associated with the letter ``P'' for ``prime'' due to a lack of imagination.

The $G_2 $ presentation is given by the usual commutative Frobenius relations, as well as the ``legs relations'' suggestively sketched in figure \ref{fig:plegs}  \[
m \circ (p^{\times \times}  \otimes id_{S^2}) = m\circ (id_{S^2} \otimes p^{\times \times} )
.\]

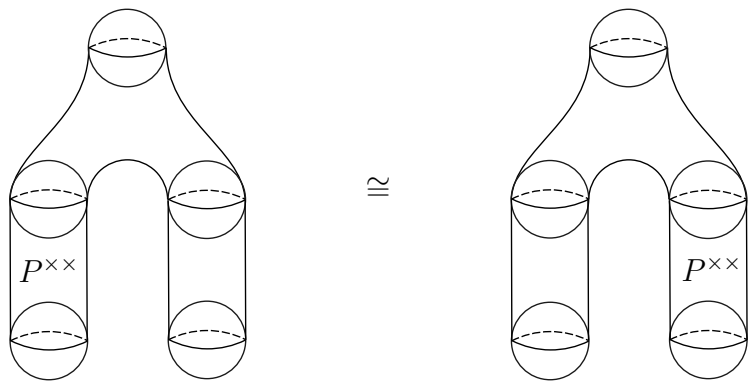
\begin{figure}[ht]
    \centering
\def\svgwidth{1\columnwidth} 
\begingroup%
  \makeatletter%
  \providecommand\color[2][]{%
    \errmessage{(Inkscape) Color is used for the text in Inkscape, but the package 'color.sty' is not loaded}%
    \renewcommand\color[2][]{}%
  }%
  \providecommand\transparent[1]{%
    \errmessage{(Inkscape) Transparency is used (non-zero) for the text in Inkscape, but the package 'transparent.sty' is not loaded}%
    \renewcommand\transparent[1]{}%
  }%
  \providecommand\rotatebox[2]{#2}%
  \newcommand*\fsize{\dimexpr\f@size pt\relax}%
  \newcommand*\lineheight[1]{\fontsize{\fsize}{#1\fsize}\selectfont}%
  \ifx\svgwidth\undefined%
    \setlength{\unitlength}{680.31496063bp}%
    \ifx\svgscale\undefined%
      \relax%
    \else%
      \setlength{\unitlength}{\unitlength * \real{\svgscale}}%
    \fi%
  \else%
    \setlength{\unitlength}{\svgwidth}%
  \fi%
  \global\let\svgwidth\undefined%
  \global\let\svgscale\undefined%
  \makeatother%
  \begin{picture}(1,0.33333333)%
    \lineheight{1}%
    \setlength\tabcolsep{0pt}%
    \put(0,0){\includegraphics[width=\unitlength,page=1]{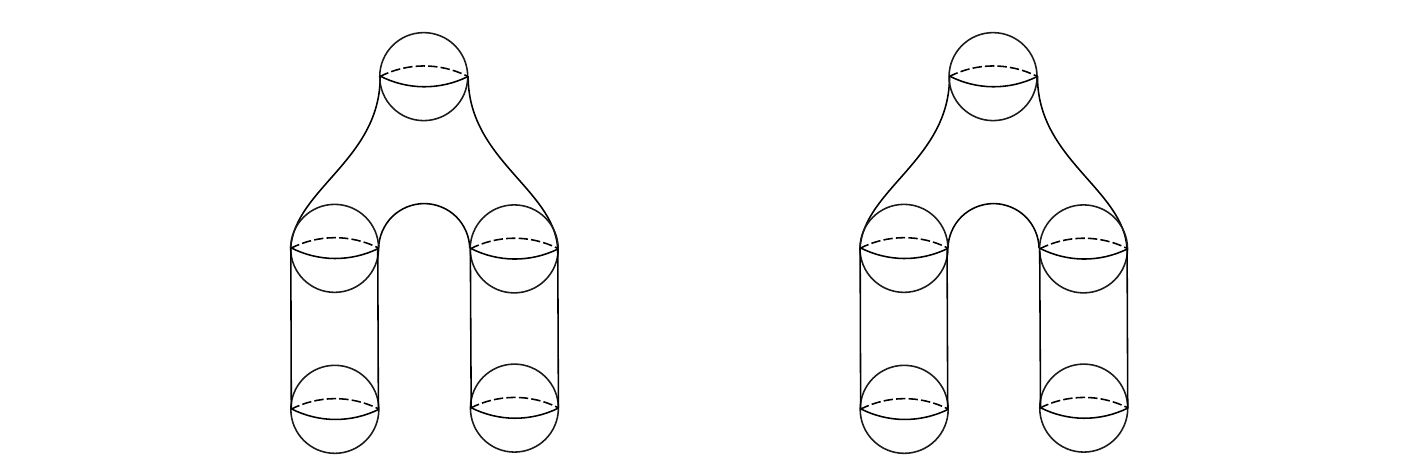}}%
    \put(0.21181461,0.0918958){\makebox(0,0)[lt]{\lineheight{1.25}\smash{\begin{tabular}[t]{l}$P^{\times \times}$\\\end{tabular}}}}%
    \put(0.7432347,0.09258073){\makebox(0,0)[lt]{\lineheight{1.25}\smash{\begin{tabular}[t]{l}$P^{\times \times}$\end{tabular}}}}%
    \put(0.49037625,0.15883929){\makebox(0,0)[lt]{\lineheight{1.25}\smash{\begin{tabular}[t]{l}$\cong $\end{tabular}}}}%
  \end{picture}%
\endgroup%

    \caption[Legs relations for the $G2$ presentation]{The ``legs relations'' for the $G_2 $ presentation of $\operatorname{Cob}( 3)_{S^2} $.}
    \label{fig:plegs}
\end{figure}

The legs relations are not trivially implied by commutativity because our boundaries are ordered. To see this, consider passing through a TFT $Z$, then both sides of the legs relations are maps from $Z(S^2) \otimes Z(S^2) $ to $Z(S^2)$. The left hand side sends the element $ a \otimes b$ to $Z(p^{\times \times} )(a) \cdot b$, while the right hand side sends it to $a \cdot Z(p^{\times \times} )(b)$.

We prove in theorem \ref{thm:trueG2thm} that
\begin{thm}
$\operatorname{Cob}( 3)_{S^2} $ has a presentation with $G_2$ generators under commutative Frobenius relations and legs relations. These relations are necessary.
\end{thm}

The $G_1 $ presentation is given by the usual commutative Frobenius relations only. We prove in theorem \ref{mainthmG1} that

\begin{thm} 
$\operatorname{Cob}( 3)_{S^2} $ has a presentation with $G_1$ generators and commutative Frobenius relations.
\end{thm}
What is perhaps surprising about the theorems is that the relations are extremely simple, even though the topology of 3-manifolds are significantly more complicated than the topology of $2$-manifolds. In $\operatorname{Cob}( 2) $ , a morphism is a 2-manifold with boundaries. The classification of $2$-manifolds shows that up to diffeomorphisms the manifold is classified by its genus. The classification of $3$-manifolds is more complicated but standard-- the theorem of Kneser-Milnor prime decomposition will play a repeated central role in our proof:

\begin{thm}
(Kneser-Milnor Prime Decomposition Theorem) 

Let $M$ be a compact, connected and orientable 3-manifold. Then there is a decomposition \[
M = P_1 \# P_2  \# \dots  \# P_m
.\]
where each $P_i$ is prime. This decomposition is unique up to insertion or deletion of $S^3$'s.
\end{thm}
A connected $3$-manifold $X$ is \textit{prime} if $X=P \# Q $ implies either $P = S^3$ or $Q=S^3$. The $\#$ denotes the \textit{connect sum} operation-- one deletes a 3-ball from both  $P$ and $Q$ and then glues them together along the sphere boundary. Meaning that one can always view $M = M_1 \# M_2 $ as $M\cong (M_1 \setminus \operatorname{int}(D^3)) \cup_{S^2} (M_2 \setminus \operatorname{int}(D^3))$. 

Notice that we can use the prime decomposition theorem to ``chop up'' a bordism $M: \sqcup^{n_{inc}}S^2 \to \sqcup^{n_{out}}S^2$ into prime factors. This is the idea behind the generators $G_1 $ and $G_2 $, which correspond to two different choices of chopping.

To show that $G_1 $ and $G_2 $ with appropriate relations are presentations of the category $\operatorname{Cob}( 3)_{S^2}$, we first consider the non-prime part--if one forgets about the prime factors in $P \subset G_2 $ and $L \subset G_1 $ for a moment and only considers the CF generators, then we are working with the category $\operatorname{Cob}( 3)_{S^2,CF}$ whose objects are disjoint unions of spheres like before, but whose morphisms are bordisms that are generated by CF generators. We prove in theorem \ref{thm:Cob2Cob3CFIso} that 
\begin{thm}
There is a canonical isomorphism $S:\operatorname{Cob}( 2) \overset{\sim}{\to} \operatorname{Cob}( 3)_{S^2, CF}$
\end{thm}
The isomorphism $S$ sends $n$-copies of $S^1$ to $n$-copies of $S^2$, and sends the multiplication map of $S^1$ to the multiplication map of $S^2$. Same goes for unit, trace, and co-multiplication. On more complicated bordisms, $S$ sends a genus $g$ bordism to a bordism with $g$-many $S^2 \times S^1$ connect sum factors. This theorem essentially shows that the commutative Frobenius part of $\operatorname{Cob}( S^2)$ is completely determined by $\operatorname{Cob}( 2)$. In particular, the commutative Frobenius relations of $G_1 $ and $G_2 $ generators are necessary.

The Cerf theory we build through geometry can be roughly summarized as follows: every composition in $\operatorname{Cob}( 3)_{S^2}$ is induced by some decomposition, which is induced by some Morse data. Two different compositions of the same morphism are induced by two different Morse data on the same underlying manifold. Let's denote the two Morse functions as $f_0 , f_1 :M \to \mathbb{R}$. We show that there exists a sequence of relations which can put a composition into a ``normal form'' using only \textit{prime relations}. 

A prime relation is a pair of Morse functions which can and will be connected by a path of functions with more than 1 singularities where the function fails to be excellent Morse, unlike commutative Frobenius relations, each of which is a single specific type of singularity. 

Call the Morse functions inducing these normal-form compositions $f_0', f_1'$. We then show the resulting normal-form compositions are related by commutative Frobenius relations only. This means that one can construct a path from $f_0 $ to $f_1 $ consisting of 3 distinct segments:
\begin{enumerate}
\item Segment 1: from $f_0 $ to $f_0' $, which induces the change from the starting composition to the normal-form composition. Since we will use only prime relations, the Morse functions inducing the individual intermediate compositions are separated by more than one singularity.
\item Segment 2: from $f_0 '$ to $f_1 '$, which has only commutative Frobenius type singularities.
\item Segment 3: from  $f_1 '$ to $f_1 $. This is the reverse of segment 1 but for the other composition.
\end{enumerate}
A generic figure depicting this process can be found in \ref{fig:cerftheorybehindthescenes}.

\begin{figure}[ht]
    \centering
\def\svgwidth{1\columnwidth} 
\begingroup%
  \makeatletter%
  \providecommand\color[2][]{%
    \errmessage{(Inkscape) Color is used for the text in Inkscape, but the package 'color.sty' is not loaded}%
    \renewcommand\color[2][]{}%
  }%
  \providecommand\transparent[1]{%
    \errmessage{(Inkscape) Transparency is used (non-zero) for the text in Inkscape, but the package 'transparent.sty' is not loaded}%
    \renewcommand\transparent[1]{}%
  }%
  \providecommand\rotatebox[2]{#2}%
  \newcommand*\fsize{\dimexpr\f@size pt\relax}%
  \newcommand*\lineheight[1]{\fontsize{\fsize}{#1\fsize}\selectfont}%
  \ifx\svgwidth\undefined%
    \setlength{\unitlength}{680.31496063bp}%
    \ifx\svgscale\undefined%
      \relax%
    \else%
      \setlength{\unitlength}{\unitlength * \real{\svgscale}}%
    \fi%
  \else%
    \setlength{\unitlength}{\svgwidth}%
  \fi%
  \global\let\svgwidth\undefined%
  \global\let\svgscale\undefined%
  \makeatother%
  \begin{picture}(1,0.16666667)%
    \lineheight{1}%
    \setlength\tabcolsep{0pt}%
    \put(0,0){\includegraphics[width=\unitlength,page=1]{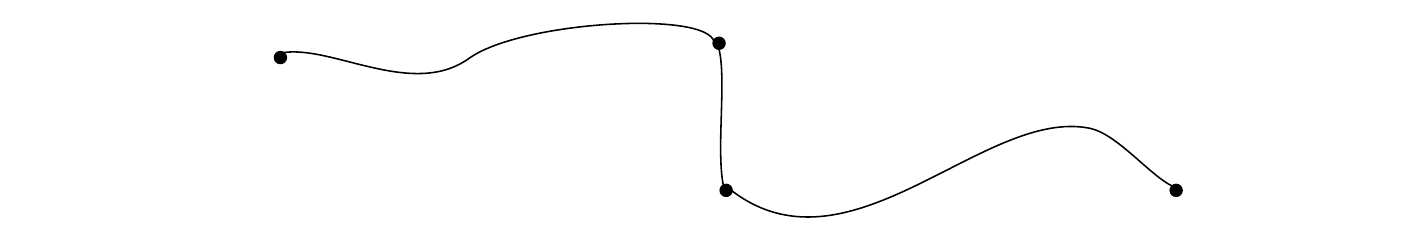}}%
    \put(0.16570476,0.09945908){\makebox(0,0)[lt]{\lineheight{1.25}\smash{\begin{tabular}[t]{l}$f_0$\end{tabular}}}}%
    \put(0.51604382,0.14269708){\makebox(0,0)[lt]{\lineheight{1.25}\smash{\begin{tabular}[t]{l}$f_0'$\end{tabular}}}}%
    \put(0.47391443,0.02517834){\makebox(0,0)[lt]{\lineheight{1.25}\smash{\begin{tabular}[t]{l}$f_1'$\end{tabular}}}}%
    \put(0.82757963,0.04624303){\makebox(0,0)[lt]{\lineheight{1.25}\smash{\begin{tabular}[t]{l}$f_1$\end{tabular}}}}%
    \put(0,0){\includegraphics[width=\unitlength,page=2]{cerftheorybehindthescenes.pdf}}%
  \end{picture}%
\endgroup%

    \caption[A Cerf theory summary]{A path from $f_0$ to $f_1$ using 3 prime relations, then 2 commutative Frobenius relations, then 2 prime relations. Note that each black dot is a Morse data which induces some composition. The prime relations is a pair of Morse data separated by more than 1 singularities, while a commutative Frobenius relation is always given by a single specific kind of singularity.}
    \label{fig:cerftheorybehindthescenes}
\end{figure}

Passing through a TFT, the two presentations of $\operatorname{Cob}( 3)_{S^2}$ give rise to two kinds of structures called P-algebras and L-algebras. Since they are both induced by $\operatorname{Cob}( 3)_{S^2} $, it is unsurprising when we show in lemma \ref{lemma:PLbijection} that there is a bijection between P-algebras and L-algebras. So it makes sense to refer to them as a single kind of structure:
\begin{defn}
Let $A$ be a commutative Frobenius algebra over $k$. Then $A$ is an $L$-algebra (equiv. $P$-algebra) if there are elements $1_p \in A$ (equiv. endomorphisms $e_p:A \to A$ that acts by multiplication by $1_p$) for every diffeomorphism class of connected oriented irreducible 3-manifold $p$.
\end{defn}
The remarkable feature of $L$ and $P$ algebras is their simplicity. We are supposed to think of such an algebra as $Z(S^2)$ for some 3d TFT $Z: \operatorname{Cob}( 3)\to \operatorname{Vect}_k$. Even though the bordisms of $S^2$ are very complicated, geometry forces $G_1 $ and $G_2 $ presentations to be extremely simple, which in turn forces the prime structures of $L$ and $P$ algebras to be also extremely simple.

Relative to existing literature on 3-dimensional TFTs, our results sit somewhat ``diagonally.'' \cite{juhasz2018defining} showed that there is a bijection between the category of $3$-dimensional TFTs and the category of \textit{J-algebras}. The $J$-algebra associated to a TFT $Z$ is a $g\geq 0$ graded algebra $\oplus_g A_g$, with maps that encode geometric operations such as the attachment of 1 and 2 handles. $Z(S^2)$ is the genus $0$ part $A_0$ of the $J$-algebra assigned to $Z$. However J-algebras do not answer the question we are addressing in this paper. For example, a J-algebra $A = \oplus _{g}A_g$ carries with it a ``1-handle map'' $A_0 \to A_1 $, which corresponds to the Morse bordism of attaching a 1-handle to a two-sphere, increasing the genus by 1. What we need to classify is all possible maps from $A_0 $ to $A_0 $, or to $A_0 \otimes A_0 $, or to $A_0 \otimes A_0 \otimes A_0 $, etc, which factors through products of $A_0 $'s exclusively. The domain of these maps can also have multiple factors of $A_0 $. From a physics point of view, while it is useful to understand maps from the space of local operators/states $A_0 $ to whatever $A_{g>0}$ means in the particular contexts,  it is more surprising to learn that the space of local operators/states of a 3-dimensional TFT is a finite dimensional commutative Frobenius algebra. Our work on the presentations of $\operatorname{Cob}( 3)_{S^2}$ and on $L$/$P$ algebras says that the degree 0 part of a $J$-algebra carries additional simple structures labelled by connected oriented irreducible 3-manifolds.

In the context of modern $(\infty,n)$-perspectives of fully-extended topological field theories, it is less clear what kind of structure must replace our $L$/$P$ algebras besides tautology. For example, while $S^1$ is a commutative Frobenius object in $\operatorname{Cob}( 2)$, $S^1$ is an $\mathbb{E}_2$-algebra in $\operatorname{Bord}_2^{fr}$ the $(\infty,2)$-category of framed bordisms. This only captures the operadic structure of $S^1$, rather than the whole PROP (products and permutations). The full PROP is, tautologically, the subcategory of $\operatorname{Bord}_2^{fr}$ spanned by disjoint unions of $S^1$. It is widely known that $S^{n-1}$ is at least an $\mathbb{E}_{n}$-algebra in $\operatorname{Bord}_n^{fr}$, but obviously there is more operadic structure in dimension $n\geq 3$. 

We propose an $\infty$-operad $\mathbb{L}^\otimes $ in definition \ref{defn:tL} and \ref{defn:L} which we conjecture to be isomorphic to the connected part of the full $\infty$-endomorphism operad of $S^2$. An $n$-ary operation is, roughly speaking, a bordism from $n$-many $S^2$ to a single $S^2$, where the bordism is itself a connected $3$-manifold. This is achieved by first specifying a collection of irreducible manifolds $P_i$, then specifying how to glue them together into a single bordism using some connect-sum configuration data. Adding to the plausibility, the $\mathbb{E}_3$ operad sits nicely inside $\mathbb{L}^\otimes $ as a suboperad. The construction is inspired by the unfinished draft of \cite{hatcher2008diffeomorphism}, which just recently was completed by \cite{boyd2026primedecompositionfibresequence} in almost all cases except, unfortunately, for bordisms with spherical boundaries.

\subsubsection{Organization Roadmap}
In Chapter \ref{ch:generators}, we introduce the generators $G_1 $ and $G_2 $ of $\operatorname{Cob}( 3)_{S^2} $ and show in proposition \ref{generate} that they generate all the morphisms.

Chapter \ref{ch:CFrelations} and chapter \ref{ch:prime relations} are dedicated to showing that commutative Frobenius relations are necessary and sufficient for the  $G_1 $ presentation, and additionally the legs relations are necessary for the $G_2 $ presentation. To achieve this, we first outline the basic strategy in section \ref{sec:outlineofstrat}. 

We then first address the commutative Frobenius relations between the shared CF generators of $G_1 $ and $G_2 $ in chapter \ref{ch:CFrelations}. This is done in stages: we first show existence in proposition \ref{prop:CFcommFrob} in section \ref{sec:existenceandsphericalMorse} by analyzing the generic Cerf theory arguments. We then show in theorem \ref{thm:Cob2Cob3CFIso} in section \ref{sec:Cob2Cob3Iso} that there is an isomorphism between $\operatorname{Cob}( 2) $  and $\operatorname{Cob}( 3)_{S^2,CF}$. We then setup some machinery in section \ref{sec:localmodels} so we can prove in section \ref{sec:sufficient} that \ref{cor:CFonlycommFrob} CF generators satisfy no more relations.

We then address the new relations that arise from the prime-labelled generators in chapter \ref{ch:prime relations}. This is also done in stages: first we show that CF generators exists exclusively outside prime factors, a phenomena we call ``no internal relations.'' This is the topic of section \ref{sec:no internal relations}. We then immediately introduce the prime relations in section \ref{sec: external prime relations}: prime commutativity \ref{primecommutativity}, legs and waist relations for $G_2 $ \ref{prop:legsandwaistrelations}, and legs relations for $G_1 $ \ref{prop:Llegs}. In section \ref{sec:external sufficient} we prove in theorem \ref{mainthmG2} and \ref{mainthmG1} that these relations are sufficient to give presentations of $\operatorname{Cob}( 3)_{S^2} $. However, some relations are not necessary. We show in section \ref{sec:legsnecessary} that, besides the commutative Frobenius relations, only the legs relations for $G_2 $ are necessary-- all other relations are implied by commutative Frobenius relations and the legs relations. This is explained in propositions \ref{waistfromCF} to \ref{colegsfromCF}. Therefore we prove in theorem \ref{thm:trueG2thm} that commutative Frobenius relations and legs relations are necessary and sufficient to form a presentation of $\operatorname{Cob}( 3) $.

Chapter \ref{ch:throughZ} is dedicated to the consequences of realizing $\operatorname{Cob}( 3)_{S^2} $ as presentations after passing through $3$-dimensional topological field theories. In section \ref{sec:PLmonoids} we define the notion of an $L$ algebra, or equivalently a $P$ algebra, in definition \ref{defn:P/L algebras}. We also explain some immediate consequences and properties afforded by $L$ algebras. We examine in section \ref{sec:Jalg} the relations between $L$ algebras and $J$-algebras which classify $2+1$-dimensional TFTs. We show that the degree 0 part of every $J$- algebra is an $L$-algebra in theorem \ref{thm:JgivesPL}, and provide a specific recipe in proposition \ref{prop:JtoLmaps} to recover all the prime endomorphisms and prime units using Heegaard splitting. Lastly in section \ref{sec:infoperad}, we define an $\infty$-operad $\mathbb{L}^\otimes $ in definition \ref{defn:tL} and \ref{defn:L}, which has $\mathbb{E}_3^{\otimes }$ as a suboperad. We conjecture in \ref{conj:LisEndfr} and \ref{conj:LisEndSO} that $\mathbb{L}^\otimes $ is isomorphic to the connected part of the $\infty$-endomorphism operad of $S^2$ in $\operatorname{Bord}_3^{fr}$ or $\operatorname{Bord}_3^{SO}$.

\newpage
\chapter{Generators} \label{ch:generators}
There are two approaches to presenting the content. One could either employ the full power of abstraction, in terms of presentations of categories outlined in Appendix \ref{appendix:presentation}, or appeal to intuition carried over from presentations of groups. We will try to accomplish both. The readers are encouraged to refresh their memories on the rigorous setup of quotient categories, generators and relations, and presentations of categories, found in the form of an expository note in the appendix.

The category we wish to find a presentation for is the subcategory of $\operatorname{Cob}(3)$ spanned by spherical objects.
\begin{defn} \label{defn:Cob(3)_S^2}
($\operatorname{Cob}(3)_{S^2} $) Consider the symmetric monoidal category $\operatorname{Cob}(3) $. It has a subcategory whose objects are disjoint unions of $S^2$. Call this category $\operatorname{Cob}(3)_{S^2}$ and note that it inherits the symmetric monoidal structure.
\end{defn}

Note that some authors include all possible codimension $1$ closed oriented manifolds when defining $\operatorname{Cob}( n) $. Even though we have demanded that our categories $\operatorname{Cob}(3)_{S^2}$ to have objects strictly the unit sphere, we did not lose any generality, because in the other convention, we have
\begin{prop}
Let $X, Y$ be diffeomorphic objects of $\operatorname{Cob}(n)$ and $f:X \to Y$ a diffeomorphism. Then $Z(X) \cong Z(Y)$. 
\end{prop}

\begin{proof}
Consider the mapping cylinder $M_f : X \to Y$ and the reverse $M'_f : Y \to X$. Unwrap the definition and it's easy to see that $M'_f \circ M_f = X \times I $, and so is the other direction. Hence $M_f$ is an isomorphism.
\end{proof}

We say a set $G$ generates morphisms of $(\mathcal{C},\otimes )$ if every morphism in $\mathcal{C}$ is a composition of elements of $G$ and all monoidal products of $G$. Note that we will also include, implicitly, the associated identities, left and right unitors, associators, and braiding. 

We say a set of equalities $R$ among (compositions of monoidal products of) elements of $G$ is sufficient if every equality of morphisms in $\mathcal{C}$ is induced by equalities in $R$. If $G$ generates morphisms and $R$ is sufficient, then we say $(G \mid R)$ is a presentation of $\mathcal{C}$. 

We want a presentation of $\operatorname{Cob}(3)_{S^2}$, both generators and relations. As we will see, there are two sets of generators, named $G_1 $ and $G_2 $, giving rise to two different presentations. The $G_1 $ presentation is associated with the letter ``L'' due to the distinct upside-down ``L''-shaped branches in figure \ref{fig:g1generate}. After passing through TFTs it would give rise to $L$-monoids and $L$-algebras. The $G_2 $ presentation is associated with the letter ``P'' for ``Prime'' due to a lack of imagination. Passing through TFTs gives rise to $P$-monoids and $P$-algebras. As $G_1 $ and $G_2 $ both generate $\operatorname{Cob}(3)_{S^2}$, we shall foreshadow that there is an equivalence between $P$-monoids and $L$-monoids, as well as between $P$-algebras and $L$-algebras.

The following notations will be used throughout the paper:
\begin{notations}
Let $M$ be a manifold. Denote by $M^\circ$ and $\operatorname{int}(M)$ the interior of $M$. Let $n\geq 1$ be a positive integer, then we frequently write $\sqcup^n M$ for short  to denote the disjoint union of $n$-many copies of $M$, instead of $\sqcup_{i=1}^n M$.
\end{notations}

Consider the following diffeomorphism classes of oriented bordisms of two-spheres:  

\begin{notations}
Define the following morphisms of $\operatorname{Cob}( 3)_{S^2} $ :
\begin{align*}
id_{S^2} &= \left[ S^2 \times I \right]  \\
m &= \left[ S^3 \setminus \operatorname{int}(\sqcup_{}^3 D^3) \right]  : S^2 \sqcup S^2 \to S^2 \\
1 &= \left[ D^3  \right] : \emptyset \to S^2\\
m^\vee &= \left[ S^3 \setminus \operatorname{int}(\sqcup_{}^3 D^3)  \right] : S^2  \to S^2\sqcup S^2 \\
\tr &= \left[ D^3 \right]  :S^2 \to \emptyset
\end{align*}
\end{notations}

\begin{mycom} \label{S3minusballs}
This set of bordisms can be neatly organized as follows: recall that the three-sphere is homeomorphic to the gluing of two three-balls along their common boundary (the north and south hemisphere glued along the equator). If we elect incoming/outgoing time direction to run from south to north, then the 1 (namely unit) map can be viewed as the southern hemisphere, and tr (namely trace) as the northern. Either one could be viewed as $S^3 \setminus \operatorname{int}( D^3)$ and distinguished by incoming/outgoing time reversal.  The identity map could be viewed as $S^3 \setminus \operatorname{int}(D^3 \sqcup D^3)$. So we have identified $S^3$ minus one, two, and three $D^3$. In addition, $S^3$ itself can be obtained by the hemisphere decomposition $\tr\circ 1$, and since composing two multiplications together increases the incoming boundary components by 1, $S^3$ minus $n\geq 3$-many $D^3$ interiors can be obtained by composing $n-2$ multiplication maps. \footnotemark

\footnotetext{One can also think of this as the operadic composition law of the (framed) little $3$-disk operad.}
\end{mycom}

\begin{notations} \label{notations:CF}
For suggestive reasons, we let \[
C = \left\{ m, 1 \right\} ,\quad F=\left\{ m^\vee, tr \right\} 
.\]

Furthermore, consider the following three sets indexed by oriented prime 3-manifolds:
\begin{align*}
P' &= \left\{ \left[ p\setminus \operatorname{int}(D^3 \sqcup D^3)\right] :S^2 \to S^2 \mid p \text{ prime} \right\} \\
L' &= \left\{ \left[ p \setminus \operatorname{int} (D^3) \right] : \emptyset \to S^2 \mid p \text{ prime} \right\}  \\
L^{\vee'} &= \left\{ \left[ p\setminus \operatorname{int}(D^3) \right] : S^2 \to \emptyset \mid p \text{ prime} \right\} 
\end{align*}
whose elements we simply refer to as $p^{\times \times}$ and $p^\times, p^{\times \vee}$. Note that $L^{\vee '}$ is the incoming/outgoing reversal of $L'$, but there is no $P^{\vee '}$ as the time reversal of any such bordism is itself.
\end{notations}

\begin{prop} \label{generate}
The homs of $\operatorname{Cob}(3)_{S^2} $ have two generating sets: 
\begin{align*}
G_1'& = C \cup F \cup L' \cup L^{\vee'} \cup \text{orientation reversal}\\
G_2' &= C \cup F \cup P' \cup \text{orientation reversal}
\end{align*}
More precisely, let $F(\text{Gra}_1 )$ and $F(\text{Gra}_2)$ be the free categories of the graphs $\text{Gra}_1$ and $\text{Gra}_2 $, which have $\left\{ \sqcup^{i} S^2 \mid i \in \mathbb{Z}_{\geq 0}\right\} \cong \mathbb{Z}_{\geq 0}$ as vertices and directed edges generated by $G_1' $ and $G_2 '$ and all tensor products. Then the tautological functors $F(\text{Gra}_1) \to \operatorname{Cob}( 3)_{S^2}$ and $F(\text{Gra}_2) \to \operatorname{Cob}( 3)_{S^2}$ are full and bijective on objects.
\end{prop}

\begin{proof}

Let $\left[ M \right] : \sqcup^{n_{inc}} S^2 \to \sqcup^{n_{out}} S^2$ be a morphism from $n_{inc}$-many two-spheres to $n_{out}$-many two-spheres. Since the braiding $\beta _{S^2, S^2}$ permutes the boundary spheres, we can, without loss of generality, assume that within incoming or outgoing boundary components the spheres are not labelled. \cite{kock_frobenius_2003}

Therefore we need to show that given an oriented bordism of unlabelled spheres $[M]$, we can find a composition in terms of $G_1 '$ and $G_2 '$ generators and their tensor products. Let $M$ be a representative of $[M]$, we will show that $M$ is diffeomorphic to gluings of manifolds which are representatives of $G_1 '$ and $G_2 '$ generators.

Consider the closed oriented manifold 
\[
\widehat{M} \cong \Big(\bigsqcup_{i=1}^{n_{inc}+n_{out}}D^3\Big) \bigcup_{\partial M} M
\]
obtained from filling in the bounding spheres. One can show that the gluing is well defined up to diffeomorphism \cite{hirsch_differential_1976}. By the Kneser-Milnor theorem, $\widehat{M}$ is diffeomorphic to a finite connected sum of oriented prime 3-manifolds $P_1, \dots , P_m$, \[
\widehat{M} \cong P_1 \# P_2 \# \dots \# P_m
.\]
Therefore $M$ itself is diffeomorphic to the same finite connected sum minus the interior of the boundary balls \[
M \cong \left(  P_1 \# P_2 \# \dots \# P_m \right) \setminus \operatorname{int}\Big(\bigsqcup_{i=1}^{n_{inc}+n_{out}}D^3\Big)
.\]
Unlike a generic Morse decomposition whose gluing surfaces are constrained by their values, the attaching balls for each connect sum operation can be freely chosen. Depending on the choices of the attaching balls, we will have two different decompositions of $M$ as gluings. To make the expression more symmetric, consider adding two trivial connect sum factors \[
M \cong S^3 \# \left(  P_1 \# P_2 \# \dots \# P_m \right) \# S^3 \setminus \operatorname{int}\Big(\bigsqcup_{i=1}^{n_{inc}+n_{out}}D^3\Big)
\]
and choose\footnote{This is not a constraint but a choice made to make apparent how subsequent choices end up being generated by $G_1'$ and $G_2 '$, much like the use of the ``standard form'' of surfaces commonly used in the classification of 2d TFT, see for eg [Kock].} that the incoming and outgoing balls be cut from these two factors \[
M \cong \Big(S^3 \setminus \operatorname{int}\Big( \bigsqcup_{i=1}^{n_{out}}D^3\Big)\Big) \# (  P_1 \# P_2 \# \dots \# P_m ) \# \Big( S^3 \setminus \operatorname{int}\Big(\bigsqcup_{j=1}^{n_{out}} D^3\Big) \Big) 
.\]

\begin{enumerate}
\item (The choice leading to $G_2 '$) 

For each prime factor $P_j$ there are two chosen embedding of three-balls \[
e^j_{1,2}: D^3 \to P_j ,\quad j=1, \dots , m, 
\]
where $e^j_2$ forms the connect sum with $e^{j+1}_1$ of $P_{j+1}$ unless $j=m$, when $e^m_2$ forms the connect sum with an embedded three ball in $S^3 \setminus \operatorname{int}(\sqcup^{ n_{inc} } D^3)$. We choose $e^1_1$ to form the connect sum with an embedded three-ball in  $S^3 \setminus \operatorname{int}(\sqcup^{n_{out}} D^3)$. Noting that \[
M \# N \cong (M\setminus \operatorname{int}(D^3)) \cup_{\partial D^3} (N \setminus \operatorname{int}(D^3)),
\]
where the boundary of the embedded three balls being identified can be thought of as the boundary being glued over, we may rewrite the connect sum as 
\begin{align*}
M &\cong \Big(S^3 \setminus \operatorname{int}\Big(\bigsqcup_{i=1}^{n_{out}}D^3\Big)\Big) \# \left(  P_1 \# P_2 \# \dots \# P_m \right) \# \Big( S^3 \setminus \operatorname{int}\Big(\bigsqcup_{j=1}^{n_{inc}} D^3\Big) \Big) \\
&\cong \Big(S^3 \setminus \operatorname{int}\Big(\bigsqcup_{i=1}^{n_{out}+1}D^3\Big)\Big) \cup_{S^2}  \Big(P_1 \setminus \operatorname{int}(D^3 \sqcup D^3)\Big) \cup_{S^2} \dots\\
& \qquad \qquad \dots \cup_{S^2} \Big(P_m \setminus \operatorname{int}(D^3 \sqcup D^3)\Big) \cup_{S^2}  \Big( S^3 \setminus \operatorname{int}\Big(\bigsqcup^{n_{int}+1} D^3\Big) \Big).
\end{align*}

Under the natural choice of $e^j_1$ bounding the outgoing sphere and $e^j_2$ bounding the incoming sphere, this composition induces 
\begin{align*}
\Big[M\Big] =& \Big[S^3 \setminus \operatorname{int}(\sqcup^{n_{out}} D^3)\Big] \circ \Big[P_1 \setminus \operatorname{int}(D^3 \sqcup D^3) \Big] \circ \dots \\
& \qquad \qquad \dots \circ \Big[P_m \setminus \operatorname{int}(D^3 \sqcup D^3)\Big] \circ \Big[S^3 \setminus \operatorname{int}(\sqcup^{n_{inc}} D^3)\Big]\\
=& \Big[S^3 \setminus \operatorname{int}(\sqcup^{n_{out}} D^3)\Big] \circ P_1^{\times \times} \circ \dots \circ P_m^{\times \times} \circ \Big[S^3 \setminus \operatorname{int}(\sqcup^{n_{inc}} D^3)\Big]
\end{align*}
and by comment \ref{S3minusballs} the morphisms $[S^3 \setminus \operatorname{int}(\sqcup^{n} D^3)]$ can be obtained as a composition of commutative Frobenius generators. See figure \ref{fig:g2generate} for a sketch of this composition.

\begin{figure}[ht]
\centering
\def\svgwidth{1\columnwidth} 
\begingroup%
  \makeatletter%
  \providecommand\color[2][]{%
    \errmessage{(Inkscape) Color is used for the text in Inkscape, but the package 'color.sty' is not loaded}%
    \renewcommand\color[2][]{}%
  }%
  \providecommand\transparent[1]{%
    \errmessage{(Inkscape) Transparency is used (non-zero) for the text in Inkscape, but the package 'transparent.sty' is not loaded}%
    \renewcommand\transparent[1]{}%
  }%
  \providecommand\rotatebox[2]{#2}%
  \newcommand*\fsize{\dimexpr\f@size pt\relax}%
  \newcommand*\lineheight[1]{\fontsize{\fsize}{#1\fsize}\selectfont}%
  \ifx\svgwidth\undefined%
    \setlength{\unitlength}{680.31496063bp}%
    \ifx\svgscale\undefined%
      \relax%
    \else%
      \setlength{\unitlength}{\unitlength * \real{\svgscale}}%
    \fi%
  \else%
    \setlength{\unitlength}{\svgwidth}%
  \fi%
  \global\let\svgwidth\undefined%
  \global\let\svgscale\undefined%
  \makeatother%
  \begin{picture}(1,0.5)%
    \lineheight{1}%
    \setlength\tabcolsep{0pt}%
    \put(0,0){\includegraphics[width=\unitlength,page=1]{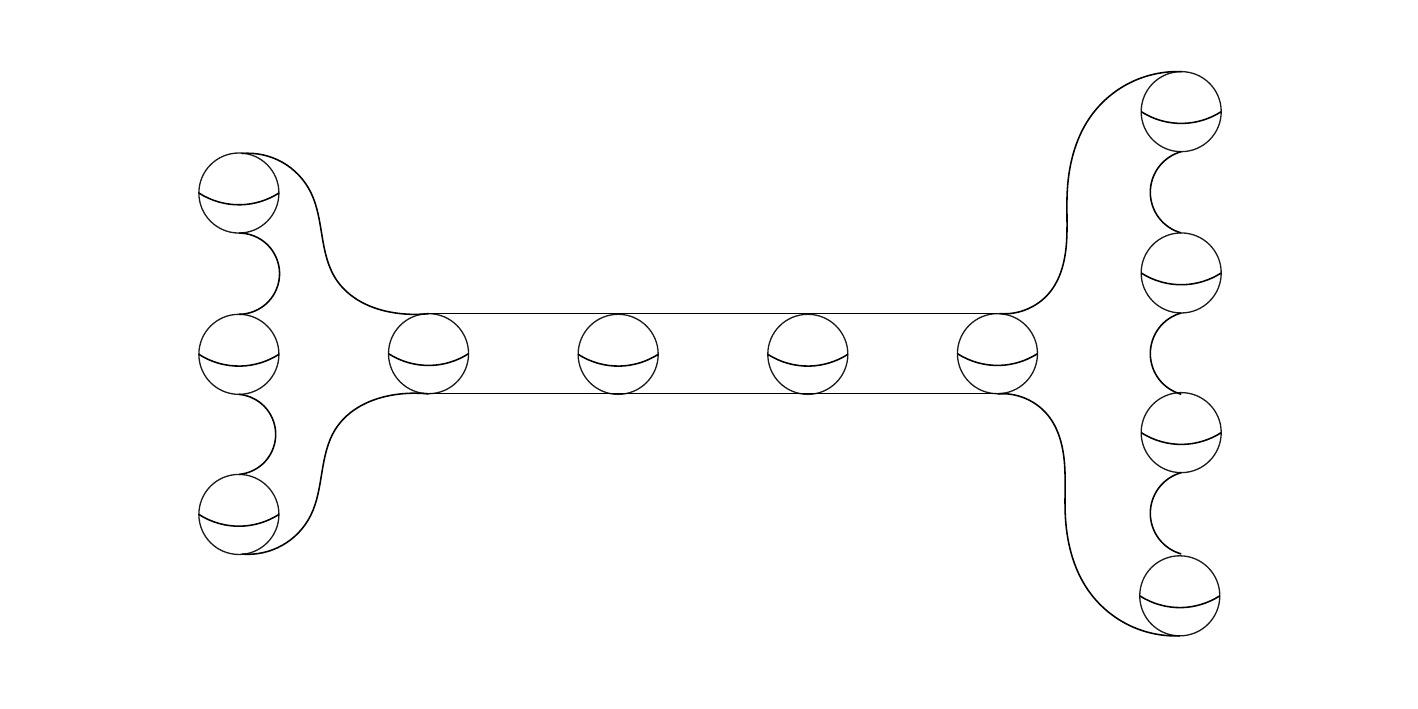}}%
    \put(0.35281871,0.2456672){\makebox(0,0)[lt]{\lineheight{1.25}\smash{\begin{tabular}[t]{l}$P_m$\end{tabular}}}}%
    \put(0.47858588,0.2456672){\makebox(0,0)[lt]{\lineheight{1.25}\smash{\begin{tabular}[t]{l}$\cdots$\end{tabular}}}}%
    \put(0.61819286,0.2456672){\makebox(0,0)[lt]{\lineheight{1.25}\smash{\begin{tabular}[t]{l}$P_1$\end{tabular}}}}%
  \end{picture}%
\endgroup%

\caption[How $G_2 $ generates morphisms]{A sketch of how $G_2 $ generates a generic bordism $[M]$ where $M$ 's prime factors are $P_1 , \dots , P_m$. Note that this picture is ``sideways'' as we usually associate the ``time'' direction as going up. In this case, time goes from left to right. }
\label{fig:g2generate}
\end{figure}
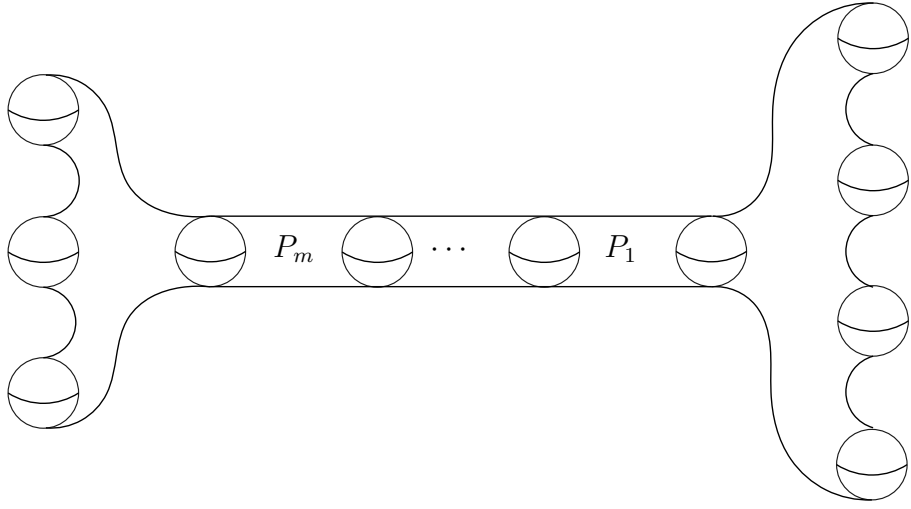


\item (The choice leading to $G_1 '$) 

For each $P_j$ there is one chosen embedding of three balls \[
e^j:D^3 \to P_j
\]
for $j=1, \dots ,m$. Take an additional $S^3$ factor and choose $m+2$ embeddings $e_j, e_0, e_{m+1}$, $j=1, \dots , m$. Let $e_j$ form the connect sum with $e^j$, let $e_0$ form the connect sum with an embedded three ball in $S^3 \setminus \sqcup^{n_{out}}D^3$ and let $e_{m+1}$ form the connect sum with an embedded three ball in $S^3 \setminus \sqcup^{n_{inc}}D^3$. Noting that \[
M\#N \cong (M\setminus \operatorname{int}(D^3)) \cup_{\partial D^3} (N\setminus \operatorname{int}(D^3))
\]
where the boundary of the embedded three balls being identified can be thought of as the boundary being glued over, we may rewrite the connect sum as 
\begin{align*}
M &\cong \Big(S^3 \setminus \operatorname{int}\Big(\bigsqcup_{i=1}^{n_{out}}D^3\Big)\Big) \# \left(  P_1 \# P_2 \# \dots \# P_m \right) \# \Big( S^3 \setminus \operatorname{int}\Big(\bigsqcup_{j=1}^{n_{inc}} D^3\Big) \Big) \\
&\cong \left( \Big(S^3 \setminus \operatorname{int}\Big(\bigsqcup_{i=1}^{n_{out}+1}D^3\Big)\Big) \sqcup \Big(S^3 \setminus \operatorname{int}\Big(\bigsqcup_{i'=1}^{n_{inc}+1}D^3\Big)\Big) \right) \bigcup_{S^2 \sqcup S^2} \left(S^3 \setminus \operatorname{int}\Big(\bigsqcup_{i''=1} ^{m+2}D^3\Big)\right) \\
& \hskip0.2\textwidth \bigcup_{\sqcup^{m} S^2} \left( \bigsqcup_{j=1}^m \left(P_j \setminus \operatorname{int}( D^3)\right) \right) 
\end{align*}
Besides the obvious choice for $e_0$ to bound an outgoing sphere and $e_{m+1}$ to bound an incoming sphere, there is no canonical choice of incoming vs outgoing for the boundaries of $e_j$ and $e^j$ for $j=1, \dots , m$. One choice is for all $e^j$ to bound outgoing boundaries and $e_j$ incoming. Then the $S^3$ which all the prime factors are connected with becomes $S^3 \setminus \operatorname{int}(\sqcup^{m+2} D^3)$ with $m+1$ many incoming boundaries and $1$ outgoing. This choice induces the following composition of $[M]$:  
\begin{align*}
\Big[M\Big] &= \Big[S^3 \setminus \operatorname{int}(\sqcup^{n_{out}+1}D^3)\Big] \circ \Big[S^3 \setminus \operatorname{int}(\sqcup^{m+2}D^3) \Big] 
\circ \Big(\bigotimes_{j=1}^m \Big[P_j \setminus \operatorname{int}(D^3)\Big] \otimes \Big[S^3 \setminus \operatorname{int}(\sqcup^{n_{inc}+1}D^3)\Big]\Big) \\
    &= \Big[S^3 \setminus \operatorname{int}(\sqcup^{n_{out}+1}D^3)\Big] \circ \Big[S^3 \setminus \operatorname{int}(\sqcup^{m+2}D^3) \Big] \circ \Big(\bigotimes_{j=1}^m P_j^\times \otimes \Big[S^3 \setminus \operatorname{int}(\sqcup^{n_{inc}+1}D^3)\Big]\Big) 
\end{align*}
where $P_j^\times = [P_j\setminus \operatorname{int}(D^3)] : \emptyset \to S^2$.
By comment \ref{S3minusballs}, the corresponding morphism $[S^3 \setminus \operatorname{int}(\sqcup^n D^3)]$ can be obtained as a composition of commutative Frobenius generators.

If we had chosen $e^j$ to bound incoming and $e_j$ to bound outgoing, then we would instead induce the following composition of $[M]$: 
\begin{align*}
\Big[M\Big] &= \Big( \bigotimes_{j=1}^m \Big[P_j \setminus \operatorname{int}(D^3)\Big] \otimes \Big[S^3 \setminus \operatorname{int}(\sqcup^{n_{out}+1}D^3)\Big] \Big)  \circ \Big[S^3 \setminus \operatorname{int}(\sqcup^{m+2}D^3)\Big] \circ \Big[S^3 \setminus \operatorname{int}(\sqcup^{n_{inc}+1}D^3)\Big]\\
&= \Big( \bigotimes_{j=1}^m P_j^{\times \vee} \otimes \Big[S^3 \setminus \operatorname{int}(\sqcup^{n_{out}+1}D^3)\Big] \Big)  \circ \Big[S^3 \setminus \operatorname{int}(\sqcup^{m+2}D^3)\Big] \circ \Big[S^3 \setminus \operatorname{int}(\sqcup^{n_{inc}+1}D^3)\Big]
\end{align*}
where now $P_j^{\times ^\vee}=[P_j\setminus D^3]:S^2 \to \emptyset$.

If we had chosen $e^j$ to bound a mixture of incoming and outgoing boundaries, the induced composition of $[M]$ would be a mixture
\begin{align*}
\Big[M\Big] &= \Big( \underset{\text{ incoming}}{\bigotimes_{e^j }} P_j^{\times \vee} \otimes \Big[S^3 \setminus \operatorname{int}(\sqcup^{n_{out}+1}D^3)\Big] \Big)  \circ \Big[S^3 \setminus \operatorname{int}(\sqcup^{m+2}D^3)\Big] \\
& \qquad \qquad \circ \Big( \underset{\text{ outgoing}}{\bigotimes_{e^j}} P_j^\times \otimes \Big[S^3 \setminus \operatorname{int}(\sqcup^{n_{inc}+1}D^3)\Big] \Big) 
\end{align*}
See figure \ref{fig:g1generate} for an illustration of a composition in the case all $e_j$ are incoming. 

\begin{figure}[ht]
    \centering
\def\svgwidth{1\columnwidth} 
\begingroup%
  \makeatletter%
  \providecommand\color[2][]{%
    \errmessage{(Inkscape) Color is used for the text in Inkscape, but the package 'color.sty' is not loaded}%
    \renewcommand\color[2][]{}%
  }%
  \providecommand\transparent[1]{%
    \errmessage{(Inkscape) Transparency is used (non-zero) for the text in Inkscape, but the package 'transparent.sty' is not loaded}%
    \renewcommand\transparent[1]{}%
  }%
  \providecommand\rotatebox[2]{#2}%
  \newcommand*\fsize{\dimexpr\f@size pt\relax}%
  \newcommand*\lineheight[1]{\fontsize{\fsize}{#1\fsize}\selectfont}%
  \ifx\svgwidth\undefined%
    \setlength{\unitlength}{680.31496063bp}%
    \ifx\svgscale\undefined%
      \relax%
    \else%
      \setlength{\unitlength}{\unitlength * \real{\svgscale}}%
    \fi%
  \else%
    \setlength{\unitlength}{\svgwidth}%
  \fi%
  \global\let\svgwidth\undefined%
  \global\let\svgscale\undefined%
  \makeatother%
  \begin{picture}(1,0.5)%
    \lineheight{1}%
    \setlength\tabcolsep{0pt}%
    \put(0,0){\includegraphics[width=\unitlength,page=1]{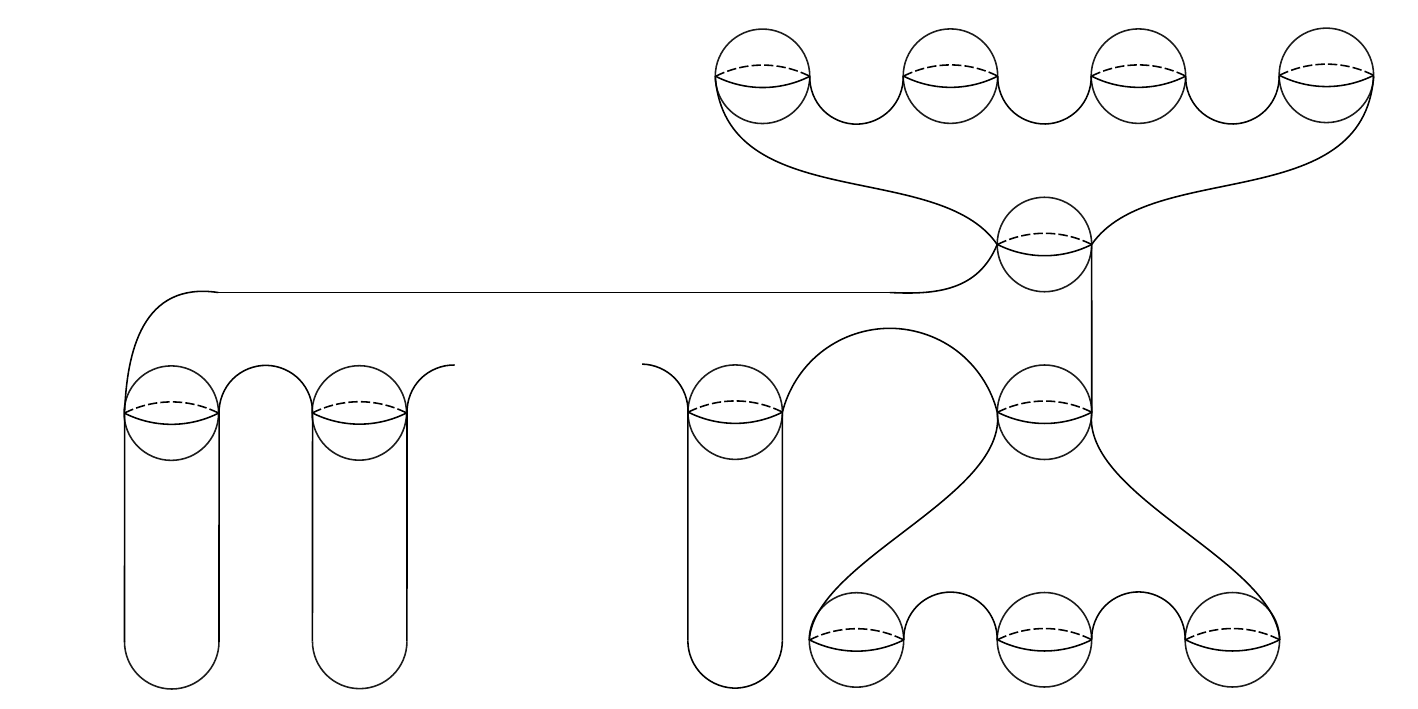}}%
    \put(0.36975396,0.12377459){\makebox(0,0)[lt]{\lineheight{1.25}\smash{\begin{tabular}[t]{l}$\dots$\end{tabular}}}}%
    \put(0.10465227,0.11862513){\makebox(0,0)[lt]{\lineheight{1.25}\smash{\begin{tabular}[t]{l}$P_1$\end{tabular}}}}%
    \put(0.23714024,0.11982129){\makebox(0,0)[lt]{\lineheight{1.25}\smash{\begin{tabular}[t]{l}$P_2$\end{tabular}}}}%
    \put(0.49996532,0.12092883){\makebox(0,0)[lt]{\lineheight{1.25}\smash{\begin{tabular}[t]{l}$P_m$\end{tabular}}}}%
  \end{picture}%
\endgroup%

    \caption[How $G_1$ generates morphisms]{A sketch of how $G_1 $ generates a generic bordism $[M]$ where $M$ 's prime factors are $P_1, \dots P_m $. This particular $[M]$ has $3$ incoming $S^2$ components and $4$ outgoing components. This figure depicts the case where all the bounding balls $e_{j}$ are incoming, namely all the prime factors are in the form of $P_i^\times : \emptyset \to S^2$.}
    \label{fig:g1generate}
\end{figure}
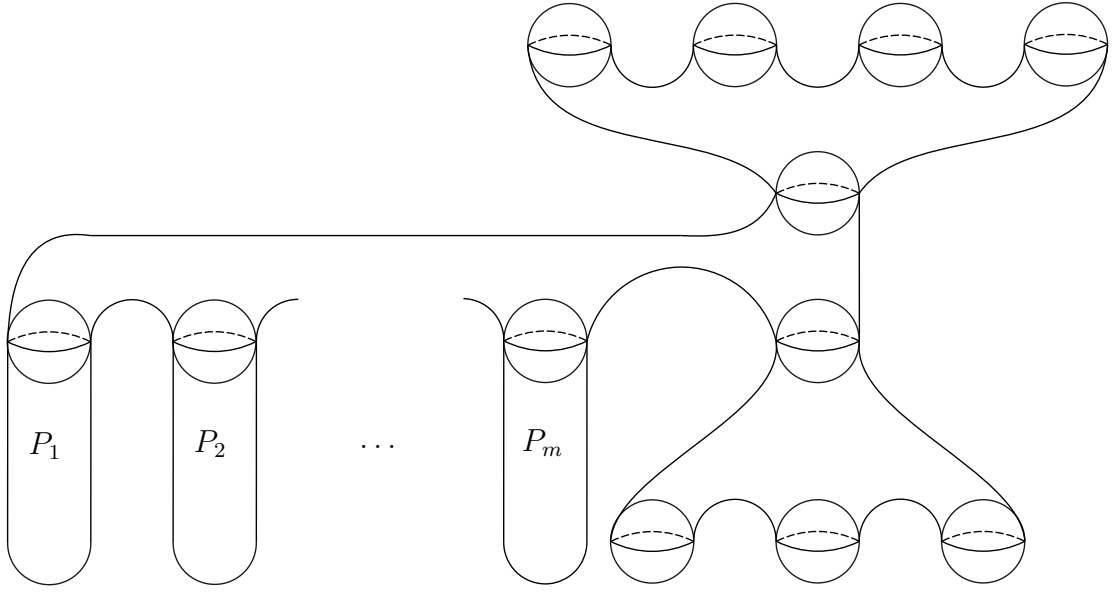

\end{enumerate}
\end{proof}


\begin{rmk} \label{rmk:S2xS1redundant}
Neither $G_1' $ nor $G_2' $ is \textit{minimal}-- $S^2 \times S^1$ is prime, but it has a standard Morse function which factors through the projection onto the $S^1$ factor: \[
S^2 \times S^1 \overset{proj}{\longrightarrow} S^1 \overset{\text{height}}{\longrightarrow} \mathbb{R}
\]
whose critical points are of index $0, 2, 1, 3$ in increasing critical values, and whose regular surfaces are either empty or a disjoint union of $S^2$. This is depicted (in a slightly wonky manner) in figure \ref{fig:s2s2pantsvsheegaard}-(a). Therefore there is a Morse gluing resembling the toric pants decomposition one dimension down, inducing the equality\[
\Big[S^2 \times S^1\Big] \cong tr \circ m \circ m^\vee \circ 1
.\]
In particular, the element $[ S^2 \times S^1 \setminus \operatorname{int}(D^3 \sqcup D^3) ] \in P \subset  G_2' $ is equal to the composition $m \circ m^\vee$, and the elements $\Big[ S^2 \times S^1 \setminus \operatorname{int}(D^3) \Big] \in L, L^\vee \subset  G_1' $ is equal to the compositions $m \circ m^\vee \circ 1$ and $tr \circ m \circ m^\vee$ respectively. 
\end{rmk}

\begin{rmk} \label{rmk:Lv'redundant}
The $L^{\vee '}$ generators are redundant -- let $P^{\times \vee} = [P\setminus \operatorname{int}( D^3)] \in L^{\vee'}$ be a morphism $S^2 \to \emptyset$. Consider $P^\times = [P\setminus D^3] \in L'$ a morphism $\emptyset \to S^2$. Since \[
P \setminus \operatorname{int}(D^3) \cong \left( P \# S^3 \right)  \setminus \operatorname{int}(D^3) 
\cong \Big(P\setminus \operatorname{int}(D^3)\Big) \cup_{S^2} \Big(S^3 \setminus \operatorname{int}(\bigsqcup_{i=1}^3 D^3)\Big) \cup_{S^2} D^3
\]
we learn that \[
P^{\times \vee} = tr \circ m \circ (P^\times \otimes id_{S^2}) \circ \lambda^{-1}_{S^2}
\]
where $\lambda_{S^2}: \emptyset \sqcup S^2 \to S^2$ is the left unitor. This relation is reflected in the sketch if figure \ref{fig:generatorrelations}.
\begin{figure}[ht]
    \centering
\def\svgwidth{1\columnwidth} 
\begingroup%
  \makeatletter%
  \providecommand\color[2][]{%
    \errmessage{(Inkscape) Color is used for the text in Inkscape, but the package 'color.sty' is not loaded}%
    \renewcommand\color[2][]{}%
  }%
  \providecommand\transparent[1]{%
    \errmessage{(Inkscape) Transparency is used (non-zero) for the text in Inkscape, but the package 'transparent.sty' is not loaded}%
    \renewcommand\transparent[1]{}%
  }%
  \providecommand\rotatebox[2]{#2}%
  \newcommand*\fsize{\dimexpr\f@size pt\relax}%
  \newcommand*\lineheight[1]{\fontsize{\fsize}{#1\fsize}\selectfont}%
  \ifx\svgwidth\undefined%
    \setlength{\unitlength}{680.31496063bp}%
    \ifx\svgscale\undefined%
      \relax%
    \else%
      \setlength{\unitlength}{\unitlength * \real{\svgscale}}%
    \fi%
  \else%
    \setlength{\unitlength}{\svgwidth}%
  \fi%
  \global\let\svgwidth\undefined%
  \global\let\svgscale\undefined%
  \makeatother%
  \begin{picture}(1,0.33333333)%
    \lineheight{1}%
    \setlength\tabcolsep{0pt}%
    \put(0,0){\includegraphics[width=\unitlength,page=1]{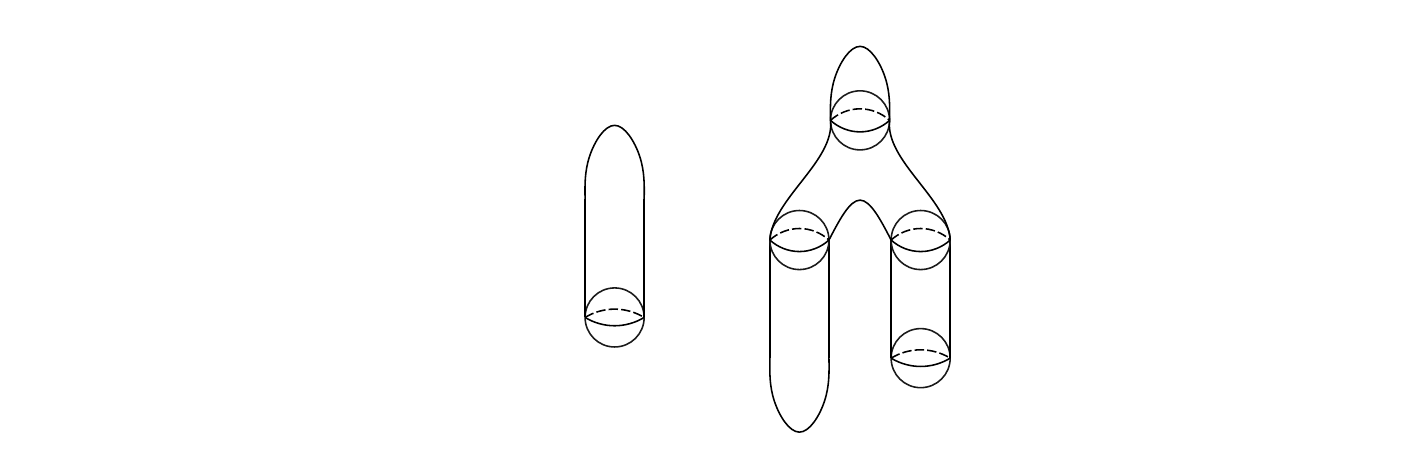}}%
    \put(0.54961943,0.10286317){\makebox(0,0)[lt]{\lineheight{1.25}\smash{\begin{tabular}[t]{l}$P^\times$\end{tabular}}}}%
    \put(0.41205652,0.15455174){\makebox(0,0)[lt]{\lineheight{1.25}\smash{\begin{tabular}[t]{l}$P^{\times \vee}$\end{tabular}}}}%
    \put(0.49300054,0.14900772){\makebox(0,0)[lt]{\lineheight{1.25}\smash{\begin{tabular}[t]{l}$\cong$\end{tabular}}}}%
  \end{picture}%
\endgroup%

    \caption[Redundancy of $L^{\vee \prime}$ generators]{Redundancy of $L^{\vee \prime}$ generators.}
    \label{fig:generatorrelations}
\end{figure}
\end{rmk}

\begin{notations}
Given the two remarks \ref{rmk:S2xS1redundant} and \ref{rmk:Lv'redundant}, we shall remove $S^2 \times S^1$ from the prime generators and delete the redundant set $L^{\vee'}$. Recalling that $S^2 \times S^1$ is the only reducible oriented prime 3-manifold, we define
\begin{align*}
P &= \left\{ p^{\times \times} = \Big[ p\setminus \operatorname{int}(D^3 \sqcup D^3) \Big] :S^2 \to S^2 \mid p\text{ oriented irreducible} \right\} \\
L &= \left\{ p^\times = \Big[ p \setminus \operatorname{int}(D^3) \Big] : \emptyset \to S^2 \mid p \text{ oriented irreducible} \right\}  
\end{align*}
\end{notations}

\begin{prop}
The following sets generate morphisms of $\operatorname{Cob}(3)_{S^2} $ and are both minimal:
\begin{align*}
G_1& = C \cup F \cup L \cup \text{orientation reversal}\\
G_2 &= C \cup F \cup P \cup \text{orientation reversal}
\end{align*}
\end{prop}

This proposition will be revisited as \ref{prop:G1G2minimal} and the proof will be presented there.

\begin{rmk}
Another distinction between $S^2 \times S^1$ and other primes, in the setting of our restricted bordism category $\operatorname{Cob}(3)_{S^2} $, is the special role it plays when compositing bordisms together: let $[M] : \sqcup^a S^2 \to \sqcup ^b S^2$ and $[N]: \sqcup^b S^2 \to \sqcup^c S^2$ be two bordisms, and let \[
\widehat{M} \cong (\sqcup^a D^3) \cup_{\sqcup^a S^2} M \cup_{\sqcup^b S^2} (\sqcup^b D^3)
\]
be the closed manifold obtained by filling the incoming and outgoing boundaries of $M$. One can show that the choice of gluing result in a well defined manifold up to diffeomorphism. Define $\widehat{N}$ analogously. Then there are prime decompositions of $\widehat{M}\cong M_1 \# M_2 \# \dots \# M_{m}$ and $\widehat{N} \cong N_1 \# \dots  \# N_n$, where each factor $M_i, N_j$ is prime. Taking the composition along the common boundary $\sqcup^b S^2$, and filling in the remaining boundaries, the resulting manifold gains connected sum factors of $S^2 \times S^1$. To see this, refer to the figure \ref{fig:g2generate} where $M$ has been put into a normal form. Consider doing the same for $N$. The pairing of the outgoing component of $M$ and the incoming components of $N$ gives rise to the new connect sum factors of $S^2\times S^1$, much like increasing genus in one dimension down: \[
\widehat{N\circ M} \cong (N_1 \# \dots \# N_n) \# \Big(\#^{b-1} ( S^2 \times S^1 ) \Big) \# (M_1 \# \dots \# M_m)
.\]
In fact, as we shall see by theorem \ref{thm:Cob2Cob3CFIso}, the morphisms in a subcategory of  $\operatorname{Cob}(3)_{S^2} $ are all $\#^g(S^2\times S^1)$ minus some (interiors of) disks, and $g$ the number of $S^2\times S^1$-factors is equivalent to the genus of some surfaces as a morphism in $\operatorname{Cob}(2)$.
\end{rmk}

\begin{mycom} \label{criticalrelations}
For now, the choice of using $G_1 $ or $G_2 $ is a matter of taste, as there is a clear bijection between the generators and that one can switch between $p^{\times \times} =[p\setminus \operatorname{int}(D^3 \sqcup D^3)] \in P$ and $p^\times = [p \setminus \operatorname{int}(D^3)] \in  L$ by relations 
\begin{align*}
\Big[p\setminus \operatorname{int}(D^3)\Big] &= \Big[p\setminus \operatorname{int}(D^3 \sqcup D^3)\Big] \circ 1\\
\Big[p\setminus \operatorname{int}(D^3 \sqcup D^3)\Big] &= m \circ (\Big[p \setminus \operatorname{int}(D^3)\Big] \otimes id_{S^2})\circ \lambda^{-1}_{S^2}
\end{align*}
where $\lambda_{S^2}: \emptyset \sqcup S^2 \to S^2$ is the left unitor. The first relation geometrically fills in a ball. The latter creates a new boundary component through $m$. In fact, there is an alternative construction \[
\Big[p\setminus \operatorname{int}(D^3 \sqcup D^3)\Big] = m \circ (id_{S^2} \otimes \Big[p \setminus \operatorname{int}(D^3)\Big])\circ \rho^{-1}_{S^2}
\]
where $\rho_{S^2}: S^2 \sqcup \emptyset \to S^2$ is the right unitor. The equivalence of the two constructions is the first example of a \textit{prime relation}. This particular one is explored in proposition \ref{prop:Llegs} and shown to be implied by the more fundamental commutative Frobenius relations in proposition \ref{prop:LlegsfromCF}.

The bijection between $G_1 $ and $G_2 $ generators is depicted in \ref{fig:g1g2bijection}. 
\begin{figure}[ht]
    \centering
\def\svgwidth{1\columnwidth} 
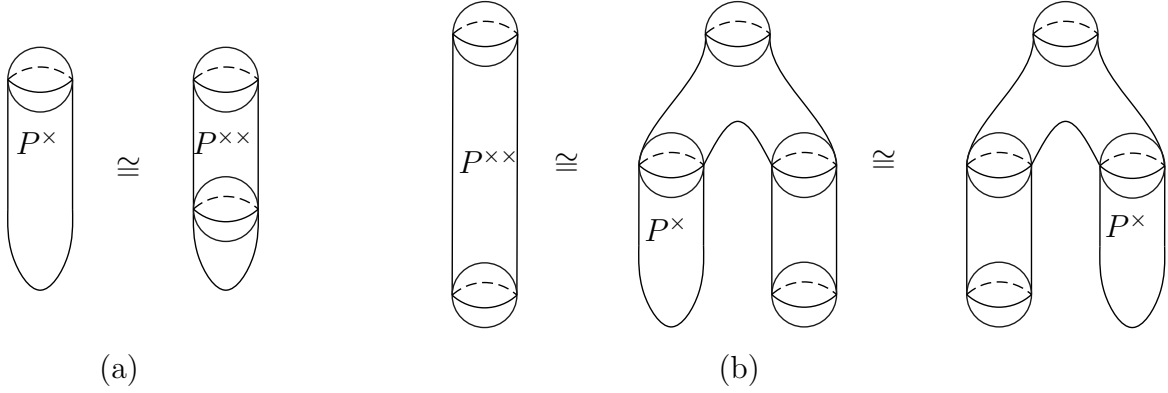

    \caption[Bijection of $G_1$ and $G_2$ generators]{Bijection between $G_1 $ and $G_2 $ generators. Figure (a) shows that all prime units $P^\times$ can be recovered by filling in the incoming ball of the prime endomorphism $P^{\times \times}$. Figure (b) shows that all prime endomorphisms can be recovered from the prime units by post-composing with the 3-dimensional pair of pants. Note that there are two possible orderings of the legs. This is known as the ``legs relations'' for $L\subset G_1 $ generators.}
    \label{fig:g1g2bijection}
\end{figure}

\end{mycom}

\newpage
\chapter{Commutative Frobenius Relations} \label{ch:CFrelations}
%
%
\section{Outline of Strategy} \label{sec:outlineofstrat}
\begin{defn} \label{defn:3d Morse Datum}
Let $M$ be a compact oriented connected 3 manifold with boundary $X_-$ and $X_+$ of codimension $1$. A \textit{prime Morse datum} is a pair $f:M\to \mathbb{R}$ and an ordered tuple $b=(b_0<b_1<\dots <b_m) \in \mathbb{R}^{m+1}$ such that 
\begin{enumerate}
\item $X_- = f^{-1}(b_0)$ and $X_+ = f^{-1}(b_m)$ are the sets of minimum resp maximum of $f$.
\item $b_i$ are regular values of $f$ such that any connected component $C \subset f^{-1}(b_{i-1}, b_i)$ is diffeomorphic to $D^3$, $S^2 \times I$, $S^3 \setminus \operatorname{int}(\sqcup^3 D^3)$, or $p\setminus \operatorname{int}(D^3 \sqcup D^3)$ where $p$ is an oriented irreducible 3-manifold.
\end{enumerate}
$C$ is called an \textit{elementary bordism} if $C$ is diffeomorphic to $D^3$ or $S^3 \setminus \operatorname{int}(\sqcup^3 D^3)$. $C$ is called a \textit{prime bordism} if $C$ is diffeomorphic to $p\setminus \operatorname{int}(D^3 \sqcup D^3)$ for $p\not \cong S^2\times S^1$ oriented irreducible. $C$ is \textit{cylindrical} if $C$ is diffeomorphic to $S^2\times I$. Note elementary bordisms admit Morse datum with only one critical point (the standard Morse function), while a cylindrical one admits datum with no critical point. Any Morse datum on a prime bordism must necessarily have more than one critical point.

We frequently suppress $b$ when discussing the prime Morse datum $f,b$ if it is clear from context which prime decomposition it induces.
\end{defn}

\begin{defn} \label{defn:Prime decomposition}
The \textit{prime decomposition} of a bordism $M$ compact oriented connected 3 manifold with boundary $X_-, X_+$ is a decomposition  \[
M = M_1 \cup_{X_1} M_2 \cup_{X_2} \dots \cup_{X_{m-1}} M_m
.\]
such that $M_i$ is diffeomorphic to a disjoint union of $D^3$, $S^2 \times I$, $S^3 \setminus \operatorname{int}(\sqcup^3 D^3)$, or $p\setminus \operatorname{int}(D^3 \sqcup D^3)$ for $p\not \cong S^2 \times S^1$ oriented irreducible. Additionally we require that $M_1 \cap \partial M = X_+$ and $M_m \cap \partial M = X_-$. 
\end{defn}

\begin{defn} \label{defn:Prime composition}
(Prime composition) of a morphism $[M]$ in $\operatorname{Cob}(3)_{S^2} $ is a collection of morphisms $[M_1], [M_2], \dots , [M_m]$ such that 
\begin{enumerate}
\item each$[M_i]$ is a tensor product of $1$, $tr$, $id$, $m$, $m^\vee$, and $[p\setminus \operatorname{int}(D^3 \sqcup D^3)]$ for $p$ oriented irreducible $3$-manifold.
\item the collection composes to $[M] = [M_1]\circ [M_2] \dots \circ [M_m]$. 
\end{enumerate}
\end{defn}

\begin{claim}
\begin{enumerate}
\item Any prime Morse datum $f, b$ on $M$ induces a prime decomposition of $M$. each factor is $M_{m-i+1}= f^{-1}(b_{i-1},b_i)$.
\item Any prime decomposition of a bordism $M$ induces a prime composition of $[M]$ such that $M_i \in [M_i]$.
\item Any prime decomposition of $M$ comes from some prime Morse datum -- once fixed some ``standard'' Morse functions for each $M_i$, and glue them together into a Morse function on $M$.
\item Any prime composition of a morphism comes from a prime decomposition of a representative.
\end{enumerate}
\end{claim}

\textbf{General strategies on showing a certain relation is true} 
Suppose we have two composition sequences, say one sequence composes to $\left[ M \right] $ and the other to $\left[ M' \right] $, and we want to show that they are equal. There are two strategies: one at the level of prime decompositions, one at the level of Morse data. 

At the level of prime decompositions, one can show that two chosen representatives $M, M'$ of $[M]$ and $[M']$ have prime decompositions inducing the corresponding compositions, and that $M$ and $M'$ are diffeomorphic. This is convenient when it is easy to argue that $M\cong M'$. One can alternatively show that that the same manifold has two prime decompositions, but it is harder-- this amounts to explicitly identifying the gluing surfaces on a fixed manifold $M$.

At the level of Morse data, one can start with two Morse functions (or rather, prime Morse data) $f_0 ,f_1 :M\to \mathbb{R}$ on the same manifold, and show that $f_0, f_1$ indeed induce the appropriate Morse decompositions. This process necessarily factors through the alternative method at the level of prime decomposition, but is strictly easier, as it gives explicitly the gluing surfaces which are regular surfaces of the Morse data. Note that if we want a non-trivial relation, $f_0, f_1$ must induce non-isomorphic /diffeomorphic prime decompositions, thus inducing different prime compositions. 

\textbf{The classification of relations take the opposite approach.} To classify all possible relations between generating bordisms is to show that for any two prime compositions of the same bordism class $\left[ M \right] $, one can always turn one composition into the other using a set of known relations. This means that there exists a sequence of prime compositions $(C_i) $ such that $C_{i-1}$ and $C_i$ are the same everywhere except for one instance of a known relation. Since the two compositions result in the same bordism class, we may choose the same representative $M$. Hence at the level of prime decompositions, one needs to show that there exists a sequence $(D_i)$ of prime decompositions of $M$ inducing the composition sequence $(C_i)$ such that $D_i$, $D_{i-1}$ are the same except for one application of a known decomposition diffeomorphism. At the level of Morse data, we need a sequence of Morse datum $(f_i, b_i)$ which induce the prime decompositions $(D_i)$

Let's review what this means in $\operatorname{Cob}(2) $. Let $M$ be a compact oriented bordism and $f_0 ,f_1:M \to \mathbb{R}$ prime data which induce compositions $C_0, C_1$. To show that commutative Frobenius relations are sufficient, one must show that there exists a sequence $(f_0, \dots , f_i, \dots , f_1)$ such that $f_{i-1}$ and  $f_i$ induces decompositions that are identical except related by a diffeomorphism of a certain kind, which in turn induces compositions that are identical except related by a commutative Frobenius relation. By noticing that Cerf theory provides a path from $f_0 $ to $f_1 $ which is excellent Morse for all but finitely many singularities, and that the types of singularities can be shown to induce commutative Frobenius relations, one chooses a path and let the sequence of Morse functions be chosen, one for each excellent Morse segment, along which with suitable choices of $b_i$ the decomposition is unique up to diffeomorphism. One summarizes by saying ``the singularities along the Cerf path exhibit the relations necessary to equate composition $C_0 $ to $C_1 $.''

The situation in $\operatorname{Cob}(3)_{S^2} $ is markedly different due to the additional prime morphisms, especially at the level of Morse data. First, external prime relations, such as legs and waist relations, are induced by Morse functions connected by a Cerf path with multiple singularities. Second, there is no longer a bijection between non-trivial singularities and commutative Frobenius relations. For example, the 0-1 birth death does not induce any relations among commutative Frobenius generators.

To show that commutative Frobenius relations between the $C=\left\{ 1, m \right\} $ and $F= \left\{ tr, m^\vee \right\} $ generators are true, we use the local form of a birth-death singularity to deduce the straight forward decompositions, hence compositions; we analyse the topological constraints on the exchanges of critical values and deduce the corresponding decompositions, which induce the compositions. To show that these relations are sufficient among CF generators, we prove in proposition \ref{prop:keyprop} and proposition \ref{prop:keypropCF} that any two Morse data inducing CF-generated compositions are connected by a path whose singularities split the path into segments, each of which is related to its neighbours by inducing a specific relation among the decompositions. To do so, we invoke the equivalence between the types of singularities and the induced decomposition relations, and show that only specific types of singularities are present on this path. This culminates in corollary \ref{cor:CFonlycommFrob}, which states that there are no more relations between CF generators. 

Then we have to deal with the prime bordisms, which have no counterparts in $\operatorname{Cob}( 2)$. There are some obvious composition relations coming from prime decompositions, which can be induced by Morse data closely related to each other. One such example is \textit{prime commutativity} introduced in \ref{primecommutativity}, which says $p^{\times \times}_1 \circ p^{\times \times}_2 = p^{\times \times}_2 \circ p^{\times \times}_1  $. If $(f,b=0)$ is a prime Morse datum inducing a prime decomposition inducing the composition on the left hand side, then ``flipping the picture upside down,'' post-composing with the map $\mathbb{R}\to \mathbb{R}$ sending $x \mapsto -x$ and keeping $b=0$ fixed will induce the composition on the right hand side. See figure \ref{fig:primecommutativity}.

Even though this relation has very simple geometric interpretations, the singularity types of a path between their Morse functions may be very complicated. Therefore we conclude that prime relations such as prime commutativity are \textit{not} classified by singularity types-- rather, they are collections of singularities along Cerf paths. This makes the classification at the Cerf-theoretic level challenging, but we are able to circumvent it through geometric arguments.

After a brief digression regarding (the lack of) \textit{internal relations}, we prove that \textit{legs and waist relations} \ref{prop:legsandwaistrelations} exist by proving that the prime decompositions, which induces the compositions under consideration, are decompositions of diffeomorphic manifolds. At the level of Morse data, the legs relations are trivial, but the story for the waist relations is similar to that of prime commutativity-- paths between the pairs of Morse functions, which induce the waist relations, will have multiple singularities where the function fails to be excellent Morse. Therefore two Morse data inducing a waist relation is connected by a path with more than one singularity.

The proof that commutative Frobenius relations, prime relations, and legs and waist relations together are sufficient can also be viewed at both the level of prime decompositions and at the level of Morse data. At the level of prime decompositions, one first show that there are ``standard'' form prime decompositions which allows one to isolate the prime morphisms Then we show that the parts generated by CF generators can be brought from one decomposition to the other using proposition \ref{prop:keypropCF}. We may hold the Morse functions over the prime part constant, and glue with the corresponding path to get the total path.



We first outline the strategy for this section. Since both $G_1 $ and $G_2 $ share $C$ and $F$ generators, we first prove that these generators satisfy the usual commutative Frobenius relations like their counterparts in $\operatorname{Cob}(2) $, and most importantly that these relations are necessary and sufficient among morphisms generated by $C$ and $F$. 

More precisely, this means that the graph with $\left\{ \sqcup^n S^2 \right\}\cong \mathbb{Z}_{\geq 0}$ vertices and $C\cup F$ (and tensor products) as directed edges, together with commutative Frobenius relations form a presentation \ref{defn:Generators and Relations} of a subcategory of $\operatorname{Cob}(3)_{S^2}$. This category is denoted $\operatorname{Cob}(3)_{S^2, CF}$ and defined as

\begin{defn} \label{defn:Cob(3)_S^2,CF}
($\operatorname{Cob}(3)_{S^2,CF} $) Let $C=\left\{ m,1 \right\} , F=\left\{ m^\vee, tr \right\} $ be defined as in \ref{notations:CF}. Consider the subcategory of $\operatorname{Cob}(3)_{S^2}$ whose morphisms are those which admit a composition in terms of $C\cup F$ generators and their tensor products. Call this category $\operatorname{Cob}(3)_{S^2,CF} $.
\end{defn}

Next, we show that $P$ generators in $G_2 $ satisfy \textit{prime commutativity} in . The more interesting question is whether the commutative Frobenius generators interact with the new prime generators in some non-trivial way. We first address that there are no ``internal'' commutative Frobenius relations inside a prime generator, and then we show that there are only a handful of ``external'' relations each for $P$ and just one for $L$. These external relations are called ``co/legs and co/waist relations'' for $P$ and ``legs relations'' for $L$. Theorem \ref{mainthmG2} shows that $G_2 $, with commutative Frobenius relations, prime commutativity, co/legs and co/waist relations, form a presentation of $\operatorname{Cob}(3)_{S^2}$. In particular, these relations are sufficient. On the other hand, theorem \ref{mainthmG1} shows that $G_1 $, with commutative Frobenius relations alone, form a presentation of $\operatorname{Cob}(3)_{S^2} $. 

The fact that commutative Frobenius relations alone are sufficient for the $G_1 $ presentation raises the question of whether the same could happen to $G_2 $. This turns out to be mostly the case, except for the co/legs relations. We present a series of propositions which shows that prime commutativity and co/waist relations are implied by the co/legs relations and commutative Frobenius relations, culminating in theorem \ref{thm:trueG2thm}, which shows that $G_2 $, with commutative Frobenius relations and legs relations, form a presentation of $\operatorname{Cob}(3)_{S^2}$. In particular, these relations are necessary.

From a higher perspective, what we aim to achieve in this section is to develop the Cerf theory of 3 manifolds equipped with a notion of ``good'' Morse function, but instead of achieving this goal through the more common Jet-stratification analysis/algebraic approach, we exploit the isomorphism \ref{thm:Cob2Cob3CFIso} between $\operatorname{Cob}(2) $ and $\operatorname{Cob}(3)_{S^2} $ and prove purely geometrically that the relations outlined in the previous paragraphs \textit{are} the Cerf theory we promised.

\section{Existence and Spherical Morse Functions} \label{sec:existenceandsphericalMorse}
\begin{defn} \label{defn:good}
Let $M$ be a compact oriented manifold possibly with boundary $\partial M$ and $f:M\to \mathbb{R}$ a Morse function. We say that $f$ is 
\begin{itemize}
\item ``admissible'' if $f : M \to \left[ a,b \right]  $ such that $a$ and $b$ are regular values and $\partial M = f^{-1}(a) \cup f^{-1}(b)$;
\item ``excellent'' if all critical values are distinct;
\item ``spherical'' if all regular level sets are disjoint unions of spheres.
\end{itemize}
\end{defn}

Standard Cerf theory, introduced in \cite{cerf_stratification_1970}, shows that two such admissible Morse functions $f_0 ,f_1 \in C^\infty(M)$ are connected by a path $f_t$ which fails to be ``excellent'' Morse only at finitely many points. For a short working exposition in the context of topological field theories, see the notes by \cite{FreedLecture23}. Passing through these bad points, the Morse function goes through either a birth-death singularity or a critical values exchange. We are interested in classifying those singularities whose regular level sets are disjoint unions of $S^2$, the defining feature of CF generators.

\begin{prop} \label{prop:CFcommFrob}
$C$ and $F$ generators satisfy commutative Frobenius relations.
\end{prop}

\begin{proof}
We handle the proof in two parts. The first part handles relations due to birth-death singularities, and the second part handles exchanges.

\begin{enumerate}
\item A birth-death singularity is the creation/annihilation of a pair of critical points whose indices are adjacent $(i,i+1)$, $i=0,1, \dots , \dim(M)$. In 3 dimensions, the possible birth-death singularities are between $(0,1)$, $(1,2)$, and $(2,3)$. Near the birth of these singularities, the path $f_t$ has the local form (up to a constant) \[
  f_t(x, y , z) = x^3 - t x \pm y^2 \pm z^2 ,\quad t\in \left[ -1, 1 \right] 		
.\]
More specifically, for $(0,1)$ birth\footnote{For the death version, take $f_{-t}$}, the local form is \[
  f_t(x,y,z) = x^3 - tx + y^2 + z^2
.\]
where the two critical points are $(\pm \sqrt{\frac{t}{3}} ,0,0)$ of index $0,1$ respectively. Geometrically, consider filling one of the missing 3-balls in $S^3 \setminus \operatorname{int}(D^3 \sqcup D^3 \sqcup D^3)$, resulting in $S^3 \setminus \operatorname{int}(D^3 \sqcup D^3)$.

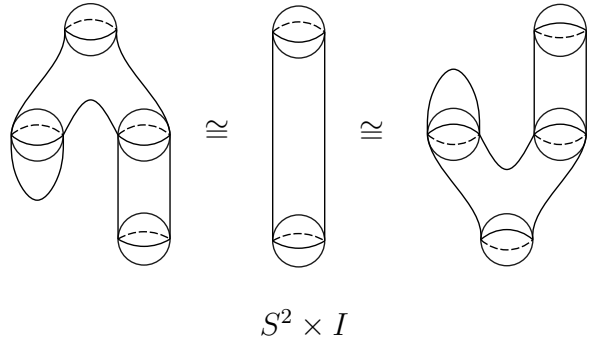
\begin{figure}[ht]
    \centering
\def\svgwidth{1\columnwidth} 
\begingroup%
  \makeatletter%
  \providecommand\color[2][]{%
    \errmessage{(Inkscape) Color is used for the text in Inkscape, but the package 'color.sty' is not loaded}%
    \renewcommand\color[2][]{}%
  }%
  \providecommand\transparent[1]{%
    \errmessage{(Inkscape) Transparency is used (non-zero) for the text in Inkscape, but the package 'transparent.sty' is not loaded}%
    \renewcommand\transparent[1]{}%
  }%
  \providecommand\rotatebox[2]{#2}%
  \newcommand*\fsize{\dimexpr\f@size pt\relax}%
  \newcommand*\lineheight[1]{\fontsize{\fsize}{#1\fsize}\selectfont}%
  \ifx\svgwidth\undefined%
    \setlength{\unitlength}{680.31496063bp}%
    \ifx\svgscale\undefined%
      \relax%
    \else%
      \setlength{\unitlength}{\unitlength * \real{\svgscale}}%
    \fi%
  \else%
    \setlength{\unitlength}{\svgwidth}%
  \fi%
  \global\let\svgwidth\undefined%
  \global\let\svgscale\undefined%
  \makeatother%
  \begin{picture}(1,0.33333333)%
    \lineheight{1}%
    \setlength\tabcolsep{0pt}%
    \put(0,0){\includegraphics[width=\unitlength,page=1]{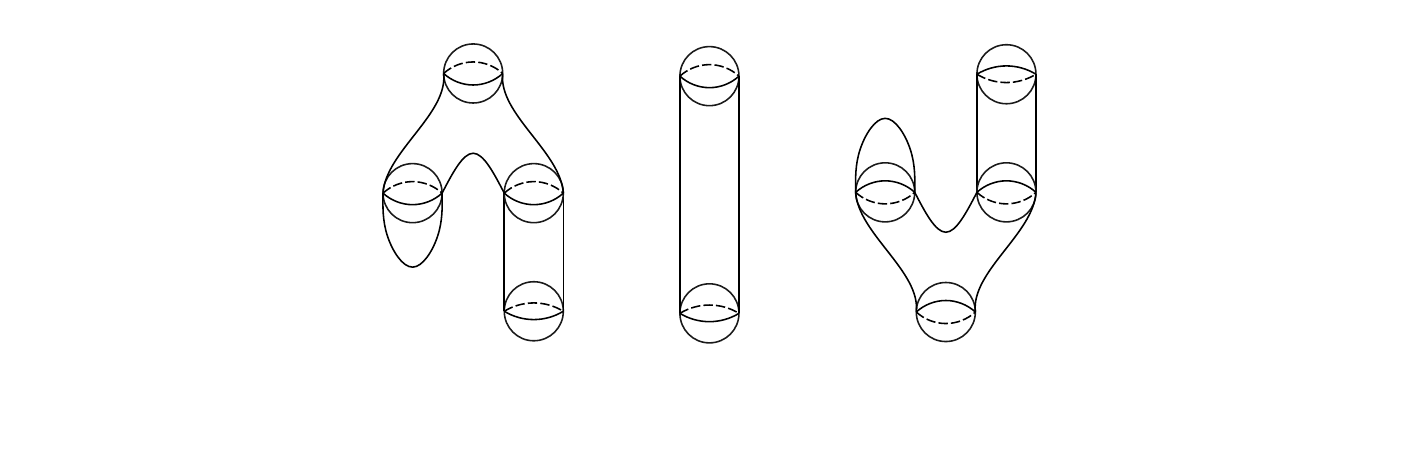}}%
    \put(0.42515872,0.19250444){\makebox(0,0)[lt]{\lineheight{1.25}\smash{\begin{tabular}[t]{l}$\cong$\end{tabular}}}}%
    \put(0.54908555,0.19137529){\makebox(0,0)[lt]{\lineheight{1.25}\smash{\begin{tabular}[t]{l}$\cong$\end{tabular}}}}%
    \put(0.46941468,0.03702234){\makebox(0,0)[lt]{\lineheight{1.25}\smash{\begin{tabular}[t]{l}$S^2 \times I$\end{tabular}}}}%
  \end{picture}%
\endgroup%

    \caption[Birth and Death of $(0,1)$ and $(2,3)$ critical points]{The birth and death of a pair of critical points of index $(0,1)$ and $(2,3)$. The first equality, when read from left to right, shows the death of critical points of indices $0$ and $1$. One can imagine pushing the left cone up until the bottom tip is level with the crouch of the pair of pants, then smoothly deform the result into a straight cylinder. When read from right to left, it depicts the birth of critical points of indices $0$ and $1$. The second equality when read from left to right depicts the birth of $2$ and $3$. }
    \label{fig:01death}
\end{figure}

The former is a representative of the pair of pants bordism $m$, while the latter is diffeomorphic to $S^2\times I$ which is a representative for the identity morphism $id_{S^2}$. The filling 3-ball is a representative of the unit bordism $1$. The first equality in figure \ref{fig:01death} provides a familiar geometric intuition. The local form of the birth of $(0,1)$ critical points encodes the geometric fact that $1$ is the unit with respect to the multiplication: \[
  m \circ (1 \otimes id_{S^2}) \cong id_{S^2}
.\]

The local form for the birth of $(2,3)$ critical points is \[
  f_t(x,y,z) = x^3 - tx -y^2 -z^2
.\]
which corresponds to the same picture but turned upside down as illustrated in figure \ref{fig:01death} as the second equality, encoding the geometric fact that $tr$ is the counit with respect to co multiplication\footnote{Strictly speaking, the left side needs to be pre-composed with a unitor. We shall ignore the unitors and associators for this part of the proof for clarity. However much of the rest of the paper work with them more carefully and explicitly.}: \[
  (tr \otimes id_{S^2})\circ m^\vee \cong id_{S^2}
.\]


\item This part of the proof is inspired by \cite{FreedLecture23}. By a classic theorem in Morse theory \cite{hirsch_differential_1976} Thm 3.4 and 3.5, if $f:\widehat{M}\to \left[ a,b \right] $ is a Morse function on a closed manifold $\widehat{M}$, and $v_k$ is the number of critical points of $f$ of index $k$, then \[
  \chi (\widehat{M}) = \sum_{k=0}^{\dim(M)}  (-1)^k v_k
.\]

We want to classify all possible connected bordisms before and after exchanges of two critical values. We demand such a bordism has two critical points at all times along the Cerf path $f_t$. Let $(i,j)$ denote the indices of the critical points whose critical values are exchanged. We know that neither $i$ nor $j$ can be $0$ -- indeed, at the singular time when the critical values coincide, the critical level set must be connected. Otherwise it takes a handle attachment to connect the index $0$ critical point to the other critical point, which is only possible if $M$ has at least 3 critical points. 

A similar argument shows that neither $i$ nor $j$ can be $3$.

So there are only three possible exchanges we need to classify:  indices $(1,1)$, $(1,2)$, and $(2,2)$ exchanges. The $(2,1)$ exchange is the reversal of the $(1,2)$ exchange by taking $f_{-t}$. 

Since $m$ and $m^\vee$ are the generators of index $1$ and $2$ in CF, we need only consider relations involving these two generators. They respectively decreases/increases the number of boundary components by $1$. By using only two of them, the only possible changes in incoming to outgoing boundary components is $1\to 1$, $2\to 2$, $3\to 1$, and $1\to 3$.

Let $M$ be such a bordism, and let $\widehat{M}$ be the closed manifold one obtains by filling in the boundary components. Then \[
  0= \chi(\widehat{M}) = v_0 -v_1 +v_2 -v_3 
.\]
where $v_0 $ is the number of incoming components and $v_3 $ is the number of outgoing components. \footnote{Secretely using the fact that one can glue two Morse functions together}

For the case of $3\to 1$, we learn that $-v_1 +v_2=-2$, and since we only have two critical points on $M$ at all times, the critical points being exchanged must have indices $(1,1)$. Therefore we are looking for two combinations of two $m$ 's which are both bordisms with three incoming and one outgoing boundary components. There are only two, illustrated in figure \ref{fig:assocandcoassoc}-(a), which are the three dimensional analogue of the associativity condition in two dimensions: \[
  m \circ (m \otimes id_{S^2}) \cong m \circ (id_{S^2} \otimes m)
.\]

\begin{figure}[ht]
    \centering
\def\svgwidth{1\columnwidth} 
\begingroup%
  \makeatletter%
  \providecommand\color[2][]{%
    \errmessage{(Inkscape) Color is used for the text in Inkscape, but the package 'color.sty' is not loaded}%
    \renewcommand\color[2][]{}%
  }%
  \providecommand\transparent[1]{%
    \errmessage{(Inkscape) Transparency is used (non-zero) for the text in Inkscape, but the package 'transparent.sty' is not loaded}%
    \renewcommand\transparent[1]{}%
  }%
  \providecommand\rotatebox[2]{#2}%
  \newcommand*\fsize{\dimexpr\f@size pt\relax}%
  \newcommand*\lineheight[1]{\fontsize{\fsize}{#1\fsize}\selectfont}%
  \ifx\svgwidth\undefined%
    \setlength{\unitlength}{680.31496063bp}%
    \ifx\svgscale\undefined%
      \relax%
    \else%
      \setlength{\unitlength}{\unitlength * \real{\svgscale}}%
    \fi%
  \else%
    \setlength{\unitlength}{\svgwidth}%
  \fi%
  \global\let\svgwidth\undefined%
  \global\let\svgscale\undefined%
  \makeatother%
  \begin{picture}(1,0.5)%
    \lineheight{1}%
    \setlength\tabcolsep{0pt}%
    \put(0,0){\includegraphics[width=\unitlength,page=1]{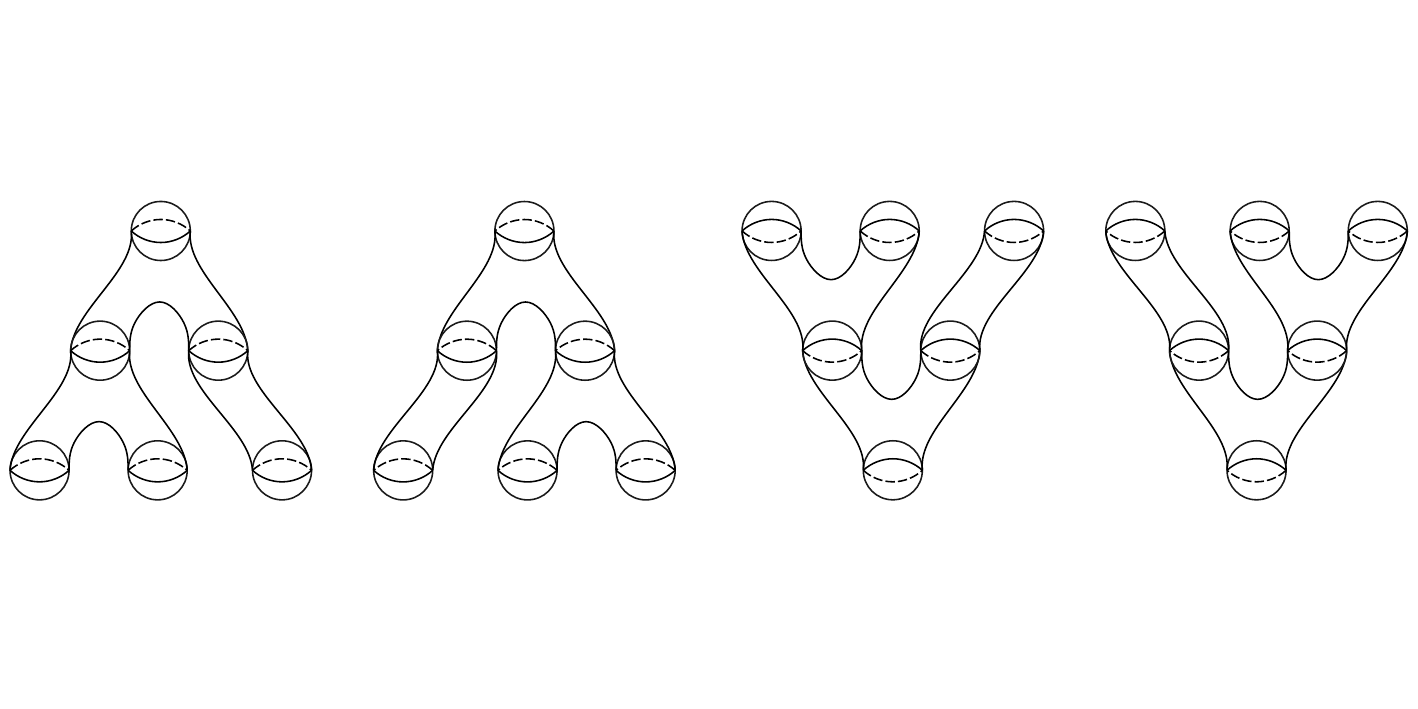}}%
    \put(0.23594689,0.24806544){\makebox(0,0)[lt]{\lineheight{1.25}\smash{\begin{tabular}[t]{l}$\cong$\end{tabular}}}}%
    \put(0.74022967,0.24800498){\makebox(0,0)[lt]{\lineheight{1.25}\smash{\begin{tabular}[t]{l}$\cong$\end{tabular}}}}%
    \put(0.23574162,0.072133){\makebox(0,0)[lt]{\lineheight{1.25}\smash{\begin{tabular}[t]{l}(a)\end{tabular}}}}%
    \put(0.74750907,0.072133){\makebox(0,0)[lt]{\lineheight{1.25}\smash{\begin{tabular}[t]{l}(b)\end{tabular}}}}%
  \end{picture}%
\endgroup%

    \caption[Associativity and Co-associativity]{(a) Associativity of multiplication. (b) Coassociativity of comultiplication.}
    \label{fig:assocandcoassoc}
\end{figure}
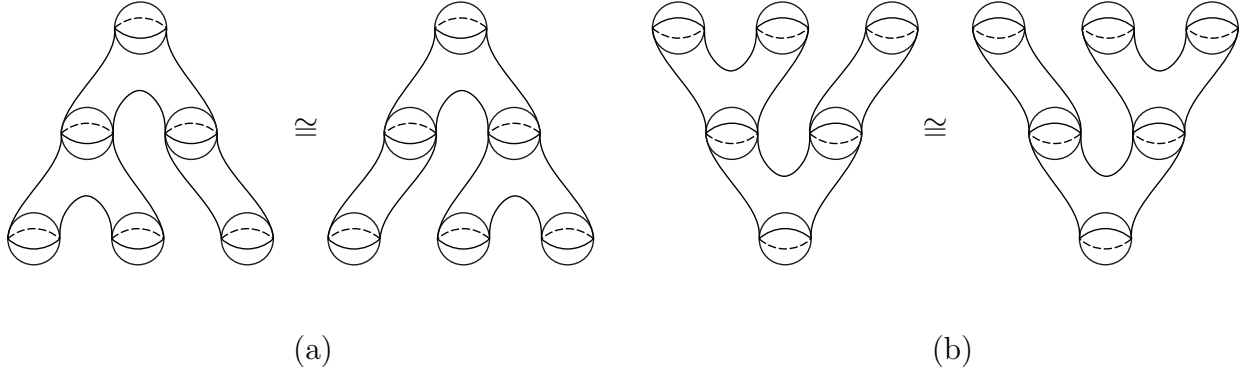
For the case of $1\to 3$, we learn that $-v_1 +v_2 =2$, meaning the exchange is of type $(2,2)$, so we are looking for ways to compose two $m^\vee$ which results in one incoming and two outgoing boundary components. The only two possible compositions illustrated in figure \ref{fig:assocandcoassoc}-(b) are the three dimensional analogue of coassociativity in two dimensions: \[
  (m^\vee \otimes id_{S^2}) \circ m^\vee \cong (id_{S^2} \otimes m^\vee) \circ m^\vee
.\]


For the case of $2\to 2$, we learn that $-v_1 +v_2 =0$, so the exchange is of type $(1,2)$. The only two possible compositions of one  $m$ and one $m^\vee$ which result in a bordism from two incoming to two outgoing boundary components are the analogue of the Frobenius relation: \[
  m^\vee \circ m \cong (m \otimes id_{S^2})\circ  (id_{S^2} \otimes m^\vee)
.\]

This is depicted in figure \ref{fig:12frob}.
\begin{figure}[ht]
    \centering
\def\svgwidth{1\columnwidth} 
\begingroup%
  \makeatletter%
  \providecommand\color[2][]{%
    \errmessage{(Inkscape) Color is used for the text in Inkscape, but the package 'color.sty' is not loaded}%
    \renewcommand\color[2][]{}%
  }%
  \providecommand\transparent[1]{%
    \errmessage{(Inkscape) Transparency is used (non-zero) for the text in Inkscape, but the package 'transparent.sty' is not loaded}%
    \renewcommand\transparent[1]{}%
  }%
  \providecommand\rotatebox[2]{#2}%
  \newcommand*\fsize{\dimexpr\f@size pt\relax}%
  \newcommand*\lineheight[1]{\fontsize{\fsize}{#1\fsize}\selectfont}%
  \ifx\svgwidth\undefined%
    \setlength{\unitlength}{680.31496063bp}%
    \ifx\svgscale\undefined%
      \relax%
    \else%
      \setlength{\unitlength}{\unitlength * \real{\svgscale}}%
    \fi%
  \else%
    \setlength{\unitlength}{\svgwidth}%
  \fi%
  \global\let\svgwidth\undefined%
  \global\let\svgscale\undefined%
  \makeatother%
  \begin{picture}(1,0.29166667)%
    \lineheight{1}%
    \setlength\tabcolsep{0pt}%
    \put(0,0){\includegraphics[width=\unitlength,page=1]{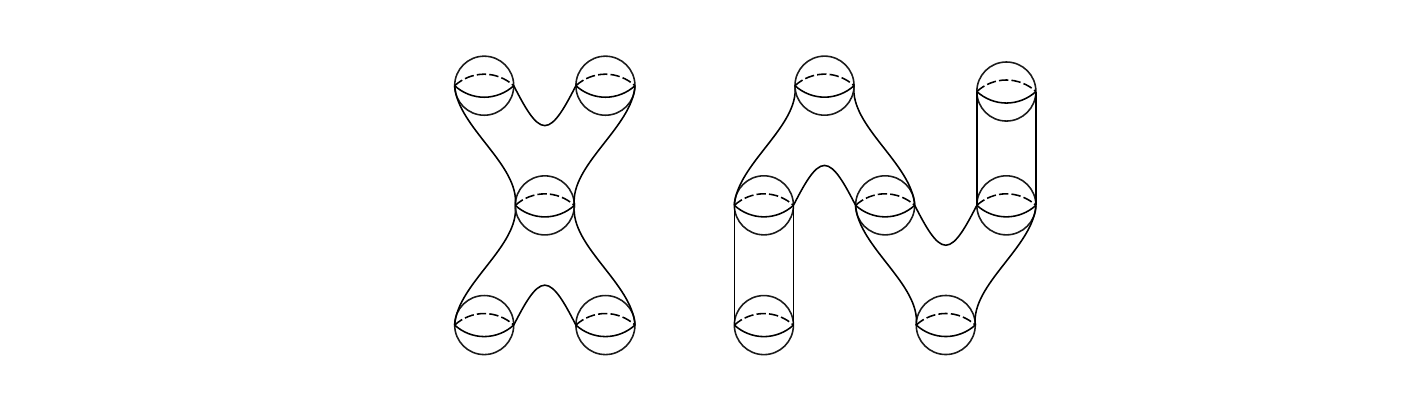}}%
    \put(0.46167628,0.14224078){\makebox(0,0)[lt]{\lineheight{1.25}\smash{\begin{tabular}[t]{l}$\cong$\end{tabular}}}}%
  \end{picture}%
\endgroup%

    \caption[Frobenius relation]{Frobenius relation}
    \label{fig:12frob}
\end{figure}


\end{enumerate}

So far we have demonstrated that CF has units, counits, associativity, coassociativity, and the Frobenius relation. By \cite{kock_frobenius_2003} these are equivalent to the relations which define a Frobenius algebra, namely these bordisms form the Frobenius PROP.  We have co/commutativity as well, upgrading to a commutative Frobenius relation. All together we see that CF generators satisfy commutative Frobenius relations.
\end{proof}

\begin{rmk} \label{rmk:12Bnotspherical}
It is worth drawing attention to the fact that even though type $(1,2)$ birth-death singularities are allowed by generic Cerf theory along path $f_t$ connecting two Morse functions $f_{0,1}:M\to \mathbb{R}$, it does not constitute a relation among the CF generators-- these generators, along with their pre-equipped Morse functions, have spherical regular surfaces. However the local form of type $(1,2)$ birth singularity is given by \[
f_t(x,y,z) = x^3 - tx -y^2 + z^2 ,\quad t\in \left[ -1,1 \right] 
.\]
For $t<0$ there are no critical points. $t=0$ has a degenerate critical point at the origin. For $t>0$ there are two critical points $(\pm \sqrt{\frac{t}{3}} ,0,0)$, of index $1$ and $2$ respectively. Let $r \in f_{t}(\pm \sqrt{\frac{t}{3}} ,0,0)$ be a regular value, one can show that the regular level surface $f_t ^{-1}(r) $ has genus at least 1 (see figure \ref{fig:12b}). Therefore there is no type $(1,2)$ birth-death relations between CF generators.

\begin{figure}[ht]
    \centering
\def\svgwidth{1\columnwidth} 
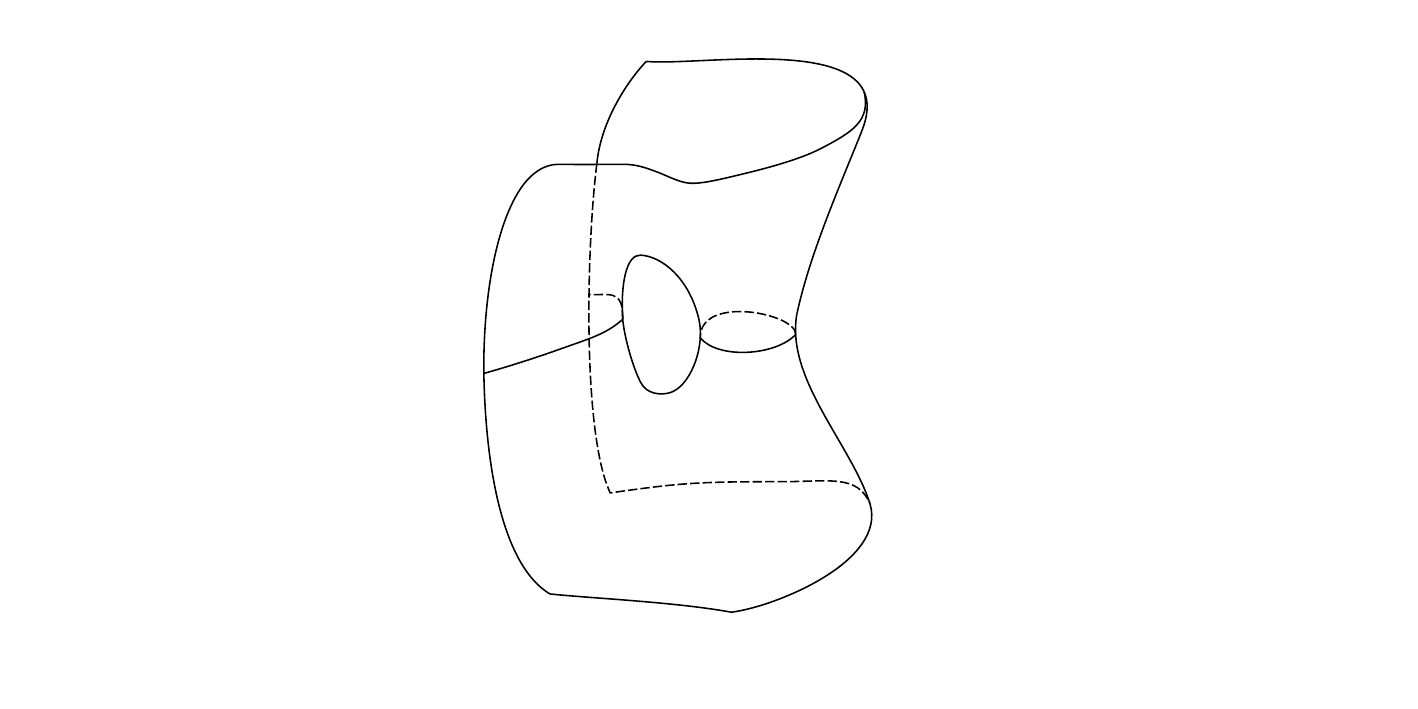

    \caption[The local picture of the birth of a pair of critical points of index $1$ and $2$]{The local picture of the regular surface $f^{-1}(r)$ after the birth of a pair of critical points $c_{1,2}$ of index $1$ and $2$, where $f(c_1) < r < f(c_2)$. Notice that this surface has genus at least 1, hence is not spherical.}
    \label{fig:12b}
\end{figure}

Cerf-theoretically, this means the type $(1,2)$ birth-death singularities cannot occur if we demand the path $f_t$ between two ``good'' Morse functions $f_0 ,f_1 $ remains ``good'' Morse. By ``good'' Morse here we mean spherical admissible Morse functions-- whose regular level sets are all disjoint unions of two-spheres. 
\end{rmk}

\begin{rmk} \label{S2xS1PantsvsHeegaard}
It is also worth noting that we did not treat the type $1 \to 1$ exchanges, which has $-v_1 +v_2 =0$, meaning the critical values being exchanged are of indices $(1,2)$. There is only one such bordism which is a composition of CF generators, analogous to the pants decomposition of the torus in two dimensions in figure \ref{fig:s2s2pantsvsheegaard}-(a) : \[
\Big[ (S^2 \times S^1 )\setminus \operatorname{int}(D^3 \sqcup D^3) \Big]  = m \circ m^\vee
.\]
\begin{figure}[ht]
    \centering
\def\svgwidth{1\columnwidth} 
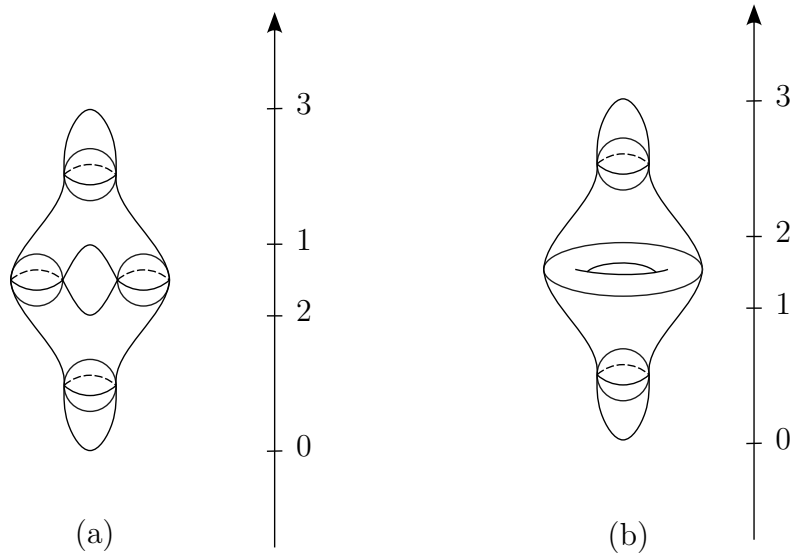

    \caption[Pants and Heegaard decomposition of $S^2 \times S^1$]{Two different decompositions of $S^2 \times S^1$. Figure (a) is the usual ``pants'' decomposition which resembles the pants decomposition of a torus in 2-dimensions. A regular surface $f^{-1}(r)$, for any $r$ between the index $2$ and $1$ critical values, is the disjoint union $S^2 \sqcup S^2$. In contrast, figure (b) is the (minimal) Heegaard decomposition, where a regular surface $f^{-1}(r)$, for any $r$ between the index $1$ and $2$ critical values, is a torus. }
    \label{fig:s2s2pantsvsheegaard}
\end{figure}

To understand what happens if one insists on a $(1,2)$ exchange to occur, consider the following Morse function introduced earlier in remark \ref{rmk:S2xS1redundant}\[
S^2 \times S^1 \overset{\text{proj}}{\longrightarrow } S^1 \overset{\text{Height}}{\longrightarrow} \mathbb{R}
.\]

The standard decomposition of this Morse function induces the composition \[
[S^2 \times S^1] = tr \circ m \circ m^\vee \circ 1
\]
where the elementary bordisms are of index $3, 1, 2, 0$ respectively, in descending critical value . If we exchange the critical values of the critical points whose indices are $(1,2)$, the indices would read $3,2,1,0$ in descending critical value. But the only way to attach a 1-handle to a ball would increase the bounding surface genus by 1. So the regular surface between index 2 and 1 is a torus, and this Morse function gives the (minimal) Heegaard splitting \ref{thm:HeegaardSplitting} of $S^2\times S^1$, illustrated in figure \ref{fig:s2s2pantsvsheegaard}-(b).


While this type of exchange is perfectly normal in generic Cerf theory, it does not give us a relation between CF generators, which requires that the Morse functions before and after have no regular level set which has a non-zero genus component.

We will say more about this in example \ref{eg:S2xS1pantstoHeegaardtopants} after introducing an important lemma.
\end{rmk}

%

\begin{lemma}
Let $\left[ M \right] $ be a morphism in $\operatorname{Cob}(3)_{S^2,CF}$. Then there exists an admissible excellent spherical Morse function on a representative $M$.
\end{lemma}

\begin{proof}
Each CF generator comes with a standard Morse function which is spherical. Since $ M $ comes with a decomposition in terms of glueing CF generators, we can glue their Morse functions together in a collared neighbourhood of the glueing boundaries after appropriate shifts. The resulting Morse function is by definition spherical and admissible on $M$. If two critical values coincide, take a small perturbation of the Morse function, rendering the values distinct and the function excellent.
\end{proof}

\begin{rmk}
The lemma addresses the existence of an admissible excellent spherical Morse function. It is always possible that there is a different composition, giving rise to a different admissible excellent spherical Morse function on $M$. Given that these compositions are the Morse data of two Morse functions, we know the full theory of Cerf can be employed. Therefore these two compositions must be related by a sequence of Cerf moves. However, the example below shows that we cannot maintain the ``spherical-ness'' of the Cerf path in 3 dimensions. Therefore it raises the question whether CF relations are sufficient. The next two sections ensure that they are indeed sufficient.
\end{rmk}

\begin{eg} \label{eg:S2xS1pantstoHeegaardtopants}
($S^2\times S^1$ Pants to Heegaard to Pants) Take any pants decomposition of $S^2 \times S^1$, induced by a Morse function we call $f_{\text{pants}}$. There is a path from $f_{\text{pants}}$ to itself, which first exchanges the order of the 1-handle and 2-handle attachments, then exchanges the order again, undoing the first switch. As we explained in remark \ref{S2xS1PantsvsHeegaard}, a Morse function along the path after the first exchange has a toric regular level set between the index $1$ and $2$ critical values, hence the interval between the two exchanges are aspherical, corresponding to the minimal Heegaard splitting of $S^2\times S^1$. Figure \ref{fig:kirby} presents the Kirby graphics of this Cerf path.

\begin{figure}[ht]
    \centering
\def\svgwidth{1\columnwidth} 
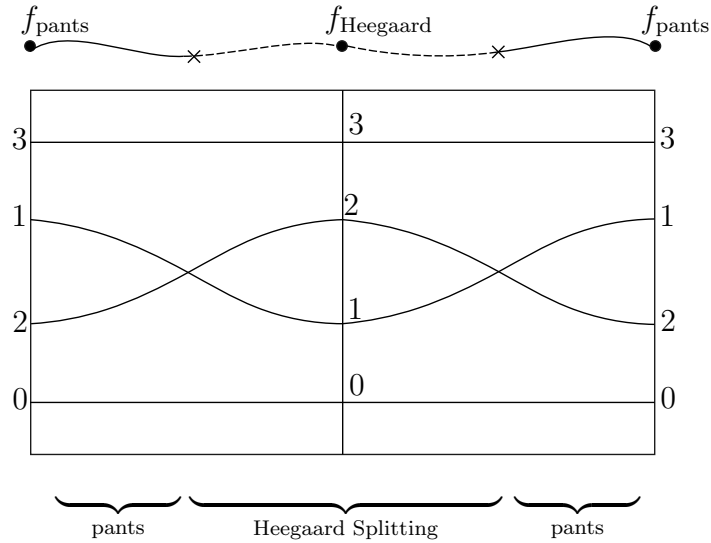

\caption[The Kirby graphics for a path $f_t$ which fails to remain spherical]{The Kirby graphics for a path at $f_{\text{pants}}: S^2 \times S^1$ which fails to remain spherical as it exchanges the index $1$ and $2$ critical points. After the exchange the Morse function along the path segment now gives the genus 1 Heegaard splitting of $S^2\times S^1$.}
    \label{fig:kirby}
\end{figure}
\end{eg}

\section{Spherical Morse Functions}\label{sec:Cob2Cob3Iso}
Before we proceed to show that the commutative Frobenius relations are sufficient among CF generators, it is necessary to examine their Morse theory. We will prove in this section that there is an isomorphism between the categories $\operatorname{Cob}( 3)_{S^2,CF} $ and $\operatorname{Cob}( 2)$ in theorem \ref{thm:Cob2Cob3CFIso}. Furthermore we show that $G_1 $ and $G_2 $ are both minimal in proposition \ref{prop:G1G2minimal}.

\begin{prop} \label{prop:sphericalnoirreducible}
Let $M$ be a closed oriented 3-manifold, and $f:M\to \mathbb{R}$ a Morse function such that for every regular value $r\in \mathbb{R}$, $f^{-1}(r) \cong \sqcup S^2$. Namely, $f$ is a spherical Morse function-- a Morse function whose regular surfaces are disjoint unions of $S^2$. Then 
\begin{enumerate}
\item $\pi_1(M)$ is free, and
\item $M$ is diffeomorphic to either $S^3$ or $\#^r \left( S^2\times S^1 \right) $ with $r\geq 1$ the rank of $\pi_1(M)$.
\end{enumerate}
\end{prop}

\begin{proof}
Let $R_f$ be the Reeb graph of $f:M\to \mathbb{R}$. Recall that for a Morse function $f$ on a compact manifold $M$, the Reeb graph $R_f$ is a finite graph such that
\begin{itemize}
\item the vertices of $R_f$ are the critical level components
\item the edges of $R_f$ are copies of intervals of regular values 
\end{itemize}

Let $q: M \to M/\sim=R_f$ be the quotient map. Then the Morse function factors through $q$:
\[\begin{tikzcd}
M & {M/\sim=R_f} & {\mathbb{R}}
\arrow["q", from=1-1, to=1-2]
\arrow["f"', curve={height=20pt}, from=1-1, to=1-3]
\arrow["{\exists \tilde{f}}", dashed, from=1-2, to=1-3]
\end{tikzcd}\]

Between consecutive critical values $a<b$, standard Morse theory shows that $f^{-1}((a,b)) \cong f^{-1}(t) \times (a,b)$ for any regular value $t\in (a,b)$.  Since our regular surfaces are disjoint unions of  $S^2$, we have by definition \[
q^{-1}(e) \cong  S^2 \times (a,b) 
.\]
for any edge $e$ which corresponds to the regular interval $(a,b)$. See figure \ref{fig:egreeb} for an example.

\begin{figure}[ht]
    \centering
\def\svgwidth{1\columnwidth} 
\begingroup%
  \makeatletter%
  \providecommand\color[2][]{%
    \errmessage{(Inkscape) Color is used for the text in Inkscape, but the package 'color.sty' is not loaded}%
    \renewcommand\color[2][]{}%
  }%
  \providecommand\transparent[1]{%
    \errmessage{(Inkscape) Transparency is used (non-zero) for the text in Inkscape, but the package 'transparent.sty' is not loaded}%
    \renewcommand\transparent[1]{}%
  }%
  \providecommand\rotatebox[2]{#2}%
  \newcommand*\fsize{\dimexpr\f@size pt\relax}%
  \newcommand*\lineheight[1]{\fontsize{\fsize}{#1\fsize}\selectfont}%
  \ifx\svgwidth\undefined%
    \setlength{\unitlength}{680.31496063bp}%
    \ifx\svgscale\undefined%
      \relax%
    \else%
      \setlength{\unitlength}{\unitlength * \real{\svgscale}}%
    \fi%
  \else%
    \setlength{\unitlength}{\svgwidth}%
  \fi%
  \global\let\svgwidth\undefined%
  \global\let\svgscale\undefined%
  \makeatother%
  \begin{picture}(1,0.5)%
    \lineheight{1}%
    \setlength\tabcolsep{0pt}%
    \put(0,0){\includegraphics[width=\unitlength,page=1]{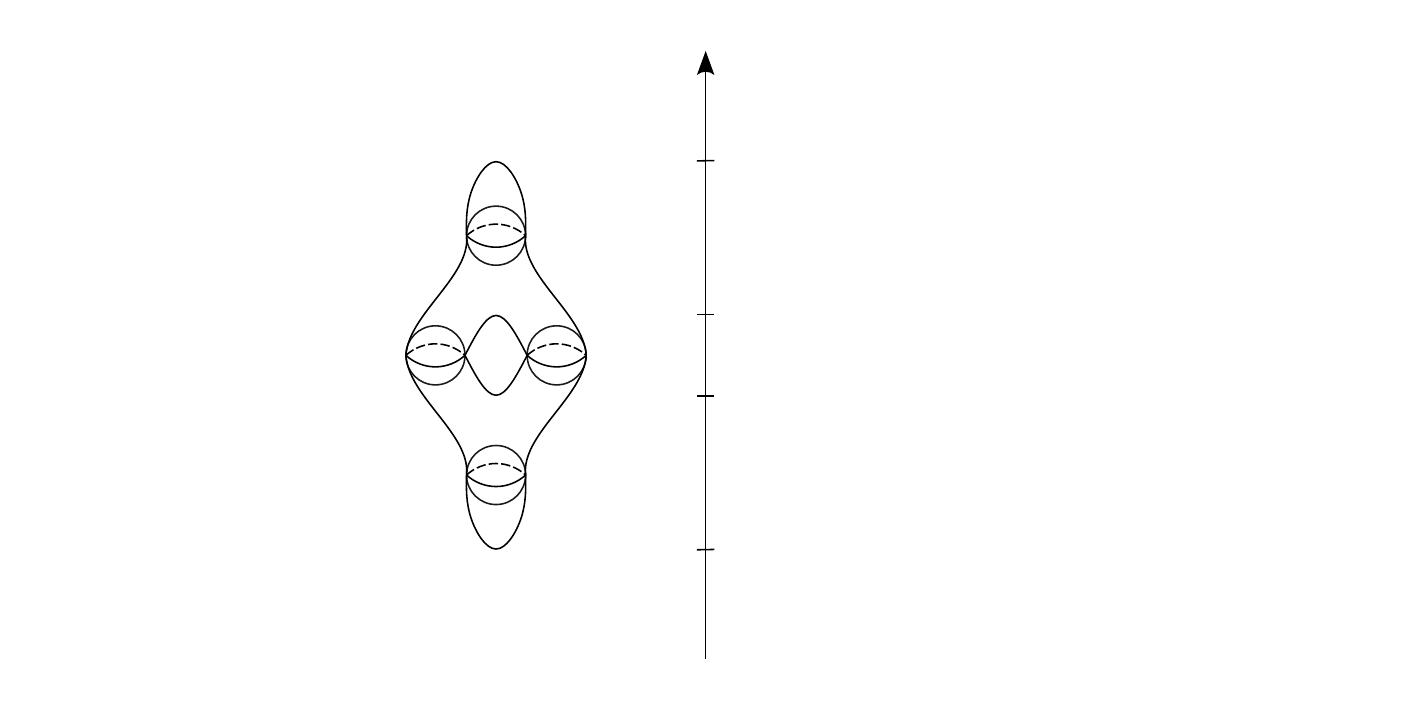}}%
    \put(0.51509171,0.10792796){\makebox(0,0)[lt]{\lineheight{1.25}\smash{\begin{tabular}[t]{l}$0$\end{tabular}}}}%
    \put(0.51509171,0.21616567){\makebox(0,0)[lt]{\lineheight{1.25}\smash{\begin{tabular}[t]{l}$2$\end{tabular}}}}%
    \put(0.51509171,0.27357705){\makebox(0,0)[lt]{\lineheight{1.25}\smash{\begin{tabular}[t]{l}$1$\end{tabular}}}}%
    \put(0.51509161,0.38232295){\makebox(0,0)[lt]{\lineheight{1.25}\smash{\begin{tabular}[t]{l}$3$\end{tabular}}}}%
    \put(0,0){\includegraphics[width=\unitlength,page=2]{egreeb.pdf}}%
    \put(0.622052,0.01564776){\makebox(0,0)[lt]{\lineheight{1.25}\smash{\begin{tabular}[t]{l}$R_f$\end{tabular}}}}%
    \put(0.30865307,0.01564776){\makebox(0,0)[lt]{\lineheight{1.25}\smash{\begin{tabular}[t]{l}$S^2 \times S^1 $\end{tabular}}}}%
    \put(0.48091755,0.01577009){\makebox(0,0)[lt]{\lineheight{1.25}\smash{\begin{tabular}[t]{l}$\mathbb{R}$\end{tabular}}}}%
    \put(0,0){\includegraphics[width=\unitlength,page=3]{egreeb.pdf}}%
    \put(0.42314704,0.0378952){\makebox(0,0)[lt]{\lineheight{1.25}\smash{\begin{tabular}[t]{l}$f$\end{tabular}}}}%
  \end{picture}%
\endgroup%

    \caption[A Reeb graph example]{An example of a Reeb graph $R_f$ associated to a Morse function $f$. Here $f:S^2 \times S^1 \to \mathbb{R}$ is the Morse function described in \ref{rmk:S2xS1redundant}, which is the projection onto the $S^1$ factor and then taking the height function.}
    \label{fig:egreeb}
\end{figure}

At a vertex $v$, the local Morse model in 3 dimensions shows the preimage of a small neighbourhood of $v$ is obtained from a union of $3$-balls by attaching a single $k$-handle, for $k\in \left\{ 0, \dots ,3 \right\} $. Since the regular level sets just above $v$ and just below are union of spheres, the only allowed local changes are: 
\begin{itemize}
\item Birth/death of a spherical component (when $v$ corresponds to an index $0$ or $3$ critical point)
\item Merge/split of spherical component(s) that never creates a torus (when $v$ corresponds to an index $1$ or $2$ critical point).
\end{itemize}

In particular, the preimage of a small vertex neighbourhood is simply connected. Indeed, a union of $3$-balls with $1/2$-handle attached along does not create $\pi _1$. The preimage of an open edge is $S^2\times (a,b)$, which is also simply connected.

\textbf{Claim} $\pi_1 (M)$ is a free group.

Cover $M$ by preimages $\left\{ U_v \right\} $ of vertex neighbourhoods and $\left\{ U_e \right\} $ of open edge neighbourhoods in $R_f$. Then each $U_v, U_e$ is simply connected.

By Seifert-van Kampen, $\pi_1 (M) = \pi_1(\text{the nerve of this cover})$. But the nerve of this cover is canonically $R_f$. So \[
\pi_1(M) \cong \pi_1 (R_f) \cong <e-v+1 \text{ many generators}>
.\]
Most importantly, we learn that $\pi_1 (M)$ is a free group.

By prime decomposition, any closed oriented $3$-manifold $M$ splits uniquely as a connected sum of prime summands, and $\pi_1 (M)=$ free product of $\pi_1 (\text{prime summands})$. In particular, if $\pi_1 (M)$ is free of rank $r$, then \[
M \cong \# ^r (S^2\times S^1) \# N
.\]
where $N$ is simply connected. By the Poincare conjecture, $N \cong S^3$.
\end{proof}

\begin{cor}
Let $M$ be a closed oriented prime $3$-manifold, and $f:M \to \mathbb{R}$ a spherical Morse function. Then $M$ is diffeomorphic to either $S^3$ or $S^2 \times S^1$.
\end{cor}

%

\begin{rmk}
The spirit of the proposition and the corollary is: no matter how hard one tries, one cannot glue CF generators into a prime 3-manifold, except for $S^3$ and $S^2\times S^1$.
\end{rmk}

\begin{cor} \label{cor:hominCF}
Every connected bordism $[M]$ in $\operatorname{Cob}(3)_{S^2, CF}$ has a representative $M$ which is diffeomorphic to either $S^3$ with the interior of a finite number of disjoint $D^3$ removed, or to $\#^r S^2 \times S^1$ with the interior of a finite number of disjoint $D^3$ removed.
\end{cor}

\begin{proof}
First cap off the incoming and outgoing boundaries of $M$, resulting in a closed oriented 3-manifold $\widehat{M}$. By definition, $CF$ generates morphisms in $\operatorname{Cob}(3)_{S^2, CF}$, so the standard spherical Morse functions on the $CF$ generators induces a spherical Morse function on $\widehat{M}$, $f : \widehat{M} \to \mathbb{R}$. Then by proposition \ref{prop:sphericalnoirreducible}, $\widehat{M} \cong S^3$ or $\widehat{M} \cong \#^r \left( S^2 \times S^1 \right) $. Hence $M \cong S^3 \setminus \left( \operatorname{int}(\sqcup D^3) \right) $ or $M \cong \#^r \left( S^2 \times S^1 \right)  \setminus \left( \operatorname{int}(\sqcup D^3) \right) $.
\end{proof}

\begin{defn} \label{defn:Sillyfunctor}
Define the \textit{silly functor} $S$ as the  map \[
S: \operatorname{Cob}( 2) \to \operatorname{Cob}( 3)_{S^2, CF}
.\]
sending
\begin{enumerate}
\item A disjoint union $\sqcup^n S^1$ to a disjoint union $\sqcup^n S^2.$
\item The unit morphism $1: \emptyset \to S^1$ to $1: \emptyset \to S^2$.
\item The trace morphism $tr: S^1 \to \emptyset $ to $tr: S^2 \to \emptyset$.
\item The multiplication $m: S^1 \sqcup S^1 \to S^1$ to $m: S^2 \sqcup S^2 \to S^2$.
\item The comultiplication $m^\vee: S^1\to S^1 \sqcup S^1$ to $m^\vee: S^2 \to S^2 \sqcup S^2$.
\item The genus $g$ bordism $[\Sigma_g \setminus int( \sqcup^{n_{inc}+n_{out}}D^2)]: \sqcup^{n_{inc}}S^1 \to \sqcup^{n_{out}}S^1$ to the bordism $ [\#^g(S^2 \times S^1) \setminus int( \sqcup^{n_{inc}+n_{out}}D^3)]: \sqcup^{n_{inc}}S^2 \to \sqcup^{n_{out}}S^2$.
\end{enumerate}
Theorem \ref{thm:Cob2Cob3CFIso} shows that the silly functor is an isomorphism.
\end{defn}

\begin{thm} \label{thm:Cob2Cob3CFIso}
$S: \operatorname{Cob}(2) \to \operatorname{Cob}(3)_{S^2,CF} $ is an isomorphism.
\end{thm}

\begin{proof}
%
By \cite{steinebrunner2026surfacecategorytropicalcurves}, $\pi_0: \operatorname{Cob}(2) \to Csp(\mathbb{Z}_{\geq 0},1)$ is an equivalence of symmetric monoidal labelled cospan categories. Slightly abusing the notations, our strategy is to construct another equivalence $\pi_0: \operatorname{Cob}(3)_{S^2, CF} \to Csp(M,[S^2 \times S^1])  $ of symmetric monoidal labelled cospan categories, where $Csp(\mathcal{M}, [S^2 \times S^1])$ is canonically isomorphic to $Csp(\mathbb{Z}_{\geq 0},1)$. Let its inverse be $F: Csp(\mathcal{M}, [S^2 \times S^1]) \to \operatorname{Cob}(3)_{S^2, CF}$. Then we show that the composite \[
\operatorname{Cob}(2) \overset{\pi_0}{\to } Csp(\mathbb{Z}_{\geq 0},1) \overset{\sim}{\to} Csp(\mathcal{M}, [S^2 \times S^1]) \overset{F}{\to } \operatorname{Cob}(3)_{S^2, CF}
.\]
is the desired functor $S: \operatorname{Cob}(2) \to \operatorname{Cob}(3)_{S^2, CF}$.

First, recall the definition of the cospan category of finite sets $Csp$. The category $Csp$ has objects finite sets and morphisms $A \to B$ is an isomorphism class of cospans $[A \to X \leftarrow B]$, where two cospans are isomorphic if there is a bijection $X \cong X'$ that is compatible with the maps from $A$ and from $B$. Disjoint union equips $Csp$ with a symmetric monoidal structure, with the empty set as the monoidal identity. If $(A, +)$ is an abelian monoid and $\alpha  \in A$, then we can define the category $Csp(A,\alpha )$ with finite sets as objects, and an arrow $A\to B$ is an equivalent class of cospans and labels
\[\begin{tikzcd}
	& A & \\
	& X \\
	A && B
	\arrow["l"', from=2-2, to=1-2]
	\arrow[from=3-1, to=2-2]
	\arrow[from=3-3, to=2-2]
\end{tikzcd}\]
The composition law of $Csp(A,\alpha )$ is rigged for the equivalence $ \pi_0:\operatorname{Cob}( 2)\to Csp(\mathbb{Z}_{\geq 0},1) $ to hold. The element $1=\alpha $ is the element that controls how much genus it increases when closing a loop via a composition. For details we refer the reader to \cite{steinebrunner2026surfacecategorytropicalcurves}.

Let $\left( \mathcal{M} = \left\{ [\#^r S^2 \times S^1] \mid r=0,1,\dots  \right\} , \# \right) $ be the abelian monoid under connect sum, whose elements are diffeomorphism classes of connect sums $[\#^r S^2 \times S^1]$. Note that $\mathcal{M}$ is canonically isomorphic to the abelian monoid $(\mathbb{Z}_{\geq 0}, +)$. Therefore there is canonical isomorphism \[
Csp(\mathbb{Z}_{\geq 0},1) \overset{\sim}{\to} Csp(\mathcal{M}, [S^2 \times S^1]) 
.\]
sending a label $l: S \to \mathbb{Z}_{\geq 0}$ to the label $l': S \to \mathbb{Z}_{\geq 0} \overset{\sim}{\to} \mathcal{M}$.

The functor $\pi_0: \operatorname{Cob}(3)_{S^2, CF} \to Csp(\mathcal{M},[S^2 \times S^1])  $ is defined as follows:
\begin{enumerate}
\item An object $\coprod_n S^2$ of $\operatorname{Cob}(3)_{S^2, CF} $ is mapped to $\pi_0 (\coprod_n S^2)$, the set of $n$ elements.
\item A bordism $W: M \to N$ is mapped to the cospan
\[\begin{tikzcd}
& {\pi_0(W)} & \\
{\pi_0(M)} && {\pi_0(N)}
\arrow[from=2-1, to=1-2]
\arrow[from=2-3, to=1-2]
\end{tikzcd}\]
and the label $l: \pi_0 (W) \to \mathcal{M}$ sends a connected component $W' : M' \to N'$ to the diffeomorphism class of the closed manifold \[
\widehat{W'} \cong \left( \sqcup_{|\pi_0 (M')|} D^3 \right) \cup W' \cup \left( \sqcup_{|\pi_0 (N')|} D^3 \right) 
.\]
where the unions are taken at the gluing boundaries. Note that $\widehat{W}$ is indeed diffeomorphic to a connect sum of $S^2 \times S^1$ by corollary \ref{cor:hominCF}.

\item Composition is well defined because the definition of $Csp(M, [S^2 \times S^1])$ is rigged specifically to do so.
\end{enumerate}

$\pi_0 $ is obviously compatible with the symmetric monoidal structure on both sides, so it's a symmetric monoidal functor. We want to show that this is an equivalence. Recall that a functor $F: \mathcal{C} \to \mathcal{D}$ of labelled cospan categories is an equivalence of categories if and only if it satisfies 

\begin{enumerate}
\item For any connected object $M_D \in \mathcal{D}$, there exists a connected object $M_C \in \mathcal{C}$ such that $F(M_C) \cong M_D$.
\item for all $M,N \in \mathcal{C}$, $F$ induces a bijection \[
\operatorname{Hom}_{\mathcal{C}}^{conn}(M,N) \cong \operatorname{Hom}_{\mathcal{D}}^{conn}(F(M),F(N)) 
.\]
\end{enumerate}

For the functor $\pi_0 : \operatorname{Cob}(3)_{S^2, CF} \to Csp(\mathcal{M}, [S^2\times S^1])$ we have defined bove, condition 1 is trivial. Hence we need to check that $\pi_0 $ induces a bijection \[
\operatorname{Hom}_{\operatorname{Cob}(3)_{S^2, CF}}^{conn}( \sqcup^{n_{inc}} S^2, \sqcup^{n_{out}} S^2) \cong \operatorname{Hom}_{Csp(\mathcal{M}, [S^2\times S^1])}^{conn}(\pi_0(\sqcup^{n_{inc}} S^2), \pi_0(\sqcup^{n_{out}} S^2))  
.\]
The left hand side is the set of diffeomorphism class of connected bordisms $[W]$ with spherical boundaries. The right hand side are all connected cospans $\pi_0(\sqcup^{n_{inc}} S^2) \to X \leftarrow \pi_0(\sqcup^{n_{out}} S^2)$ with label $l: X \to \mathcal{M}$. By definition, a cospan  $A \to X \leftarrow B$ of finite sets is connected if $X$ is a set with one elements. Hence a connected morphism belonging to the right hand side is an element $l(*) \in \mathcal{M}$. Therefore the right hand side is simply isomorphic to $\mathcal{M}$.

Under this identification, $\pi_0$ induces the map \[
\operatorname{Hom}_{\operatorname{Cob}(3)_{S^2, CF}}^{conn}( \sqcup^{n_{inc}} S^2, \sqcup^{n_{out}} S^2) \to  \left\{ [\#^r \left( S^2 \times S^1 \right) ] \mid r=0, 1, \dots  \right\} 
.\]
sending $[W]$ to $[\widehat{W}]$.

We will construct an inverse. Let $V \cong \#^r \left( S^2 \times S^1 \right) $. Choose an orientation-preserving embedding \[
\iota: \left( \sqcup^{n_{inc}} D^3 \right) \sqcup \left( \sqcup^{n_{out}} D^3  \right)  \hookrightarrow V
.\]
Then $V \cong \widehat{W_{\iota}}$, where $W_\iota = V \setminus im(\iota)^\circ$. For any other such embedding $\iota'$, there exists a diffeomorphism $\psi: V \to V$ with $\iota ' = \psi \circ \iota$, and such that $\psi|_{W_\iota} : W_\iota \to W_{\iota'}$ is a boundary and orientation preserving diffeomorphism.\cite[Ch8 Thm 3.2]{hirsch_differential_1976}. Therefore the map \[
[V] \mapsto [W_{\iota}]
.\]
is independent on the choice of $\iota$. It is clearly the inverse to the assignment $[W] \mapsto [\widehat{W}]$.

Hence we have shown that the map \[
\operatorname{Hom}_{\operatorname{Cob}(3)_{S^2, CF}}^{conn}( \sqcup^{n_{inc}} S^2, \sqcup^{n_{out}} S^2) \to  \left\{ [\#^r \left( S^2 \times S^1 \right) ] \mid r=0, 1, \dots  \right\} 
.\]
is a bijection. Therefore \[
\operatorname{Hom}_{\operatorname{Cob}(3)_{S^2, CF}}^{conn}( \sqcup^{n_{inc}} S^2, \sqcup^{n_{out}} S^2) \cong \operatorname{Hom}_{Csp(\mathcal{M}, [S^2\times S^1])}^{conn}(\pi_0(\sqcup^{n_{inc}} S^2), \pi_0(\sqcup^{n_{out}} S^2))  
.\]
is also a bijection, and $\pi_0 : \operatorname{Cob}(3)_{S^2, CF} \to Csp(\mathcal{M}, [S^2\times S^1])$ is an equivalence of labelled cospan categories, in particular an equivalence of categories.

We can construct a functor $F: Csp(\mathcal{M}, [S^2 \times S^1]) \to \operatorname{Cob}(3)_{S^2, CF}$ as the (quasi-)inverse to $\pi_0 $ in the following way:
\begin{enumerate}
\item $F$ sends an order $n$ finite set to $\sqcup^n S^2$. In particular, $F$ sends the empty set to the empty manifold.
\item For a labelled cospan $f: A \to X \leftarrow B :g $ with label $l: X \to \mathcal{M}$,  its image under $F$ is the bordism $W: \sqcup^{|A|}S^2 \to \sqcup^{|B|} S^2$ with $|X|$-many connected component, such that over the element $x\in X$, the corresponding component denoted $W_x$ is diffeomorphic to \[
	W_x \cong l(x) \setminus \Big( \underset{b\in g^{-1}(x)}{\coprod_{a \in f^{-1}(x)}} D^3 \Big) 
.\]
\end{enumerate}

Combining all the relevant functors together, we see that the composition \[
S: \operatorname{Cob}(2) \overset{\pi_0}{\to } Csp(\mathbb{Z}_{\geq 0},1) \overset{\sim}{\to} Csp(\mathcal{M}, [S^2 \times S^1]) \overset{F}{\to } \operatorname{Cob}(3)_{S^2, CF}
.\]
sends
\begin{enumerate}
\item A disjoint union $\sqcup^n S^1$ to a disjoint union $\sqcup^n S^2.$
\item The unit morphism $1: \emptyset \to S^1$ to $1: \emptyset \to S^2$.
\item The trace morphism $tr: S^1 \to \emptyset $ to $tr: S^2 \to \emptyset$.
\item The multiplication $m: S^1 \sqcup S^1 \to S^1$ to $m: S^2 \sqcup S^2 \to S^2$.
\item The comultiplication $m^\vee: S^1\to S^1 \sqcup S^1$ to $m^\vee: S^2 \to S^2 \sqcup S^2$.
\item The genus $g$ bordism $[\Sigma_g \setminus int( \sqcup^{n_{inc}+n_{out}}D^2)]: \sqcup^{n_{inc}}S^1 \to \sqcup^{n_{out}}S^1$ to the bordism $ [\#^g(S^2 \times S^1) \setminus int( \sqcup^{n_{inc}+n_{out}}D^3)]: \sqcup^{n_{inc}}S^2 \to \sqcup^{n_{out}}S^2$.
\end{enumerate}

\end{proof}


\begin{proof}
	(Alternative) An alternative and easier proof is the following: Let $S: \operatorname{Cob}(2)\to \operatorname{Cob}(3)_{S^2, CF}$ be given by 
\begin{enumerate}
\item $S$ sends  $\amalg^n S^1$ to  $\amalg^n S^2$.
\item $S$ sends the unique genus $g$ bordism $\amalg^{n_{inc}}S^1 \to \amalg^{n_{out}} S^1 $ to the unique bordism $[W]: \amalg^{n_{inc}}S^2 \to \amalg^{n_{out}}S^2$ where $\widehat{W} \cong \#^g (S^2 \times S^1)$. 
\end{enumerate}
Then $S(1)=1$, $S(tr)=tr$, $S(m)=m$, and $S(m^\vee) = m^\vee$. By the classification of morphisms in $\operatorname{Cob}( 3)_{S^2,CF} $ in corollary \ref{cor:hominCF}, $S$ induces a bijection \[
\operatorname{Hom}_{\operatorname{Cob}(2)}\Big(\coprod^{n_{inc}}S^1, \coprod^{n_{out}}S^1\Big) \to \operatorname{Hom}_{\operatorname{Cob}(3)_{S^2, CF}}\Big(\coprod^{n_{inc}}S^2, \coprod^{n_{out}}S^2\Big) 
.\]
sending the unique genus $g$ bordism to the unique bordism with $g$-many $S^2 \times S^1$ factors.
\end{proof}

\begin{prop} \label{prop:G1G2minimal}
Both $G_1 $ and $G_2 $ are minimal.
\end{prop}

\begin{proof}
Let $P$ be an oriented irreducible 3-manifold. Consider $[P] = tr \circ [P\setminus \operatorname{int}(D^3)] = tr \circ [P\setminus \operatorname{int}(D^3 \sqcup D^3)] \circ 1$. By proposition \ref{cor:hominCF}, if $[P \setminus \operatorname{int}(D^3)]$ or $[P\setminus \operatorname{int}(D^3 \sqcup D^3) ]$ were generated by commutative Frobenius generators, then $P$ is either $S^3$ or $S^2\times S^1$, which contradicts the premise. Uniqueness of prime decomposition shows that a fixed prime unit cannot be generated by other prime units, hence all prime units are necessary. Theorem \ref{thm:Cob2Cob3CFIso} shows that the commutative Frobenius generators are necessary. Hence $G_1 $ is minimal. The same argument shows that $G_2 $ is also minimal.
\end{proof}

\section{Local Models of Excellent Morse Failures}\label{sec:localmodels}
Proposition \ref{prop:CFcommFrob} gives the existence of commutative Frobenius relations among CF generators. We would like to know if they are sufficient-- namely, are there other necessary relations if one wants to give a presentation of the category $\operatorname{Cob}( 3)_{S^2,CF} $? This section introduces a tool which we call \textit{local models of failures}, which the next section uses alongside theorem \ref{thm:Cob2Cob3CFIso} to prove sufficiency.

\begin{defn} \label{defn:Local Model}
(Local Model) Let $f:M\to \mathbb{R}$ be a Morse function. A local model of critical points and indices $((m_j,i_j))_{j=1}^{n}$ is a connected submanifold $U\subset  M$ along with a Morse decomposition induced by $f|_U$: \[
U \cong U_1 \cup_{X_1} U_2 \cup_{X_2 } \cup \dots \cup_{X_{n-1}} U_n \subset  M
.\]
such that each $U_j$ is an elementary cobordism with critical point $m_j \in U_j$ of index $i_j$ satisfying \[
f(m_1)\leq f(m_2)\leq \dots \leq f(m_n)
.\]

If it is clear from context which critical points $(m_j)$ and which submanifold $U$ are being considered, we say the local model is of $(i_j)$ on $U$. Note that a local model on $U$ by definition is equipped with a inducing Morse function defined on at least $U$.

We say two local models $U \cong \cup _{i\in I} U_i, V \cong \cup_{i \in I} V_i$ are diffeomorphic if there exists a diffeomorphism $h:U \to V$ such that $h(U_j) = V_j$ preserves each boundary components (as a set, not necessarily acting as the identity on each boundary component) and $h$ sends critical points to corresponding critical points. 
\end{defn}

\begin{conventions}
This is a subtle convention. Given two local models $U_i, U_{i}'$ of the same critical points on $M$, we assume that they are two Morse decompositions of the same underlying connected submanifold $U\subset  M$. This is so that we can speak of a path of functions between two local models: it is a path of functions on a submanifold $U$ between the two Morse functions inducing the two local models $U_i$ and $U_i'$.
\end{conventions}

\begin{defn} \label{defn:Path of functions between two local models}
(Path of functions between two local models) Let $f_0, f_1 : M \to \mathbb{R}$ be Morse functions inducing two different local models on a connected submanifold $U \subset  M$. A path of functions between local models of $U$ is a path $f_t : I \times U \to \mathbb{R}$ from $f_0|_U $ to $f_1|_U$. Note we do not require that $f_t$ be defined on all of the ambient manifold $M$.
\end{defn}

\begin{eg}
For the associativity relation of $S^2$, figure \ref{fig:assocandcoassoc}-(a) gives two local models for $U=M$ and there exists an evident path of local models of $(1,1)$. This path is described by smoothly moving the two index $2$ critical points towards each other in height, and their critical values ultimately coincide and cross. In this case this path of local models is the same as a path of Morse functions on the whole $M$, which exchanges the critical values of the two critical points. Note that the local models are \textit{not} diffeomorphic: it exchanges boundary components.
\end{eg}


\begin{eg}
Consider the two torus $\Sigma_1 = S^1 \times S^1$ with $U$ the shaded region in figure \ref{fig:localmodelbirth}. Equip the torus on the left with the standard height function. In figure (a), $U$ has no critical points, so it is itself a trivial local model. In figure (b), $U$ has an index 1 critical point and an index $2$ critical point. In order to induce this on the same $\Sigma_1$, the Morse function on $\Sigma_1$ is no longer the height function in figure b. However standard Cerf theory shows there is a path $f_t$ between the height function and the Morse function inducing figure (b). When restricted to $U$, this path gives a path $f_t|_U$ of functions between the trivial local model of $U$ and the local model of indices $(1,2)$ of $U$.
\end{eg}

\begin{figure}[ht]
    \centering
\def\svgwidth{1\columnwidth} 
\begingroup%
  \makeatletter%
  \providecommand\color[2][]{%
    \errmessage{(Inkscape) Color is used for the text in Inkscape, but the package 'color.sty' is not loaded}%
    \renewcommand\color[2][]{}%
  }%
  \providecommand\transparent[1]{%
    \errmessage{(Inkscape) Transparency is used (non-zero) for the text in Inkscape, but the package 'transparent.sty' is not loaded}%
    \renewcommand\transparent[1]{}%
  }%
  \providecommand\rotatebox[2]{#2}%
  \newcommand*\fsize{\dimexpr\f@size pt\relax}%
  \newcommand*\lineheight[1]{\fontsize{\fsize}{#1\fsize}\selectfont}%
  \ifx\svgwidth\undefined%
    \setlength{\unitlength}{680.31496063bp}%
    \ifx\svgscale\undefined%
      \relax%
    \else%
      \setlength{\unitlength}{\unitlength * \real{\svgscale}}%
    \fi%
  \else%
    \setlength{\unitlength}{\svgwidth}%
  \fi%
  \global\let\svgwidth\undefined%
  \global\let\svgscale\undefined%
  \makeatother%
  \begin{picture}(1,0.33333333)%
    \lineheight{1}%
    \setlength\tabcolsep{0pt}%
    \put(0,0){\includegraphics[width=\unitlength,page=1]{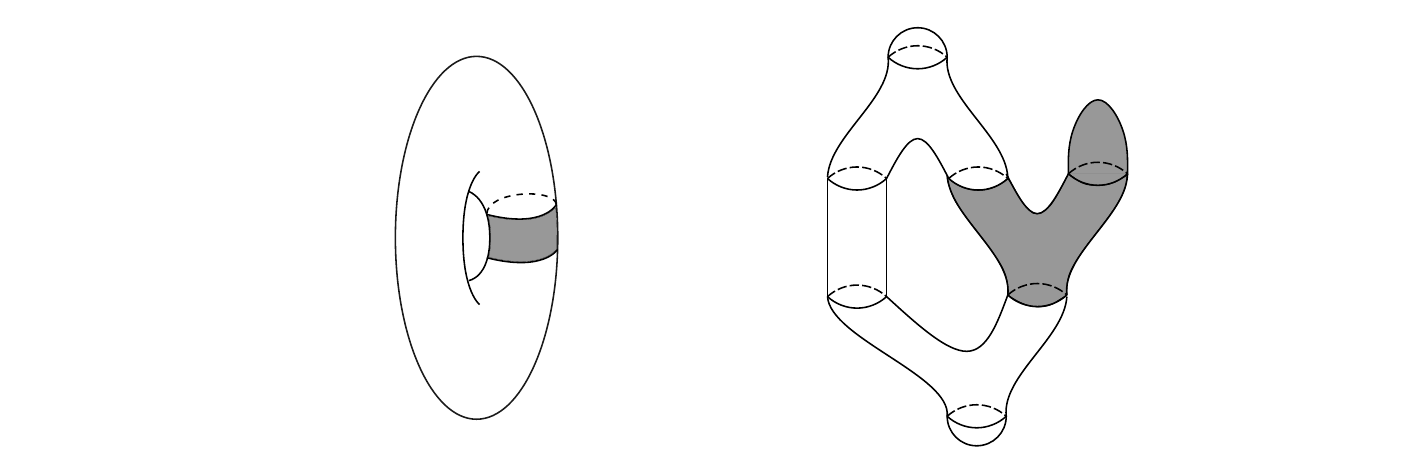}}%
    \put(0.31838802,0.01372356){\makebox(0,0)[lt]{\lineheight{1.25}\smash{\begin{tabular}[t]{l}(a)\end{tabular}}}}%
    \put(0.61373653,0.01909612){\makebox(0,0)[lt]{\lineheight{1.25}\smash{\begin{tabular}[t]{l}(b)\end{tabular}}}}%
    \put(0.40546954,0.16309013){\makebox(0,0)[lt]{\lineheight{1.25}\smash{\begin{tabular}[t]{l}$U$\end{tabular}}}}%
    \put(0.8130764,0.16257617){\makebox(0,0)[lt]{\lineheight{1.25}\smash{\begin{tabular}[t]{l}$U$\end{tabular}}}}%
  \end{picture}%
\endgroup%

    \caption[A local model of the birth of a pair of critical points]{A local model of the birth of a pair of critical points.}
    \label{fig:localmodelbirth}
\end{figure}
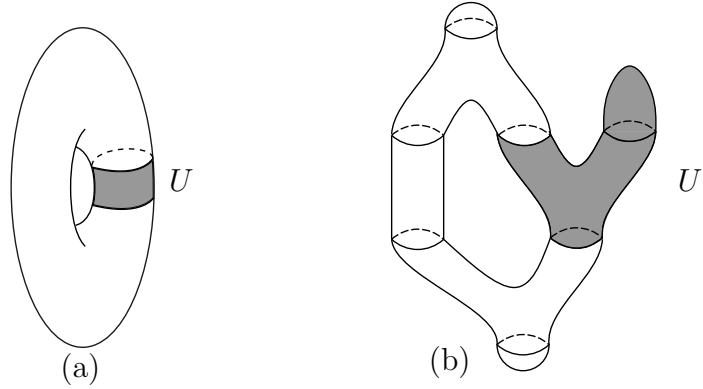

\begin{notations}
Let $f_0 ,f_1 $ be two admissible excellent Morse functions on $M$, and $\gamma :\left[ 0,1 \right] \to C^\infty(M)$ a path from $f_0 $ to $f_1 $. Let $f_t = \gamma (t)$. Then standard Cerf theory applies and $f_t$ is admissible excellent Morse for all but finitely many $t$, say at $\left\{ t_1 ,\dots t_n \right\} $. 
\begin{enumerate}
\item  We say $f_{t_i}$ is of failure type $(j,j+1)_{B}$ or $ (j,j+1)_D $ if two critical points of index $j, j+1$ were created or destroyed at time $t_i$. 
\item We say $f_{t_i}$ is of failure type $(k,l)_X$ if two critical points of index $k,l$ have the same critical value at $f_{t_i}$. The ordering is such that $f_{t_i - \epsilon} (x_k) < f_{t_i - \epsilon}(x_j)$ if $x_k$ is the critical point of index $k$.
\item We say $f_{t_i}$ is of local failure type $(j,j+1)_{B,D}$ if it is of failure type $(j,j+1)_{B,D}$.
\item We say $f_{t_i}$ is of local failure type $(k,l)_{X, M,M'}$ if it is of failure type $(k,l)_X$ and $M,M'$ are local models of $(k,l)$ before and after the exchange.
\item We say $f_{t_i}$ is of trivial local failure type $(k,l)_{X}$ if it is of local failure type $(k,l)_{X,M,M'}$ where $M,M'$ are diffeomorphic local models of $(k,l)$.
\end{enumerate}

In this language, a \textit{relation} between two (non-cylindrical) cobordism generators is either a local failure type $(j,j+1)_{B,D}$ or a non-trivial local failure type $(k,l)_{X,M,M'}$, where $M,M'$ are not the same (i.e. diffeomorphic) local models. To classify relations between CF generators, we want the local models before and after a failure to have only spherical boundaries.
\end{notations}

\begin{rmk}
The need to distinguish failure and local failure is because even though birth/death is purely local, exchanges of critical values are not. However the purpose of introducing local models is so we can isolate a local region and treat the whole region as a whole. Although this is not usually a focal issue in standard Cerf theory or in $\operatorname{Cob}(3)_{S^2, CF}$, presently it is necessary to distinguish $(1,1)_{X}$ exchanges in $\operatorname{Cob}(2) $ by their local models, as illustrated in the following example.
\end{rmk}

\begin{eg}
Consider the standard height function on the two-torus $T=\Sigma_1 \cong S^1 \times S^1$, embedded in  $\mathbb{R}^3$ as in figure \ref{fig:tiltingtorus}. Consider rotating the torus by $\theta(t)$ about the $x$-axis, where $\theta(t) = \pi t$. Then there is a path of maps parametrized by $t$ \[
f_t: T=\Sigma_1 \cong S^1 \times S^1 \hookrightarrow \mathbb{R}^3 \overset{R_x(\theta(t))}{\longrightarrow } \mathbb{R}^3 \overset{proj_z}{\relbar\joinrel\twoheadrightarrow} \mathbb{R}
.\]
such that at $\theta(0)=0$ one recovers the standard height function. As one increases $t$, the critical values of the two index $1$ critical points move towards each other, ultimately coinciding at some $\theta_c(t_c)$. For a time $t=t_c + \epsilon$ slightly after $t_c$, the critical values become distinct again.

\begin{figure}[ht]
    \centering
\def\svgwidth{1\columnwidth} 
\begingroup%
  \makeatletter%
  \providecommand\color[2][]{%
    \errmessage{(Inkscape) Color is used for the text in Inkscape, but the package 'color.sty' is not loaded}%
    \renewcommand\color[2][]{}%
  }%
  \providecommand\transparent[1]{%
    \errmessage{(Inkscape) Transparency is used (non-zero) for the text in Inkscape, but the package 'transparent.sty' is not loaded}%
    \renewcommand\transparent[1]{}%
  }%
  \providecommand\rotatebox[2]{#2}%
  \newcommand*\fsize{\dimexpr\f@size pt\relax}%
  \newcommand*\lineheight[1]{\fontsize{\fsize}{#1\fsize}\selectfont}%
  \ifx\svgwidth\undefined%
    \setlength{\unitlength}{680.31496063bp}%
    \ifx\svgscale\undefined%
      \relax%
    \else%
      \setlength{\unitlength}{\unitlength * \real{\svgscale}}%
    \fi%
  \else%
    \setlength{\unitlength}{\svgwidth}%
  \fi%
  \global\let\svgwidth\undefined%
  \global\let\svgscale\undefined%
  \makeatother%
  \begin{picture}(1,0.33333333)%
    \lineheight{1}%
    \setlength\tabcolsep{0pt}%
    \put(0,0){\includegraphics[width=\unitlength,page=1]{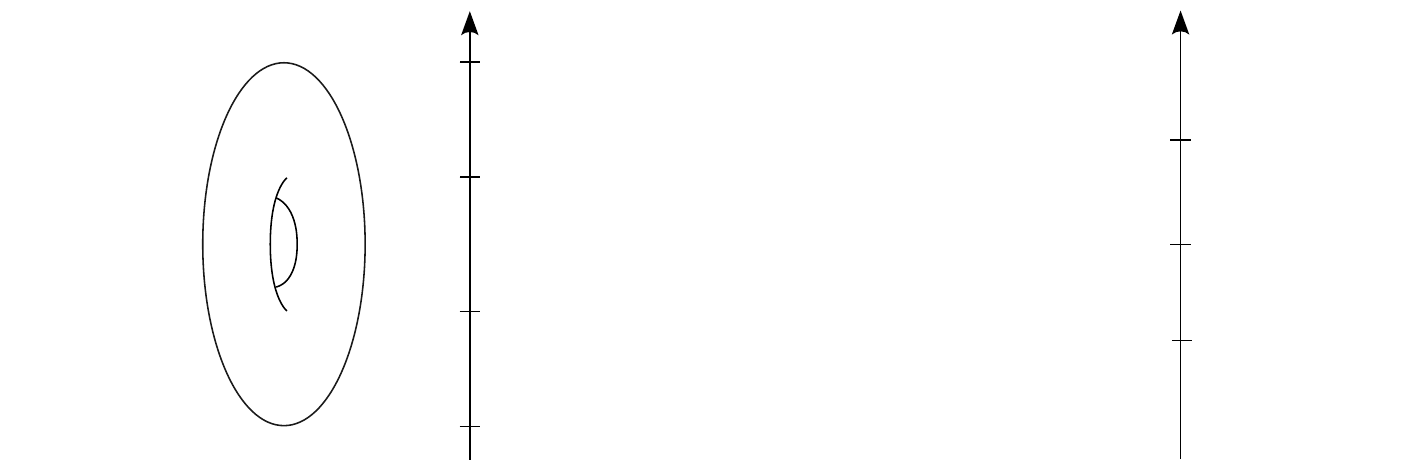}}%
    \put(0.34276403,0.02713814){\makebox(0,0)[lt]{\lineheight{1.25}\smash{\begin{tabular}[t]{l}0\end{tabular}}}}%
    \put(0.34253508,0.10815729){\makebox(0,0)[lt]{\lineheight{1.25}\smash{\begin{tabular}[t]{l}1\end{tabular}}}}%
    \put(0.34187978,0.20294158){\makebox(0,0)[lt]{\lineheight{1.25}\smash{\begin{tabular}[t]{l}1\end{tabular}}}}%
    \put(0.34280351,0.28415934){\makebox(0,0)[lt]{\lineheight{1.25}\smash{\begin{tabular}[t]{l}2\end{tabular}}}}%
    \put(0.8563776,0.22893499){\makebox(0,0)[lt]{\lineheight{1.25}\smash{\begin{tabular}[t]{l}2\end{tabular}}}}%
    \put(0.85830841,0.08662497){\makebox(0,0)[lt]{\lineheight{1.25}\smash{\begin{tabular}[t]{l}0\end{tabular}}}}%
    \put(0.85545387,0.15582581){\makebox(0,0)[lt]{\lineheight{1.25}\smash{\begin{tabular}[t]{l}1,1\end{tabular}}}}%
    \put(0,0){\includegraphics[width=\unitlength,page=2]{tiltingtorus.pdf}}%
  \end{picture}%
\endgroup%

    \caption[The tilting torus]{The tilting torus. The figure on the left is the standard embedding of the 2-torus. As one rotates the torus, the critical values of the index 1 critical points move towards each other, ultimately coinciding as depicted in the right figure. The critical level set is depicted as the two dashed circles wedged together at two points. For a better visualization, we encourage the reader to cut a few donuts or bagles like so in real life. }
    \label{fig:tiltingtorus}
\end{figure}
Note however that $f_{t_c}$ is of trivial local failure type $(1,1)_X$-- the local models of $(1,1)$, before and after the exchange of critical values, are diffeomorphic-- both are of the form \[
\text{pants} \cup_{S^1 \sqcup S^1} \text{copants}
.\]

Therefore in this tilting torus example, the exchange of the critical values of the two index $1$ critical points do not constitute a relation. 
\end{eg}

\begin{eg}
By remark \ref{rmk:12Bnotspherical}, the birth of the pair of critical points of index $(1,2)$ is not a relation between CF generators because the local model after the birth has torus boundaries.
\end{eg}

\begin{eg}
If $\dim(M)=2$, the possible failure types are \[
\left\{ (0,1)_B, (0,1)_D, (1,2)_B, (1,2)_D, (1,1)_X  \right\} 
.\]
The possible relations, aka non-trivial local failure types are \[
\left\{ (0,1)_B, (0,1)_D, (1,2)_B, (1,2)_D, (1,1)_{X, \text{assoc}}, (1,1)_{X, \text{coassoc}}, (1,1)_{X, \text{Frobenius}} \right\} 
.\]
where the pairs of local models are the evident ones.
\end{eg}

\begin{eg}
If $\dim(M)=3$, the possible failure types are \[
\left\{ (0,1)_B, (0,1)_D, (2,3)_B, (2,3)_D, (1,1)_X, (2,2)_X, (1,2)_X  \right\} 
.\]
The possible relations, aka non-trivial local failure types are \[
\left\{ (0,1)_B, (0,1)_D, (2,3)_B, (2,3)_D, (1,1)_{X, \text{assoc}}, (2,2)_{X, \text{coassoc}}, (1,2)_{X, \text{Frobenius}} \right\} 
.\]
where the pairs of local models are the evident ones.
\end{eg}

\section{Sufficiency of Commutative Frobenius Relations Among CF Generators} \label{sec:sufficient}
Recall that there is a \textit{silly functor} defined in \ref{defn:Sillyfunctor}, \[
S: \operatorname{Cob}( 2) \to \operatorname{Cob}( 3)_{S^2, CF}
\]
sending disjoint unions of $n$-copies of $S^1$ to disjoint unions of $n$-copies of $S^2$, and the genus $g$ bordism \[
	[\Sigma_g \setminus int( \sqcup^{n_{inc}+n_{out}}D^2)]: \sqcup^{n_{inc}}S^1 \to \sqcup^{n_{out}}S^1
\]
to the bordism \[
	[\#^g(S^2 \times S^1) \setminus int( \sqcup^{n_{inc}+n_{out}}D^3)]: \sqcup^{n_{inc}}S^2 \to \sqcup^{n_{out}}S^2.
\]

Let's examine what this functor does to local failure models. Unsurprisingly, it maps $\operatorname{Cob}(2) $ non-trivial local failure types, to $\operatorname{Cob}(3)_{S^2,CF}$ relations by 
\begin{align*}
(0,1)_{B,D} &\mapsto (0,1)_{B,D} \\
(1,2)_{B,D} &\mapsto (2,3)_{B,D}\\
(1,1)_{X,\text{assoc}} & \mapsto (1,1)_{X,\text{assoc}}\\
(1,1)_{X,\text{coassoc}}&\mapsto (2,2)_{X,\text{coassoc}}\\
(1,1)_{X,\text{Frobenius}}&\mapsto (1,2)_{X,\text{Frobenius}}
\end{align*}

It also behaves well with the associated Morse data: note that $S([M]\circ [N]) = S([M])\circ S([N])$ and $S([\Sigma_g]) = [\#^g S^2\times S^1]$. In particular, if $M$ is a representative of $[M]$ with Morse function $f$ inducing decomposition $M=M_1 \cup \dots  \cup M_n$, then there exists a manifold $N $ a representative of $S([M])$ with Morse function $F$ inducing decomposition $N=N_1 \cup \dots \cup N_n$ such that the Morse decompositions are mapped to each other:  \[
S([M_i]) = [N_i],\quad  \forall i = 1, \dots , n
.\]

See figure \ref{fig:sillymap} for a simple example. Note that there are infinitely many choices of such $F$ on $N$, and, by composing with suitable monotonically increasing functions, in fact infinitely many $F$ that additionally has the same as critical values as those of $f$. Therefore WLOG we impose the condition that $F$ has the same critical values as those of $f$.

\begin{notations}
For simplicity and abusing the notations, we refer to any representative $N \in S([M])$ as $S(M)$, and any Morse function $F:N \to \mathbb{R}$, with the same critical values as $f$ and inducing $N= N_1 \cup \dots N_n$ such that $S([M_i])=[N_i]$, as $S(f)$.
\end{notations}
\begin{figure}[ht]
    \centering
\def\svgwidth{1\columnwidth} 
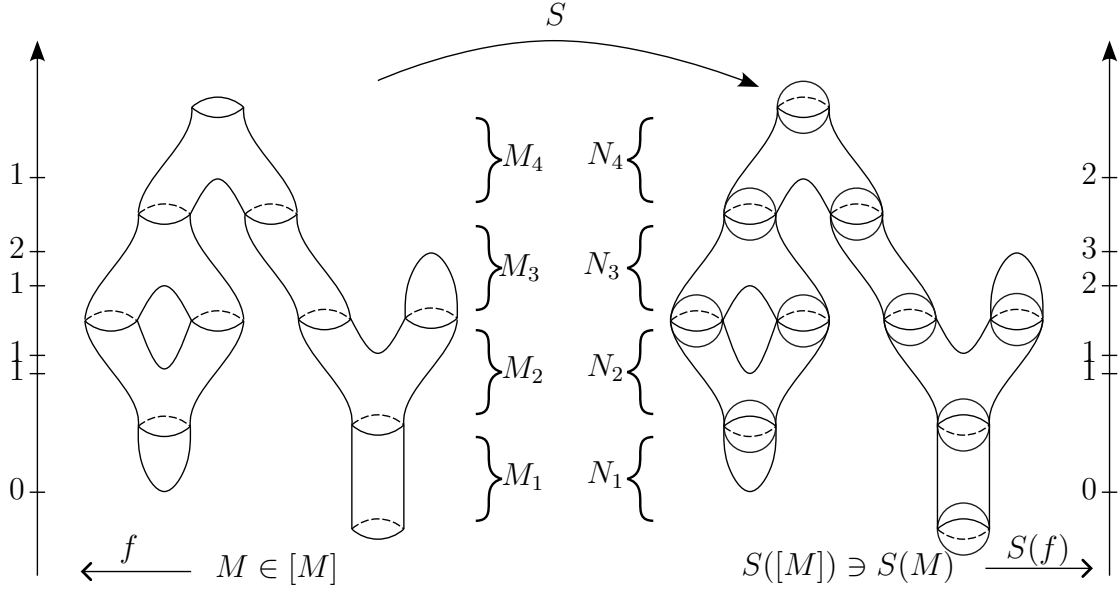

\caption[Morse data between $\operatorname{Cob}( 2) $ and $\operatorname{Cob}( 3)_{S^2} $ ]{The behavior of Morse data under the isomorphism $S: \operatorname{Cob}( 2) \to \operatorname{Cob}( 3)_{S^2}$. $S(f) : S(M) \to \mathbb{R}$ is any Morse function such that it has the same critical values as $f$, and induces decomposition $S(M) = N_1  \cup \dots \cup N_n$ such that $S([M_i])=[N_i]$.}
    \label{fig:sillymap}
\end{figure}

Naively, when one is given a morphism $[M]$ in $\operatorname{Cob}( 2) $ and two different compositions related by a single 2-dimensional commutative Frobenius relation, one may use the silly functor $S$ on $[M]$ and its two compositions to get a morphism $S([M])$ in $\operatorname{Cob}( 3)_{S^2, CF} $ with two compositions related by a 3-dimensional commutative Frobenius relation. 

Because the silly functor $S: \operatorname{Cob}( 2) \to \operatorname{Cob}( 3)_{S^2, CF}$ is an isomorphism, this also works the other way around: given a morphism $[N]$ in $\operatorname{Cob}( 3)_{S^2, CF}$ and two distinct compositions related by a single commutative Frobenius relation, one can use the inverse silly functor to get $S^{-1}([N])$ in $\operatorname{Cob}( 2)$ and two compositions related by a single commutative Frobenius relations.

The next proposition \ref{prop:keyprop} says that this process is repeatable, and in fact has a Cerf-theoretic lift: given two Morse functions $f_0 ,f_1 :M \to \mathbb{R}$ and a path $f_t$ between them, one can construct a path $F_t $ between $F_0 , F_1 :S(M) \to \mathbb{R}$, where $S(M)$ is a representative of $S([M])$, and $F_{0,1}$ is induced by $f_{0,1} $. Along this path except with at finitely many isolated singularities, the function is not only admissible excellent Morse, but importantly also spherical. The isolated singularities also match those (non-trivial) singularities of $f_t$ in the sense that they are the $S$-image of local failures.

The strategy of the proof is simple: to construct $F_t: S(M)\to \mathbb{R}$, follow $f_t:M\to \mathbb{R}$ and use it as a blueprint. More specifically, whenever one encounters a specific singularity along $f_t$, isolates a local region $U$ of the manifold $M$ where the topology change occurs, and consider its image $S(U)$ in $S(M)$. All the commutative Frobenius relations in $\operatorname{Cob}( 3)_{S^2,CF} $ can be described by some basic paths $\tilde{F}_t$, so we declare that over $S(U)$ the path of functions takes the form $\tilde{F}_t$, and outside $U$ the path of functions is constant. Then we can glue the two paths of functions together with a bump function over a collared neighbourhood of $\partial S(U)$.

The reader is warned that the wording of the proposition and technical aspects of the proof may obscure the intuition and the proof strategy outlined above.

\begin{prop} \label{prop:keyprop}
Let $\left[ M \right] $ be a morphism in $\operatorname{Cob}(2) $ and $f_0 ,f_1 :M \to \mathbb{R}$ two admissible excellent Morse functions. Let $f_t$ be a path from $f_0 $ to $f_1 $ where $f_t $ fails to be excellent Morse at time $\left\{ t_1 <t_2 < \dots < t_n \right\} $ of non-trivial local failure type $\left\{ T_1, T_2, \dots , T_n \right\}  $, and possibly of trivial local failure type $(1,1)_X$ at other times $\left\{ \widehat{t}_1< \dots < \widehat{t}_m \right\} $. Then there exists a path $F_t$ connecting $S(f_0 ) ,S(f_1 ) : S(M) \to \mathbb{R} $ which fails to be admissible excellent spherical Morse only at time $\left\{ t_1, \dots , t_n \right\} $ of local failure type $\left\{ S(T_1), \dots , S(T_n) \right\} $.
\end{prop}

\begin{proof}
WLOG we can choose the critical values of $F_0$ and $F_1$ to be the same as those of $f_0 , f_1$. Indeed, for any $\left\{ x_1 <\dots x_k \right\} $ and $\left\{ y_1 <\dots <y_k \right\} $, there exists a monotonically increasing smooth function mapping  $x_i$ to $y_i$, and we can always take a smooth interpolation from the identity to this function.

Take a strict total ordering of all the times that $f_t$ fails to be excellent Morse: $\left\{ \tau_1 ,\tau_2, \dots , \tau _{n+m} \right\} $ such that each $\tau_j$ is either in $\left\{ t_1 , \dots , t_n \right\} $ or $\left\{ \widehat{t}_1, \dots , \widehat{t}_m \right\} $, corresponding to non-trivial local failures or trivial ones. Let $\epsilon>0$ be small enough such that $\cap_{j=1}^{m+n} (\tau_j -\epsilon, \tau_j + \epsilon) = \emptyset$, or equivalently \[
\cap_{i=1}^n (t_i - \epsilon, t_i + \epsilon) \cap_{j=1}^m (\widehat{t}_j - \epsilon, \widehat{t}_j + \epsilon)= \emptyset
.\]

We shall construct the path $F_t: I \times M \to \mathbb{R}$ from $F_0 = S(f_0 ) $ to $F_1 = S(f_1 )$ satisfying the desired property. First we split the path $f_t$ into $n+m+1$ segments: \[
(0, \tau _1 -\epsilon), (\tau _1 +\epsilon, \tau _2 -\epsilon), \dots , (\tau _{n+m} +\epsilon, 1) 
.\]

Over any such segment, $f_t$ is excellent good Morse and the Morse data are diffeomorphic. For the ``base case'', let $F_{\tau_1 -\epsilon}$ have the same critical values as $f_{\tau _1 -\epsilon}$, and let $F_t$ be the smooth interpolation from $F_0$ to $F_{\tau_1 -\epsilon}$.

Suppose $F_t$ is defined up to and including $\tau_i-\epsilon$, we will construct $F_t$ up to and including $\tau_{i+1}-\epsilon$. Over the interval $(\tau_i -\epsilon, \tau_i +\epsilon)$, $f_t$ fails to be excellent Morse exactly once at $\tau_i$. There are three possible cases:
\begin{enumerate}
\item the failure is of trivial local failure type $(1,1)_X$.

The Morse data before and after the failure are diffeomorphic, meaning $f_{\tau _i-\epsilon}$ and $f_{\tau _i+\epsilon}$ have diffeomorphic Morse decompositions. Therefore we let $F_t$ be a smooth interpolation from $F_{\tau _i -\epsilon}$ to $F_{\tau_{i+1} -\epsilon}$, where $F_{\tau_{i+1}-\epsilon} = S(f_{\tau _{i+1}-\epsilon})$ has the same critical values as $f_{\tau _{i+1}-\epsilon}$.

\item the failure is of non-trivial local failure type $(j, j+1)_B$ or $(j,j+1)_D$.

We will construct $F_t$ over the time interval $(\tau_i -\epsilon, \tau_i + \epsilon)$ such that $F_t$ is of local failure type $S((j,j+1)_{B,D})$ at $t=\tau_i $.

The birth or death of a pair of critical points of indices $(S(j),S(j)+1)$ is described locally by a path $\tilde{F}_t$ provided by standard Cerf theory. Let $U\subset S(M) $ be a neighbourhood over which $\tilde{F}_t$ is defined. Outside of $U\subset  S(M)$, we can take  $\tilde{G}_t: (\tau _i -\epsilon, \tau _i + \epsilon) \times U^c \to \mathbb{R}$ to be the same as $F_{\tau_i -\epsilon }|_{U^c}$ for all $t \in (\tau _i -\epsilon, \tau _i + \epsilon)$. We then glue $\tilde{F}$ and $\tilde{G}$ together using a bump function, and set the result to be the desired $F_t$ over the interval $(\tau_i -\epsilon, \tau _i +\epsilon)$. We then take a smooth interpolation from $F_{\tau _i +\epsilon}$ to $F_{\tau _{i+1} -\epsilon}=S(f_{\tau _{i+1}-\epsilon})$.

\item the failure is of non-trivial local failure type $(j,k)_{X,w,w'}$ where $w,w'$ are local models of $(j,k)$ before and after the exchange.

We will construct $F_t$ over the time interval $(\tau _i - \epsilon, \tau _i + \epsilon)$ such that $F_t$ is of local failure type $S((j,k)_{X,w,w'})$ at $t=\tau _i$. We use a similar approach to the birth-death type of failure, except instead of gluing a path defined on a neighbourhood, we glue a path defined on the relevant local model. 

First note that for any pairs of local models associated with the $S$-image of a non-trivial local failure type, our settings are simple enough that there exists an evident path of functions on the local model between the two inducing Morse functions. So let $U\subset S(M)$ be a local model induced by $F_{\tau_i -\epsilon}$ of the pair of critical points of index $S(j), S(k)$ whose critical values are being exchanged, namely $U$ is the underlying submanifold of the local model $S(w)$. Choose any local model $\cup_l U_l$ on $U$ diffeomorphic to $S(w')$, then there exists a path of functions $\tilde{F}_t: (\tau_i - \epsilon, \tau_i + \epsilon) \times U \to \mathbb{R}$ from $F_{\tau _i -\epsilon}|_U$ to the Morse function inducing $\cup_l U_l$, such that it is of the desired failure type at $t=\tau _i$. Outside $U$, we can take $\tilde{G}_t : (\tau_i - \epsilon, \tau _i + \epsilon)\times U^c \to \mathbb{R}$ to be the same as $F_{\tau _i - \epsilon}|_{U^c}$. Glue $ \tilde{F}$ and $\tilde{G}$ together to get the desired $F_t: (\tau _i -\epsilon, \tau _i + \epsilon)$ and take a smooth interpolation to $S(f_{\tau _{i+1}-\epsilon})$ at time $\tau _{i+1}-\epsilon$.
\end{enumerate}

Since there are only finitely many $t$ such that $f_t$ fails to be excellent Morse, this process eventually must terminate and all the smooth interpolations are taken from the constructed $F_{\tau_i + \epsilon}$ to $S(f_1)$. Also since $F_0$ is spherical, smooth interpolations preserves spherical-ness, and all local failures are spherical except at the singular point, we have that the constructed path $F_t$ is admissible excellent spherical Morse at all but time $\left\{ t_1 , \dots , t_n \right\} $.
\end{proof}

%
%

The power of proposition \ref{prop:keyprop} is that it allows a Cerf-path in $\operatorname{Cob}( 2)$ to lift to a spherical Cerf-path in $\operatorname{Cob}( 3)_{S^2,CF} $. Together with the silly functor $S$, we can show that there is always a spherical Cerf-path between any admissible excellent spherical Morse functions on $N$, a representative of a morphism $[N]$ in $\operatorname{Cob}( 3)_{S^2,CF} $.

\begin{prop}\label{prop:keypropCF}
Let $\left[ N \right] $ be a morphism in $\operatorname{Cob}(3)_{S^2, CF} $ and $F_0 ,F_1 :N \to \mathbb{R}$ two admissible excellent spherical Morse functions. Then there exists a path $F_t$ in $C^\infty(N)$ from $F_0 $ to $F_1 $ which is admissible excellent spherical Morse at all but finitely many $t$. Moreover, the types of failure are of CF type.
\end{prop}

\begin{proof}
Consider $f_0 = S^{-1}(F_0)$ and $f_1 = S^{-1}(F_1)$, both of which are Morse functions on $M$ a representative of $S^{-1}(\left[ N \right] )$. Let $f_t$ be any path of functions from $f_0$ to $f_1$ on $M$. Then by proposition \ref{prop:keyprop}, there exists a path $F_t$ from $S(f_0 ) = S(S^{-1}(F_0)) = F_0$ to $S(f_1 )=S(S^{-1}(F_1))=F_1 $ which fails to be admissible excellent spherical Morse at finitely many times, and the types of failure are of commutative Frobenius type.
\end{proof}

\begin{cor} \label{cor:CFonlycommFrob}
CF bordism generators satisfy no more relations.
\end{cor}

\begin{rmk}
Proposition \ref{prop:keypropCF} and corollary \ref{cor:CFonlycommFrob} addresses the major subtlety of the isomorphism between $\operatorname{Cob}( 2) $ and $\operatorname{Cob}( 3)_{S^2,CF} $: while two compositions of the same morphism $[N]$ in $\operatorname{Cob}( 3)_{S^2,CF}$ are always induced by two Morse functions $F_0, F_1:N\to \mathbb{R}$, a path $F_t^! $ from $F_0 $ to $F_1 $ may fail to be spherical. An example is example \ref{eg:S2xS1pantstoHeegaardtopants}, where $F_0 =F_1: S^2 \times S^1 \to \mathbb{R}$ are both pants decompositions, but the path $F_t^!$ has a segment which is the Heegaard splitting, in particular not spherical.

However by proposition \ref{prop:keypropCF} and corollary \ref{cor:CFonlycommFrob}, we now know that there is always a path $F_t$ from $F_0 $ to $F_1 $ that is spherical all the way except at singularities of commutative Frobenius type. So any two compositions of the same morphisms $[N]$ in $\operatorname{Cob}( 3)_{S^2, CF}$ can always be turned into each other using only commutative Frobenius relations.
\end{rmk}

%
%

%

\chapter{Prime Relations} \label{ch:prime relations}

\section{No Internal Relations} \label{sec:no internal relations}


Let $p^{\times \times} =\left[ p \setminus \operatorname{int}(D^3 \sqcup D^3) \right]  \in P$ be a prime endomorphism. We say a relation is \textit{internal }to $p^{\times \times} $ if there exists a composition of CF generators and $P$ generators such that the composition is again  $p^{\times \times} $, other than the trivial composition $p^{\times \times}  = id_{S^2} \circ p^{\times \times}  = p^{\times \times}  \circ id_{S^2}$. One might reasonably suspect that internal relations could arise: two Morse data on a representative $p$ could induce very different Morse decompositions. Even for two Morse functions that are identical except being related by a simple birth, one could ask if a new critical point is coming from that of a CF generator. But since the prime factors of a manifold is unique up to diffeomorphism, the only $P$ generator that could possibly show up is $p^{\times \times} $ itself. This places a very strong constraint on the possible form of the internal relation.

The following lemma guarantees that there are no such internal relations.

\begin{lemma}\label{sphereinp}
Let $p $ be an oriented prime 3 manifold and $f: p \setminus \operatorname{int}( D^3 \sqcup D^3 )  \to \mathbb{R}$ an admissible Morse function. If a regular level set has a component diffeomorphic to $S^2$, namely $ S \cong S^2 \subset  f^{-1}(r)  $ for some regular value $r$, then either 
\begin{enumerate}
\item $p \cong S^2 \times S^1$
\end{enumerate}
Or one of the following mutually exclusive conditions must be true
\begin{enumerate}\setcounter{enumi}{1}
\item $S$ is the gluing sphere of $p\setminus \operatorname{int}( D^3\sqcup D^3 ) \cong \Big( p\setminus \operatorname{int}(D^3 \sqcup D^3)  \Big) \cup_S S^2\times I$, namely $S$ bounds $S^2\times I$ with the other boundary being the incoming or outgoing boundary of $p\setminus \operatorname{int}(D^3 \sqcup D^3)$, or
\item $S$ is the gluing sphere of $p\setminus \operatorname{int}( D^3 \sqcup D^3  ) 
\cong \Big( p \setminus \operatorname{int}(D^3 \sqcup D^3 \sqcup D^3) \Big)  \cup_S D^3$, or
\item $S$ is the gluing sphere of $\Big( p\setminus \operatorname{int}( D^3\sqcup D^3 )  \Big) 
\cong \operatorname{int}( p\setminus D^3 )  \cup_S \left( S^2\times I \setminus \operatorname{int}( D^3) \right)  $.
\end{enumerate}
\end{lemma}

\begin{proof}
This is a fancy restatement of the fact that an oriented 3-manifold is prime if and only if it's irreducible, except for $S^2 \times S^1$. If $S$ is non-separating, one can take a tubular neighbourhood of the connecting path and show that $p$ must contain a connected sum factor of $S^2\times S^1$. Since $p$ is prime, $p$ must be $S^2\times S^1$. If $S$ is separating, then $p$ has a connected sum factor with attaching sphere $S$. Since $p$ is prime, the other prime factor must be $S^3$, namely $S$ bounds a three-ball in $p$. This three ball could contain the deleted incoming ball, the deleted outgoing ball, neither deleted balls, or both deleted balls. If the three-ball bound by $S $ contains only the incoming or the outgoing ball, then $S$ must be one of the boundary components of $S^2\times I$, the other being the boundary of the incoming/outgoing ball. If the three-ball bound by $S$ contains neither the incoming nor the outgoing three balls, then $S$ is the boundary of a simple $D^3$, glued to the rest of the manifold via $S$. If the three-ball bound by $S$ contains both the incoming and the outgoing ball, then $S$ is the gluing boundary of $\left( S^2 \times I \right)  \setminus \operatorname{int}(D^3)$ to the topologically non-trivial part, which is  $p \setminus \operatorname{int}(D^3)$. See figure \ref{fig:sphereinprimeendo} for a sketch of the latter three cases.
\end{proof}

\begin{figure}[ht]
    \centering
\def\svgwidth{1\columnwidth} 
\begingroup%
  \makeatletter%
  \providecommand\color[2][]{%
    \errmessage{(Inkscape) Color is used for the text in Inkscape, but the package 'color.sty' is not loaded}%
    \renewcommand\color[2][]{}%
  }%
  \providecommand\transparent[1]{%
    \errmessage{(Inkscape) Transparency is used (non-zero) for the text in Inkscape, but the package 'transparent.sty' is not loaded}%
    \renewcommand\transparent[1]{}%
  }%
  \providecommand\rotatebox[2]{#2}%
  \newcommand*\fsize{\dimexpr\f@size pt\relax}%
  \newcommand*\lineheight[1]{\fontsize{\fsize}{#1\fsize}\selectfont}%
  \ifx\svgwidth\undefined%
    \setlength{\unitlength}{680.31496063bp}%
    \ifx\svgscale\undefined%
      \relax%
    \else%
      \setlength{\unitlength}{\unitlength * \real{\svgscale}}%
    \fi%
  \else%
    \setlength{\unitlength}{\svgwidth}%
  \fi%
  \global\let\svgwidth\undefined%
  \global\let\svgscale\undefined%
  \makeatother%
  \begin{picture}(1,0.33333333)%
    \lineheight{1}%
    \setlength\tabcolsep{0pt}%
    \put(0,0){\includegraphics[width=\unitlength,page=1]{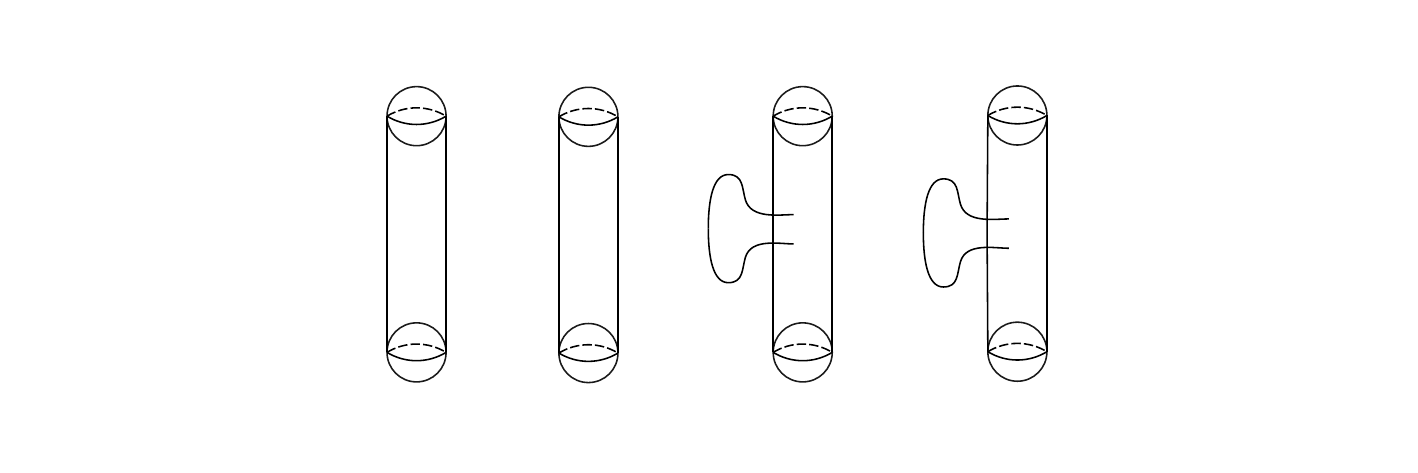}}%
    \put(0.28251171,0.16345018){\makebox(0,0)[lt]{\lineheight{1.25}\smash{\begin{tabular}[t]{l}$p$\end{tabular}}}}%
    \put(0.40678186,0.20464675){\makebox(0,0)[lt]{\lineheight{1.25}\smash{\begin{tabular}[t]{l}$p$\end{tabular}}}}%
    \put(0.55549882,0.20613613){\makebox(0,0)[lt]{\lineheight{1.25}\smash{\begin{tabular}[t]{l}$p$\end{tabular}}}}%
    \put(0.66651287,0.16439107){\makebox(0,0)[lt]{\lineheight{1.25}\smash{\begin{tabular}[t]{l}$p$\end{tabular}}}}%
    \put(0.34087143,0.16324355){\makebox(0,0)[lt]{\lineheight{1.25}\smash{\begin{tabular}[t]{l}$\cong$\end{tabular}}}}%
    \put(0.44878785,0.16214078){\makebox(0,0)[lt]{\lineheight{1.25}\smash{\begin{tabular}[t]{l}$\cong$\end{tabular}}}}%
    \put(0.60122604,0.16357034){\makebox(0,0)[lt]{\lineheight{1.25}\smash{\begin{tabular}[t]{l}$\cong$\end{tabular}}}}%
    \put(0,0){\includegraphics[width=\unitlength,page=2]{sphereinprimeendo.pdf}}%
  \end{picture}%
\endgroup%

    \caption[Spheres in a prime endomorphism]{Classifications of spherical regular surface in a prime endomorphism. The shaded sphere is the regular surface $S\cong S^2$. Not depicted is the case $p=S^2\times S^1$. Note that $p$ here denotes the diffeomorphism class of a connected closed oriented prime 3-manifold, but only up to some finite number of $D^3$ removed.}
    \label{fig:sphereinprimeendo}
\end{figure}


\begin{mycom}
Note that there exist separating spheres in $S^2\times S^1$, for example one could pick $S$ to be the level set near a global minimum. Therefore even though condition 2, 3, 4 are mutually exclusive, they are each compatible with condition 1. 
\end{mycom}

\begin{prop}
There are no internal relations among generators in $\mathcal{G}_2$.
\end{prop}

\begin{proof}
Suppose there is a composition sequence of a prime endomorphism $p^{\times \times} $. Then there exists a Morse function on $p\setminus\operatorname{int}(D^3\sqcup D^3)$ which induces a Morse decomposition containing a factor $\widehat{p}$ diffeomorphic to $p\setminus \operatorname{int}(D^3\sqcup D^3)$ itself. Since $p$ is connected, one of the bounding spheres of the Morse factor $\widehat{p}$ must be the gluing sphere. By lemma \ref{sphereinp}, since $p$ is not $S^2\times S^1$, condition 2 must be true. Hence by naturality of the coherence data (unitors, associators and braiding), the composition can be brought to the form \[
p^{\times \times}  = p^{\times \times}  \circ (\left[ b_1  \right] \circ \left[ b_2  \right] \circ \dots ) 
\]
if the gluing sphere is incoming, or of the form \[
p^{\times \times}  = (\left[ b_1  \right] \circ \left[ b_2  \right] \circ \dots )\circ p^{\times \times}   
\]
if the gluing sphere is outgoing. Here the morphisms $[b_i]$ are generated by CF generators.

Lemma \ref{sphereinp} forces both parentheses to compose to the identity $id_{S^2} = [S^2 \times I]$. And by proposition \ref{prop:keyprop}, they can do so using commutative Frobenius relations. We end up getting the trivial identity relations $p^{\times \times} = p^{\times \times} \circ id_{S^2} = id_{S^2} \circ p^{\times \times} $, and neither case constitutes an internal relation by definition.
\end{proof}

\begin{rmk}
Cerf-theoretically, we have shown that given two Morse functions on $p\setminus \operatorname{int}(D^3 \sqcup D^3)$ and a path connecting them, new critical points coming from CF generators must together cancel to an identity cylinder, attached to the incoming or the outgoing part of $p\setminus \operatorname{int}(D^3 \sqcup D^3)$. They can do so with commutative Frobenius relations. They cannot cancel to a ball like in condition 3 since the only way such ball would attach to the remainder of the manifold is if the remainder has three boundary components, but our prime endomorphisms $p^{\times \times} $ all only have two boundary components. They cannot cancel to a pair of pants like in condition 4 since the only way such pants would attach to the remainder of the manifold is if the remainder has one boundary components. 

While it is true that the Morse data for the prime part might change drastically, it does not concern us-- over the prime part we care only about the diffeomorphism type. 
\end{rmk}

The analogous version for $G_1 $ generators is as follows: 

\begin{lemma}
Let $p $ be an oriented prime 3 manifold and $f: p \setminus  \operatorname{int}(D^3)   \to \mathbb{R}$ an admissible Morse function. If a regular level set has a component diffeomorphic to $S^2$, namely $ S \cong S^2 \subset  f^{-1}(r)  $ for some regular value $r$, then either 
\begin{enumerate}
\item $p \cong S^2 \times S^1$
\end{enumerate}
Or one of the two following mutually exclusive conditions must be true
\begin{enumerate}\setcounter{enumi}{1}
\item $S$ is the gluing sphere of $p\setminus \operatorname{int}(D^3)  \cong \left( p\setminus \operatorname{int}(D^3) \right)  \cup_S S^2\times I$, namely $S$ bounds $S^2\times I$ with the other boundary being the boundary of $p\setminus \operatorname{int}(D^3)$, or
\item $S$ is the gluing sphere of $p\setminus  \operatorname{int}(D^3) \cong p \setminus \operatorname{int}(D^3 \sqcup D^3 ) \cup_S D^3$.
\end{enumerate}
\end{lemma}

\begin{proof}
The proof is identical to the proof of \ref{sphereinp}, except the three ball bound by $S$ may now contain the only boundary ball or not at all. The former gives condition 2 while the later gives condition 3. See figure \ref{fig:sphereinprimeunit} for an illustration of the two cases.
\end{proof}
\begin{figure}[ht]
    \centering
\def\svgwidth{1\columnwidth} 
\begingroup%
  \makeatletter%
  \providecommand\color[2][]{%
    \errmessage{(Inkscape) Color is used for the text in Inkscape, but the package 'color.sty' is not loaded}%
    \renewcommand\color[2][]{}%
  }%
  \providecommand\transparent[1]{%
    \errmessage{(Inkscape) Transparency is used (non-zero) for the text in Inkscape, but the package 'transparent.sty' is not loaded}%
    \renewcommand\transparent[1]{}%
  }%
  \providecommand\rotatebox[2]{#2}%
  \newcommand*\fsize{\dimexpr\f@size pt\relax}%
  \newcommand*\lineheight[1]{\fontsize{\fsize}{#1\fsize}\selectfont}%
  \ifx\svgwidth\undefined%
    \setlength{\unitlength}{680.31496063bp}%
    \ifx\svgscale\undefined%
      \relax%
    \else%
      \setlength{\unitlength}{\unitlength * \real{\svgscale}}%
    \fi%
  \else%
    \setlength{\unitlength}{\svgwidth}%
  \fi%
  \global\let\svgwidth\undefined%
  \global\let\svgscale\undefined%
  \makeatother%
  \begin{picture}(1,0.33333333)%
    \lineheight{1}%
    \setlength\tabcolsep{0pt}%
    \put(0,0){\includegraphics[width=\unitlength,page=1]{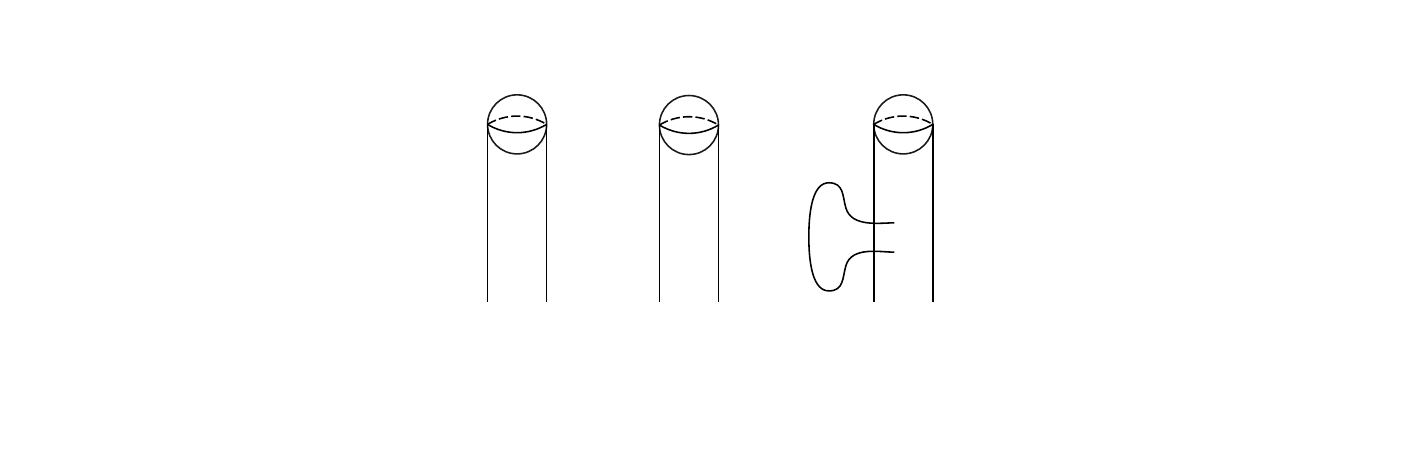}}%
    \put(0.35344388,0.15764974){\makebox(0,0)[lt]{\lineheight{1.25}\smash{\begin{tabular}[t]{l}$p$\end{tabular}}}}%
    \put(0.477714,0.19884631){\makebox(0,0)[lt]{\lineheight{1.25}\smash{\begin{tabular}[t]{l}$p$\end{tabular}}}}%
    \put(0.62643102,0.20033569){\makebox(0,0)[lt]{\lineheight{1.25}\smash{\begin{tabular}[t]{l}$p$\end{tabular}}}}%
    \put(0.41180356,0.15744311){\makebox(0,0)[lt]{\lineheight{1.25}\smash{\begin{tabular}[t]{l}$\cong$\end{tabular}}}}%
    \put(0.51972001,0.15634034){\makebox(0,0)[lt]{\lineheight{1.25}\smash{\begin{tabular}[t]{l}$\cong$\end{tabular}}}}%
    \put(0,0){\includegraphics[width=\unitlength,page=2]{sphereinprimeunit.pdf}}%
  \end{picture}%
\endgroup%

    \caption[Spheres in a prime unit]{Classifications of spherical regular surface in a prime unit. The shaded sphere is the regular surface $S\cong S^2$. Not depicted is the case $p=S^2 \times S^1$. Note that $p$ here denotes the diffeomorphism class of a connected closed oriented prime 3-manifold, but only up to some finite number of $D^3$ removed.}
    \label{fig:sphereinprimeunit}
\end{figure}

\begin{prop}
There are no internal relations among generators in $G_1 $.
\end{prop}

\begin{proof}
Suppose there is a composition of a prime unit $p^\times$. Then there exists a Morse function on $p\setminus \operatorname{int}(D^3)$ which induces a Morse decomposition containing a factor $\widehat{p}$ diffeomorphic to $p\setminus \operatorname{int}(D^3)$. Since $p$ is connected, the bounding sphere of $\widehat{p}$ must be a gluing sphere. By the previous lemma, since $p$ is not $S^2\times S^1$, condition 2 must be true. Hence by naturality of the unitors, associators, and braiding, the composition can be brought to the form \[
p^\times = ([b_1 ]\circ [b_2 ] \circ \dots ) \circ p^\times
.\]
where the parenthesis must compose to $[S^2 \times I]$ and the morphisms $[b_i]$ are generated by CF generators.  By proposition \ref{prop:keyprop} can do so with commutative Frobenius relations. We end up with the identity relation $p^\times = id_{S^2} \circ p^\times$, which by definition is not an internal relation.
\end{proof}

\begin{rmk}
While this may be an obvious statement, the lack of internal relations means that under no circumstance can a CF generator exist ``inside'' a prime unit/endomorphism. Therefore intuitively we can think of the prime units/endomorphisms as belonging to a different ``species'' of bordisms, even when studying the underlying Cerf theory.
\end{rmk}

%

\section{External Prime Relations} \label{sec: external prime relations}

\begin{prop} \label{primecommutativity}
(Prime Commutativity) Let $p^{\times \times}_1 , p^{\times \times}_2  \in P$ be prime endomorphisms. Then \[
p^{\times \times}_1 \circ p^{\times \times}_2 = p^{\times \times}_2 \circ p^{\times \times}_1 
.\]
\end{prop}

\begin{figure}[ht]
    \centering
\def\svgwidth{1\columnwidth} 
\begingroup%
  \makeatletter%
  \providecommand\color[2][]{%
    \errmessage{(Inkscape) Color is used for the text in Inkscape, but the package 'color.sty' is not loaded}%
    \renewcommand\color[2][]{}%
  }%
  \providecommand\transparent[1]{%
    \errmessage{(Inkscape) Transparency is used (non-zero) for the text in Inkscape, but the package 'transparent.sty' is not loaded}%
    \renewcommand\transparent[1]{}%
  }%
  \providecommand\rotatebox[2]{#2}%
  \newcommand*\fsize{\dimexpr\f@size pt\relax}%
  \newcommand*\lineheight[1]{\fontsize{\fsize}{#1\fsize}\selectfont}%
  \ifx\svgwidth\undefined%
    \setlength{\unitlength}{680.31496063bp}%
    \ifx\svgscale\undefined%
      \relax%
    \else%
      \setlength{\unitlength}{\unitlength * \real{\svgscale}}%
    \fi%
  \else%
    \setlength{\unitlength}{\svgwidth}%
  \fi%
  \global\let\svgwidth\undefined%
  \global\let\svgscale\undefined%
  \makeatother%
  \begin{picture}(1,0.33333333)%
    \lineheight{1}%
    \setlength\tabcolsep{0pt}%
    \put(0,0){\includegraphics[width=\unitlength,page=1]{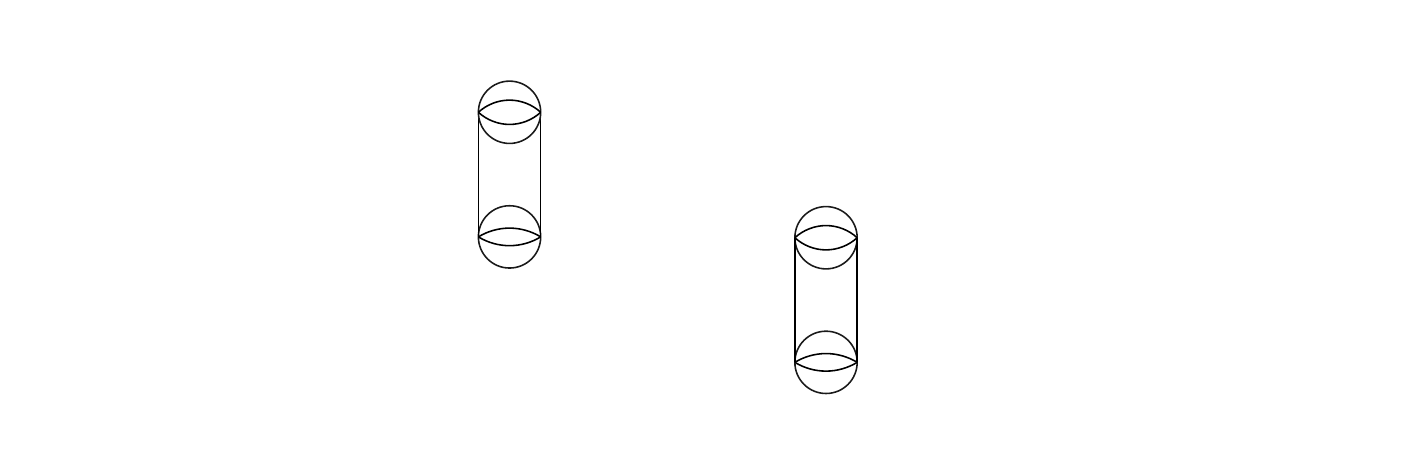}}%
    \put(0.46757476,0.16723779){\makebox(0,0)[lt]{\lineheight{1.25}\smash{\begin{tabular}[t]{l}$\cong$\end{tabular}}}}%
    \put(0.34075705,0.20369193){\makebox(0,0)[lt]{\lineheight{1.25}\smash{\begin{tabular}[t]{l}$p_1^{\times \times}$\end{tabular}}}}%
    \put(0,0){\includegraphics[width=\unitlength,page=2]{primecommutativity.pdf}}%
    \put(0.56405099,0.20309575){\makebox(0,0)[lt]{\lineheight{1.25}\smash{\begin{tabular}[t]{l}$p_2^{\times \times}$\end{tabular}}}}%
    \put(0.56394793,0.11905231){\makebox(0,0)[lt]{\lineheight{1.25}\smash{\begin{tabular}[t]{l}$p_1^{\times \times}$\end{tabular}}}}%
    \put(0,0){\includegraphics[width=\unitlength,page=3]{primecommutativity.pdf}}%
    \put(0.34065406,0.11964849){\makebox(0,0)[lt]{\lineheight{1.25}\smash{\begin{tabular}[t]{l}$p_2^{\times \times}$\end{tabular}}}}%
  \end{picture}%
\endgroup%

    \caption[Prime endomorphism commutativity]{Prime endomorphism commutativity. Geometrically turning the left hand side upside down. Equivalently one can post-compose the Morse function by $x\mapsto -x$.}
    \label{fig:primecommutativity}
\end{figure}

\begin{proof}
By definition, connected sum is commutative, hence  $p_1 \# p_2 \cong p_2 \# p_1 $. Deleting a disk from both $ p_1, p_2  $, we have \[
	\left(p_1 \setminus\operatorname{int}( D^3)\right) \# \left(p_2 \setminus \operatorname{int}(D^3)\right) = \left(p_2 \setminus \operatorname{int}(D^3)\right) \# \left(p_1 \setminus \operatorname{int}(D^3)\right)
.\] 
Noting that $\#$ is the same as gluing over a bounding sphere, we have the desired relation.

Morse-theoretically, any Morse datum $(f,b=0)$ inducing the decomposition \[
	\left(p_1 \setminus \operatorname{int}(D^3 \sqcup D^3)\right) \cup_{S^2} \left(p_2 \setminus \operatorname{int}(D^3 \sqcup D^3)\right)
\]
also induces, via ``time reversal'' by post-composing with $\mathbb{R} \to \mathbb{R}$ sending $x \mapsto -x$ , the decomposition \[
	\left(p_2 \setminus \operatorname{int}(D^3 \sqcup D^3)\right) \cup_{S^2} \left(p_1 \setminus \operatorname{int}(D^3 \sqcup D^3)\right)
\]
.
\end{proof}

\begin{rmk} \label{rmk:Cerftheoryofprimecomm}
	Note however that this relation does not correspond to any specific type of singularities along Cerf-paths. Let $(f,b=0)$ be a Morse datum inducing $p^{\times \times}_1 \circ p^{\times \times}_2 $ and $(-f,b=0)$ be the time reversal Morse datum inducing $p^{\times \times}_2 \circ p^{\times \times}_1 $. Then a path $f_t$ from $f$ to $-f$ will generally be very complicated and contain multiple singularities. This is a feature of \textit{prime relations}. A prime relation is induced by a \textit{pair of Morse functions}  $f_0 ,f_1 : M \to \mathbb{R}$, rather than a specific type of singularities along a path from $f_0 $ to $f_1 $.
\end{rmk}

\begin{prop} \label{prop:legsandwaistrelations}
(Legs and Waist Relations) Let $p^{\times \times} \in P$ be a prime endomorphism, $m \in C$ the multiplication map, and $m^\vee\in F$ the comultiplication map. Then the following relations are always true: 
\begin{align}
p^{\times \times}  \circ m &= m \circ \left(p^{\times \times} \otimes id_{S^2}\right) = m \circ (id_{S^2} \otimes p^{\times \times} )\\
m^\vee \circ p^{\times \times}  &= (p^{\times \times} \otimes id_{S^2}) \circ m^\vee = (id_{S^2} \otimes p^{\times \times} ) \circ m^\vee
\end{align}
We shall call the first equality ``waist relations'' and the second equality ``legs relations,'' and the second line are the ``co-waist relations'' and ``co-legs relations.''
\end{prop}

\begin{figure}[ht]
    \centering
\def\svgwidth{1\columnwidth} 
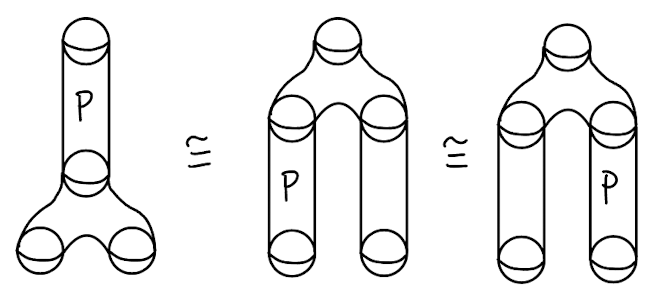

    \caption[Legs and waist relations]{A cartoon for legs and waist relations. All manifolds depicted are diffeomorphic, and boundary-preserving within each line. $p^{\times \times} $ here denotes the manifold $p\setminus \operatorname{int}(D^3 \sqcup D^3) $.}
    \label{fig:legsandwaist}
\end{figure}


\begin{proof}
Consider $p\setminus \operatorname{int}(\sqcup^3 D^3)$, but instead of $p$ itself, take the trivial gluing $p \cong \left( p \setminus \operatorname{int}(D^3 )\right)  \cup_S D^3$ along a two sphere $S$ (or equivalently the trivial connect sum $p \cong p \# S^3$). There are six ways to cut out three disjoint three-balls away from the gluing boundary such that a factor $p\setminus \operatorname{int}(D^3 \sqcup D^3)$ is present: Topologically, one ball must be cut from $p\setminus \operatorname{int}(D^3)$ to recover the prime endomorphism $p \setminus \operatorname{int}(D^3 \sqcup D^3)$, the other two must be cut from the trivially glued on three-ball. Those two deleted balls can form $(2,0), (1,1) $ or $ (0,2)$ (incoming, outgoing) boundaries. In the $(2,0)$ case, $m$ has two incoming boundaries, so $S$ must be the outgoing boundary of $m$ and the incoming boundary of $p^{\times \times} $, giving $p^{\times \times} \circ m$. In the $(0,2)$ case, $m^\vee$ has two outgoing boundaries, so $S$ must be the incoming boundary of $m^\vee$ and the outgoing boundary of $p^{\times \times} $, giving $m^\vee \circ p^{\times \times} $.

The case $(1,1)$ is only slightly more complicated. If $S$ is the incoming boundary of $p^{\times \times} $, then the only possible compositions are $(p^{\times \times}  \otimes id_{S^2})\circ m^\vee$ and $(id_{S^2} \otimes p^{\times \times} )\circ m^\vee$ (recall that the boundaries are labelled, so these two are different morphisms in $\operatorname{Cob}( 3)_{S^2, CF} $). If $S$ is the outgoing boundary of $p^{\times \times} $, then the only possible compositions are $m\circ (p^{\times \times} \otimes  id_{S^2}) $ and $m\circ (id_{S^2} \otimes p^{\times \times} )$.

All of the cases mentioned above are diffeomorphic as manifolds with boundaries. To establish equalities as bordisms in $\operatorname{Cob}(3)_{S^2}$, match the number of incoming and outgoing boundaries to give the claimed relations.
\end{proof}

\begin{rmk}
The Cerf-theoretic interpretation for the (co)legs and (co)waist relations are worth explaining. The (co)legs relations do not have any interesting Cerf theory-- for example, if the Morse datum $(f,b)$ induces \[
p\setminus \operatorname{int}\Big(\bigsqcup_{i=1}^3 D^3\Big) = \Big( S^3 \setminus \operatorname{int}(\sqcup_{i=1}^3 D^3) \Big) \bigcup_{S^2 \sqcup S^2} \Big( (p \setminus \operatorname{int}(D^3 \sqcup D^3)) \bigsqcup \left( S^2 \times I \right)  \Big) 
.\]
which induces the composition $m\circ (p^{\times \times} \otimes id_{S^2})$, then the same Morse datum $(f,b)$ also induces $m \circ (id_{S^2} \otimes p^{\times \times} )$ by switching the factors $p \setminus \operatorname{int}(D^3 \sqcup D^3)$ and $S^2 \times I$. So echoing the comments made in remark \ref{rmk:Cerftheoryofprimecomm}, the (co)legs relations are induced by the pair of Morse functions $f, f:p \setminus \operatorname{int}(\sqcup^3 D^3) \to \mathbb{R}$. We can always take the constant path from $f$ to itself when constructing a Cerf path.

The (co)waist relations are, in contrast, induced by two different Morse functions $f_{\text{upper}}, f_{\text{lower}} : p \setminus (\sqcup^3 D^3) \to \mathbb{R}$, where the subscript denotes where the prime endomorphism $p^{\times \times} $ sits in relations to $m$ or $m^\vee$ (later in the composition or earlier in the composition). Again, a path $f_t$ from $f_ \text{upper}$ to $f_ \text{lower}$ is generally very complicated and contain multiple singularities. One can imagine that, even in the simplest case, the index $2$ critical point of $m$ must be moved past all critical points of $p^{\times \times} $, so the path $f_t$ contains at least as many singularities as $f_{\text{upper}}$ has critical points, minus one. This is a feature of prime relations, as first discussed in \ref{rmk:Cerftheoryofprimecomm}.

\end{rmk}

\begin{rmk}
As there are countably many diffeomorphism classes of prime 3-manifolds, there are countably many such legs and waist relations.
\end{rmk}

\begin{rmk}
Notice that $p^{\times \times} $ is not an algebra homomorphism with respect to $m$: for this to hold, we would need $p^{\times \times} \circ m = m\circ (p^{\times \times}  \otimes p^{\times \times} )$, but the left hand side has one factor of $p$ while the right has two. Nor does it preserve the unit: in fact, it turns the unit into the corresponding prime unit $p^{\times \times} \circ 1 = p^\times$.
\end{rmk}

One may wonder if there are similar relations involving the prime units and the commutative Frobenius generators. The answer is yes-- but it's not very interesting or useful. Since a prime unit's boundary is outgoing, it can only be attached to incoming boundaries. Thus we are led to the following relations. Note that the proposition \ref{prop:LlegsfromCF} immediately afterwards renders these relations unnecessary, as they are implied by commutativity of the commutative Frobenius relations. In fact, the proof of theorem \ref{mainthmG1} does not make use of these relations at all, in contrast to the proof of theorem \ref{mainthmG2} which uses all of the prime endomorphism relations.

\begin{prop} \label{prop:Llegs}
(Legs Relations) Let $p^\times \in L$ be a prime unit, $m \in C$ the multiplication map and $m^\vee \in F$ the comultiplication map. Then the following relations are always true: \[
m \circ (p^\times \otimes id_{S^2}) \circ \lambda^{-1}_{S^2} = m\circ (id_{S^2} \otimes p^\times)\circ \rho ^{-1}_{S^2}
.\]
where $\lambda_{S^2}: \emptyset \sqcup S^2 \overset{\sim}{\to} S^2$ is the left unitor and $\rho_{S^2} : S^2 \sqcup \emptyset \overset{\sim}{\to} S^2$ is the right unitor.
We shall call this equality the ``legs relations.''
\end{prop}

\begin{figure}[ht]
    \centering
\def\svgwidth{1\columnwidth} 
\begingroup%
  \makeatletter%
  \providecommand\color[2][]{%
    \errmessage{(Inkscape) Color is used for the text in Inkscape, but the package 'color.sty' is not loaded}%
    \renewcommand\color[2][]{}%
  }%
  \providecommand\transparent[1]{%
    \errmessage{(Inkscape) Transparency is used (non-zero) for the text in Inkscape, but the package 'transparent.sty' is not loaded}%
    \renewcommand\transparent[1]{}%
  }%
  \providecommand\rotatebox[2]{#2}%
  \newcommand*\fsize{\dimexpr\f@size pt\relax}%
  \newcommand*\lineheight[1]{\fontsize{\fsize}{#1\fsize}\selectfont}%
  \ifx\svgwidth\undefined%
    \setlength{\unitlength}{680.31496063bp}%
    \ifx\svgscale\undefined%
      \relax%
    \else%
      \setlength{\unitlength}{\unitlength * \real{\svgscale}}%
    \fi%
  \else%
    \setlength{\unitlength}{\svgwidth}%
  \fi%
  \global\let\svgwidth\undefined%
  \global\let\svgscale\undefined%
  \makeatother%
  \begin{picture}(1,0.33333333)%
    \lineheight{1}%
    \setlength\tabcolsep{0pt}%
    \put(0,0){\includegraphics[width=\unitlength,page=1]{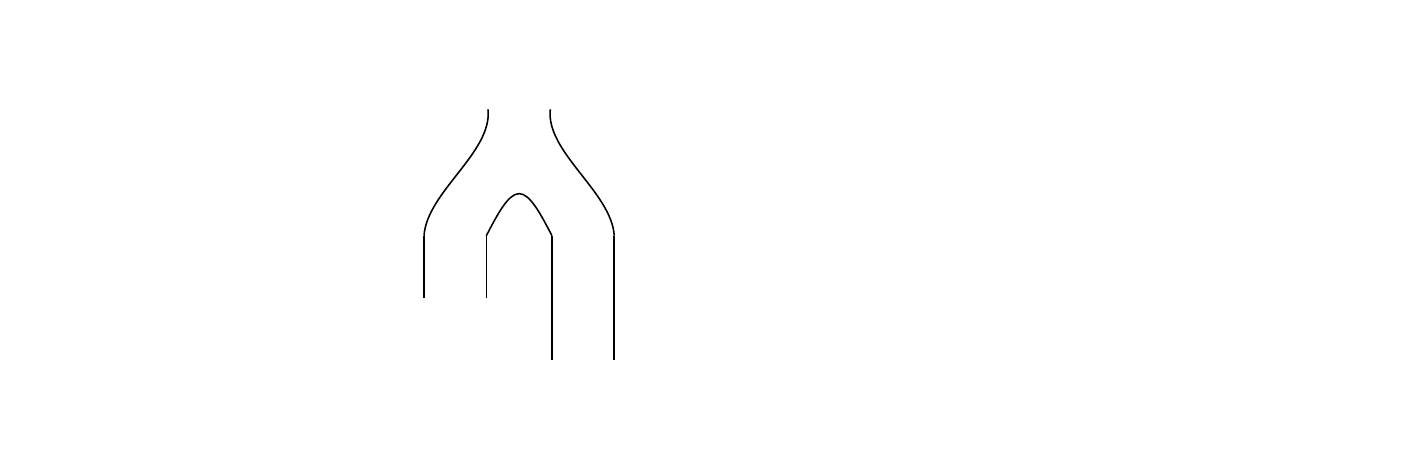}}%
    \put(0.30913973,0.12047736){\makebox(0,0)[lt]{\lineheight{1.25}\smash{\begin{tabular}[t]{l}$p^\times$\end{tabular}}}}%
    \put(0,0){\includegraphics[width=\unitlength,page=2]{llegs.pdf}}%
    \put(0.66495056,0.12047736){\makebox(0,0)[lt]{\lineheight{1.25}\smash{\begin{tabular}[t]{l}$p^\times$\end{tabular}}}}%
    \put(0,0){\includegraphics[width=\unitlength,page=3]{llegs.pdf}}%
    \put(0.49381984,0.16247292){\makebox(0,0)[lt]{\lineheight{1.25}\smash{\begin{tabular}[t]{l}$\cong$\end{tabular}}}}%
    \put(0,0){\includegraphics[width=\unitlength,page=4]{llegs.pdf}}%
  \end{picture}%
\endgroup%

    \caption[Legs relations for $G_1$]{Legs relations for $p^\times L \subset G_1$.}
    \label{fig:llegs}
\end{figure}

\begin{proof}
Geometrically this is an obvious statement. Consider the gluing \[
p\setminus  \operatorname{int}\Big(D^3\bigsqcup D^3\Big)  \cong \Big( p \setminus \operatorname{int}(D^3) \Big)  \bigcup_{S^2} \Big( S^3 \setminus \operatorname{int}(\bigsqcup_{i=1}^3 D^3) \Big)
.\]
Since the sphere boundary of a prime unit is outgoing, one of the sphere boundaries of the $S^3\setminus \operatorname{int}(\sqcup^3 D^3)$ factor must be incoming. Choosing one of the remaining two as outgoing and the other as incoming, this decomposition induces two compositions of $[p\setminus \operatorname{int}(\sqcup^2 D^3)]$ (which we recognize as the corresponding prime endomorphism, see \ref{criticalrelations}). Since the incoming boundary is either $\emptyset \sqcup S^2$ or  $S^2 \sqcup \emptyset$, we can resolve the two cases with a left or right unitor respectively, giving the desired equality of morphisms.
\end{proof}

\textbf{Warning:} what follows, and what will follow the next two sections, lean heavy on coherence conditions of symmetric monoidal categories. We made the deliberate choice of carefully exhibiting unitors, associators, and braiding, because it enables the reader to apply the results in a non-strict symmetric monoidal category $\mathcal{C}$ after passing through a $3$-dimensional topological field theories $Z: \operatorname{Cob}( 3) \to \mathcal{C}$.

\begin{notations} \label{coherencenotations}
Let $(\mathcal{C}, \otimes , 1_{\otimes })$ be a symmetric monoidal category and $\lambda: 1_ \otimes \otimes  (-) \Rightarrow Id_{\mathcal{C}}$ the natural isomorphism which gives the left unitors, and $\rho : (-) \otimes 1_{\otimes } \Rightarrow Id_{\mathcal{C}}$ the natural isomorphism which gives the right unitors. Denote their morphisms at an object $A\in \mathcal{C}$ by $\lambda_A : 1_{\otimes } \otimes A \overset{\sim}{\to} A$ and $\rho _A : A \otimes 1_{\otimes } \overset{\sim}{\to} A$. Denote by $\beta _{A,B}$ the braiding $A \otimes B \overset{\sim}{\to} B \otimes A$. Denote by $\alpha_{A,B,C}$ the associator $(A \otimes B) \otimes C \overset{\sim}{\to} A \otimes (B \otimes C)$
\end{notations}

\begin{prop} \label{prop:LlegsfromCF}
Leg relations for $L$ are implied by commutativity of the commutative Frobenius relations.
\end{prop}
\begin{proof}
We want to show that $m\circ ( p^\times \otimes id_{S^2} )  \circ \lambda^{-1}_{S^2} =m \circ (id_{S^2} \otimes p^\times) \circ \rho^{-1}_{S^2} $ is a consequence of commutative Frobenius relations.

Let $f: A\to A'$ and $g:B\to B'$ be two morphisms of a symmetric monoidal category. Then by naturality of braiding , the following diagram commutes 
\[\begin{tikzcd}
{A\otimes B} & {A'\otimes B'} \\
{B\otimes A} & {B'\otimes A'}
\arrow["{f\otimes g}", from=1-1, to=1-2]
\arrow["{\beta _{A,B}}"', from=1-1, to=2-1]
\arrow["{g\otimes f}", from=2-1, to=2-2]
\arrow["{\beta _{B',A'}}"', from=2-2, to=1-2]
\end{tikzcd}\]
In particular, for $f=p^\times: \emptyset \to S^2 $ and $g=id_{S^2}$, we have
\[
p^\times \otimes id_{S^2} = \beta _{S^2 \sqcup S^2} \circ (id_{S^2} \otimes p^\times) \circ \beta_{\emptyset \sqcup S^2} = \beta  _{S^2 \sqcup S^2} \circ (id_{S^2} \otimes p^\times) \circ (\rho^{-1}_{S^2} \circ \lambda_{S^2})
\]

Therefore we may substitute the above relation into the left hand side of the legs relation \[
m\circ (p^\times \otimes id_{S^2}) \circ \lambda^{-1}_{S^2} = m \circ \beta _{S^2 \sqcup S^2} \circ (id_{S^2} \otimes p^\times) \circ (\rho^{-1}_{S^2} \circ \lambda_{S^2}) \circ \lambda^{-1}_{S^2} = m \circ (id_{S^2} \otimes p^\times) \circ \rho^{-1}_{S^2}
.\]
where the second equality is due to commutativity $m = m\circ \beta _{S^2 \sqcup S^2}$. We recognize the right hand side as the desired composition.
\begin{figure}[ht]
    \centering
\def\svgwidth{1\columnwidth} 
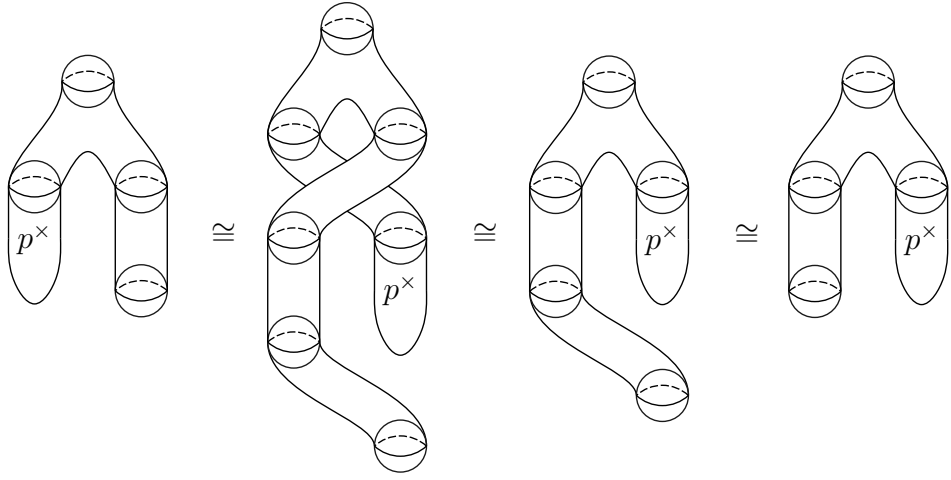

    \caption[Redundancy of $G_1$ legs relations]{Redundancy of $L$ legs relations.}
    \label{fig:llegsfromcf}
\end{figure}
\end{proof}

\begin{rmk}
Note that in $\operatorname{Cob}(3)_{S^2}$, there is equality of morphisms $[p\setminus (D^3 \sqcup D^3)]= m \circ ([p \setminus D^3] \otimes id_{S^2}) \circ \lambda^{-1}_{S^2} = m\circ (id_{S^2} \otimes [p \setminus D^3])\circ \rho ^{-1}_{S^2} $, or more succinctly  \[
p^{\times \times}  = m\circ (p^\times \otimes id_{S^2}) \circ \lambda^{-1}_{S^2} = m \circ (id_{S^2} \otimes p^\times ) \circ \rho ^{-1}_{S^2}
.\]
But this does not constitute a relation in the usual sense-- the left hand side is a $P$ generator while the right hand sides involve $L$ generators. They belong to different presentations, $G_2$ and $G_1$. More on this in lemma \ref{lem:consistencyislegs} and the discussion right before it.
\end{rmk}

\section{Sufficiency of External Relations} \label{sec:external sufficient}
Equipped with prime commutativity, legs and waist relations, co-legs and co-waist relations, we are ready to state a sufficiency theorem among $G_2 $ generators.

\begin{thm} \label{mainthmG2}
$\operatorname{Cob}( 3)_{S^2} $ has a presentation with $G_2$ generators under commutative Frobenius relations, prime commutativity, legs and waist relations, and co-legs and co-waist relations.
\end{thm}

\begin{proof}
Proposition \ref{generate} shows that $G_2 $ is generating. We need to show that the relations listed above are sufficient. Let $[M]$ be a bordism, and $f_0, f_1:M\to \mathbb{R}$ admissible excellent Morse functions inducing two generating compositions. We claim that there is a path from $f_0$ to $f_1$ which exhibits finitely many prime-commutativity relations, co/legs and co/waist relations, or fails to be excellent Morse at finitely many times due to commutative Frobenius relations. Thus the two compositions of $[M]$ are related by a sequence of prime relations and commutative Frobenius relations.

First, either $M$ is closed and $[M] \in \operatorname{End}(\emptyset)$, or $[M]$ has a non-empty boundary. We treat the two cases slightly differently.

If $[M]$ has at least one boundary component $S$, say outgoing, then for the Morse decomposition induced by $f_0 $, 
\begin{enumerate}
\item take a path $\gamma $ from a point on a particular prime endomorphism $p^{\times \times} $ to a point 
$x\in S$.

\item since $M$ is compact, it is finitely generated. Therefore the path $\gamma $ intersects finitely many generating bordisms. In particular it may intersect the boundary of an individual bordism multiple times.

\item The idea is to ``carry'' $p^{\times \times} $ along this path until we reach the outgoing boundary. Take a chronological ordering of the bordism pieces intersected by $\gamma $ as $p\setminus \operatorname{int}(D^3\sqcup D^3), M_1, M_2, \dots M_I$ . For each $M_i$ take a chronological ordering of the boundary components intersected by $\gamma $ as $b_1^i, b_2^i, \dots b_{J_i}^i $. Since the path originates from $p^{\times \times} $, $b_1^1$ is the boundary along which $p^{\times \times}  $ is composed with $[M_1]$. Since the path terminates on $S$, we have $b_{J_I}^I = S$.

If $b_1^i, b_{J_i}^i$ are different boundary components, apply the appropriate relation. After the relation is applied, $[p]$ is composed with $[M_i]$ along boundary component $b_{J_i}^i$ and, since $[M_i]$ and $[M_{i+1}]$ were glued along $b_{J_i}^i=b_1^{i+1}$, $p^{\times \times} $ is also composed with $[M_{i+1}]$ along boundary component $b_1^{i+1}$.

If $b_1^i = b_{J_i}^i$ are the same boundary component, namely $\gamma $ enters and exists $[M_i]$ through the same boundary component, then $[M_{i+1}] = [M_{i-1}]$. We take the identity relation (aka do nothing). Then $p^{\times \times} $ is still composed with $[M_i]$ along $b_1^i = b_{J_i}^i$ and with  $[M_{i+1}]$ along $b_1^{i+1}$.

\item The fact that $\gamma $ ends on $S=b_{J_I}^I$ guarantees that this process terminates with $p^{\times \times} $ ending up composed with $[M_I]$ along $S$.

\item We will then repeat this finite process for the remaining finitely many prime factors, each time choosing the target $S$ as the outgoing boundary of the previously chosen prime factor.

\item Now that all the prime bordisms are composed in a sequence, we use naturality to bring the composition to the form \[
\left( (p^{\times \times}_1 \circ p^{\times \times}_2 \circ \dots \circ p^{\times \times}_n) \otimes id  \right) \circ s_0
.\]
where $s_0$ is a composition of CF bordisms which equals $[\#^g \left( S^2 \times S^1 \right)  \setminus \operatorname{int}(\sqcup^{n_{inc}+n_{out}}D^3) ]$. Let $f_0 ': M\to  \mathbb{R}$ be a corresponding admissible Morse function. 

\item Repeat this process for $f_1 $, choosing the same outgoing boundary component $S$ at the beginning, bringing its corresponding composition to the form \[
\left( (p^{\times \times}_1 \circ p^{\times \times}_2 \circ \dots \circ p^{\times \times}_n) \otimes id  \right) \circ s_1
.\]
where $s_1$ is again a (possibly different) composition of CF bordisms which equals $[\#^g \left( S^2 \times S^1 \right)  \setminus \operatorname{int}(\sqcup^{n_{inc}+n_{out}}D^3) ]$. Let $f_1 ':M\to \mathbb{R}$ be a corresponding admissible Morse function.

\item By proposition \ref{prop:keypropCF}, there exists a path \[
\widehat{f}_t':I \times \left( \#^g \left( S^2\times S^1 \right) \setminus \operatorname{int}(\sqcup^{n_{inc}+n_{out}}D^3) \right) \to \mathbb{R}
\]
from $f_0' |_{\#^g \left( S^2\times S^1 \right)  \setminus \operatorname{int}(\sqcup^{n_{inc}+n_{out}}D^3)}$ to $f_1' |_{\#^g \left( S^2\times S^1 \right) \setminus \operatorname{int}(\sqcup^{n_{inc}+n_{out}}D^3)}$ which is admissible excellent spherical Morse at all but finitely many times where the failures are of CF type.

\item Hence there is a sequence of commutative Frobenius relations which sends $s_0 $ to $s_1 $. Taking the reverse process from $f_1 '$ to $f_1 $, we have obtained a sequence of relations which sends one composition to the other. Cerf-theoretically, there is a path $f_0$ to $f_0'$ exhibiting all the prime relations used in turning the composition into the standard form. We glue the path $\widehat{f}_t'$ and the constant path $f_0'|_{p_1 \# p_2 \dots \# p_n \setminus \operatorname{int}(D^3 \sqcup D^3)}$ together, forming a path $f_t'$ from $f_0'$ to $f_1 '$. Taking the path from $f_1 $ to $f_1 '$ in reverse, we have constructed a path from $f_0 $ to $f_1 $ with the desired property.
\end{enumerate}

If $[M]$ has empty boundary component, we must deal with the fact that there is no preferred boundary component to stack the prime factors. At first glance the situation is further complicated by the fact that $f_0 $ and $f_1 $ may achieve minima and maxima at different points on $M$, so there is no preferred point towards which to send the prime factors. However, we may treat $f_0 $, $f_1 $ independently. One can still choose regular values $r_0, r_1 $ close to say the maxima, such that the regular surfaces $S_0 = f_0 ^{-1}(r_0 ) ,S_1 = f_1 ^{-1}(r_1 ) $ are still diffeomorphic to $S^2$. We may repeat the same process, bringing all the prime factors towards the respective maxima using only prime commutativity and legs and waist relations, such that the composition induced by $f_0 $ is brought to the form \[
tr \circ p^{\times \times}_1 \circ p^{\times \times}_2 \circ \dots \circ p^{\times \times}_n \circ s_0 
.\]
and the composition induced by $f_1 $ brought to the form \[
tr \circ p^{\times \times}_1 \circ p^{\times \times}_2 \circ \dots \circ p^{\times \times}_n \circ s_1
.\]
where $s_0 ,s_1 $ are two composition sequences of CF bordisms which equals $[\#^g \left( S^2 \times S^1 \right)  \setminus \operatorname{int}( D^3)]$. Then proposition \ref{prop:keyprop} provides a sequence of commutative Frobenius moves turning $s_0 $ into $s_1 $.
\end{proof}

\begin{thm} \label{mainthmG1}
$\operatorname{Cob}( 3)_{S^2} $ has a presentation with $G_1$ generators and commutative Frobenius relations.
\end{thm}

It is very tempting to treat this as a corollary of theorem \ref{mainthmG2} by exploiting the fact that every $p^\times \in L$ can be written as $p^{\times \times}  \circ 1$ and then ``carry'' the prime factors to some extrema or boundary component like we did in the proof of \ref{mainthmG2}. However this is not a correct proof of the sufficiency of commutative Frobenius relations, at least not on the nose-- we would have to use the legs and waist relations of $G_2$, as well as prime commutativity of $G_2$.
What we need is instead a sequence of relations which preserves the prime $L$ morphisms of $G_1$. 

One way to fix this approach is by noting that even though $p^\times = p^{\times \times} \circ 1$ inevitably uses $p^{\times \times} $, we can rewrite $p^{\times \times} $ in terms of $p^\times$ to get \[
p^\times = p^{\times \times} \circ 1 = \Big(m\circ (p^\times \otimes id_{S^2})\circ \lambda^{-1}_{S^2}\Big) \circ 1
.\]
Notice that the right most term is entirely in terms of $G_1 $ generators. Then one only needs to show that prime commutativity, legs and waist relations of $P$ have analogous versions in terms of $m\circ (p^\times \otimes id_{S^2} \circ \lambda^{-1}_{S^2})$ as a whole, instead of $p^{\times \times} $, using only commutative Frobenius relations. Therefore each use of prime commutativity or legs and waist relations of $P$ does in fact result in a composition in terms of $G_1$ generators.

Below we present a cleaner proof without breaking up the $L$ generators. The proof requires the repeated use of the following lemma, see notations \ref{coherencenotations}:

\begin{lemma}\label{lemma:pullaside}
Let  $(\mathcal{C}, 1_{\otimes }, \alpha , \beta , \lambda, \rho )$ be a symmetric monoidal category, and let $f: A \to A'$, $g: B \to B'$, $h: C \to A \otimes B$, $p: 1_ \otimes \to S$ be morphisms. Then the following is true as morphisms $C \to A' \otimes B'$ \[
\left( f \otimes \left( p \otimes g \right)  \right) \circ \left( id_A \otimes \lambda ^{-1}_B \right) \circ h = \left( f \otimes \left( id_S \otimes g \right)  \right) \circ \left( id_A \otimes \beta_{B,S} \right) \circ \alpha _{A,B,S} \circ \left( h \otimes p \right) \circ \rho ^{-1}_C
\]
\end{lemma}

\begin{rmk}
Notice that one can use bifunctoriality and write the right hand side as \[
\left( f \otimes \left( id_S \otimes g \right)  \right) \circ \left( id_A \otimes \beta_{B,S} \right) \circ \alpha _{A,B,S} \circ \left( h \otimes id_S \right) \circ (id_C \otimes p) \circ \rho ^{-1}_C
.\]
\end{rmk}

\begin{mycom}
The take home lesson is this: given any morphism $M : C \to C'$ which is equal to a composition containing a factor of $p: 1_ \otimes \to S$, $M$ is necessarily of the form \[
M = H \circ \left( f \otimes \left(p \otimes g\right) \right) \circ (id_A \otimes \lambda^{-1}_B) \circ h
.\]
for some $H: A' \otimes \left( S \otimes B' \right) \to C'$. We have shown that $M$ has an alternative composition as \[
M = H \circ \left( f \otimes \left( id_S \otimes g \right)  \right) \circ \left( id_A \otimes \beta_{B,S} \right) \circ \alpha _{A,B,S} \circ \left( h \otimes p \right) \circ \rho ^{-1}_C
.\]
We can clean up the mess by packaging all but the last two terms together, into the form \[
M = \overline{H} \circ \left( h \otimes p \right) \circ \rho ^{-1}_C
.\]
This means that we can always extract/isolate the morphism $p$ by ``pulling it aside/passing through the right.''

The alternative form in the remark enables us to write \[
M = \overline{H'} \circ (id_C \otimes p) \circ \rho ^{-1}_C
.\]
which gives an extraction/isolation of $p$ by ``pulling aside and pushing down.'' These are best summarized intuitively in the following figure \ref{fig:pullaside}.

\begin{figure}[ht]
    \centering
\def\svgwidth{1\columnwidth} 
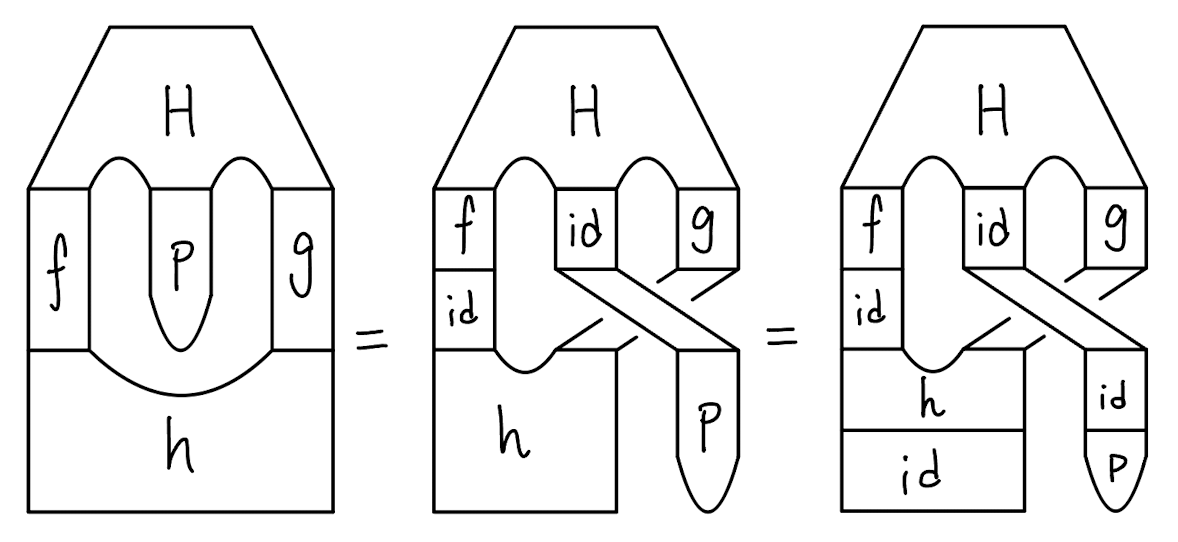

    \caption[Pulling aside a map out of $1_\otimes$]{Pulling aside a map out of $1_ \otimes $.}
    \label{fig:pullaside}
\end{figure}

\end{mycom}

\begin{proof}[Proof of Lemma \ref{lemma:pullaside}]
Starting with the left hand side $\left( f \otimes \left( p \otimes g \right)  \right) \circ \left( id_A \otimes \lambda ^{-1}_B \right) \circ h$, first rewrite it using bifunctoriality as \[
\left( f \otimes \left( id_S \otimes g \right)  \right) \circ \left( id_A \otimes \left( \left( p \otimes id_B \right) \circ \lambda^{-1}_B \right)  \right) \circ h
.\]
Using the braiding relation $\left( p \otimes id_B \right) \circ \lambda^{-1}_B = \beta _{B,S}\circ \left( id_B \otimes p \right) \circ \rho ^{-1}_B$, the above is equivalent to \[
\left( f \otimes \left( id_S \otimes g \right)  \right) \circ \left( id_A \otimes \left( \beta _{B,S}\circ \left( id_B \otimes p \right) \circ \rho ^{-1}_B \right)  \right) \circ h
.\]
Rewrite this expression using bifunctoriality as \[
\left( f \otimes \left( id_S \otimes g \right)  \right) \circ \left( id_A \otimes \beta _{B,S} \right) \circ \left( id_A \otimes \left( id_B \otimes p \right)  \right) \circ \left( id_A \otimes \rho ^{-1}_B \right) \circ h
.\]
Using a coherence condition [Kelly 64] of unitors and aassociators, the following diagram commutes for all objects $X,Y$,
\[\begin{tikzcd}
{(X \otimes Y) \otimes 1_\otimes} && {X\otimes Y} \\
{X \otimes (Y \otimes 1_\otimes)}
\arrow["{\rho_{X \otimes Y}}", from=1-1, to=1-3]
\arrow["{\alpha_{X,Y,1_\otimes}}"', from=1-1, to=2-1]
\arrow["{id_X \otimes \rho_Y}"', from=2-1, to=1-3]
\end{tikzcd}\]
In particular for $X=A, Y=B$, we have $id_A \otimes \rho^{-1}_B= \alpha _{A,B, 1_ \otimes } \circ \rho ^{-1}_{A \otimes B} $, so the left hand side is equal to  \[
\left( f \otimes \left( id_S \otimes g \right)  \right) \circ \left( id_A \otimes \beta _{B,S} \right) \circ \left( id_A \otimes  \left( id_B \otimes p \right)  \right)  \circ \alpha _{A,B, 1_ \otimes } \circ \rho ^{-1}_{A \otimes B} \circ h
.\]
We may pass the associator to the left using naturality and get \[
\left( f \otimes \left( id_S \otimes g \right)  \right) \circ \left( id_A \otimes \beta_{B,S} \right) \circ  \alpha _{A,B,S} \circ \left( id_{A \otimes B}  \otimes p   \right) \circ \rho ^{-1}_{A \otimes B} \circ h
.\]
Shifting our attention to the last two terms, note that the naturality of the right unitor makes the following diagram commute
\[\begin{tikzcd}
C & {C\otimes 1_\otimes} \\
{A\otimes B} & {(A\otimes B) \otimes 1_\otimes}
\arrow["{\rho^{-1}_C}", from=1-1, to=1-2]
\arrow["h"', from=1-1, to=2-1]
\arrow["{h \otimes id_{1_\otimes}}", from=1-2, to=2-2]
\arrow["{\rho^{-1}_{A\otimes B}}"', from=2-1, to=2-2]
\end{tikzcd}\]
Hence $\rho ^{-1}_{A \otimes B} \circ h = \left( h \otimes id_{1_ \otimes } \right) \circ \rho ^{-1}_C$, and we get \[
\left( f \otimes \left( id_S \otimes g \right)  \right) \circ \left( id_A \otimes \beta_{B,S} \right) \circ  \alpha _{A,B,S} \circ \left( id_{A \otimes B}  \otimes p   \right) \circ \left( h \otimes id_{1_ \otimes } \right) \circ \rho ^{-1}_C
.\]
Combining the third and second to last terms, we get the desired right hand side \[
\left( f \otimes \left( id_S \otimes g \right)  \right) \circ \left( id_A \otimes \beta_{B,S} \right) \circ  \alpha _{A,B,S} \circ (h \otimes p) \circ \rho ^{-1}_C
.\]
\end{proof}

\begin{proof}[Proof of Theorem \ref{mainthmG1}]
Proposition \ref{generate} shows that $G_1 $ generates morphisms. So we need to show that commutative Frobenius relations are sufficient and necessary.

The idea of the proof for sufficiency is simple: use lemma \ref{lemma:pullaside} repeatedly by induction over finitely many prime units, isolating and grouping them into a manageable composite, such that whatever remains is generated by non-prime morphisms.  We then invoke proposition \ref{prop:keyprop}, freely changing the remaining morphisms using only commutative Frobenius relations. 

Precisely, let $[M]$ be a morphism in $\operatorname{Cob}(3)_{S^2}$ from $n_{inc}$-many $S^2$ to $n_{out}$-many $S^2$. Pick any representative and let its unique (non-$S^2 \times S^1$) prime factors be $p_1 , p_2 , \dots p_m$. Hence $[M]$ is $G_1 $-generated by commutative Frobenius generators as well as prime units $p^\times_1, \dots , p^\times_m: \emptyset \to S^2$. We want to show that any two such compositions are related by, and only by, commutative Frobenius relations.

Without loss of generality, consider the prime unit $p^\times_1$ in either compositions. As a morphism whose domain is $\emptyset $, the compositions necessarily take the form 
\begin{align*}
[M] &= H_1 \circ \left( f_1 \otimes (p^\times_1 \otimes g_1 ) \right) \circ (id_{A_1 } \otimes \lambda^{-1}_{B_1 }) \circ h_1 \\
[M] &= H_2 \circ \left( f_2 \otimes (p^\times_1 \otimes g_2 ) \right) \circ (id_{A_2 } \otimes \lambda^{-1}_{B_2 }) \circ h_2 
\end{align*}
for some morphisms $f_{1,2} : A_{1,2} \to A_{1,2}' $, $g_{1,2}: B_{1,2} \to B_{1,2}'$, $h_{1,2}: \otimes ^{n_{inc}} S^2 \to A_{1,2} \otimes B_{1,2}$, $H_{1,2}: A_{1,2}' \otimes (S^2 \otimes B_{1,2}') \to \otimes ^{n_{out}} S^2$. These morphisms themselves are generated by commutative Frobenius generators and the remaining prime units $p^\times_2 ,\dots , [p_m\setminus D^3]$.

Focusing our attention to the first composition, lemma \ref{lemma:pullaside} allows us to use only coherences to set the first composition as \[
[M] = \left( f_1  \otimes \left( id_{S^2} \otimes g_1  \right)  \right) \circ \left( id_{A_1 } \otimes \beta_{B_1 ,S^2} \right) \circ \alpha _{A_1 ,B_1 ,S^2} \circ \left( h_1 \otimes id_{S^2} \right) \circ (id_{\otimes ^{n_{inc}} S^2} \otimes p^\times_1) \circ \rho ^{-1}_{\otimes ^{n_{inc}} S^2}
.\]

Keeping the last two terms untouched, rewrite this composition in a shorten form as \[
[M] = \overline{H}_1^1 \circ (id_{\otimes ^{n_{inc}} S^2} \otimes p^\times_1)\circ \rho ^{-1}_{\otimes ^{n_{inc}}S^2}
.\]
where $\overline{H}_1^1 = \left( f_1  \otimes \left( id_{S^2} \otimes g_1  \right)  \right) \circ \left( id_{A_1 } \otimes \beta_{B_1 ,S^2} \right) \circ \alpha _{A_1 ,B_1 ,S^2} \circ \left( h_1 \otimes id_{S^2} \right): \otimes ^{n_{inc}+1} S^2 \to \otimes ^{n_{out}}S^2 $. Note that since $\overline{H}_1 ^1$ no longer has the $p^\times_1$ factor, it contains only the remaining prime factors $p^\times_2, \dots p^\times_m$. The $1$ in the superscript of $\overline{H}_1^1$ records this fact.

We may inductively repeat this process for each of the remaining prime factors, giving a sequence of equalities \[
\overline{H}_1^i = \overline{H}_1^{i+1} \circ (id_{\otimes ^{inc+i} S^2} \otimes p^\times_{i+1})\circ \rho ^{-1}_{\otimes ^{inc + i}S^2}
.\]
where $\overline{H}_1^j$ is generated by commutative Frobenius generators and prime factors $p^\times_{j+1}, \dots p^\times_m$. In particular, at the termination of this process, $\overline{H}_1^m $ is generated by commutative Frobenius generators alone.

Therefore using purely coherence relations, the first composition can be rewritten as\footnote{We are essentially casting the underlying 3 manifold into a normal form.} \[
[M] = \overline{H}_1^m \circ (id_{\otimes ^{inc + m} S^2} \otimes p^\times_m) \circ \rho ^{-1}_{\otimes ^{n_{inc}+m} S^2} \circ \dots \circ (id_{\otimes^{n_{inc}} S^2} \otimes p^\times_1) \circ \rho ^{-1}_{\otimes ^{n_{inc}} S^2}
.\]

Similarly, using purely coherence relations, the second composition can be rewritten as \[
[M] = \overline{H}_2^m \circ (id_{\otimes ^{inc + m} S^2} \otimes p^\times_m) \circ \rho ^{-1}_{\otimes ^{n_{inc}+m} S^2} \circ \dots \circ (id_{\otimes^{n_{inc}} S^2} \otimes p^\times_1) \circ \rho ^{-1}_{\otimes ^{n_{inc}} S^2}
.\]

Since both $\overline{H}_{1,2}^m$ are generated entirely by commutative Frobenius generators, proposition \ref{prop:keyprop} shows that they are related by commutative Frobenius relations. Recalling that we have not used any non-coherence relations, we have shown that any two $G_1$ compositions of $[M]$ can be turned into each other using commutative Frobenius relations alone.

That the commutative Frobenius relations are necessary follows from theorem \ref{thm:Cob2Cob3CFIso}. 

%
\end{proof}

\section{Necessity of Legs Relations} \label{sec:legsnecessary}
This section exhibits all coherence data. See notation \ref{coherencenotations}. The reader is urged to reference the graphical summary of the proofs whenever possible.

\begin{mycom}
\textbf{(Important comment)} Before we prove that some of the prime relations of $P$ are redundant, we need to make an important observation.

There is an important geometric relation which shall play a role in the proofs. It is the relation we introduced earlier in comment \ref{criticalrelations}: \[
[p \setminus \operatorname{int}(D^3 \sqcup D^3)] = m \circ ([p \setminus \operatorname{int}(D^3)] \otimes id_{S^2}) \circ \lambda^{-1}_{S^2}= m \circ (id_{S^2} \otimes [p \setminus \operatorname{int}(D^3)] ) \circ \rho ^{-1}_{S^2}
.\]
where the second equality is the same as legs relations for $L$, and the morphism $[p\setminus \operatorname{int}(D^3)] : \emptyset \to S^2$ is interpreted as a prime unit. This equality, along with the simple $[p\setminus \operatorname{int}(D^3)] = [p\setminus \operatorname{int}(D^3 \sqcup D^3)] \circ 1$, showed that there is a bijection between the $G_1 $ generators and the $G_2 $ generators. 


It is necessary and important to stress that these two equalities do not constitute relations for either $G_1 $ or $G_2 $ presentations, as they are expressed in terms of a mix of $P$ generators (on the left hand sides) and $L$ generators (on the right hand sides). \footnote{These two equalities could be thought of as ``consistency conditions'' or ``compatibility conditions'' (across $P$ and $L$, along with the existing commutative Frobenius generators). Topology forces these conditions to be true in $\operatorname{Cob}(3)_{S^2} $.}

Note, however, if one chooses to express the right hand side $L$ generators in terms of the $P$ generators using $[p\setminus \operatorname{int}(D^3)] = [p\setminus \operatorname{int} (D^3 \sqcup D^3)] \circ 1$ 
then they do become genuine relations. The conditions read 
\begin{equation}\label{eqn:consistency}
p^{\times \times}  = m \circ \left( (p^{\times \times}  \circ 1) \otimes id_{S^2} \right) \circ \lambda^{-1}_{S^2} = m \circ \left( id_{S^2} \otimes \left( p^{\times \times}  \circ 1 \right)  \right)  \circ \rho ^{-1}_{S^2}
\end{equation}
which we will show is equivalent to legs relations for $P$.
\end{mycom}

\begin{lemma}\label{lem:consistencyislegs}
The following two terms are both equal to $p^{\times \times} $ if and only if the legs relations for $P$ holds:
\[
m \circ \left( (p^{\times \times}  \circ 1) \otimes id_{S^2} \right) \circ \lambda^{-1}_{S^2} ,\quad   m \circ \left( id_{S^2 } \otimes  (p^{\times \times}  \circ 1)  \right) \circ \rho ^{-1}_{S^2} \]
\end{lemma}

\begin{proof}

Given the legs relation $m\circ \left( p^{\times \times}  \otimes id_{S^2} \right) = m \circ \left( id_{S^2} \otimes p^{\times \times} \right) $, precompose both sides with $(1 \otimes id_{S^2}) \circ \lambda^{-1}_{S^2}$, to get \[
m\circ \left( \left( p^{\times \times}  \circ 1 \right)  \otimes id_{S^2} \right) \circ \lambda^{-1}_{S^2} = m \circ \left( 1 \otimes  p^{\times \times} \right) \circ \lambda ^{-1}_{S^2}
.\]
The left hand side is in the desired form. We just need the right to be $p^{\times \times} $. Rewrite the middle term on the right hand side as  $\left(1 \otimes id_{S^2}\right) \circ \left( id_\emptyset \otimes p^{\times \times}  \right) $, and use naturality of the unitor $ \lambda ^{-1}_{S^2} \circ p^{\times \times}  = (id_\emptyset \otimes p^{\times \times} ) \circ \lambda^{-1}_{S^2}$, the right hand side becomes \[
m\circ (1 \otimes id_{S^2}) \circ \lambda^{-1}_{S^2} \circ p^{\times \times}  = id_{S^2} \circ p^{\times \times}  =p^{\times \times}  
\]
where the first equality is due to $1$ being the multiplicative unit.

Note that similarly by precomposing both sides with $(id_{S^2} \otimes 1)\circ \rho ^{-1}_{S^2}$ one can show that legs relations give \[
	p^{\times \times}  = m\circ \left( id_{S^2} \otimes \left( p^{\times \times} \circ 1 \right)  \right) \circ \rho ^{-1}_{S^2}
.\]
which does not come as a surprise-- the right hand sides are related by braiding and commutativity, or can be thought of as a legs relations for $L$ in disguise.

Now for the only if direction, which is summarized in figure \ref{fig:plegsfromconsistency}, given that 
\[
p^{\times \times}  = m\circ \left(\left( p^{\times \times}  \circ 1 \right)  \otimes id_{S^2}\right) \circ \lambda_{S^2}^{-1} = m \circ \left(id_{S^2} \otimes \left( p^{\times \times}  \circ 1 \right)  \right) \circ \rho _{S^2}^{-1}
.\]
Substitute this into $m\circ (p^{\times \times}  \otimes id_{S^2})$to get \[
m \circ (p^{\times \times}  \otimes id_{S^2}) = m\circ \left( \left(  m \circ (id_{S^2} \otimes \left( p^{\times \times}  \circ 1 \right) ) \circ \rho _{S^2}^{-1}
\right) \otimes id_{S^2}  \right) 
.\]
By bifunctoriality of $\otimes $, this is \[
=m\circ (m \otimes id_{S^2}) \circ \left( (id_{S^2} \otimes \left( p^{\times \times}  \circ 1 \right) ) \otimes id_{S^2} \right)  \circ (\rho ^{-1}_{S^2} \otimes id_{S^2})
.\]

Use the precise form of associativity $m\circ (m \otimes id_{S^2})= m\circ (id_{S^2} \otimes m) \circ \alpha_{S^2, S^2, S^2}$ on the left most two terms giving \[
=m \circ (id_{S^2} \otimes m) \circ \alpha _{S^2,S^2,S^2} \circ \left(\left(id_{S^2}\otimes \left( p^{\times \times}  \circ 1 \right) \right)\otimes id_{S^2}\right) \circ (\rho ^{-1}_{S^2} \otimes id_{S^2})
.\]

by naturality of the associator, we may pass it through the right, giving \[
=m\circ (id_{S^2} \otimes m) \circ (id_{S^2} \otimes \left(\left( p^{\times \times}  \circ 1 \right)  \otimes id_{S^2}\right)) \circ \alpha_{S^2, \emptyset, S^2} \circ (\rho ^{-1}_{S^2} \otimes id_{S^2})
.\]
which we rewrite using bifunctoriality as \[
=m\circ \left(id_{S^2} \otimes \left(m \circ \left(\left( p^{\times \times}  \circ 1 \right)  \otimes id_{S^2}\right)\right)\right) \circ \left(\alpha_{S^2, \emptyset, S^2} \circ \left(\rho ^{-1}_{S^2} \otimes id_{S^2}\right)\right)
.\]
By coherence of the unitors and the associators, $\alpha_{S^2, \emptyset, S^2} \circ (\rho ^{-1}_{S^2} \otimes id_{S^2}) = id_{S^2} \otimes \lambda_{S^2}^{-1}$. We may use bifunctoriality again to get \[
=m\circ \left( id_{S^2} \otimes \left( m \circ \left( \left( p^{\times \times}  \circ 1 \right) \otimes id_{S^2}\right) \circ \lambda_{S^2}^{-1} \right)  \right) 
.\]
henceforth we reach the desired right hand side $m\circ (id_{S^2} \otimes p^{\times \times} )$.
\begin{figure}[ht]
    \centering
\def\svgwidth{1\columnwidth} 
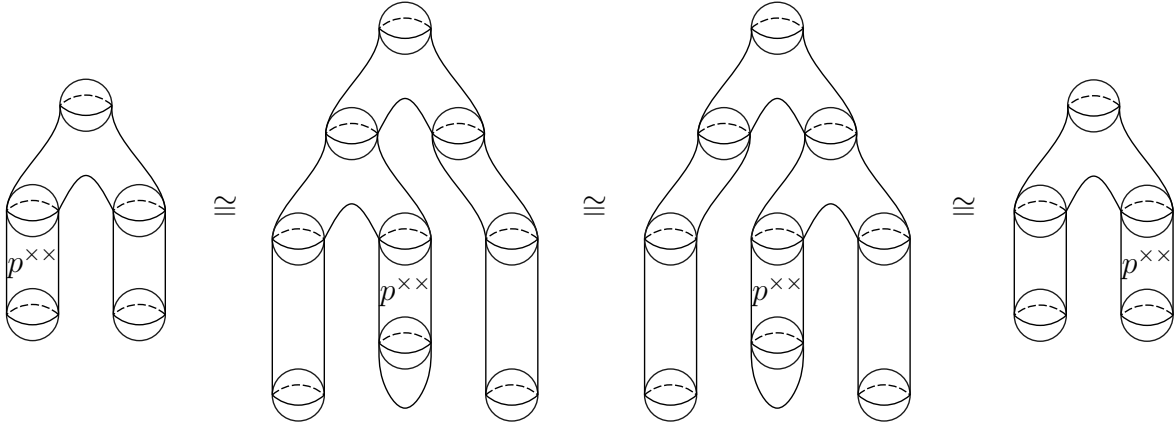

\caption[$G_{1,2}$ bijection induces $G_2$ legs relations]{The only if direction of lemma \ref{lem:consistencyislegs}.}
    \label{fig:plegsfromconsistency}
\end{figure}
\end{proof}

\begin{rmk}
Notice that the two terms in the lemma above are equal due to braiding and commutativity, the proof of which is apparent once one interprets $[p\setminus \operatorname{int}(D^3 \sqcup D^3)]\circ 1 = [p\setminus \operatorname{int}(D^3)] : \emptyset \to S^2$ as a prime unit. The proof for legs relation for $L$ (which uses commutativity) applies directly. However, we do \textit{not} need this fact for the if and only if statement to be true, so commutativity is not involved in the proof of the lemma.
\end{rmk}

\begin{mycom}
\textbf{Important comment continued} The proofs of following propositions require that the legs relations hold, or equivalently that all three terms in lemma \ref{lem:consistencyislegs} are equal. 
This means that, with the exception of legs relations and commutative Frobenius relations, all other prime relations of $G_2 $ are redundant.
\end{mycom}

\begin{notations}
Due to the repeated occurrence of the term $p^{\times \times}  \circ 1$ in the following proofs, we will adopt the short hand $x=p^{\times \times}  \circ 1$. The reason not to use the equality of composition $[p\setminus \operatorname{int}(D^3)] = [p\setminus \operatorname{int}(D^3 \sqcup D^3)] \circ 1 = p^{\times \times} \circ 1$ is that the left hand side is naturally interpreted as a generator of $G_1 $. We write $x$ to remind ourselves that $x$ is generated by $G_2$.
\end{notations}

\begin{prop} \label{waistfromCF}
Waist relations for $P$ are implied by legs relations and associativity of the commutative Frobenius relations.
\end{prop}

\begin{proof} 
We need to show that $p^{\times \times} \circ m = m\circ (p^{\times \times} \otimes id_{S^2})$ is induced by commutative Frobenius relations. This proof is graphically summarized by figure \ref{fig:waistfromcf}

First, by lemma \ref{lem:consistencyislegs}, the following relations hold  \[
p^{\times \times}  = m\circ (x \otimes id_{S^2}) \circ \lambda_{S^2}^{-1} = m \circ (id_{S^2} \otimes x) \circ \rho _{S^2}^{-1}
.\]
where $x = p^{\times \times}  \circ 1 : \emptyset \to S^2$. Substitute this for $p^{\times \times} $ on the left hand side to get \[
m\circ (x \otimes id_{S^2} ) \circ \left( \lambda_{S^2} ^{-1} \circ m \right) 
.\]

Unitors are natural isomorphisms. By the naturality of the left unitor $\lambda$, the following diagram commutes 
\[\begin{tikzcd}
A & {1_\otimes \otimes A} \\
B & {1_\otimes \otimes B}
\arrow["{\lambda_A^{-1}}", from=1-1, to=1-2]
\arrow["f", from=1-1, to=2-1]
\arrow["{id_{1_\otimes} \otimes f}", from=1-2, to=2-2]
\arrow["{\lambda_B^{-1}}", from=2-1, to=2-2]
\end{tikzcd}\]
for all $f:A\to B$ in the monoidal category. In particular when $f$ is $m: S^2 \sqcup S^2 \to S^2$, we learn that $\lambda_{S^2}^{-1} \circ m = \left( id_{\emptyset} \otimes m  \right) \circ \lambda^{-1}_{S^2 \sqcup S^2} $. Substituting this to the right most parentheses, we get \[
m \circ (x \otimes  id_{S^2}) \circ (id_{\emptyset }\otimes m) \circ \lambda^{-1}_{S^2 \sqcup S^2}
.\]

By a coherence\cite{kelly_maclanes_1964} of unitors and associators, the following diagram commutes 
\[\begin{tikzcd}
{A \otimes B} & {1_\otimes\otimes(A\otimes B)} \\
& {(1_\otimes\otimes A)\otimes B}
\arrow["{\lambda^{-1}_{A\otimes B}}", from=1-1, to=1-2]
\arrow["{\lambda^{-1}_A \otimes id_B}"', from=1-1, to=2-2]
\arrow["{\alpha_{1_\otimes, A, B}}"', from=2-2, to=1-2]
\end{tikzcd}\]
In particular for $A=B=S^2$, we learn that $\lambda^{-1}_{S^2\sqcup S^2} = \alpha_{\emptyset, S^2, S^2} \circ \lambda^{-1}_{S^2} \otimes id_{S^2}$. Substituting in to get \[
m\circ (x \otimes id_{S^2} ) \circ (id_{\emptyset } \otimes m) \circ \alpha _{\emptyset, S^2, S^2} \circ (\lambda^{-1}_{S^2} \otimes id_{S^2})
.\]
We will use bifunctoriality $(x \otimes id_{S^2}) \circ (id_{\emptyset } \otimes m) = x \otimes m = (id_{S^2} \otimes m) \circ (x \otimes (id_{S^2} \otimes  id_{S^2}))$, and pass the associator to the left using naturality to get \[
m \circ (id_{S^2} \otimes m) \circ \alpha _{S^2, S^2, S^2} \circ ((x \otimes id_{S^2}) \otimes id_{S^2}) \circ (\lambda^{-1}_{S^2} \otimes id_{S^2})
.\]
Now use associativity on the left most two terms, giving \[
m\circ (m \otimes id_{S^2}) \circ ((x \otimes id_{S^2}) \otimes id_{S^2}) \circ (\lambda^{-1}_{S^2} \otimes id_{S^2})
.\]
Combine the terms to get \[
m\circ \left( \left( m\circ \left( x \otimes id_{S^2} \right) \circ \lambda^{-1}_{S^2} \right) \otimes id_{S^2} \right) 
.\]
and henceforth we reach the desired right hand side $m\circ (p^{\times \times}  \otimes id_{S^2})$

\begin{figure}[ht]
    \centering
\def\svgwidth{1\columnwidth} 
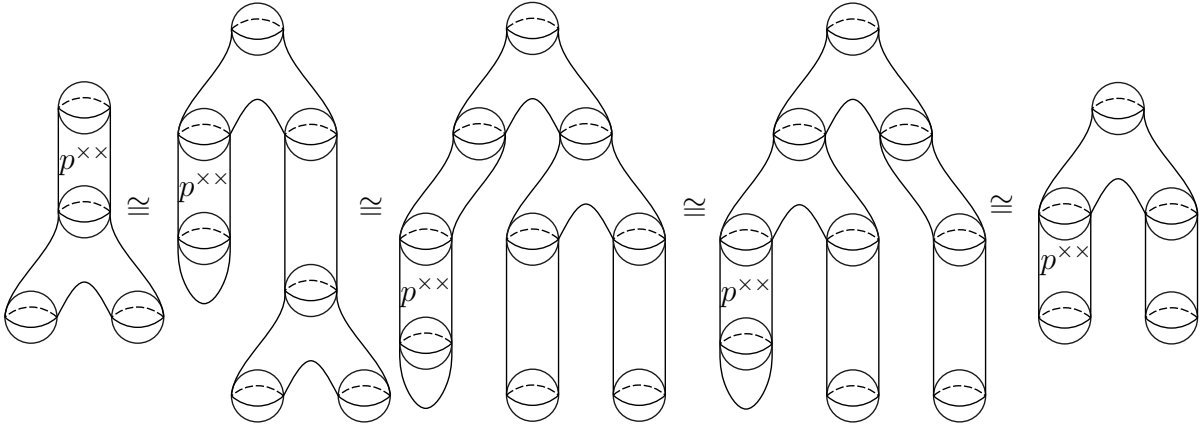

    \caption[Redundancy of $G_2$ waist relations]{Redundancy of $G_2$ waist relations.}
    \label{fig:waistfromcf}
\end{figure}

\end{proof}

\begin{prop} \label{primecommfromCF}
The prime commutativity relations are implied by legs relations and associativity and commutativity of the commutative Frobenius relations.
\end{prop}

\begin{proof}
We need to show that $p^{\times \times}  \circ p^{\times \times \prime}  = p^{\times \times \prime} \circ p^{\times \times} $ is a consequence of associativity and commutativity. The idea of the proof is summarized in figure \ref{fig:primecommfromcf}.

First, by lemma \ref{lem:consistencyislegs}, consider the relations  \[
p^{\times \times}  = m\circ (x \otimes id_{S^2}) \circ \lambda_{S^2}^{-1} = m \circ (id_{S^2} \otimes x) \circ \rho _{S^2}^{-1}
.\]
where $x = p^{\times \times}  \circ 1 : \emptyset \to S^2$. Similarly let $x' = p^{\times \times \prime}  \circ 1 : \emptyset \to S^2$ and write \[
p^{\times \times \prime}  = m\circ (x' \otimes id_{S^2}) \circ \lambda_{S^2}^{-1} = m \circ (id_{S^2} \otimes x') \circ \rho _{S^2}^{-1}
.\]

Substitute $m\circ (x \otimes id_{S^2}) \circ \lambda_A^{-1}$ for $p^{\times \times} $ and $m\circ (id_{S^2}^{-1} \otimes x') \circ \rho _{S^2}^{-1}$ for $p^{\times \times \prime}$ to get \[
p^{\times \times} \circ p^{\times \times \prime}  = m \circ (x \otimes id_{S^2}) \circ \lambda^{-1}_{S^2} \circ m \circ (id_{S^2} \otimes x') \circ \rho ^{-1}_{S^2}
.\]

By naturality of the left unitor explained in the proof of proposition \ref{waistfromCF}, we have $\lambda^{-1}_{S^2} \circ m = (id_{\emptyset} \otimes m) \circ \lambda^{-1}_{S^2 \sqcup S^2}$. So our expression is the same as \[
m\circ (x \otimes id_{S^2}) \circ (id_{\emptyset } \otimes m) \circ \lambda^{-1}_{S^2 \sqcup S^2} \circ (id_{S^2}\otimes  x') \circ \rho ^{-1}_{S^2}
.\]
We may use naturality of $\lambda$ again to obtain $\lambda^{-1}_{S^2 \sqcup S^2} \circ (id_{S^2} \otimes x') = (id_{\emptyset } \otimes (id_{S^2} \otimes x')) \circ \lambda^{-1}_{S^2 \sqcup \emptyset}$, and use bifunctoriality of $\otimes $ to turn the above expression to \[
m\circ  (id_{S^2 } \otimes m) \circ \left( x \otimes (id_{S^2} \otimes x') \right) \circ \lambda^{-1}_{S^2 \sqcup \emptyset} \circ \rho ^{-1}_{S^2}
.\]
By associativity $m\circ (m \otimes id_{S^2}) = m\circ (id_{S^2} \otimes m) \circ \alpha_{S^2, S^2, S^2}$, we may rewrite the left most two terms and get \[
m\circ  (m \otimes id_{S^2}) \circ \alpha^{-1}_{S^2, S^2, S^2} \circ \left( x \otimes (id_{S^2} \otimes x') \right) \circ \lambda^{-1}_{S^2 \sqcup \emptyset} \circ \rho ^{-1}_{S^2}
.\]
Passing the associator $\alpha ^{-1}_{S^2, S^2, S^2}$ to the right using naturality, we get  \[
m\circ  (m \otimes id_{S^2}) \circ  \left( \left( x \otimes id_{S^2} \right)  \otimes x' \right) \circ \alpha_{\emptyset, S^2, \emptyset} \circ \lambda^{-1}_{S^2 \sqcup \emptyset} \circ \rho ^{-1}_{S^2}
.\]

Noting that a coherence condition of \cite{kelly_maclanes_1964} introduced in the proof of proposition \ref{waistfromCF}, when $B = \emptyset$ and  $A = S^2$, gives  $\alpha ^{-1}_{\emptyset, S^2, \emptyset} \circ \lambda^{-1}_{S^2 \sqcup \emptyset }  =\lambda^{-1}_{S^2} \otimes id_{\emptyset}$, we can combine the resulting terms by bifunctoriality and arrive at the expression \[
m\circ \left( m \circ (x \otimes id_{S^2}) \circ \lambda^{-1}_{S^2}  \right) \otimes (id_{S^2} \circ x' \circ id_{\emptyset }) \circ \rho^{-1}_{S^2}
.\]
Apply the leg relation for $L$, we get \[
m \circ \left( m\circ (id_{S^2} \otimes x) \circ \rho_{S^2}^{-1} \right) \otimes \left( id_{S^2} \circ x' \circ id_{\emptyset } \right) \rho ^{-1}_{S^2}
.\]
Next, we reorganize the terms by bifunctoriality in preparation of a use of associativity: \[
m\circ (m \otimes id_{S^2}) \circ \left( (id_{S^2} \otimes x) \otimes x'  \right) \circ (\rho ^{-1}_{S^2} \otimes  id_{\emptyset}) \circ \rho ^{-1}_{S^2}
.\]
Use associativity, and pass the resulting associator to the right using naturality to get \[
m\circ (id_{S^2} \otimes m) \circ \left( id_{S^2} \otimes (x \otimes x') \right) \circ \alpha_{S^2, \emptyset, \emptyset} \circ (\rho ^{-1}_{S^2} \otimes id_{\emptyset}) \circ \rho ^{-1}_{S^2}
.\]
Combine the second and third terms, we reach the expression \[
[p] \circ [p'] = m\circ \left( id_{S^2} \otimes (m \circ (x \otimes x')) \right)  \circ \alpha_{S^2, \emptyset, \emptyset} \circ  \left( \rho ^{-1}_{S^2} \otimes id_{\emptyset} \right) \circ \rho ^{-1}_{S^2}
.\]
A critical observation is that, by braiding, commutativity, and  naturality of the unitors, $m\circ (x \otimes x') = m\circ (x' \otimes x)$. So we have in fact reached an expression which is invariant under the exchange $p^{\times \times}  \leftrightarrow p^{\times \times \prime} $ and $x \leftrightarrow x'$. Therefore we have the desired equality $p^{\times \times}  \circ p^{\times \times \prime} = p^{\times \times \prime}  \circ p^{\times \times} $. Note that we explicitly used associativity in one of our steps and commutativity is the symmetric exchange.
\begin{figure}[ht]
    \centering
\def\svgwidth{1\columnwidth} 
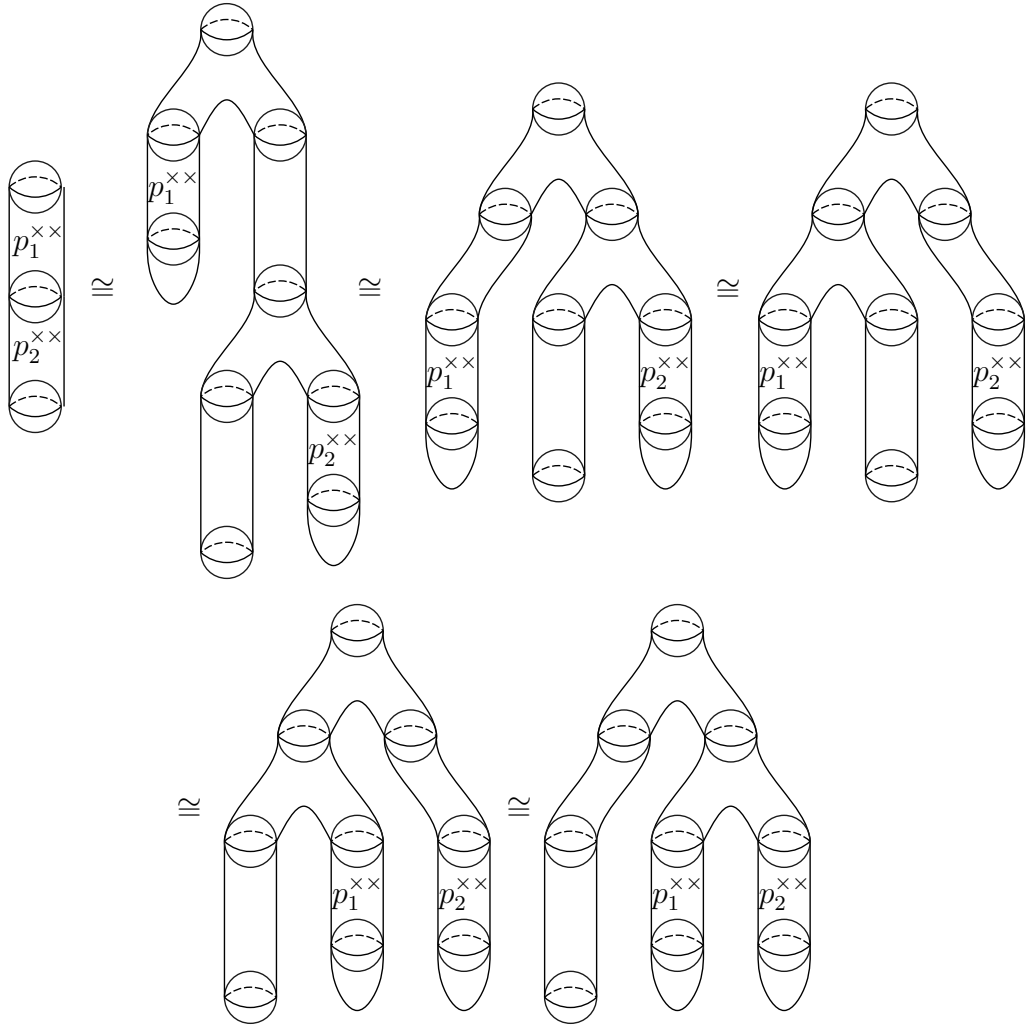

    \caption[Redundancy of prime commutativity]{Redundancy of prime commutativity. Note that the final figure is symmetric under the exchange of $p^{\times \times} _1 \circ 1 \leftrightarrow p^{\times \times} _2 \circ 1$.}
    \label{fig:primecommfromcf}
\end{figure}

\end{proof}

\begin{prop} \label{cowaistfromCF}
Cowaist relations are implied by legs relations and Frobenius.
\end{prop}

\begin{proof}
We want to show $(p^{\times \times}  \otimes id_{S^2}) \circ m^\vee = m^\vee \circ p^{\times \times}  $ is induced by the legs relations and the Frobenius relation. This proof is summarized graphically by figure \ref{fig:cowaistfromcf}

First, by lemma \ref{lem:consistencyislegs}, consider the relations  \[
p^{\times \times}  = m\circ (x \otimes id_{S^2}) \circ \lambda_{S^2}^{-1} = m \circ (id_{S^2} \otimes x) \circ \rho _{S^2}^{-1}
.\]
where $x = p^{\times \times}  \circ 1 : \emptyset \to S^2$. Substitute the first equality into the left hand side, we get \[
\left( \left( m\circ \left( x  \otimes id_{S^2} \right) \circ \lambda^{-1}_{S^2} \right) \otimes id_{S^2}  \right) \circ m^\vee
.\]

Rewrite this in a more convenient form \[
\left( m \otimes id_{S^2} \right) \circ \left( \left(x \otimes id_{S^2} \right) \otimes id_{S^2} \right) \circ \left( \lambda^{-1}_{S^2} \otimes id_{S^2} \right) \circ m^\vee
.\]
and recall the consistency condition of \cite{kelly_maclanes_1964} that $\lambda^{-1}_{S^2} \otimes id_{S^2} = \alpha^{-1}_{\emptyset, S^2, S^2} \circ \lambda^{-1}_{S^2 \sqcup S^2}$, we can substitute it into the expression above and use naturality of the associator to pass it to the left, giving \[
\left( m \otimes id_{S^2} \right) \circ \alpha^{-1}_{S^2, S^2, S^2} \circ \left( x  \otimes id_{S^2 \sqcup S^2} \right) \circ \lambda^{-1}_{S^2 \sqcup S^2} \circ m^\vee
.\]
By naturality of the left unitor, \[
\left( m \otimes id_{S^2} \right) \circ \alpha ^{-1}_{S^2, S^2, S^2} \circ \left( id_{S^2} \otimes id_{S^2 \sqcup S^2} \right) \circ \left( x \otimes m^\vee \right) \circ \lambda^{-1}_{S^2}
.\]
Rewrite this as \[
\left( m \otimes id_{S^2} \right) \circ \alpha ^{-1}_{S^2, S^2, S^2} \circ \left( id_{S^2} \otimes m^\vee \right) \circ \left( x  \otimes id_{S^2} \right) \circ \lambda^{-1}_{S^2}
.\]
and use the Frobenius relation on the first three terms, giving \[
m^\vee \circ m \circ \left( x \otimes id_{S^2} \right) \circ \lambda^{-1}_{S^2}
.\]
which gives the desired right hand side upon using lemma \ref{lem:consistencyislegs}.

\begin{figure}[ht]
    \centering
\def\svgwidth{1\columnwidth} 
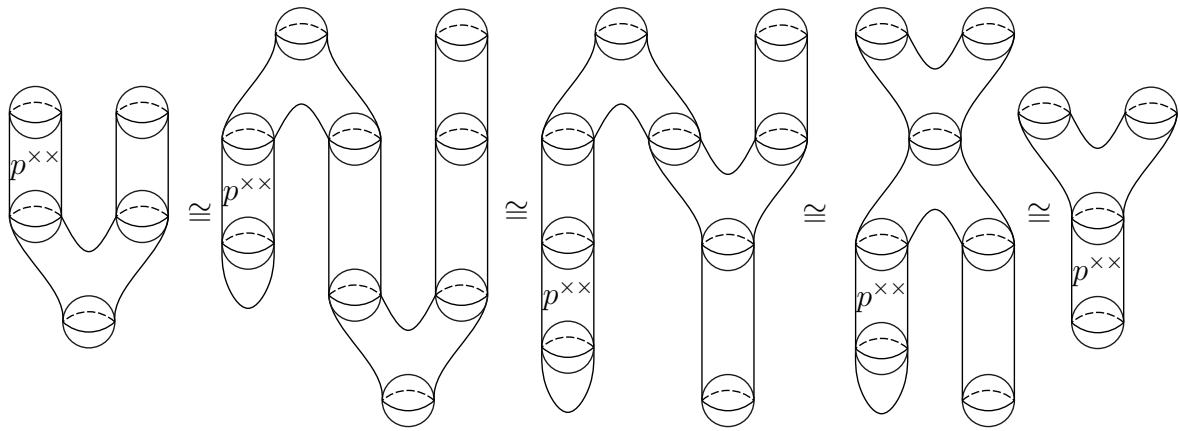

    \caption[Redundancy of $G_2$ cowaist relations]{Redundancy of cowaist relations.}
    \label{fig:cowaistfromcf}
\end{figure}
\end{proof}

\begin{rmk}
The other cowaist relation also holds by essentially the same proof \[
\left( id_{S^2} \otimes p^{\times \times}  \right) \circ m^\vee = m^\vee \circ p^{\times \times} 
.\]
therefore we have the following
\end{rmk}

\begin{prop} \label{colegsfromCF}
Colegs relations are implied by legs relations and the Frobenius relation.
\end{prop}

\begin{proof}
We want to show $\left( p^{\times \times}  \otimes id_{S^2} \right) \circ m^\vee = \left( id_{S^2} \otimes p^{\times \times}  \right) \circ m^\vee$. By legs relations and Frobenius, cowaist relations make both sides equal to $m^\vee \circ p^{\times \times} $.
\end{proof}

\begin{mycom}
The last remaining question is: are the legs relations redundant as well, despite the fact that they hold in $\operatorname{Cob}(3)_{S^2}$ due to topology? To phrase the question differently, are the legs relations implied by commutative Frobenius relations alone? Then answer is negative-- the legs relations for $P$ are not a consequence of just commutative Frobenius relations. 
\end{mycom}

\begin{prop} \label{Plegsnecessary}
Legs relations for $P$ are not induced by commutative Frobenius relations.
\end{prop}

\begin{proof}
Let $\mathbb{R}^2$ be the commutative Frobenius algebra with the Hadamard product. Namely, the multiplication map $m: \mathbb{R}^2 \otimes \mathbb{R}^2 \to \mathbb{R}^2$ is given by $(a,b) \cdot (c,d) = (ac,bd)$. The unit with respect to this multiplication is the element $(1,1)$. Let $e_1 = (1,0)$ and $e_2 = (0,1)$ be the standard basis. The comultiplication is defined by $m^\vee (e_1) = e_1 \otimes e_1 $ and $m^\vee (e_2 ) = e_2 \otimes e_2 $, then extended by linearity. Trace is given by $tr(a,b) = a+b$. One easily verifies that this is a commutative Frobenius algebra.

Let $Gra_2$ be the graph with $\left\{ \sqcup^{i\geq 0} S^2 \right\}\cong Z_{\geq 0}$ as the vertex set, and $G_2 $ and all its monoidal products as directed edges, and $F(Gra_2)$ the free symmetric monoidal category on the graph $Gra_2$. Let $F(Gra_2) /CF$ be the quotient category by the least congruence containing commutative Frobenius relations (see appendix \ref{defn:Free Category}, \ref{defn:Quotient Category} for the definitions used here).

Suppose that the commutative Frobenius algebra $(\mathbb{R}^2, \text{Hadamard})$ is $Z(S^2)$ for some symmetric monoidal functor $Z: F(Gra_2)/CF \to \text{Vect}$. Fix a connected closed irreducible 3-manifold $p$. Consider the assignment $Z(p^{\times \times} ) = $ rotation clockwise by $\pi/2$. Send the remaining prime endomorphisms to the identity. Then \[
Z(m) \circ \Big(Z(p^{\times \times} ) \otimes Z(id_{S^2})\Big) (e_1 \otimes  e_2 ) = -e_2 \cdot e_2 =-e_2 
\]
while \[
Z(m) \circ \Big(Z(id) \otimes Z(p^{\times \times} )\Big) (e_1 \otimes e_2 ) = e_1 \cdot e_1 =e_1 
\]

So it is not true that $m \circ (p^{\times \times} \otimes id_{S^2}) = m\circ (id_{S^2} \otimes p^{\times \times} )$.

Therefore the legs relations are not in the least congruence containing commutative Frobenius relations. 


%
\end{proof}

%
\begin{thm}\label{thm:trueG2thm}
$\operatorname{Cob}( 3)_{S^2} $ has a presentation with $G_2$ generators under commutative Frobenius relations and legs relations. These relations are necessary.
\end{thm}

\begin{proof}
By theorem \ref{mainthmG2}, $G_2 $ generates homs of $\operatorname{Cob}(3)_{S^2}$, and commutative Frobenius relations, co/legs and co/waist relations, and prime commutativity relations are sufficient. By propositions \ref{waistfromCF} - \ref{cowaistfromCF}, co/waist relations, prime commutativity, and colegs relations are induced by commutative Frobenius relations and legs relations, therefore they are not necessary. Theorem \ref{thm:Cob2Cob3CFIso} shows that commutative Frobenius relations are necessary, and proposition \ref{Plegsnecessary} shows that the legs relations are also necessary. Therefore the commutative Frobenius relations and legs relations are sufficient and necessary.
\end{proof}

\newpage
\textbf{Sumary Page}

Define the following morphisms of $\operatorname{Cob}( 3)_{S^2} $ :
\begin{align*}
id_{S^2} &= \left[ S^2 \times I \right]  \\
m &= \left[ S^3 \setminus \operatorname{int}(\sqcup_{}^3 D^3) \right]  : S^2 \sqcup S^2 \to S^2 \\
1 &= \left[ D^3  \right] : \emptyset \to S^2\\
m^\vee &= \left[ S^3 \setminus \operatorname{int}(\sqcup_{}^3 D^3)  \right] : S^2  \to S^2\sqcup S^2 \\
\tr &= \left[ D^3 \right]  :S^2 \to \emptyset
\end{align*}

$\operatorname{Cob}( 3)_{S^2}$ has two sets of generators:
\begin{align*}
G_1& = C \cup F \cup L \cup \text{orientation reversal}\\
G_2 &= C \cup F \cup P \cup \text{orientation reversal}
\end{align*}
where 
\begin{align*}
C &= \left\{ 1, m \right\} \\
F &= \left\{ tr, m^\vee \right\} \\
P &= \left\{ p^{\times \times} = \Big[ p\setminus \operatorname{int}(D^3 \sqcup D^3) \Big] :S^2 \to S^2 \mid p\text{ oriented irreducible} \right\} \\
L &= \left\{ p^\times = \Big[ p \setminus \operatorname{int}(D^3) \Big] : \emptyset \to S^2 \mid p \text{ oriented irreducible} \right\}  
\end{align*}

The $G_2 $ presentation of $\operatorname{Cob}( 3)_{S^2} $ satisfy the following complete list of relations: \[
\begin{cases}
\begin{aligned}
&m \circ (1 \otimes id_{S^2}) = id_{S^2}, &\text{unit}\\
&(tr \otimes id_{S^2}) \circ m^\vee = id_{S^2}, &\text{counit}\\
&m\circ (m \otimes id_{S^2}) = m\circ (id_{S^2} \otimes m), &\text{associativity}\\
&(m^\vee \otimes id_{S^2}) \circ m^\vee = (id_{S^2} \otimes m^\vee) \circ m^\vee, &\text{coassociativity} \\
&m^\vee \circ m = (m \otimes id_{S^2}) \circ \alpha^{-1}_{S^2, S^2, S^2} \circ (id_{S^2} \otimes m^\vee), &\text{Frobenius relation}\\
&[p] \circ [p'] = [p'] \circ [p], &\text{prime commutativity*}\\
&m\circ ([p]\otimes id_{S^2}) = m \circ (id_{S^2} \otimes [p]), &\text{legs relations}\\
&[p] \circ m = m \circ ([p] \otimes id_{S^2}), &\text{waist relations*}\\
&([p] \otimes id_{S^2}) \circ m^\vee = (id_S^2 \otimes [p])\circ m^\vee, &\text{colegs relations*}\\
&m^\vee \circ [p] = ([p] \otimes id_{S^2}) \circ m^\vee ,& \text{cowaist relations*}
\end{aligned}
\end{cases}
\]
The relations marked * are implied by the non-* relations.

The $G_1 $ presentation of $\operatorname{Cob}( 3)_{S^2} $ satisfy the following complete list of relations: \[
\begin{cases}
\begin{aligned}
&m \circ (1 \otimes id_{S^2}) = id_{S^2}, &\text{unit}\\
&(tr \otimes id_{S^2}) \circ m^\vee = id_{S^2}, &\text{counit}\\
&m\circ (m \otimes id_{S^2}) = m\circ (id_{S^2} \otimes m), &\text{associativity}\\
&(m^\vee \otimes id_{S^2}) \circ m^\vee = (id_{S^2} \otimes m^\vee) \circ m^\vee, &\text{coassociativity} \\
&m^\vee \circ m = (m \otimes id_{S^2}) \circ \alpha _{S^2, S^2, S^2} \circ (id_{S^2} \otimes m^\vee), &\text{Frobenius relation}\\
&m\circ ([p]\otimes id_{S^2}) = m \circ (id_{S^2} \otimes [p]), &\text{legs relations*}
\end{aligned}
\end{cases}
\]
The relation marked * is implied by the non-* relations.

\newpage
\chapter{Passing through TFTs} \label{ch:throughZ}
\section{P and L Monoids} \label{sec:PLmonoids}
Inspired by the list of relations which define the $G_1$ and $G_2$ presentations of $\operatorname{Cob}(3)_{S^2}$, we define the following monoids in symmetric monoidal categories, in anticipation that they will come from a 3d TFT.

\begin{defn} \label{defn:Sketch, P/L monoids}
(Sketch, P/L monoids) Let $(A, 1, m, tr, m^\vee)$ be a commutative Frobenius monoid in a symmetric monoidal category $(\mathcal{C}, \otimes , 1_{\otimes })$. We say that $A$ is a P-monoid if there exist endomorphisms $e_p \in \operatorname{End}_{\mathcal{C}}(A)   $ where $p$ is a diffeomorphism class of oriented irreducible prime 3 manifolds, satisfying 
\begin{enumerate}
\item (Prime commutativity) $e_p \circ e_{p'}=e_{p'} \circ  e_p $ 
\item (Legs and waist relations) $e_p \circ m = m \circ (e_p \otimes id_A) = m\circ (id_A \otimes e_p)$
\item (Colegs and cowaist relations) $m^\vee \circ e_p = (e_p \otimes id_A) \circ m^\vee = (id_A \otimes e_p) \circ m^\vee$.
\end{enumerate}

We say that $A$ is a $L$-monoid if there exist morphisms $1_p \in \operatorname{Hom}_{\mathcal{C}}(1_ \otimes , A) $ where $p$ is a diffeomorphism class of irreducible oriented prime 3 manifolds, satisfying 
\begin{enumerate}
\item (Legs relations) $m \circ (1_p \otimes id_A) = m\circ (id_A \otimes 1_p)$.
\end{enumerate}
\end{defn}

The sketch definition has an immediate consequence:

\begin{lemma} \label{lemma:PLbijection}
There is a bijection of $P$ monoids and $L$ monoids.
\end{lemma}

\begin{proof}
Let $(A,1,m,tr, m^\vee)$ be a $P$ monoid, with $e_p$ a prime endomorphism. We may define the corresponding prime unit as $1_p = e_p \circ 1$.  

Similarly if $A$ is an $L$ monoid with $1_p$ a prime unit, we may define the corresponding prime endomorphism $e_p$ as $e_p = m \circ (1_p \otimes id_A) = m \circ (id_A \otimes 1_p)$.

It is easy to check that the proposed prime units and prime endomorphisms satisfy the desired conditions.
\end{proof}

\begin{rmk}
By essentially the same proofs of propositions \ref{waistfromCF} - \ref{colegsfromCF}, not all of the conditions above are necessary. Proposition \ref{prop:LlegsfromCF} shows that the legs relations for $A$ to be an $L$-monoid is induced by commutativity of $A$ being a commutative Frobenius algebra. Hence all commutative Frobenius monoids equipped with morphisms $1_p \in \operatorname{Hom}_{\mathcal{C}}(1_ \otimes , A) $, where $p$ is a diffeomorphism class of irreducible oriented prime 3 manifold, is an $L$-monoid. Similarly, proposition \ref{primecommfromCF}, \ref{waistfromCF}, \ref{colegsfromCF}, and \ref{cowaistfromCF} shows that prime commutativity, waist relations, colegs and cowaist relations are implied by commutative Frobenius relations and legs relations. Therefore any commutative Frobenius monoid equipped with prime endomorphisms $e_p \in \operatorname{Hom}_{\mathcal{C}}(A) $ satisfying the legs relations $m\circ (e_p \otimes id_A) = m\circ (id_A \otimes e_p)$ is a $P$-monoid.
\end{rmk}

\begin{mycom}
In light of lemma \ref{lemma:PLbijection} and the previous remark, one may wonder why there is a slight asymmetry between $P$ monoids and $L$ monoids. Both are commutative Frobenius algebras, both come equipped with maps labelled by diffeomorphism classes of oriented prime 3 manifolds, yet $P$ monoids require that the maps satisfy some relations while $L$ monoids don't. This is because for any commutative Frobenius object $X$ there is a canonical injection \[
\operatorname{Hom}_{\mathcal{C}}(1_{\otimes }, X)  \hookrightarrow \operatorname{End}_{\mathcal{C}}(X) 
.\]
sending $f: 1_ \otimes \to X$ to $m \circ (f \otimes id_X) \circ \lambda^{-1}_X : X \cong 1_{\otimes } \otimes X \to X \otimes X \to X $.

Hence, heuristically speaking, one may find that there are more endomorphisms. The legs relations say that only those induced by a prime unit are prime endomorphisms in a $P$-monoid. This is best illustrated in the following two examples.
\end{mycom}

\begin{prop}
If $(\mathcal{C}, \otimes , 1_ \otimes )$ is the category $(\textbf{Sets}, \times, \left\{ * \right\} )$, then a prime endomorphism $e_p$ of a $P$ monoid of $\mathcal{C}$ acts by multiplication by the element $1_p(*)=e_p(1_A) $, where $1_p$ is the prime unit of the corresponding $L$ monoid.
\end{prop}

\begin{proof}
Let $A$ be a $P$ monoid and $e_p:A \to A$ be a prime endomorphism. Then it satisfy the legs and waist relation $ e_p \circ m = m\circ (e_p \otimes id_A) = m\circ (id_A \otimes e_p) : A \otimes A \to A$. Meaning that \[
e_p(a \cdot b) = e_p(a) \cdot b = a\cdot e_p(b)
.\]
Let $1_A \in A$ be the multiplicative unit with respect to $m$. Then \[
e_p(a) = e_p(a \cdot 1_A) = e_p(a) \cdot 1_A = a \cdot e_p(1_A)
.\]
for all $a\in A$. Therefore $e_p$ acts by multiplication by the element $e_p(1_A)$. 

Since $1_A\in A$ is the image of the unit map $1: \left\{ * \right\} \to A$, we may identify $e_p(1_A)=e_p(1(*)) $ with the composition $e_p \circ 1$. By lemma \ref{lemma:PLbijection}, the composition $1_p = e_p \circ 1$ is the prime unit in the corresponding $L$ monoid. 
\end{proof}

\begin{rmk}
The same proof establishes that in the category $(\text{Veck}_k, \otimes , k)$ a prime endomorphism $e_p$ acts by multiplication by the element $1_p: k \to A$, or equivalently $e_p \circ 1$ where $1: k \to A$ is the unit. Note that there are more linear endomorphisms of a vector space than linear maps of the form  $k\to A$. If $A$ is finite dimensional, the former is given by $n\times n$ matrices while the latter by $n\times 1$. Hence for there to be a bijection between $P$ and $L$ monoids (algebras), it is necessary to impose the legs relations on the prime endomorphisms to rule out the endomorphisms not induced by prime units.
\end{rmk}

We are now finally ready to give the most compact definition for $P$ and $L$ monoids.

\begin{defn} \label{defn:P/L monoids}
(P/L monoids) Let $(A, 1, m, tr, m^\vee)$ be a commutative Frobenius monoid in a symmetric monoidal category $(\mathcal{C}, \otimes , 1_{\otimes })$. Then the following are equivalent,

\begin{enumerate}
\item There exist morphisms $1_p \in \operatorname{Hom}_{\mathcal{C}}(1_ \otimes , A) $, where $p$ is a diffeomorphism class of oriented irreducible prime 3 manifolds. We say $A$ is a $L$-monoid.
\item There exist endomorphisms $e_p \in \operatorname{End}_{}(A) $ that satisfies $m\circ (e_p \otimes id_A) = m\circ (id_A \otimes e_p)$, where $p$ is a diffeomorphism class of oriented irreducible prime 3 manifolds. We say $A$ is a $P$-monoid.
\end{enumerate}
The equivalence is exhibited by $1_p = e_p \circ 1$ and $e_p = m \circ (1_p \otimes id_A) = m \circ (id_A \otimes 1_p)$.
\end{defn}

Restricting ourselves to the setting of $(\text{Vect}_k, \otimes , k)$, we can say a bit more about the prime units and prime endomorphisms:

\begin{defn} \label{defn:P/L algebras}
(P/L algebras) Let $(A,1,m,tr,m^\vee)$ be a commutative Frobenius algebra over $k$. We say $A$ is a $L$ algebra if there are elements $\left\{ 1_p \in A \right\} $ for all $p$ diffeomorphism classes of oriented irreducible prime 3 manifold. Equivalently, we say $A$ is a $P$ algebra if there are linear endomorphisms $\left\{ e_p \right\} $ that acts by multiplications by the element $1_p$. The equivalence is exhibited by $1_p = e_p(1)$ and $e_p = m \circ (1_p \otimes id_A) = m\circ (id_A \otimes e_p)$.
\end{defn}

\begin{thm} \label{thm:Z(S2)isPLmonoid}
Let $Z: \operatorname{Cob}(3)_{S^2} \to \mathcal{C}$ be a symmetric monoidal functor. Then $Z(S^2) $ is a $P/L$- monoid. In the case $\mathcal{C} = \operatorname{Vect}_k$, $Z(S^2)$ is a $P/L$-algebra.
\end{thm}

\begin{proof}
By proposition \ref{prop:CFcommFrob} and corollary \ref{cor:CFonlycommFrob}, $(Z(S^2), Z(1), Z(m), Z(tr), Z(m^\vee))$ is a commutative Frobenius algebra. The prime units $1_p \in \operatorname{Hom}_{\mathcal{C}}(1_ \otimes , Z(S^2)) $ are given by $Z([p\setminus D^3])$.
\end{proof}

\begin{thm} \label{thm:3dTFTgivesPLmonoid}
Let $Z: \operatorname{Cob}(3) \to \mathcal{C}$ be an oriented 3-dimensional topological field theory. Then $Z(S^2)$ is a $P/L$-monoid. In the case $\mathcal{C} = \operatorname{Vect}_k$, $Z(S^2)$ is a $P /L$-algebra.
\end{thm}

\begin{proof}
This is a corollary to the previous theorem \ref{thm:Z(S2)isPLmonoid}: every $3$-dimensional TFT induces a symmetric monoidal functor \[
\operatorname{Cob}(3)_{S^2} \hookrightarrow \operatorname{Cob}(3) \overset{Z}{\to } \mathcal{C}
.\]
sending $S^2$ to $Z(S^2)$. Therefore the image $Z(S^2)$ is a $P/L$-monoid/algebra.
\end{proof}

\begin{mycom}
What is surprising about $P/L$ algebras and their existence through theorem \ref{thm:Z(S2)isPLmonoid} and \ref{thm:3dTFTgivesPLmonoid} is their simplicity. The only novel structures on an $L$ algebra, besides the commutative Frobenius structures, are the existence of distinguished elements $1_p$ labelled by diffeomorphism classes of oriented irreducible prime 3 manifolds. There are no relations among them. This is a feature imposed by geometry, not a bug. The only novel structures on a $P$ algebra besides the commutative Frobenius structures are the distinguished endomorphisms $e_p$. Geometry forces these to act by multiplication by $1_p$, and there are no additional relations. Hence definition \ref{defn:P/L algebras} and theorem \ref{thm:3dTFTgivesPLmonoid} together can be thought of as a statement of the ``structural collapse'' on the algebraic structures of $Z(S^2)$.
\end{mycom}

\begin{rmk}
We make the following observations. Let $Z: \operatorname{Cob}(3) \to \mathcal{C}$ be an oriented 3-dimensional topological field theory. Denote by $e_M : Z(S^2) \to Z(S^2)$ the image $Z([M \setminus (D^3 \sqcup D^3)])$ where $M$ is a closed oriented 3 manifold. Then \[
e_{M \# N} = e_M \circ e_N = e_N \circ e_M
.\]
In particular, if $P, Q$ were oriented prime 3 manifolds, we recover the commutativity of prime endomorphisms \[
e_{P \# Q} = e_P \circ e_Q = e_Q \circ e_P
.\]
This means that the assignment $M \mapsto e_M$ is a monoid map \[
(\left\{ \text{Closed oriented 3 manifolds} \right\} / \cong , \#) \to \operatorname{End}_{\mathcal{C}}(Z(S^2)) 
.\]
In particular, this map sends $S^3$ to the identity $id_{Z(S^2)}$.
\end{rmk}

\begin{eg}
For $Z: \operatorname{Cob}(3) \to \text{Vect}_k$, where $k$ is algebraically closed, since the prime endomorphisms $\left\{ e_p \right\} $ are pairwise commuting, there exists a basis on $Z(S^2)$ such that all the $e_p$ are simultaneously upper-triangular. Equivalently, there exists a flag  \[
0 = V_0 \subset  V_1 \subset  \dots \subset V_n = Z(S^2) ,\quad \dim V_i = i
.\]
such that every $e_M$ preserves every $V_i$. Since each quotient $V_i / V_{i-1}$ is one-dimensional, $e_M$ acts on the quotient by scalar multiplications. Call this scalar $\chi_i(M) \in k$. Since $e_{M\# N} = e_M \circ e_N$, their actions on the quotient $V_i / V_{i-1}$ must satisfy \[
\chi_i(M \# N) = \chi_i(M) \chi_i (N)
.\]
Therefore this ``character theory'' of closed oriented 3 manifolds is completely determined by the character theory of oriented irreducible prime 3 manifolds. 
\end{eg}

\begin{eg}
In fact, we can say more in some special cases. Consider the same setup as above. Let $\mathcal{M}$ be the set of diffeomorphism classes of oriented irreducible prime 3 manifolds. Consider  \[
A = k \left[ \left\{ e_p | p\in \mathcal{M} \right\}  \right] 
.\]
This is a commutative subalgebra of $\operatorname{End}_{\mathcal{C}}(Z(S^2)) $. Since $Z(S^2)$ is finite dimensional, so is $A$. 

If $Z(S^2)$ happens to be semisimple as an $A$-module, then $Z(S^2)$ is a direct sum of simple $A$-modules. But $A$ is commutative, so all its simple modules are one-dimensional. Therefore \[
Z(S^2) \cong \bigoplus_\lambda k_\lambda
.\]

Over any line $k_\lambda$,  every $a\in A$ acts by some scalar. Denote this scalar as $\lambda(a)$ or, if $a=e_M$ for some closed 3 manifold $M$, by $\chi_\lambda (M)$. Then we still have \[
\chi_\lambda(M \# N) = \chi_\lambda(M) \chi_\lambda(N) 
.\]

The map $a \mapsto \lambda (a)$ defines $\lambda: A \to k$ for each $\lambda$. We may group the same\footnote{We consider two lines $k_{\lambda_1 }, k_{\lambda_2 } $ the same of $\lambda_1 = \lambda_2$ as maps $A \to k$. } $k_\lambda$'s together and obtain the direct sum decomposition of $Z(S^2)$ into spaces with the same character \[
Z(S^2) = \bigoplus_ \lambda V_\lambda 
.\]
Then each $e_M$ acts by  \[
e_M |_{V_\lambda } = \chi_\lambda (M) id_{V_\lambda}
.\]
or equivalently \begin{equation} \label{eqn:e_Mact}
e_M = \sum_{\lambda}^{ } \chi_{\lambda} (M) \pi_{\lambda}
.\end{equation}
where $\pi _{\lambda}: Z(S^2) \to V_\lambda$ are the projection maps.

The prime endomorphisms are those $e_M$ where $M$ is an oriented irreducible prime 3 manifold $P$. So these act by 
\begin{equation} \label{eqn:e_Pact}
e_P = \sum_{\lambda}^{ } \chi_{\lambda} (P) \pi_{\lambda}
.\end{equation}
\end{eg}

\begin{eg}
The example above provides a simple way of computing the invariants $Z(M)$ for any closed 3 manifold $M$. Let $Z : \operatorname{Cob}(3) \to \text{Vect}_k$ where $k$ is algebraically closed. If $Z(S^2)$ is semisimple as an $A$-module where $A$ is the commutative subalgebra $A = k[\left\{ e_p | p \in \mathcal{M} \right\} ] \subset  \operatorname{End}_{\text{Vect}_k}(Z(S^2)) $, then the example above shows that $e_P$ acts by equation \ref{eqn:e_Pact}. If $1: k \to Z(S^2)$ is the unit map and $tr: Z(S^2) \to k$ the trace map, then for any closed 3 manifold $M\cong P_1 \# \dots \# P_m$, we have \[
Z(M) = tr \circ e_M \circ 1 = \sum_{\lambda}^{ } (tr \circ \pi_{\lambda} \circ 1 ) \chi_\lambda(M) =  \sum_{\lambda}^{ } (tr \circ \pi_{\lambda} \circ 1 ) \prod_{i=1}^m \chi_\lambda(P_i)
.\]
Note however that the simplicity of this formula comes at the cost of needing to compute the characters $\chi_\lambda(P_i)$ which may be equally challenging as computing $Z(M)$ via other methods.
\end{eg}

\section{Relations to J-Algebras} \label{sec:Jalg}

\begin{defn} \label{defn:Trace of Surgery}
(Trace of Surgery) Let $M$ be an oriented $n$-manifold and $0\leq k\leq n$. Let $\phi : S^{k-1} \times D^{n-k} \hookrightarrow M$ be an embedding along which a $k$-handle is attached. The resulting manifold is \[
M' = M \cup_{\phi } (D^k \times D^{n-k})
.\]
The \textit{trace} of this handle attachment is the $(n+1)$-dimensional cobordism $W$ from $M$ to $M'$ obtained by attaching a $k$-handle to $M \times [0,1]$ along $M\times \left\{ 1 \right\} $: \[
W = (M \times [0,1]) \cup_{\phi } (D^k \times D^{n-k+1})
.\]
Its incoming boundary is $M \times \left\{ 0 \right\} $ and the outgoing boundary is the manifold $M'$.
\end{defn}

Recall \cite{juhasz2018defining}'s assignment of a $J$-algebra associated to a 3-dimensional oriented TFT $Z: \operatorname{Cob}(3) \to \text{Vect}_k$ \footnote{Note Juhasz's original paper the unusual convention is to denote by $\operatorname{Cob}(n)$ the category of $n$-manifolds and bordisms. We will not adopt this convention.} is denoted $J(Z)$ which consists of the following data 

\begin{enumerate}
\item $A_g = Z(\Sigma_g)$ is the vector space associated to the genus $g$ surface $\Sigma_g$ under $Z$, for all $g\geq 0$
\item linear maps $\alpha_g: A_g \to A_{g-1} $, which we call ``genus reduction maps.'' \\
This is the $Z$-image of trace of the attachment of a 2-handle around a hole, which eliminates it.
\item linear maps $\delta _{j, g-j}: A_g \to A_j \otimes A_{g-j}$, which we call ``splitting maps.'' \\
This is the $Z$-image of the trace of the splitting of $\Sigma_g $ into $\Sigma_j \amalg \Sigma_{g-j}$.
\item linear maps  $\mu _{i,j} : A_i \otimes A_j \to A_{i+j}$, which we call ``merging maps.''\\
This is the $Z$-image of the trace of the merging of $\Sigma_i \amalg \Sigma_j $ into $\Sigma_{i+j}$.
\item linear maps $\omega_g : A_g \to A_{g+1}$, which we call ``genus creation maps.'' \\
This is the $Z$-image of the trace of the attachment of a $1$-handle.
\item The linear map $tr: A_0 \to k$, which we call the ``trace map.'' \\
This is $Z(D^3)$, where $D^3 : S^2 \to \emptyset$.
\item The linear map $1: k \to A_0$, which we call the ``unit map.''\\
This is $Z(D^3)$, where $D^3 : \emptyset \to S^2 $
\item involutions $*_g: A_g \to A_g$
\item representations $\rho _g : MCG(\Sigma_g) \to \operatorname{Aut}(A_g) $.
\end{enumerate}

8 and 9 come from the following construction: First note that $MCG(\Sigma_g)$ here means $\operatorname{Diff}^+(\Sigma_g ) / \operatorname{Diff}^+_0(\Sigma_g )  $, namely isotopy classes of orientation preserving diffeomorphisms. The representations send a particular orientation-preserving diffeomorphism $d\in \operatorname{Diff}^+(\Sigma_g ) $ to $Z(Cyl(d))$ where $Cyl(d)$ is the mapping cylinder. In particular, there is a specific order $2$ diffeomorphism giving rise to the involutions $*_g$-- this is defined to be the diffeomorphism which sends the $\Sigma_g$ in its standard embedding in $\mathbb{R}^3$ to itself after performing a $\pi $-rotation about the $z$-axis.

The genus creation/reduction maps, and the splitting/merging maps are the images of the trace of elementary surgeries on surfaces. 

These data, along with appropriate compatibility conditions, fit into an algebraic structure named ``J-Algebras.'' A theorem by \cite{juhasz2018defining} is that these data completely determine the 3-dimensional TFT $Z$. In fact, the statement is stronger:

\begin{thm} \label{thm:Jalg=3dTFT}
[Juhasz] The assignment $Z \mapsto J(Z)$ is an equivalence between the symmetric monoidal category of 3-dimensional TFTs and $\textbf{J-Alg}$.
\end{thm}

Here the category $\textbf{J-Alg}$ is the category whose objects are $J$-algebras and whose morphisms are \textit{J-algebra homomorphisms}. We refer the reader to the original article for the precise definition of such morphisms.

Our goal for this section is to briefly analyse where $P$ and $L$ algebras fit inside a $J$-algebra.

The following data of the $J$-algebra associated with a TFT $Z$ that directly involves the genus $0$ piece:

\begin{enumerate}
\item $A_0 = Z(S^2)$ is the vector space associated to the two-sphere.
\item linear maps $\alpha_1: A_1 \to A_{0} $,
\item linear maps $\delta _{0, g}: A_g \to A_0 \otimes A_{g}$,
\item linear maps  $\mu _{g,0} : A_g \otimes A_0 \to A_{g}$, 
\item linear maps $\omega_0 : A_0 \to A_{1}$, 
\item The linear map $tr: A_0 \to k$, 
\item The linear map $1: k \to A_0$, 
\item involutions $*_0: A_0 \to A_0$
\item representations $\rho _0 : MCG(S^2) \to \operatorname{Aut}(A_0) $.
\end{enumerate}

Note that 8 and 9 are actually both trivial for the following reason: $MCG$ here means isotopy classes of orientation preserving diffeomorphisms, so $MCG(S^2)= \left\{ [id_{S^2}] \right\} $ is the group of one element. Therefore $\rho_0$ maps $[id_{S^2}]$ trivially to the identity of $A_0$. The diffeomorphism of $S^2$ defined by rotation about the $z$-axis by $\pi $ is therefore also isotopic to the identity, hence passing through $Z$ gives the identity on $A_0$. By definition this is the involution $*_g$. Therefore we see that 8 and 9 contains no information.

Furthermore, we see that 6 and 7 agree with the unit and trace maps of $A_0$ as a commutative Frobenius algebra. In the case $g=0$, $\delta_{0,g} = m^\vee $ and $\mu_{g,0} = m$ recovers the co/multiplication maps.

Therefore, the non-trivial information associated to the commutative Frobenius algebra $A_0$ is the following:
\begin{enumerate}
\setcounter{enumi}{1}
\item the genus reduction map $\alpha_1: A_1 \to A_0 $ 
\item the splitting maps $\delta_{0,g}: A_g \to A_0 \otimes A_g$, $g>0$
\item the merging maps $\mu_{g,0}: A_g \otimes A_0 \to A_g$, $g>0$
\item the genus creation map $\omega_0 : A_0 \to A_1$.
\end{enumerate}


Even though the four classes of maps above are the only maps in the data of $J$-algebras that directly involves $A_0 $, they do not generate all relevant maps. For example, they do not generate prime endomorphisms $e_p : A_0 \to A_0$ where $p$ is irreducible.

\begin{thm} \label{thm:JgivesPL}
The degree $0$ part of every $J$-algebra is a $P / L$-algebra. Namely, every $J$-algebra induces a $P/L$-algebra. 
\end{thm}

\begin{proof}
By theorem \ref{thm:Jalg=3dTFT}, for any $J$-algebra $J$ there exists a 3-dimensional oriented TFT $Z: \operatorname{Cob}(3)\to \operatorname{Vect}_k$ such that $J(Z) = J$, where $J(Z)$ consists of the data 1-9 along with compatibility conditions. By theorem \ref{thm:3dTFTgivesPLmonoid}, every 3-dimensional TFT $Z$ determines a $P /L$ algebra $Z(S^2)$.
\end{proof}

This is an extremely unsatisfying proof. We will present a more geometric construction that illustrates precisely \textit{how} the data of a $J$-algebra determines, for example, the prime units. For this, we will need Heegaard Splitting.

\begin{thm} \label{thm:HeegaardSplitting}
[Heegaard] Let $M$ be a connected closed oriented $3 $-manifold. Then there exists two handle bodies $M_1 , M_2 $ of genus $g$ and an orientation-reversing diffeomorphism $f: \partial M_1 \to \partial M_2 $ such that \[
M \cong M_1 \cup_f M_2 
.\]
\end{thm}

\begin{rmk}
The Heegaard genus $g(M)$ is defined as the minimal genus of all possible Heegaard splitting of $M$. In fact \cite{martelli2016introduction}, only $S^3$ (and those $M$ diffeomorphic to it) has $g=0$, and only lens spaces (other than $S^3$) or $S^2 \times S^1$ has $g=1$. All other 3-manifolds have $g\geq 2$.
\end{rmk}

Since $M_1 $ is evidently diffeomorphic to $M_1 \cup_{\partial M_1 } \partial M_1 \times [0,1]$, we may write \[
M \cong M_1 \cup_{\partial M_1 } \partial M_1 \times [0,1] \cup_f M_2 
.\]
and after factoring through diffeomorphisms $\phi_{1,2}: \Sigma_g \to \partial M_{1,2} $, \[
M \cong M_1 \cup_{\phi_1 ^{-1}} Cyl(h) \cup_{\phi_2 } M_2 
.\]
where $h =\phi_2 ^{-1} \circ f\circ \phi_1 : \Sigma_g \to \Sigma_g $ is orientation preserving.

Therefore we have the following variant of the Heegaard splitting:

\begin{thm}
Let $M$ be a connected closed oriented $3 $-manifold. Then there exists two handle bodies $M_1 , M_2 $ of genus $g$ and $h \in \operatorname{Diff}^+(\Sigma_g ) $ such that \[
M \cong M_1 \cup_{\phi_1 ^{-1}} Cyl(h) \cup_{\phi_2 } M_2 
.\]
where $Cyl(h) = (\Sigma_g \times [0,1] \amalg \Sigma_g) / (x,1) \sim h(x)$ is the mapping cylinder of $h$, $\phi_{1}: \Sigma_g \to \partial M_{1}$ is an orientation-preserving diffeomorphism, and $\phi_{2}: \Sigma_g \to \partial M_{2}$ is an orientation-reversing diffeomorphism.
\end{thm}

and the boundary variant:

\begin{thm} \label{thm:Heegaardvarpar}
Let $M$ be a connected closed oriented $3 $-manifold with boundary $\partial M \cong \amalg^{a+b} S^2$. Then there exists two connected oriented 3-manifolds $M_1 , M_2 $ with boundaries $\partial M_1 = \amalg^a S^2 \amalg \Sigma^1$ and $\partial M_2 = \amalg^b S^2 \amalg \Sigma^2$, where $\Sigma^{1,2}$ are of of genus $g$, and $h \in \operatorname{Diff}^+(\Sigma_g ) $ such that \[
M \cong M_1 \cup_{\phi_1 ^{-1}} Cyl(h) \cup_{\phi_2 } M_2 
.\]
where $Cyl(h) = (\Sigma_g \times [0,1] \amalg \Sigma_g) / (x,1) \sim h(x)$ is the mapping cylinder of $h$, $\phi_{1}: \Sigma_g \to \Sigma^1$ is an orientation-preserving diffeomorphism, and $\phi_{2}: \Sigma_g \to \Sigma^2$ is an orientation-reversing diffeomorphism.
\end{thm}

\begin{prop} \label{prop:JtoLmaps}
Given any $A= \oplus_{i\geq 0 } A_i$ an $J$-algebra over $k$, and any irreducible oriented $3$-manifold $P$, let $1_P: k \to A_0$ and $e_P : A_0  \to A_0$ be a prime unit and a prime endomorphism of the $L$-algebra $A_0 $. Then there exists $h \in MCG(\Sigma_g)$ for some $g$ such that \[
1_P = \alpha_1 \circ \alpha_2 \dots \circ \alpha _g \circ \rho_g(h) \circ \omega_{g-1} \circ \omega _{g-2} \circ \dots \circ \omega_0 \circ 1
.\]
and that \[
e_P = \alpha_1 \circ \alpha_2 \dots \circ \alpha _g \circ \rho_g(h) \circ \omega_{g-1} \circ \omega _{g-2} \circ \dots \circ \omega_0
.\]
Namely, $1_P: k \to A_0$ is the composition \[
k \overset{1}{\to } A_0 \overset{\omega _0}{\to } A_1 \overset{\omega _1}{\to } \dots \overset{\omega _{g-1}}{\to }A_g \overset{\rho _g(h)}{\to } A_g \overset{\alpha _g}{\to } A_{g-1} \overset{\alpha_{g-1}}{\to } \dots \overset{\alpha _2}{\to } A_1 \overset{\alpha_1 }{\to }A_0 
.\]
and $e_P: A_0 \to A_0 $ is the composition \[
A_0 \overset{\omega _0}{\to } A_1 \overset{\omega _1}{\to } \dots \overset{\omega _{g-1}}{\to }A_g \overset{\rho _g(h)}{\to } A_g \overset{\alpha _g}{\to } A_{g-1} \overset{\alpha_{g-1}}{\to } \dots \overset{\alpha _2}{\to } A_1 \overset{\alpha_1 }{\to }A_0 
.\]
\end{prop}

\begin{proof}
Applying theorem \ref{thm:Heegaardvarpar} to $P\setminus D^3$ in the case $(a,b)=(0,1)$, there exists a genus $g$ handle body $M_1 $ and a connected closed oriented $3$-manifold $M_2 $ with boundary $S^2 \amalg \Sigma_g$ and  $h\in \operatorname{Diff}^+(\Sigma_g ) $ such that \[
P\setminus D^3 \cong M_1 \cup Cyl(h) \cup M_2 
.\]
Recall that any handle body of genus $g$ can be obtained, by definition, by attaching $g$-many 1-handles to a $0$-handle. Hence in the case $(a,b)=(0,1)$, we learn that $P$ is built by first attaching $g$-many 1-handles to a $0$-handle, which results in $M_1 $.  Then we glue on a mapping cylinder $Cyl(h)$. Then we attach $g$-many $2$-handles, which form $M_2 $. We stop short of attaching the final $3$-handle, which would had result in a closed manifold.

Passing through the TFT $Z$ corresponding to the $J$-algebra $A$, and noting that $\omega _i$ is the image of the trace of $1$-handle attachments while $\alpha_i$ is the image of the trace of $2$-handle attachments, we obtain the desired composition for $1_P$.

Doing the same to $P\setminus (D^3 \sqcup D^3)$ in the case $(a,b)=(1,1)$, we obtain the desired composition for $e_P$.
\end{proof}

%
%
%

\section{The $\infty$-operad $\mathbb{L}^\otimes $} \label{sec:infoperad}
We now shift gears to the $(\infty,1)$-categorical setting. In this exploratory section, we aim to define an $\infty$-operad $\mathbb{L}^\otimes $ and conjecture it as equivalent to the connected part of the $\infty$-endomorphism operad $\operatorname{End}_{\Omega^2 \text{Bord}_3^{fr}}^\otimes (S^2) $. There is also an oriented version of the proposal. 

We will do this in steps: first we will briefly review the constructions of the $\infty$-little $n$-cube operad $\mathbb{E}_n^\otimes $, then we will define the $\infty$-operad $\mathbb{L}^\otimes $ and $\operatorname{End}_{\Omega^2 \text{Bord}_3}^\otimes (S^2) $, and lastly we will give two precise conjectures which state that there exists a map of $\infty$-operads which is an equivalence.

\begin{mycom}
When we say $\mathcal{C}$ is an $\infty$-category we mean either $\mathcal{C}$ is a complete Segal space or a quasi-category. Although it is usually clear from context which model we are referring, we will default to using quasi-categories unless specified otherwise.

Many influential work on the bordism categories were done in the model of ($n$-fold) complete Segal spaces such as\cite{lurie_classification_2008, Calaque_2019, scheimbauer2014factorization}, while $\infty$-operads were explored by \cite{LurieHigherAlgebra} in the model of quasi-categories. 

Both complete Segal spaces and quasi categories are valid models of $(\infty,1)$-categories, due to the fact that the homotopy theory of complete Segal spaces is the same as the homotopy theory of Quasi-categories, hence one loses nothing by choosing a model. In fact, one may move from one model to the other via \cite{joyal2007quasi} style constructions, such as theorem \ref{thm:JoyalTierney} in the appendix.
\end{mycom}

The following category will show up repeatedly in the section

\begin{defn} \label{defn:Fin*}
($\mathcal{F}in_*$) Let $\langle n \rangle = \left\{ *, 1, \dots ,n \right\} $. The Segal category of finite pointed set $\mathcal{F}in_*$ is defined to have 
\begin{enumerate}
\item objects are $\langle n \rangle $ for each $n\geq 1$.
\item a morphism $\alpha : \langle m \rangle \to \langle n \rangle  $ is a map such that  $\alpha (*)=*$.
\end{enumerate}

A morphism $\alpha : \langle m \rangle \to \langle n \rangle $ is \textit{inert} if for each $i \in 1, \dots ,n$, the inverse $\alpha ^{-1}(i)$ has only one element; it is \textit{active} if $\alpha ^{-1}(*) = \left\{ * \right\} $.
\end{defn} 

Recall that an $\infty$-operad as defined in \cite{lurie2017higheralgebra} is a morphism of $\infty$-categories \[
p: \mathcal{O}^\otimes \to N(\mathcal{F}in_*)
.\]
such that every inert morphism has $p$-coCartesian lifts, the inert morphisms $\rho ^i \langle n \rangle \to \langle 1 \rangle $ sending $i$ to $1$ and everything else to $*$ induce ``projection maps,'' and there is equivalence of $\infty$-categories $\mathcal{O}^\otimes _{\langle n \rangle } \to \mathcal{O}^n$ where $\mathcal{O}^\otimes _{\langle n \rangle } \equiv p ^{-1}(\langle n \rangle )$, $\mathcal{O} \equiv \mathcal{O}^{\otimes }_{\langle 1 \rangle }$. A precise definition is recorded in \ref{defn:infoperad}.


First let's recall the standard construction of the $\infty$-little k-cube operad $\mathbb{E}_k^ \otimes $. 

Let $\Box^k = (-1,1)^k $ be the $k$-dimensional open cube. We say an embedding $f: \Box^k \to \Box^k$ is rectilinear if it is given by \[
f(x_1 , \dots ,x_k) = (a_1 x_1 +b_1 , \dots , a_k x_k+b_k)
.\]
for some constants $a_i>0$ and $b_i$. For any finite set $S$, we say an open embedding $\Box^k \times S \to \Box^k$ is rectilinear if it is rectilinear on each connected component. Denote all such rectilinear embeddings as $\operatorname{Rect}(\Box^k \times S, \Box^k)$ and equip it with the subspace topology of $\mathbb{R}^{2k \mid S \mid }$. Note that the images of the $|S|$-many $\Box^k$ are disjoint by definition.

\begin{defn} \label{defn:Ek}
Let $^t \mathbb{E}_k^{\otimes }$ be the topological category defined as follows 
\begin{enumerate}
\item The objects are $\langle n \rangle $, the same as those of $\mathcal{F}in_*$ 
\item An arrow $f:\langle m \rangle \to \langle n \rangle $ consists of the following data:
\begin{enumerate}
\item a morphism $\alpha : \langle m \rangle \to \langle n \rangle $ in $\mathcal{F}in_*$.
\item for each $1\leq j\leq n$, a rectilinear embedding in \[
\operatorname{Rect}(\Box^k \times \alpha ^{-1}(j), \Box^k)
.\]
\end{enumerate}
\item The topology on the mapping spaces $\operatorname{Map}_{^t \mathbb{E}_3^\otimes }(\langle m \rangle , \langle n \rangle )$ is induced by the presentation \[
\operatorname{Map}_{^t \mathbb{E}_3^\otimes }(\langle m \rangle , \langle n \rangle ) = \coprod_{\alpha : \langle m \rangle \to \langle n \rangle } \prod_{1\leq j\leq n} \operatorname{Rect}(\Box^k \times \alpha ^{-1}(j), \Box^k)
.\]
\item The composition \[
\operatorname{Map}_{^t \mathbb{E}_3^\otimes }(\langle n \rangle , \langle n' \rangle ) \times \operatorname{Map}_{^t \mathbb{E}_3^\otimes }(\langle n' \rangle , \langle n'' \rangle ) \to \operatorname{Map}_{^t \mathbb{E}_3^\otimes }(\langle n \rangle , \langle n'' \rangle )
.\]
sends $\alpha : \langle n \rangle , \langle n' \rangle , \beta : \langle n' \rangle \to \langle n'' \rangle $ to $\beta \circ \alpha : \langle n \rangle \to \langle n'' \rangle $, and the rectilinear embeddings \[
f_j : \Box^k \times \alpha ^{-1}(j) \to \Box^k ,\quad f'_{j'}: \Box^k \times \beta ^{-1}(j') \to \Box^k
.\]
are sent to \[
f'_{j'} \circ (\coprod_{j \in \alpha ^{-1}(\beta ^{-1}(j'))} f_j) : \Box^k \times (\beta \circ \alpha ) ^{-1} (j') \to \Box^k
.\]
\end{enumerate}

We call $^t \mathbb{E}_k^{\otimes }$ the topological little $k$-cube operad.

Define the $\infty$-category $\mathbb{E}_k^\otimes $ as the topological nerve of $^t \mathbb{E}_k^\otimes $.
\end{defn}

\begin{prop}
$\mathbb{E}_k^\otimes $ is an $\infty$-operad.
\end{prop}

%
%

\begin{defn}
We call $\mathbb{E}_k^\otimes $ the $\infty$-little $k$-cube operad. 
\end{defn}

\begin{rmk}
There are many variations to definition \ref{defn:Ek} by replacing the spaces $\operatorname{Rect}(\Box^k \times \alpha ^{-1}(j), \Box^k)$ by some other ones. There are two important variants:
\begin{enumerate}
\item One may replace $\Box^k$ with $D^k$, and define a rectilinear embedding  $D^k \to D^k$ as one which extends to $\mathbb{R}^k \to \mathbb{R}^k$ as $v \mapsto \lambda v + v_0 $ for some $\lambda > 0$ and $v_0 \in \mathbb{R}^k$. That is, a uniform radial scaling and a translation. The resulting topological category is referred to as the ``little $k$-disk operad'' and its topological nerve is once again an $\infty$-operad, rightfully called the ``$\infty$-little $k$-disk operad.'' 

By \cite{may1997operads} and \cite{Lambrechts_Volić_2014}, for any finite set $S$ there is a canonical homotopy equivalence between $\operatorname{Rect}(\Box^k \times S, \Box^k)$, $\operatorname{Rect}(D^k \times S, D^k)$ and $\operatorname{Conf}_{|S|}(\mathbb{R}^k)$ the configuration space of $|S|$-many distinct points in $\mathbb{R}^k$. Hence the topological little $k$-cube operad is homotopy equivalent to the topological little $k$-disk operad. Passing through the homotopy coherent nerve gives an equivalence between the $\infty$-little $k$-cube operad and the $\infty$-little $k$-disk operad. 

\item One may replace $\Box^k$ with $D^k$, and define a ``projective isometry'' $f: D^k \to D^k$ as one which extends to $\mathbb{R}^k \to \mathbb{R}^k$ as $v \mapsto \lambda Rv + v_0 $ for some $R\in SO(k)$, $\lambda > 0$ and $v_0 \in \mathbb{R}^k$. That is, a rotation followed by a uniform radial scaling, then a translation. The resulting topological category $^t \mathbb{E}_{SO(k)}^\otimes $ is referred to as the ``framed little $k$-disk operad'' and its topological nerve $\mathbb{E}_{SO(k)}^\otimes $ is once again an $\infty$-operad, rightfully called the ``$\infty$-framed little $k$-disk operad.''

This operad (both the topological and the $\infty$-version) is not equivalent to the ordinary little $k$-disk/cube operad for $k\geq 2$. Heuristically speaking, the $SO(k)$ rotational information assigned to the embedded disks is the classifying space $BSO(k)$. When $k\geq 2$ this space is no longer contractible, hence the homotopy theory of projective isometries is different from the homotopy theory of rectilinear embeddings.

\item Similar to variation 2, one may further allow $R \in O(k)$ to include orientation reversal. The resulting topological category is denoted $^t \mathbb{E}_{O(k)}^{\otimes }$ and the topological nerve $\mathbb{E}^\otimes _{O(k)}$ is an $\infty$-operad.
\end{enumerate}
\end{rmk}

\begin{eg}
It is widely understood that $S^{k-1}$ is an $\mathbb{E}_k^\otimes $-algebra object in $\operatorname{Bord}_k^{fr}$. Precisely, this means that there is a map of $\infty$-operads \[
\mathbb{E}_k^{\otimes } \to \Omega^{k-1} \operatorname{Bord}_k^{fr, \otimes }
.\]
sending $\langle 1 \rangle $ to the object $S^2$. 

Additionally, $S^{k-1}$ is understood to be an $\mathbb{E}_{SO(k)}^\otimes $-algebra object in $\Omega^{k-1}\operatorname{Bord}_k^{SO(k), \otimes }$, and an $\mathbb{E}_{O(k)}^{\otimes}$-algebra object in $\Omega^{k-1}\operatorname{Bord}_k^{O(k), \otimes }$. \footnote{Note the infamous trap in naming conventions: associated to the \textit{framed} little $k$-disk operad is $S^{k-1}$ as an \textit{oriented} manifold.} Though, at the time of this writing, we are not aware of a thorough treatment in existing literature.
\end{eg}

Recall \cite{hatcher2008diffeomorphism}'s construction of the space of connect sum configurations:
\begin{notations}
Let $G$ be a graph and $v$ be a vertex. Let $HE(v) = DE(v) \cup E(v)$ be the set of half edges of $v$, written as a union of unpaired half edges $DE(v)$ and the paired half edges $E(v)$. We say $v$ has valency $|E(v)|$ and deficiency $|DE(v)|$.
\end{notations}

\begin{defn} \label{defn:C(M)}
($C(M)$) Let $M= P_1 \# \dots \#P_m \#(S^2\times S^1)^g$ be a closed oriented 3-manifold, where $P_i$ are oriented irreducible 3-manifolds. Let $G$ be a finite connected graph with $m$ of its vertices labelled $1, \dots , m$ and such that all unlabelled vertices have valence $\geq 3$ and the fundamental group is free of rank $g$. Let
\[
C_G(M) = \left\{ (r_e, l_e, D^3_e \hookrightarrow P_i, D_e^3 \hookrightarrow S^3_v ) \right\}
\]
over all edges $e$ and all unlabelled vertices $v$ of $G$, where $r_e>0, l_e>0 $ and embeddings $D^3_e \hookrightarrow P_i$ and $D^3_e \hookrightarrow S^3_v$, the latter being an isometry on its boundary. Let $\overline{C_G}(M)$ be the same as $C_G(M)$ except now allow $l_e$ to go to $0$. Let $G/e$ denote the graph by contracting the edge $e$. Then define \[
C(M) = \coprod_{G} \overline{C_G}(M) / \sim
.\]
where a configuration with $l_e=0$ is identified with the same one after contracting the edge $e$.
\end{defn}

The space $C(M)$ is best thought of as all the possible ways to ``build '' $M$ from $P_i$ factors. To incorporate spherical boundaries, we need to introduce unpaired half edges corresponding to $S^2$ boundaries, one of which is distinguished and corresponds to the outgoing boundary.

\begin{defn} \label{defn:C(M)par}
($C^{\partial, n}_{SO}(M)$) Let $M= P_1 \# \dots \#P_m \#(S^2\times S^1)^g$ be a closed oriented 3-manifold, where $P_i$ are oriented irreducible 3-manifolds. Let $G$ be a finite connected graph with $m$ of its vertices labelled $1, \dots , m$ and such that all unlabelled vertices have valence $\geq 3$ and the fundamental group is free of rank $g$, and such that it has $n+1$-many unpaired half edges, one of which is distinguished. Let
\[
C_{SO,G}^{\partial ,n}(M) = \left\{ (r_e, l_e, D^3_e \hookrightarrow P_i, D^3 \hookrightarrow S^3_v, D^3_\epsilon \hookrightarrow P_i, D^3_\epsilon \hookrightarrow S^3_v ) \right\}
\]
over all edges $e$ and all unpaired half edges $\epsilon$ and all unlabelled vertices $v$ of $G$, where $r_e>0, l_e>0 $ and embeddings $D^3_e \hookrightarrow P_i$ and $D^3_e \hookrightarrow S^3_v$, the latter being an isometry on its boundary. The half-edge embeddings $D^3_\epsilon \hookrightarrow P_i, D^3_\epsilon \hookrightarrow S^3_v $ are orientation-preserving and have disjoint images. Let $\overline{C_{SO,G}^{\partial ,n}}(M)$ be the same as $C_{SO,G}^{\partial ,n}(M)$ except now allow $l_e$ to go to $0$. Let $G/e$ denote the graph by contracting the edge $e$. Then define \[
C_{SO}^{\partial ,n}(M) = \coprod_{G} \overline{C_{SO}^{\partial ,n}}(M) / \sim
.\]
\end{defn}

The space $C^{\partial ,n}_{SO} (M)$ is $C(M)$ but now one remembers incoming and outgoing spherical boundary data. This is the oriented version-- all the irreducible factors are oriented, and all edge embeddings $D^3_e \hookrightarrow S^3_v, D^3_e \hookrightarrow P_i$ oriented, and the boundary embeddings $D^3_\epsilon \hookrightarrow S^3_v, D^3_\epsilon \hookrightarrow P_i$ (those corresponding to half edges) are also oriented. 

\noindent \textbf{Variation}
There exists an obvious variation, one where we upgrade the orientation to the more restrictive framed setting. Call the space $C^{\partial ,n}_{fr}(M)$, where $M$ is now a framed closed $3$-manifold, and the $P_i$ factors are also framed. All embeddings are framing-preserving embeddings. \footnote{We will be intentionally vague about this, since this is only a sketch definition.} 

\begin{rmk}
The definition for the configuration space of the un-oriented case is less straightforward, so we won't attempt to present it here. However, we can make some suggestions. Call this configuration space $C^{\partial ,n}_{O}(M)$, where $M= P_1 \# \dots \# P_m \#(S^2\times S^1)^g \# (S^2 \overset{\sim}{\times}S^1)^{g'}$. In the unoriented case there exists a reducible prime $S^2 \widetilde{\times} S^1$ which is a $S^2$ bundle over $S^1$, defined as \[
S^2 \widetilde{\times} S^1 = S^2 \times [0,1] / (x,0) \sim (r(x),1)
.\]
where $r:S^2 \to S^2$ is the antipodal map, in particular orientation reversing.\footnote{Unlike $S^2\times S^1$ where the identification can be made as $(x,0)\sim (x,1)$.} Hence we need the underlying graphs to have fundamental groups free of rank $g+g'$.

We would need the attaching tubes $D^3_e \times I_{e}$ as well as the orientation-reversing version of these tubes, defined as \[
\left( D^3_e \times I_e  \right) \coprod D^3 / (x,1) \sim r(x)
.\]
where $r:D^3 \times D^3$ is again the antipodal map.
\end{rmk}

Setting aside the issue of the unoriented configuration space, we define an operad which captures the information in the oriented configuration space:

\begin{defn} \label{defn:tL}
($^t \mathbb{L}^\otimes $) is a topological category defined as follows
\begin{enumerate}
\item The objects are $\langle n \rangle $, the same as $\mathcal{F}in_*$.
\item A morphism from $\langle m \rangle $ to $\langle n \rangle $ consists of the following data:
\begin{enumerate}
\item a morphism $\alpha : \langle m \rangle \to \langle n \rangle $ in $\mathcal{F}in_*$
\item for every $j = 1, \dots , n$, a non-negative integer $g_j \geq 0$ and an element $h_j \in \operatorname{BDiff}(\coprod_{i \in I_j} P_i)$, where $P_i$ are framed irreducible 3-manifolds.
\item for every $j=1, \dots ,n$, a framed connect sum configuration $f_j \in C^{\partial , \mid \alpha ^{-1}(j) \mid } _{fr}(M)$ where $M=\left( \#_{i\in I_j} P_i \right) \# (S^2\times S^1)^{g_j}$.
\end{enumerate}
\item Composition is given by
\begin{enumerate}
\item send $\alpha : \langle m \rangle \to \langle n \rangle $ and $\beta : \langle n \rangle \to \langle n' \rangle $ to the composition $\beta \circ \alpha : \langle m \rangle \to \langle n' \rangle $.
\item If $g_j\geq 0$, $h_j \in \operatorname{BDiff}(\coprod_{i \in I_j}P_i)$ for $j = 1, \dots , n'$ and $g_k\geq 0$, $h_k \in \operatorname{BDiff}(\coprod_{i \in I_k}P_{i})$ for $k = 1, \dots , n''$ are the data to be composed, their composition is given by $\overline{g}_k = g_k + \sum_{j \in \alpha ^{-1}(k)}^{ } g_j $ and $\overline{h}_k  $ is the image of $(h_j, h_k)_{j\in \alpha ^{-1}(k)}$ under the natural map \[
\left( \prod_{j \in \alpha ^{-1}(k)} \operatorname{BDiff}\left(\coprod_{i \in I_j} P_i \right) \right)  \times \operatorname{BDiff}\left(\coprod_{i \in I_k} P_i \right) \to \operatorname{BDiff}(\coprod_{i \in \cup_{j \in \alpha ^{-1}(k)} I_j \cup I_k}P_i ) 
.\]
\item If $f_j $ for $j = 1, \dots , n'$ and $f_k$ for $k=1, \dots , n''$ are the framed connect sum configurations over the graphs $G_j$ and $G_k$ respectively, their composition $\overline{f}_k$ is a configuration over the graph $\overline{G}_k$ obtained by connecting the incoming half-edges of $G_k$ with the outgoing half-edges of $\left\{ G_j, j\in \alpha ^{-1}(k) \right\}$. The configuration $\overline{f}_k$ itself is the union of the configurations $\left\{ f_j \right\}_{j\in \alpha ^{-1}(k)}$ and $f_k$.
\end{enumerate}
\end{enumerate}
\end{defn}

\begin{rmk}
The natural map in item 3.b requires some explaining. There is a natural map \[
\operatorname{BDiff}(X ) \times \operatorname{BDiff}(Y ) \to \operatorname{BDiff}(X \coprod Y ) 
.\]
which is induced by the natural map $\operatorname{Diff}(X ) \times \operatorname{Diff}(Y ) \to \operatorname{Diff}(X\coprod Y ) $. If $X$ is not diffeomorphic to $Y$, the target is just $\operatorname{Diff}(X ) \times \operatorname{Diff}(Y ) $ and this map is the identity. If $X$ is diffeomorphic to $Y$, however, the target now must include the diffeomorphism swapping the two, which is a separate path component from the identity component.

In our model of $\operatorname{BDiff}(X) = \operatorname{Emb}(X, \mathbb{R}^\infty) / \operatorname{Diff}(X ) $, we can give a geometric construction. Fix two embeddings \[
i_0, i_1 : \mathbb{R}^\infty \to \mathbb{R}^\infty
.\]
with disjoint images. Then send $[e_X] \in \operatorname{BDiff}(X ) = \operatorname{Emb}(X, \mathbb{R}^\infty) / \operatorname{Diff}(X ) $ and $[e_Y] \in \operatorname{BDiff}(Y )= \operatorname{Emb}(Y, \mathbb{R}^\infty) / \operatorname{Diff}(Y )  $ to the class of embedding $X \coprod Y \to \mathbb{R}^\infty$ given by \[
x \mapsto i_0 (e_X(x)) ,\quad y \mapsto i_1 (e_Y(y))
.\]
This is well defined because precomposing $e_X$ and $e_Y$ by diffeomorphisms corresponds to precomposing the union embedding by an element of $\operatorname{Diff}(X\coprod Y) $.
\end{rmk}

\textbf{Variation} Define $^t \mathbb{L}^\otimes _{SO}$ similarly, except now 2.b the $P_i$ are oriented irreducible 3-manifolds, and 2.c the configuration $f_j \in C^{\partial , \mid \alpha ^{-1}(j) \mid } _{SO}(M)$ is an oriented connect sum configuration.

\begin{defn} \label{defn:L}
Denote by $\mathbb{L}^\otimes $ the topological nerve of $^t \mathbb{L}^\otimes $, and by $\mathbb{L}^\otimes _{SO}$ the topological nerve of $^t \mathbb{L}^\otimes _{SO}$.
\end{defn}

\begin{prop}
The $\infty$-categories $\mathbb{L}^\otimes $ and $\mathbb{L}^\otimes _{SO}$ are $\infty$-operads.
\end{prop}

\begin{proof}
We show that $\operatorname{Sing}(^t \mathbb{L}^\otimes )$ is canonically isomorphic to the simplicial category $\mathcal{O}^\otimes $, constructed via \cite[Notation 2.1.1.22]{lurie2017higheralgebra}, of a fibrant simplicial coloured operad $\mathcal{O}$. Therefore $\mathbb{L}^\otimes = N(\operatorname{Sing}(^t \mathbb{L}^\otimes )) = N(\mathcal{O}^\otimes ) = N^\otimes (\mathcal{O})$ is the operadic nerve of $\mathcal{O}$, which by \cite[Proposition 2.1.1.27]{lurie2017higheralgebra} is an $\infty$-operad.

Let $\mathcal{O}$ be the simplicial coloured operad with a single object $D^3$, with  \[
\operatorname{Mul}_{\mathcal{O}}(\left\{ D^3 \right\} _{i\in I}, D^3) = \operatorname{Sing}(\operatorname{Map}_{^t \mathbb{L}^\otimes } (\langle n \rangle , \langle 1 \rangle )^a)
\]
where $a: \langle n \rangle \to \langle 1 \rangle $ is the unique active map, sending $*$ to $*$ and everything else to $1$, and $\text{Map}_{^t \mathbb{L}^\otimes }(\langle n \rangle , \langle 1 \rangle )^a$ is the subspace of morphisms with data $a$ as the map in $\mathcal{F}in_*$.
The singular complex of any topological space is Kan. So $\mathcal{O}$ is a fibrant simplicial one-coloured operad. By construction, there is a canonical isomorphism between $\operatorname{Sing}(^t \mathbb{L}^\otimes )$ and $\mathcal{O}^\otimes $. Hence $\mathbb{L}^\otimes $ is an $\infty$-operad.

The proof for the oriented version is analogous.
\end{proof}

The $\infty$-operad $\mathbb{L}^\otimes $ has a few notable sub-operads, induced by sub-categories of $^t \mathbb{L}^\otimes $. Same goes for the oriented version $\mathbb{L}_{SO}^\otimes $. We shall introduce them below and show that one of them is equivalent to $\mathbb{E}_3^{\otimes }$.

\begin{defn}
($^t \mathbb{L}^{*, \otimes }$) Consider the subcategory $^t \mathbb{L}^{*, \otimes }$ of the topological category $^t \mathbb{L}^{\otimes }$ defined by 
\begin{enumerate}
\item The objects are $\langle n \rangle $, the same as before and as those of $\mathcal{F}in_*$.
\item A morphism from $\langle m \rangle $ to $\langle n \rangle $ consists of the following data:
\begin{enumerate}
\item a morphism $\alpha : \langle m \rangle \to \langle n \rangle  $ in $\mathcal{F}in_*$.
\item for every $j=1, \dots ,n$, a framed connect sum configuration $f_j \in C^\partial _{fr, G}(S^3) \subset C^\partial _{fr}(S^3)$, where $G$ is the graph with a single vertex $v$ and $d\geq 1$ unpaired half-edges.
\end{enumerate}
\item The composition is the same as those of $^t \mathbb{L}^\otimes $.
\end{enumerate}
\end{defn}

\begin{rmk}
This is a subcategory of $^t \mathbb{L}^\otimes $ because we have chosen all $g_j$ to be $0$, and all $h_j \in \operatorname{BDiff}(\emptyset ) $ to be the unique point.\footnotemark{} Then $M$ ,which is supposed to appear in the configuration $C^\partial _{fr}(M)$, is now the unit under connect sum, which is $S^3$. 

\footnotetext{There is a unique function  $f: \emptyset \to \emptyset$ which is vacuously smooth and vacuously its own inverse. Hence $\operatorname{Diff}(\emptyset  ) = \left\{ f \right\}  $. Recall that our model for the classifying spaces is $\operatorname{BDiff}(M) \cong \operatorname{Emb}(M, \mathbb{R}^\infty) / \operatorname{Diff}(M ) $. There is only one map from the empty set to any (possibly empty) set: $e: \emptyset \to S$, which is vacuously a smooth embedding when $S = \mathbb{R}^\infty$. Hence $\operatorname{BDiff}((\emptyset) ) \cong * $ the one-point space in our model. } 
\end{rmk}

\begin{defn}
Let $\mathbb{L}^{*, \otimes }$ be the topological nerve of $^t \mathbb{L}^{*, \otimes }$. It is an $\infty$-operad. Moreover, it is an $\infty$-sub-operad of $\mathbb{L}^\otimes $.
\end{defn}

\begin{prop}
There is a canonical isomorphism between $^t \mathbb{L}^{*, \otimes }$ and $^t \mathbb{E}_3^\otimes $. 
\end{prop}

\begin{proof}
The two topological categories have the same objects. So we only need to show that they agree on morphisms. Fix a $j = 1, \dots , n$. A framed connect sum configuration $f_j \in C^{\partial }_{fr, G}(S^3)$ where $G$ is the single vertex $v$ graph with $d\geq 1$ unpaired half edges is a tuple \[
f_j = (D^3_\epsilon \hookrightarrow S^3)_{\epsilon \in DE(v)}
.\]
where the embeddings are rectilinear and whose images are disjoint. Here $ DE(v)$ is the set of half-edges at the only vertex $v$. We may delete the interior of the image of the rectilinear embedding $D^3_{\epsilon^!} \hookrightarrow S^ 3$ corresponding to the one distinguished half-edge $\epsilon^! \in DE(v)$. After which we are left with a collection of rectilinear disjoint embeddings \[
D^3_\epsilon \hookrightarrow (S^3 \setminus \operatorname{int}(im(D^3_{\epsilon^!})) )\cong D^3 ,\quad  \epsilon \in DE(v) - \epsilon^!
.\]
We may identify this with the morphisms of $^t\mathbb{E}_3^\otimes $\[
\operatorname{Rect}(D^3 \times \alpha ^{-1}(j), D^3)
.\]
where $d = |\alpha ^{-1}(j)| +1$.
\end{proof}

\begin{cor}
There is a canonical isomorphism of $\infty$-operads between $\mathbb{L}^{*, \otimes }$ and $\mathbb{E}_3^\otimes $. Moreover, $\mathbb{E}_3^\otimes $ is an $\infty$-sub-operad of $\mathbb{L}^\otimes $.
\end{cor}

Another sub-operad of interest is the ``spike ball'' operad, defined below.

\begin{defn}
(Spike Ball Operad) The ``spike ball'' operad $^t \mathbb{L}^{\bigstar, \otimes }$ is the subcategory of the topological category $^t \mathbb{L}^\otimes $ defined by 
\begin{enumerate}
\item The objects are $\langle n \rangle $, the same as before and as those of $\mathcal{F}in_*$.
\item A morphism from $\langle m \rangle $ to $\langle n \rangle $ consists of the following data:	
\begin{enumerate}
\item a morphism $\alpha : \langle m \rangle \to \langle n \rangle $ in $\mathcal{F}in_*$.
\item for every $j = 1, \dots , n$, an element $h_j \in \operatorname{BDiff}(\coprod_{i \in I_j} P_i ) $, where $P_i$ are framed irreducible $3$-manifolds.
\item for every $j=1, \dots , n$, a framed connect sum configuration $f_j \in C^{\partial , \mid \alpha ^{-1}(j) \mid }_{fr, G}(M) \subset C^{\partial, \mid \alpha ^{-1}(j) \mid } _{fr}(M)$, where $M = \#_{i \in I_j} P_i $, and $G$ is a graph with an unlabelled vertex $v$ and $\mid \alpha ^{-1}(j) \mid $ unpaired half edges at $v$, and connected to $v$ are $I_j$-labelled vertices. So $G$ is the star graph with one central vertex $v$ and $|I_j|$-many leaves, where $v$ also has $\mid \alpha ^{-1}(j) \mid $ unpaired half-edges.
\end{enumerate}
\item The composition is the same as those of  $^t \mathbb{L}^\otimes $.
\end{enumerate}

An example of the type of star graph $G$ can be found in figure \ref{fig:spikeball}-(a). 
Let $\mathbb{L}^{\bigstar, \otimes }$ be the topological nerve of $^t \mathbb{L}^{\bigstar, \otimes }$. It is an $\infty$-category and moreover an $\infty$-operad. Call it the ``$\infty$-spike-ball operad.''
\end{defn}

\begin{figure}[ht]
    \centering
\def\svgwidth{1\columnwidth} 
\begingroup%
  \makeatletter%
  \providecommand\color[2][]{%
    \errmessage{(Inkscape) Color is used for the text in Inkscape, but the package 'color.sty' is not loaded}%
    \renewcommand\color[2][]{}%
  }%
  \providecommand\transparent[1]{%
    \errmessage{(Inkscape) Transparency is used (non-zero) for the text in Inkscape, but the package 'transparent.sty' is not loaded}%
    \renewcommand\transparent[1]{}%
  }%
  \providecommand\rotatebox[2]{#2}%
  \newcommand*\fsize{\dimexpr\f@size pt\relax}%
  \newcommand*\lineheight[1]{\fontsize{\fsize}{#1\fsize}\selectfont}%
  \ifx\svgwidth\undefined%
    \setlength{\unitlength}{680.31496063bp}%
    \ifx\svgscale\undefined%
      \relax%
    \else%
      \setlength{\unitlength}{\unitlength * \real{\svgscale}}%
    \fi%
  \else%
    \setlength{\unitlength}{\svgwidth}%
  \fi%
  \global\let\svgwidth\undefined%
  \global\let\svgscale\undefined%
  \makeatother%
  \begin{picture}(1,0.33333333)%
    \lineheight{1}%
    \setlength\tabcolsep{0pt}%
    \put(0,0){\includegraphics[width=\unitlength,page=1]{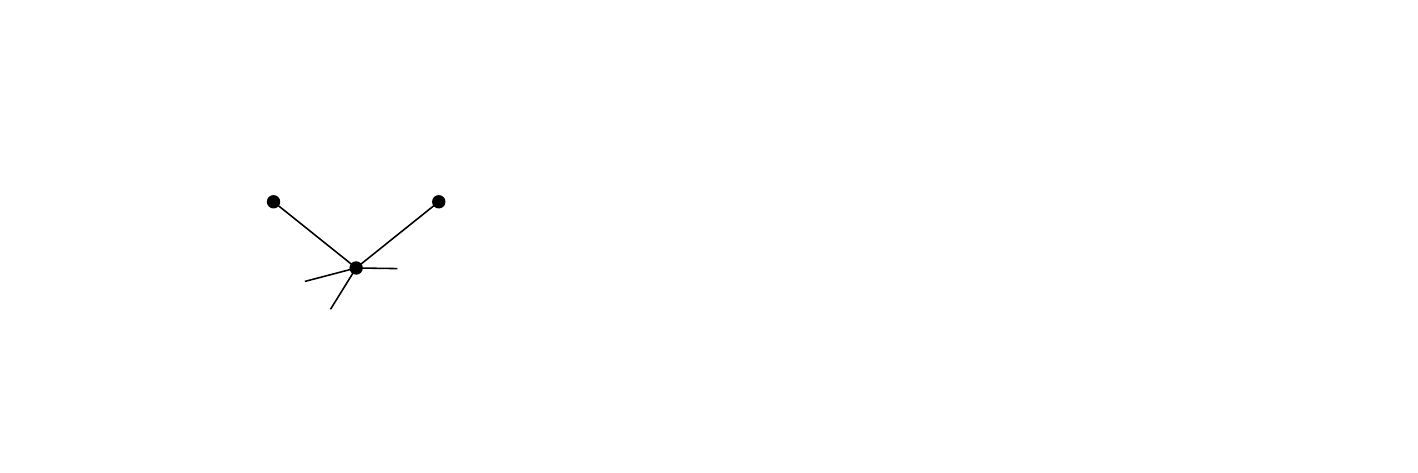}}%
    \put(0.1880864,0.20644983){\makebox(0,0)[lt]{\lineheight{1.25}\smash{\begin{tabular}[t]{l}$1$\end{tabular}}}}%
    \put(0.30607866,0.20644983){\makebox(0,0)[lt]{\lineheight{1.25}\smash{\begin{tabular}[t]{l}$2$\end{tabular}}}}%
    \put(0,0){\includegraphics[width=\unitlength,page=2]{spikeball.pdf}}%
    \put(0.49190156,0.26511154){\makebox(0,0)[lt]{\lineheight{1.25}\smash{\begin{tabular}[t]{l}$P_1$\end{tabular}}}}%
    \put(0.77482694,0.26077296){\makebox(0,0)[lt]{\lineheight{1.25}\smash{\begin{tabular}[t]{l}$P_2$\end{tabular}}}}%
    \put(0.24160784,0.02153772){\makebox(0,0)[lt]{\lineheight{1.25}\smash{\begin{tabular}[t]{l}(a)\end{tabular}}}}%
    \put(0.72752304,0.01073958){\makebox(0,0)[lt]{\lineheight{1.25}\smash{\begin{tabular}[t]{l}(b)\end{tabular}}}}%
  \end{picture}%
\endgroup%

    \caption[Spike ball operad example]{(a) An example of the star graph $G$ in the spike ball operad. In this particular case, $G$ has $\mid \alpha ^{-1}(j) \mid + 1 =3$ half edges, and $I_j=\left\{ 1,2 \right\} $ labelled vertices. (b) illustrate the particular bordism $[M] \in \operatorname{BDiff}( M) $represented by $G$ and the connect sum configuration $f$. The two dashed spheres are considered incoming, and the distinguished half edge is the big sphere, considered outgoing.}
    \label{fig:spikeball}
\end{figure}

\begin{rmk}
This is a subcategory of $^t \mathbb{L}^\otimes $ because we have chosen all $g_j$ to be $0$. Also note that $^t \mathbb{L}^{*, \otimes }$ is a subcategory of $^t \mathbb{L}^{\bigstar, \otimes }$, where for any $j = 1, \dots , n$, $I_j = \emptyset$ so there are no irreducible factors in $\operatorname{BDiff}(\coprod_{i \in I_j}P_i) $.
\end{rmk}

\begin{cor}
$\mathbb{E}_3^\otimes $ is an $\infty$-sub-operad of $\mathbb{L}^{\bigstar, \otimes }$.
\end{cor}

Lastly, we need to define the $\infty$-endomorphism operad of $S^2$. This is most easily done using complete Segal spaces as there is previous work by Calaque and Scheimbauer. 

\begin{defn} 
Let $V$ be a finite dimensional real vector space. Define $\left( P\Omega^{n-1}Bord_n^V \right)_k$ to be the set of tuples \[
(M, t_0, \dots , t_k)
.\]
such that
\begin{enumerate}
\item $t_0 \leq t_1 \leq \dots \leq t_k$ an ordered tuple of $k+1$ real numbers.
\item $M$ is a closed $n$-dimensional submanifold of $V \times \mathbb{R}$ and the composition $\pi :M \hookrightarrow V\times \mathbb{R} \twoheadrightarrow \mathbb{R}$ is proper.
\item For each $0\leq i\leq k$, and each $x\in \pi ^{-1}(t_i)$, the map $\pi $ is submersive at $x$.
\end{enumerate}
Viewing $\left( P\Omega ^{n-1}Bord_n^V \right)_k$ as a subspace of \[
\operatorname{Sub}(V\times \mathbb{R})\times \mathbb{R}^{k+1} ,\quad \operatorname{Sub}(V\times \mathbb{R}) = \coprod_{[M]} \operatorname{Emb}(M, V\times \mathbb{R}) / \operatorname{Diff}(M)
.\]
we equip it with the subspace topology.

Note that $\left( P\Omega^{n-1}Bord_n^V \right)_\bullet$ is a simplicial space. The simplicial face maps are given by forgetting a $t$, and the simplicial degeneracy maps are given by duplicating a $t$.
\end{defn}

Let $P\Omega^{n-1}Bord_n = \varinjlim_{V \subseteq \mathbb{R}^\infty} P\Omega^{n-1}Bord_n^V$. At the $k$-th level, we may think of it to consist of tuples \[
(M, t_0, \dots ,t_k)
.\]
where $M$ is now a closed $n$-dimensional submanifold of $\mathbb{R}^\infty$ (technically $\mathbb{R}^\infty \times \mathbb{R}$).

\begin{claim}
$P\Omega^{n-1}Bord_n$ is a Segal space.
\end{claim}

\begin{defn} 
We define the $(\infty,1)$ category of $n-1$ manifolds and bordisms as the completion of the Segal space $P\Omega^{n-1}Bord_n$. Loosely speaking this category consists of the following data:
\begin{itemize}
\item The objects are closed $n-1$ manifolds, corresponding to the $n-1$ dimensional preimage $\pi ^{-1}(t_0)\subset M$ where $M$ is an $n$-manifold.
\item The 1-morphisms are compact bordisms, corresponding to the $n$-dimensional preimage $\pi ^{-1}([t_0 ,t_1 ])\subset  M$.
\item The 2-morphisms are diffeomorphisms between bordisms.
\item The 3-morphisms are isotopes of diffeomorphisms
\item  $\dots $
\end{itemize}
\end{defn}

\textbf{Variations} By demanding the manifolds of $P\Omega^{n-1}\operatorname{Bord}_n$ be framed or oriented, we obtain the standard framed or oriented variations of $\Omega^{n-1}\operatorname{Bord}_n$, denoted as $\Omega^{n-1}\operatorname{Bord}_n^{fr}$  or $\Omega^{n-1}\operatorname{Bord}_n^{SO}$.

\begin{mycom}
There are a few versions of $PBord^V$ floating in the literature. They differ mostly on condition 3. The original sketch was proposed by \cite{lurie_classification_2008}, which was later corrected by \cite{Calaque_2019} . Our definition most closely resembles that of \cite{jordan2016topological}, which was adapted from \cite{Calaque_2019}.
\end{mycom}

\begin{rmk}
If $\mathcal{C}$ is a symmetric monoidal $(\infty,n)$-category with unit $\mathbf{1}$, then $\Omega \mathcal{C} = \operatorname{End}_{\mathcal{C}}(\mathbf{1})= \operatorname{Hom}_{\mathcal{C}}(\mathbf{1}, \mathbf{1}) $ is a symmetric monoidal $(\infty,n-1)$-category. As the notation suggests, $\Omega^{k-1} Bord_k$ has an alternative construction as the $k-1$ looping of the $(\infty,k)$-category $Bord_k$. In fact, this construction is more precise than the one we presented above-- see \cite{Calaque_2019} if our presentation is insufficient.
\end{rmk}

\begin{rmk}
The $(\infty,1)$-category $\Omega^{n-1}Bord_n$ inherits a symmetric monoidal structure from $Bord_n$. We shall be intentionally vague about what we mean by ``symmetric monoidal structure.'' Beyond the heuristics of disjoint union, a precise definition can be found in [Scheimbauer]. The symmetric monoidal category $\Omega^{n-1}Bord_n$ determines 
\footnote{More precisely, one extracts a quasi-category via [Joyal Tierney]. This quasi-category is the $\infty$-category which is symmetric monoidal in the sense of Lurie HA, and it is the underlying $\infty$-category of the $\infty$-operad $\Omega^{n-1}Bord_n^\otimes$.} 
an $\infty$-operad $p: \Omega^{n-1}Bord_n^\otimes \to N(\mathcal{F}in_*)$. Roughly speaking, the fibre of $p$ over $\langle m \rangle $ consists of $m$-tuples of $n-1$ dimensional closed manifolds, and a morphism over $\alpha : \langle m \rangle \to \langle m' \rangle $ is a bordism. 
\end{rmk}

\begin{defn} \label{defn:Endomorphism operad}
(Endomorphism operad) Given $p: \mathcal{C}^\otimes \to N(\mathcal{F}in_*)$ be an $\infty$-operad, and let $X \in \mathcal{C}$ be an object. By definition there is an equivalence $\mathcal{C}^\otimes _{\langle n \rangle } \cong \mathcal{C}^n$ induced by the inert maps $\rho^i : \langle n \rangle \to \langle 1 \rangle $. Let $X^n$ denote the object of $\mathcal{C}^\otimes_{\langle n \rangle }$ corresponding to the $n$-tuple $(X, \dots , X) \in \mathcal{C}^n$. The $\infty$-endomorphism operad of $X$ is defined as the full $\infty$-suboperad \[
\operatorname{End}_{\mathcal{C}}^\otimes (X) \subseteq \mathcal{C}^\otimes 
.\]
spanned by the objects $X^n$ in each fiber $\mathcal{C}^\otimes_{\langle n \rangle }$.
\end{defn}

\begin{notations}
Let $\operatorname{Mul}_{\mathcal{C}}(\left\{ X \right\}_{1\leq i\leq n},Y)$ denote the union of the components of $\operatorname{Map}_{\mathcal{C}^\otimes }(X^n,X)$ over the unique (active) map $\langle n \rangle \to \langle 1 \rangle $ sending $*$ to $*$ and everything else to $1$.
\end{notations}

Heuristically, the $\infty$-endomorphism operad of $X$ is the one-coloured $\infty$-operad whose multimorphisms are given by those of $\mathcal{C}$: \[
\operatorname{Map}_{\operatorname{End}_{\mathcal{C}}^\otimes (X) }(X^n, X) = \operatorname{Mul}_{\mathcal{C}}(\{X, \dots , X\}, X)
.\]

\begin{defn} 
Let $\Omega^{n-1}Bord_n$ be the symmetric monoidal $\infty$-category of $n$-dimensional bordisms. The $\infty$-endomorphism operad of $S^{n-1}$ is $\operatorname{End}_{\Omega^{n-1}Bord_n}^{\otimes }(S^{n-1}) $. The connected $\infty$-endomorphism operad of $S^{n-1}$ is the sub-operad
\[
\operatorname{End}_{\Omega^{n-1}Bord_n}^{\otimes , conn}(S^{n-1}) \hookrightarrow  \operatorname{End}_{\Omega^{n-1}Bord_n}^{\otimes }(S^{n-1})
\]
consisting of connected bordisms.

Similarly, we define the framed $\infty$-endomorphism operad $\operatorname{End}_{\Omega^{n-1}Bord_n^{fr}}^{\otimes }(S^{n-1}) $, its connected sub-operad $\operatorname{End}_{\Omega^{n-1}Bord_n^{fr}}^{\otimes, conn }(S^{n-1}) $ and the oriented $\infty$-endomorphism operad $\operatorname{End}_{\Omega^{n-1}Bord_n^{SO}}^{\otimes }(S^{n-1}) $, along with its connected sub-operad $\operatorname{End}_{\Omega^{n-1}Bord_n^{SO}}^{\otimes, conn }(S^{n-1}) $.
\end{defn}

%
%
%
%

\begin{conj} \label{conj:LisEndfr}
There is a map of $\infty$-operads $f: \mathbb{L}^\otimes \to \operatorname{End}_{\Omega^2 \operatorname{Bord}_3^{fr}}^\otimes (S^2) $, whose image is $\operatorname{End}_{\Omega^2 \operatorname{Bord}_3^{fr}}^{\otimes , conn}(S^2)$. Furthermore, $f$ is an equivalence of $\infty$-operads $\mathbb{L}^\otimes \overset{\sim}{\to} \operatorname{End}_{\Omega^2 \operatorname{Bord}_3^{fr}}^{\otimes , conn}(S^2)$.
\end{conj}

\begin{conj} \label{conj:LisEndSO}
There is a map of $\infty$-operads $f_{SO}: \mathbb{L}_{SO}^\otimes \to \operatorname{End}_{\Omega^2 \operatorname{Bord}_3^{SO}}^\otimes (S^2) $, whose image is $\operatorname{End}_{\Omega^2 \operatorname{Bord}_3^{SO}}^{\otimes , conn}(S^2)$. Furthermore, $f_{SO}$ is an equivalence of $\infty$-operads $\mathbb{L}_{SO}^\otimes \overset{\sim}{\to} \operatorname{End}_{\Omega^2 \operatorname{Bord}_3^{SO}}^{\otimes , conn}(S^2)$.
\end{conj}

\begin{rmk}
The reason behind the proposal is this: by Hatcher's unfinished draft \cite{hatcher2008diffeomorphism}, which was proven by \cite{boyd2026primedecompositionfibresequence}, for every oriented 3 manifold $M$ whose irreducible factors are $P_1 , \dots , P_m$, with possibly finitely many irreducible prime factors $S^2 \times S^1$, then there exists a homotopy fibre sequence \[
C(M) \to \operatorname{BDiff}(M ) \to \operatorname{BDiff}(\coprod_{1\leq i\leq m}P_i ) 
.\]
where $C(M)$ is the space of connect sum configurations, and the second map is roughly the splitting of $M \hookrightarrow \mathbb{R}^\infty$ along splitting spheres, resulting in a collection of irreducible factors with $D^3$ removed and embedded in $\mathbb{R}^\infty$. We then cap the missing balls, which give the boundary-less irreducible factors $\coprod_{i} P_i \hookrightarrow \mathbb{R}^\infty$.

A multimorphism in any of the above $\infty$-endomorphism operads is a bordism from copies of  $S^2$ to a single outgoing $S^2$. Such a bordism is an element $(M, t_0 ,t_1 )$ of \[
\operatorname{BDiff}(M ) \times \mathbb{R}^2 \cong \operatorname{Emb}(M, \mathbb{R}^\infty) / \operatorname{Diff}(M ) \times \mathbb{R}^2
.\]

Hence in order to model the multimorphisms by an operad, we need the operad to essentially have $\operatorname{BDiff}(M ) $ built in, as well as ways to specify its boundaries. The space of connect sum configurations $C(M)$ in the homotopy fibration, over a base point, is our intended carrier of this information. To specify the boundaries, we modified he space of connect sum configurations to include embeddings $B^3 \hookrightarrow S^3, P_i$ corresponding to newly introduced half edges. Hence the operad $\mathbb{L}^\otimes $ and the oriented cousin $\mathbb{L}_{SO}^\otimes $ has exactly enough information to start from a collection of irreducible 3-manifolds $\coprod_i P_i \hookrightarrow \mathbb{R}^\infty$ and form a connect sum through a connect sum configuration $f \in C_{fr}^\partial (M)$ or $C_{SO}^\partial (M)$. The configuration $f$ also specifies where along the bounding 3-balls are, so one could remove their interior.
\end{rmk}


\begin{eg}
Recall that $\mathbb{L}^\otimes $ has an $\infty$-suboperad $\mathbb{L}^{\bigstar, \otimes }$, the ``spike-ball'' $\infty$-operad. Consider a subset of morphism of $\langle 1 \rangle \to \langle 1 \rangle $ given by the following data
\begin{enumerate}
\item the morphism $\alpha : \langle 1 \rangle \to \langle 1 \rangle $ is the identity in $\mathcal{F}in_*$.
\item an element $h \in \operatorname{BDiff}(P )  $ where $P$ is any fixed framed irreducible 3-manifold.
\item A configuration $f \in C^{\partial, 1}_{fr, G}(P)$, where $G$ is the graph with a single unlabelled vertex $v$ connected via the edge $e$ to a single labelled vertex $p$, where there are two additional unpaired half edges $\left\{ \epsilon_1 , \epsilon_2  \right\} $ at the unlabelled vertex $v$: figure \ref{fig:simplespikeball}

\begin{figure}[ht]
    \centering
\def\svgwidth{1\columnwidth} 
\begingroup%
  \makeatletter%
  \providecommand\color[2][]{%
    \errmessage{(Inkscape) Color is used for the text in Inkscape, but the package 'color.sty' is not loaded}%
    \renewcommand\color[2][]{}%
  }%
  \providecommand\transparent[1]{%
    \errmessage{(Inkscape) Transparency is used (non-zero) for the text in Inkscape, but the package 'transparent.sty' is not loaded}%
    \renewcommand\transparent[1]{}%
  }%
  \providecommand\rotatebox[2]{#2}%
  \newcommand*\fsize{\dimexpr\f@size pt\relax}%
  \newcommand*\lineheight[1]{\fontsize{\fsize}{#1\fsize}\selectfont}%
  \ifx\svgwidth\undefined%
    \setlength{\unitlength}{680.31496063bp}%
    \ifx\svgscale\undefined%
      \relax%
    \else%
      \setlength{\unitlength}{\unitlength * \real{\svgscale}}%
    \fi%
  \else%
    \setlength{\unitlength}{\svgwidth}%
  \fi%
  \global\let\svgwidth\undefined%
  \global\let\svgscale\undefined%
  \makeatother%
  \begin{picture}(1,0.33333333)%
    \lineheight{1}%
    \setlength\tabcolsep{0pt}%
    \put(0,0){\includegraphics[width=\unitlength,page=1]{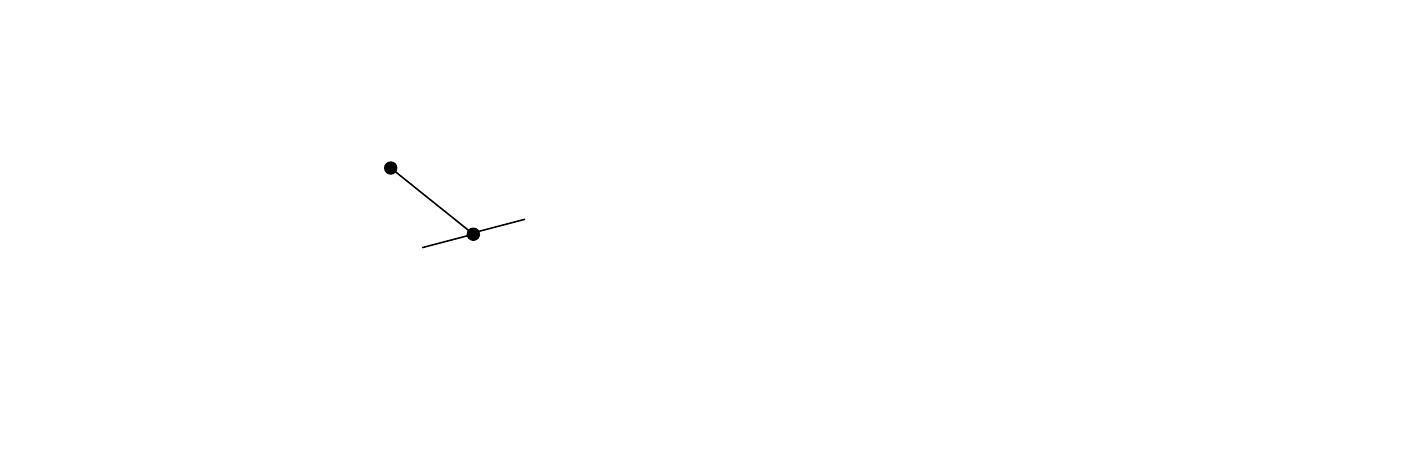}}%
    \put(0.27079601,0.23026384){\makebox(0,0)[lt]{\lineheight{1.25}\smash{\begin{tabular}[t]{l}$1$\end{tabular}}}}%
    \put(0,0){\includegraphics[width=\unitlength,page=2]{simplespikeball.pdf}}%
    \put(0.56497154,0.26750412){\makebox(0,0)[lt]{\lineheight{1.25}\smash{\begin{tabular}[t]{l}$P_1$\end{tabular}}}}%
    \put(0.32431746,0.04535166){\makebox(0,0)[lt]{\lineheight{1.25}\smash{\begin{tabular}[t]{l}(a)\end{tabular}}}}%
    \put(0.80059306,0.0131321){\makebox(0,0)[lt]{\lineheight{1.25}\smash{\begin{tabular}[t]{l}(b)\end{tabular}}}}%
  \end{picture}%
\endgroup%

    \caption[A simple spike ball example]{A simple spike ball example. (a) gives the graph $G$. (b) provides a visualization for the bordism represented by $G$ and a connect sum configuration $f$.}
    \label{fig:simplespikeball}
\end{figure}
\end{enumerate}

The configuration $f$ is equal to the following data \[
(r_e, l_e, D^3_e \hookrightarrow S^3_v, D^3_e \hookrightarrow P, D^3_{\epsilon_1 }\hookrightarrow S^3_v, D^3_{\epsilon_2 } \hookrightarrow S^3_v)
.\]
where all the embeddings in $S^3_v$ are rectilinear and have disjoint images. We may delete the interior of the image of one of the half-edges, say $D^3_{\epsilon_2 }$, so what remains of $S^3_v$ is diffeomorphic to a ball $D^3$. Then the configuration $f$ contains the following data \[
(r_e, l_e, D^3_e \hookrightarrow D^3, D^3_e \hookrightarrow P, D^3_{\epsilon_1 }\hookrightarrow D^3 )
.\]

Note that there exists a map \[
\operatorname{Conf}_2(D^3) \times \operatorname{BDiff}(P) \to \mathbb{L}^{\bigstar, \otimes }
.\]
Sending a configuration of $2$ points in $D^3$ to the centre of the embeddings $D^3_e, D^3_{\epsilon_1 } \hookrightarrow D^3$ and an element $h \in \operatorname{BDiff}(P ) $ to the chosen element as part of the data of $\langle 1 \rangle \to \langle 1 \rangle $. Taking the inclusion into the operad $\mathbb{L}^{\otimes }$, passing through the endomorphism operad, and finally through a TFT $Z$, we get a map \[
\operatorname{Conf}_2(D^3) \times \operatorname{BDiff}(P) \to \mathbb{L}^{\bigstar, \otimes } \hookrightarrow \mathbb{L}^\otimes \to \operatorname{End}_{\Omega \operatorname{Bord}_3^{fr}}^\otimes (S^2) \to \operatorname{End}_{\mathcal{C}}^\otimes (Z(S^2))
.\]
whose image is over the $1$-ary operations. Therefore there is a map \[
\operatorname{Conf}_2(D^3) \times \operatorname{BDiff}(P) \to \operatorname{Map}_{\mathcal{C}}(Z(S^2), Z(S^2))
.\]
for any TFT $Z: \Omega^2\operatorname{Bord}_3^{fr} \to \mathcal{C}$ and for any framed irreducible 3-manifold $P$.
\end{eg}

\appendix
\chapter{Appendix: Definitions and Common Results}

\section{Presentations of a Category} \label{appendix:presentation}
\begin{defn} \label{defn:Free Category}
(Free Category) Let $G$ be a directed graph. The \textbf{free category} $F(G)$ has objects the vertices of $G$ and morphisms all composable arrows. The empty path is the identity of a vertex.
\end{defn}

\begin{defn} \label{defn:Binary Relation}
(Binary Relation) Let $\mathcal{C}$ be a category and $a,b \in \operatorname{obj}(\mathcal{C})$. A \textbf{binary relation} $R_{a,b}$ on $\operatorname{Hom}_{\mathcal{C}}(a,b) $ is a subset $R_{a,b} \subset  \operatorname{Hom}_{\mathcal{C}}(a,b)\times \operatorname{Hom}_{\mathcal{C}}(a,b) $. We write $f R_{a,b} g$ if $(f,g) \in R_{a,b}$. We say $R$ is a \textbf{binary relation on homs} of $\mathcal{C}$ if $R = \left\{ R_{a,b} \mid a,b \in \operatorname{obj} (\mathcal{C}) \right\} $ is a collection of binary relations on all pairs of objects of $\mathcal{C}$. If a pair of morphisms $(f,f')$ is in $R_{a,b}$, we write $f R_{a,b} f'$, or $f=f'$ if the context is clear. 
\end{defn}

\begin{defn} \label{defn:Congruence}
(Congruence) A \textbf{congruence} $R$ on a category $\mathcal{C}$ is an assignment to each pair of objects $a,b \in \operatorname{obj}(\mathcal{C})$ an equivalence relation $R_{a,b}$ on $\operatorname{Hom}_{\mathcal{C}}(a,b) $ that is compatible with compositions: if $f R_{a,b} f'$, meaning if $(f,f')$ are equivalent, then for all $g: a' \to a$ and all $h: b \to b'$ one has $(hfg) R_{a',b'} (hf'g)$; if $e R_{b,c} e'$, then $(ef) R_{a,c} (e'f')$. We sometimes write $f \sim_R f'$ instead to denote the fact that $R$ is a congruence rather than a binary relation.
\end{defn}

\begin{thm}
(First isomorphism theorem of category theory) Let $H: \mathcal{C} \to \mathcal{D}$ be a functor. Then $H$ defines a congruence $\equiv_H$ on $\mathcal{C}$ by $f \equiv_H f'$ if and only if $H(f) = H(f')$. Moreover, $H$ factors uniquely through the quotient: $H = \overline{H} \circ Q$ where $Q: \mathcal{C} \to \mathcal{C} / \equiv_H$ and $\overline{H}: \mathcal{C} / \equiv_H \to \mathcal{D}$ where $\overline{H}(a) = H(a)$ and $\overline{H}([f]) = H(f)$.
\end{thm}

\begin{proof}
$\equiv_H$ is obviously symmetric, reflexive, and transitive. Functoriality of $H$ guarantees that $\equiv_H$ is compatible with compositions. The functor $\overline{H}$ is well defined by definition, and $\overline{H}$ is unique since any functor $H'$ with $H = H' \circ Q$ must satisfy $ H'([f])= H'(Q(f))= H(f) $ for any morphism in $\mathcal{C}/ \equiv_H$. Since $Q$ is identity on objects, $H'$ must coincide with $H$ (and hence $\overline{H}$) on objects also. This completely determines $H'$ as the same as $\overline{H}$.
\end{proof}

\begin{lemma}
(Least Congruence) Given any binary relations $R$ on homs of $\mathcal{C}$, there is a least congruence $R'$ on $\mathcal{C}$ with $R \subset  R'$.
\end{lemma}

\begin{defn} \label{defn:Quotient Category}
(Quotient Category) Let $\mathcal{C}$ be a category and $R$ a congruence on $\mathcal{C}$. Then the \textbf{quotient category }$\mathcal{C}/R$ has the same objects as $\mathcal{C}$ and the morphisms $\operatorname{Hom}_{\mathcal{C}/R}(a,b) = \operatorname{Hom}_{\mathcal{C}}(a,b) / R_{a,b}$ are those of $\mathcal{C}$ modulo the equivalence relation $R_{a,b}$.

If $R$ is a binary relation, the quotient category $\mathcal{C} / R$ is defined as  $\mathcal{C} / R'$ where $R'$ is the least congruence containing $R$.
\end{defn}

\begin{prop}
(Universal Property of Quotient Category) \cite{mac_lane_categories_1994} Let $\mathcal{C}$ be a category with $R$ a binary relation on the homs. Let $\mathcal{C}/ R$ be the quotient category. Then the quotient functor $Q=Q_R : \mathcal{C} \to \mathcal{C}/R$, which is identity on objects and sends a morphism to its equivalence class,\footnote{Hence $Q$ is surjective on hom sets, and is therefore a full functor.} satisfy the following universal property: If $H: \mathcal{C} \to \mathcal{D}$ is any functor such that  $f R_{a,b} f'$ implies $Hf = Hf'$ for all $f$ and $f'$, then there exists a unique functor $H' : \mathcal{C}/R \to \mathcal{D}$ with $H' \circ Q_R = H$.
\end{prop}

\begin{defn} \label{defn:Generators and Relations}
(Generators and Relations) Let $G$ be a directed graph and $R$ a binary relation on $F(G)$, the free category of $G$. Then we say $F(G)/R $ has \textbf{generators }$G$ and \textbf{relations} $R$. An isomorphism $F(G) /R \to \mathcal{C}$ and $\langle G\mid R \rangle $ form a \textbf{presentation }of $\mathcal{C}$. We say the morphisms of $\mathcal{C}$ are \textbf{generated} by the edges of $G$ and the objects by the vertices, and that the binary relation $R$ is \textbf{sufficient}.
\end{defn}

\begin{rmk}
Sufficiency of relations is equivalent to the following more convenient description: Let $G$ be a directed graph and $R$ a binary relation on $F(G)$. Let $\overline{H} : F(G)/R \to \mathcal{C}$ be a functor, where $F(G) / R $ denotes the quotient category $F(G) / \sim_R$ for $\sim_R$ the least congruence containing $R$. Consider the functor $H : F(G) \to \mathcal{C}$ which is the composition $H = \overline{H} \circ Q$. Then $R$ being sufficient (equivalently $\overline{H}$ being an isomorphism) is the same as $\overline{H}$ being bijective on objects and on homs. But since $Q$ is bijective on objects, the former is equivalent to $H$ being bijective on objects. The injectivity part of the later says $\overline{H}([f]) = \overline{H}	([f'])$ if and only if $[f] = [f']$, i.e. $f \sim_R f'$. Notice that the congruence ``kernel'' $\equiv_H$ defined by $f \equiv_H f' \Leftrightarrow H(f) = H(f')$ can be equivalently defined as $f \equiv_H f' \Leftrightarrow \overline{H}([f]) = \overline{H}([f'])$. So $\overline{H}$ being fully faithful implies that $\sim_R = \equiv_H$. 

Translating the abstract nonsense into English, a relation $R$ is sufficient if every morphism in $\mathcal{C}$ is induced by at least one morphism in $F(G)$ and every equality of morphisms in $\mathcal{C}$ is due to equivalence induced by $R$ in $F(G)$.

Note that this does not speak of whether $R$ is the most efficient way of achieving sufficiency. 
\end{rmk}

\begin{defn} \label{defn:Smaller Relations}
(Smaller Relations) Let $R,R'$ be binary relations on homs of $F(G)$. We say $R'$ is smaller than $R$, denoted $R' \subsetneq R$, if $R'_{a,b} \subseteq R_{a,b}$ for all pairs of objects $a,b$ and there exists at least one pair of objects $X,Y$ such that $R'_{X,Y} \subsetneq R_{X,Y}$. In English, $R'$ is smaller than $R$ if it can be obtained from $R$ by removing some binary relations.
\end{defn}

\begin{defn} \label{defn:Necessary Relations}
(Necessary Relations) Let $\overline{H}, G, R$ be a presentation of $\mathcal{C}$. Let $H:F(G) \to \mathcal{C}$ be the composition $\overline{H} \circ Q_R$. We say $R$ is necessary if for any $R'$ smaller than $R$, the unique functor $\overline{H}' : F(G) / \sim_{R'} \to \mathcal{C}$, through which $H$ factors as $H = \overline{H}' \circ Q_{R'}$ is no longer an isomorphism. 
\end{defn}

\begin{rmk}
In this case, $\overline{H}'$ will fail faithfulness while still remaining full and bijective on objects. A pair of morphisms $(f,f') \in R_{X,Y} $ will be mapped to two distinct equivalence classes in $F(G) / \sim_{R'}$, but sent to the same morphism in $\mathcal{C}$ under $H$. Hence the two distinct equivalence classes will be sent to the same morphism by $\overline{H}'$.

Heuristically, the least congruence containing $R'$ will be too small, resulting in a quotient $F(G) /\sim_{R'}$ being too big. 
\end{rmk}

This paper is concerned with constructing presentations of $\operatorname{Cob}(3)_{S^2}$. There are two presentations-- both having non-negative integers (or equivalently $\sqcup^{n\geq 0} S^2$ )as vertices, but one graph has $G_1 $ (and all tensor products) as edges, while the other graph has  $G_2 $ (and all tensor products) as edges. By abuse of notation we shall refer to the two graphs as $G_1 $ and $G_2 $. Since their edges are morphisms, the free category is just $\operatorname{Cob}(3)_{S^2} $ except one forgets if two compositions are equal. The two binary relations on homs are commutative Frobenius relations for the $G_1$ case, and commutative Frobenius relations $+$ legs relations for the $G_2 $ case.

The first section is largely concerned with showing that the functors $F(G_{1,2})\to \operatorname{Cob}(3)_{S^2}$ are bijective on objects and full, so whatever relations we plan on choosing has a chance to result in presentations. The second section explicitly constructs two binary relations $R_{1,2}$ on $F(G_{1,2})$. We then show that these relations are sufficient, namely they give rise to presentations of $\operatorname{Cob}(3)_{S^2}$. Then we show that these relations can be thinned down to be necessary. Hence we obtain two minimal presentations of the category $\operatorname{Cob}(3)_{S^2} $.

\section{Towards $\infty$}

\begin{defn} \label{defn:infcat}
($(\infty,1)$-category) An $(\infty,1)$-category is a complete Segal space $X_\bullet$. That is, a simplicial object in topological spaces satisfying completeness and the Segal condition.
\end{defn}

See \cite{lurie_classification_2008} for details about completeness and Segal condition. See \cite[section 2]{zhao2013extended} for a comparison between the different settings of complete Segal spaces. Here we specialize to the setting of topological spaces and retreat to compactly generated Hausdorff spaces or CW spaces if pathologies arise.

\begin{defn} \label{defn:Quasi-Category}
(Quasi-Category) A quasi-category is a simplicial set $\mathcal{C}$ such that every map of simplicial sets $f_0: \Lambda^n_i \to \mathcal{C}$ can be extended to an $n$-simplex $f: \Delta^n \to \mathcal{C}$ for $0<i<n$. We also call it an  $\infty$-category.
\end{defn}

\begin{thm} \label{thm:JoyalTierney}
(Joyal Tierney) There is a Quillen equivalence between the model catgegory for quasi-categories and the model category for complete Segal spaces. In particular, to each complete Segal space $X_\bullet$ there is a quasi-category $Q_\bullet$ given by $Q_n = \operatorname{Sing}(X_n)_0$.
\end{thm}

\begin{notations}
Let $\text{Cat}_\Delta$ be the category of simplicial categories, the objects are simplicial categories and the morphisms are simplicial maps (i.e. maps of simplices that are compatible with face and degeneracy maps). Let $\text{sSet}$ be the category of simplicial sets and simplicial maps. Let $[n]=\left\{ 0<1<\dots <n \right\} $ be the category whose objects are $0,1,\dots n$ and whose maps are orderings $i<j$. 
\end{notations}

\begin{defn} \label{defn:Ordinary Nerve}
(Ordinary Nerve) The ordinary nerve of an ordinary category $\mathcal{C}$ is a simplicial set defined as follows: \[
N(\mathcal{C})_n = \operatorname{Hom}_{Cat}([n], \mathcal{C}) 
.\]
Specifically, $N(\mathcal{C})_n$ is the set of $n$-many composable arrows in $\mathcal{C}$: \[
X_0 \overset{f_0}{\to } X_1 \dots \overset{f_{n-1}}{\to }X_n
.\]
with the face map given by

\begin{itemize}
\item $d_0 $ forgets the first object and the first arrow.
\item $d_n$ forgets the last object and the last arrow.
\item  $d_i$ for $0<i<n$ forgets the object $X_i$ and replaces $f_{i-1}$ and $f_{i}$ with their composition.
\end{itemize}
and the degeneracy maps given by inserting the identity morphisms.
\end{defn}

To work in the setting of $\infty$-categories, one needs to define appropriate notions of nerves for simplicial categories and for topological categories. We must introduce the following notion

\begin{defn} \label{defn:Thickened[n]}
($\mathfrak{C}[\Delta^n]$) is the simplicial category defined as 
\begin{enumerate}
\item The objects are $0, 1, \dots ,n$ 
\item The mapping simplicial sets are given by \[
\operatorname{Map}_{\mathfrak{C}[\Delta^n]}(i,j) = 
\begin{cases}
\emptyset , &\text{if }j<i\\
N(P_{i,j}) ,&\text{if } i\leq j
\end{cases}
.\]
where $P_{i,j}$ is the partially ordered set \[
P_{i,j} = \left\{ I \subseteq \left\{ 0,1, \dots ,n \right\} | (i,j \in I) \quad\text{and}\quad \forall k  \in I, i\leq j\leq k \right\} 
.\]
\item For $i<j,k$, the composition \[
\operatorname{Map}_{\mathfrak{C}[\Delta^n]}(i,j) \times \operatorname{Map}_{\mathfrak{C}[\Delta^n]}(j,k) \to \operatorname{Map}_{\mathfrak{C}[\Delta^n]}(i,k) 
.\]
is induced by the map 
\begin{align*}
P_{i,j}\times P_{j,k} &\to P_{i,k} \\
(I_1, I_2) &\mapsto I_1 \cup I_2
\end{align*}
\end{enumerate}
\end{defn}

\begin{eg}
Recall that the simplicial nerve $N(\mathcal{C})$ of a simplicial category $\mathcal{C}$ is a simplicial set $N(\mathcal{C})\in \text{sSet}$, whose $n$-simplices are given by \[
N(\mathcal{C})_n = \operatorname{Hom}_{\text{Cat}_\Delta}(\mathfrak{C}[\Delta^n],\mathcal{C}) 
.\]
\end{eg}

\begin{prop}
\cite{lurie_higher_2009} Let $\mathcal{C}$ be a simplicial category such that $\operatorname{Map}_{\mathcal{C}}(X,Y)$ is a Kan complex for any $X,Y\in \mathcal{C}$. Then its simplicial nerve $N(\mathcal{C})$ is an $\infty$-category.
\end{prop}

\begin{notations}
Let $\mathcal{C}$ be a topological category. Then there is an associated simplicial category $\text{Sing}(\mathcal{C})$ whose objects are those of $C$ and whose homs are simplicial spaces given by \[
\operatorname{Map}_{\operatorname{Sing}(\mathcal{C})}(x,y) = \operatorname{Sing}(\operatorname{Map}_{\mathcal{C}}(x,y))
.\]
\end{notations}

\begin{rmk}
A standard result is that $\operatorname{Sing}: tCat \to Cat_\Delta $ is the adjoint of the geometric realization functor $|.|: Cat_\Delta \to tCat$, which sends a simplicial category $\mathcal{C}$ to the topological category $|\mathcal{C}|$ whose objects are the same as those of $\mathcal{C}$ and whose mapping spaces are given by \[
\operatorname{Map}_{|\mathcal{C}|}(x,y) = |\operatorname{Map}_{\mathcal{C}}(x,y)|
.\]
\end{rmk}

\begin{prop} \label{prop:tnerveinfcat}
The topological nerve $N(\mathcal{C})$ of any topological category $\mathcal{C}$ is defined as the simplicial nerve of $\text{Sing}(\mathcal{C})$. It is an $\infty$-category.
\end{prop}

\begin{rmk}
The First few levels contain the following data:
\end{rmk}

\begin{rmk}
Note that if $\mathcal{C}$ were an ordinary category, one may define a simplicial category $\mathcal{C}_s$ and a topological category $\mathcal{C}_t$. Their simplicial nerve and topological nerve agrees with the ordinary nerve.
\end{rmk}

\begin{defn} \label{defn:infoperad}
($\infty$-operad) \cite{LurieHigherAlgebra} An $\infty$-operad is a functor $p: \mathcal{O}^{\otimes } \to N(\mathcal{F}in_*)$ between $\infty$-categories satisfying
\begin{enumerate}
\item for every inert morphism $f: \langle m \rangle \to \langle n \rangle $ in $N(\mathcal{F}in_*)$ and every object $C \in \mathcal{O}^\otimes _{\langle m \rangle } = p ^{-1}(\langle m \rangle )$, there exists a $p$-coCartesian morphism $\overline{f}: C \to C'$ in $\mathcal{O}^\otimes$ lifting $f$. In particular, $f$ induces a functor $f_! \mathcal{O}^\otimes _{\langle m \rangle } \to \mathcal{O}^\otimes _{\langle n \rangle } $.
\item Let $C \in \mathcal{O}^\otimes _{\langle m \rangle }, C' \in \mathcal{O}^\otimes _{\langle n \rangle } $, and $f: \langle m \rangle \to \langle n \rangle $ a morphism in $\mathcal{F}in_*$. Let $\operatorname{Map}^f_{\mathcal{O}^\otimes }(C,C')$ be the union of the connected components of $\operatorname{Map}_{\mathcal{O}^\otimes }(C,C')$ over $f$. Choose $p$-coCartesian morphisms $C' \to C'_i$ over the inert morphisms $\rho ^i: \langle n \rangle \to \langle 1 \rangle $ for $1\leq i\leq n$. Then the induced map \[
\operatorname{Map}_{\mathcal{O}^\otimes }^f(C, C') \to \prod_{1\leq i\leq n} \operatorname{Map}^{\rho ^i \circ f}_{\mathcal{O}^\otimes }(C, C'_i)
\]
is a homotopy equivalence.
\item For every collection of objects $C_1 , \dots , C_n \in \mathcal{O}^{\otimes }_\langle 1 \rangle $, there exists an object $C \in \mathcal{O}^\otimes _{\langle n \rangle } $ and a collection of $p$-coCartesian morphisms $C \to C_i$ covering $\rho ^i: \langle n \rangle \to \langle 1 \rangle $.
\end{enumerate}
We denote the fibre $\mathcal{O}^\otimes_{\langle 1 \rangle } = p ^{-1}(\langle 1 \rangle )$ as $\mathcal{O}$ and refer it as the ``underlying $\infty$-category.''
\end{defn}

\begin{rmk}
If condition $(1)$ and $(2)$ are satisfied, then condition $(3)$ is equivalent to
\begin{enumerate}
\setcounter{enumi}{2}
\item There is an equivalence of $\infty$-categories \[
\phi : \mathcal{O}^\otimes _{\langle n \rangle } \to \mathcal{O}
.\]
for all $n\geq 0$ determined by the functors $\rho ^i_!: \mathcal{O}^\otimes _{\langle n \rangle } \to \mathcal{O}$.
\end{enumerate}
\end{rmk}

\begin{defn} \label{defn:Algebra Object}
(Algebra Object) Let $\mathcal{O}^\otimes , \mathcal{O}'^\otimes $ be $\infty$-operads. An $\infty$-operad map from $\mathcal{O}^\otimes $ to $\mathcal{O}'^\otimes $ is a map of simplicial sets $f: \mathcal{O}^\otimes \to \mathcal{O}'^\otimes $such that 
\begin{enumerate}
\item The following digram commutes
\[\begin{tikzcd}
{\mathcal{O}^\otimes} && {\mathcal{O}'^\otimes} \\
& {N(\mathcal{F}in_*)}
\arrow["f", from=1-1, to=1-3]
\arrow[from=1-1, to=2-2]
\arrow[from=1-3, to=2-2]
\end{tikzcd}\]
\item $f$ carries inert morphisms to inert morphisms.
\end{enumerate}
Denote by $\operatorname{Alg}_{\mathcal{O}}(\mathcal{O}'^\otimes )$ the full subcategory of $\operatorname{Fun}(\mathcal{O}^\otimes , \mathcal{O}'^\otimes )$ spanned by $\infty$-operad maps.
\end{defn}

\textbf{Acknowledgements} The author would like to thank Ian Agol, Richard Bamler, Ansuman Bardalai, Jacob Erlikman, Michael Hutchings, Theo Johnson-Freyd, David Nadler, Qiuyu Ren, Luuk Stehouwer, and Jan Steinebrunner for the helpful discussions.

\printbibliography

\end{document}